\documentclass[10]{article}
\usepackage{amsfonts,amssymb,amsmath,amsthm,cite}
\usepackage{graphicx}

\usepackage[applemac]{inputenc}
\usepackage{amsmath,amssymb,amsthm, hyperref, euscript}
\usepackage[matrix,arrow,curve]{xy}
\usepackage{graphicx}
\usepackage{tabularx}
\usepackage{float}
\usepackage{hyperref}
\usepackage{tikz}
\usepackage{slashed}
\usepackage{mathrsfs}
\usepackage{multirow}

\usetikzlibrary{matrix}

\usepackage[T1]{fontenc}
\usepackage{amsfonts,cite}
\usepackage{graphicx}

\usepackage[applemac]{inputenc}

\usepackage[sc]{mathpazo}
\DeclareFontFamily{T1}{pzc}{}
\DeclareFontShape{T1}{pzc}{m}{it}{1.8 <-> pzcmi8t}{}
\DeclareMathAlphabet{\mathpzc}{T1}{pzc}{m}{it}
\theoremstyle{plain}
\newtheorem{prop}{Proposition}[section]

\newtheorem{lem}[prop]{Lemma}%[section]
\newtheorem{cor}[prop]{Corollary}%[section]
\newtheorem{thm}[prop]{Theorem}%[section]
\newtheorem{theorem}[prop]{Theorem}
\newtheorem{lemma}[prop]{Lemma}

\newtheorem{corollary}[prop]{Corollary}

\theoremstyle{definition}
\newtheorem{defn}[prop]{Definition}%[section]
\newtheorem{empt}[prop]{}%[section]
\newtheorem{rem}[prop]{Remark}%[section]

\theoremstyle{definition}
\newtheorem{definition}[prop]{Definition}
\newtheorem{example}[prop]{Example}

\newtheorem{axiom}[prop]{Axiom}

\newtheorem{remark}[prop]{Remark}

\numberwithin{equation}{section}

\DeclareMathOperator{\Dom}{Dom}              %% domain of an operator
\newcommand{\vertiii}[1]{{\left\vert\kern-0.25ex\left\vert\kern-0.25ex\left\vert #1
    \right\vert\kern-0.25ex\right\vert\kern-0.25ex\right\vert}}
\newcommand{\Ga}{\Gamma}                     %% short for  \Gamma
\newcommand{\Coo}{C^\infty}                  %% smooth functions
\usepackage[sc]{mathpazo}
\newbox\ncintdbox \newbox\ncinttbox %% noncommutative integral symbols
\setbox0=\hbox{$-$} \setbox2=\hbox{$\displaystyle\int$}
\setbox\ncintdbox=\hbox{\rlap{\hbox
    to \wd2{\hskip-.125em \box2\relax\hfil}}\box0\kern.1em}
\setbox0=\hbox{$\vcenter{\hrule width 4pt}$}
\setbox2=\hbox{$\textstyle\int$} \setbox\ncinttbox=\hbox{\rlap{\hbox
    to \wd2{\hskip-.175em \box2\relax\hfil}}\box0\kern.1em}

\newcommand{\Id}{\mathrm{Id}}                %% identity map
\newcommand{\A}{\mathcal{A}}                 %% an algebra
\renewcommand{\a}{\alpha}                    %% short for  \alphapha
\newcommand{\E}{\mathcal{E}}                 %% space of distributions
\newcommand{\C}{\mathbb{C}}                  %% complex numbers
\newcommand{\G}{\mathcal{G}}                 %% Moyal L^2-filtration
\renewcommand{\H}{\mathcal{H}}               %% Hilbert space
\newcommand{\half}{\tfrac{1}{2}}             %% small fraction  1/2
\newcommand{\hookto}{\hookrightarrow}        %% abbreviation
\newcommand{\La}{\Lambda}                    %% short for \Lambda
\newcommand{\la}{\lambda}                    %% short for \lambda
\newcommand{\M}{\mathcal{M}}                 %% Moyal multplr algebra

\newcommand{\N}{\mathbb{N}}                  %% nonnegative integers
\newcommand{\om}{\omega}                     %% short for \omega
\newcommand{\eps}{\varepsilon}                    %% tensor product
\newcommand{\Q}{\mathbb{Q}}                  %% rational numbers
\newcommand{\R}{\mathbb{R}}                  %% real numbers
\newcommand{\set}[1]{\{\,#1\,\}}             %% set notation
\renewcommand{\SS}{\mathcal{S}}              %% Schwartz space
\DeclareMathOperator{\supp}{\mathfrak{supp}}            %% support
\newcommand{\T}{\mathbb{T}}                  %% circle as a group
\renewcommand{\th}{\theta}                   %% short for \theta
\newcommand{\Z}{\mathbb{Z}}                  %% integers
\renewcommand{\.}{\cdot}                     %% anonymous variable
\newcommand{\sA}{\mathcal{A}} 
\newcommand{\sH}{\mathcal{H}}       %%
\newcommand{\sS}{\mathcal{S}}       %%

\newcommand{\Om}{\Omega}       %%

\newcommand{\al}{\alpha}          %% short for  \alpha
\newcommand{\bt}{\beta}           %% short for  \beta
\newcommand{\ga}{\gamma}          %% short for  \gamma
\newcommand{\Th}{\Theta}          %% short for  \Theta
\renewcommand{\th}{\theta}        %% short for  \theta
\def\<#1|#2>{\langle#1\stroke#2\rangle} %% \braket (Dirac notation)
\def\?#1|#2?{\{#1\stroke#2\}}        %% left-linear pairing

\def\<#1,#2>{\langle#1,#2\rangle}            %% bilinear pairing
\def\ee_#1{e_{{\scriptscriptstyle#1}}}       %% basis projector
\def\wick:#1:{\mathopen:#1\mathclose:}       %% Wick-ordered operator

\newbox\ncintdbox \newbox\ncinttbox %% noncommutative integral symbols
\setbox0=\hbox{$-$}
\setbox2=\hbox{$\displaystyle\int$}
\setbox\ncintdbox=\hbox{\rlap{\hbox
           to \wd2{\box2\relax\hfil}}\box0\kern.1em}
\setbox0=\hbox{$\vcenter{\hrule width 4pt}$}
\setbox2=\hbox{$\textstyle\int$}
\setbox\ncinttbox=\hbox{\rlap{\hbox
           to \wd2{\hskip-.05em\box2\relax\hfil}}\box0\kern.1em}

\newcommand{\stroke}{\mathbin|}   %% (for `\bbraket' and such)
\newcommand{\End}{\mathrm{End}}       %%
\newcommand{\Hom}{\mathrm{Hom}}       %%
\newcommand{\be}{\begin{equation}}
\renewcommand{\ee}{\end{equation}}
\newcommand{\bea}{\begin{eqnarray}}
\newcommand{\eea}{\end{eqnarray}}
\newcommand{\bean}{\begin{eqnarray*}}
	\newcommand{\eean}{\end{eqnarray*}}
\newcommand{\brray}{\begin{array}}
	\newcommand{\erray}{\end{array}}

\title{Coverings of Spectral Triples}
\begin{document}
\maketitle  \setlength{\parindent}{0pt}
\begin{center}
\author{
{\textbf{Petr R. Ivankov*}\\
e-mail: * monster.ivankov@gmail.com \\
}
}
\end{center}

\vspace{1 in}

%\begin{abstract}
\noindent

\paragraph{}	
It is well-known that any covering space of a Riemannian manifold has the natural structure of a Riemannian manifold. This article contains a noncommutative generalization of this fact. Since any  Riemannian manifold with a Spin-structure defines a spectral triple, the spectral triple can be regarded as a noncommutative Spin-manifold. Similarly there is an algebraic construction which is a noncommutative generalization of topological covering. This article contains a construction of spectral triple on the "noncommutative covering space".

%\end{abstract}
\tableofcontents

\section{Motivation. Preliminaries}
\subsection{Prototype. Coverings of Riemannian manifolds}
 \paragraph*{} This article proves a noncommutative generalization of the following proposition. 

\begin{prop}\label{comm_cov_mani}(Proposition 5.9 \cite{koba_nomi:fgd})
	\begin{enumerate}
		\item Given a connected manifold $M$ there is a unique (unique up to isomorphism) universal covering manifold, which will be denoted by $\widetilde{M}$.
		\item The universal covering manifold $\widetilde{M}$ is a principal fibre bundle over $M$ with group $\pi_1(M)$ and projection $p: \widetilde{M} \to M$, where $\pi_1(M)$ is the first homotopy group of $M$.
		\item The isomorphism classes of covering spaces over $M$ are in 1:1 correspondence with the conjugate classes of subgroups of $\pi_1(M)$. The correspondence is given as follows. To each subgroup $H$ of $\pi_1(M)$, we associate $E=\widetilde{M}/H$. Then the covering manifold $E$ corresponding to $H$ is a fibre bundle over $M$ with fibre $\pi_1(M)/H$ associated with the principal bundle  $\widetilde{M}(M, \pi_1(M))$. If $H$ is a normal subgroup of $\pi_1(M)$, $E=\widetilde{M}/H$ is a principal fibre bundle with group $\pi_1(M)/H$ and is called a regular covering manifold of $M$.
	\end{enumerate}
\end{prop}

\begin{empt}
	If $\widetilde{M}$ is a covering space of Riemannian manifold $M$ then it is possible to give $\widetilde{M}$ a Riemannian structure such that $\pi: \widetilde{M} \to M$ is a local isometry (this metric is called the {\it covering metric}). cf. \cite{do_carmo:rg} for details.
\end{empt}

   Gelfand-Na\u{i}mark theorem \cite{arveson:c_alg_invt} states the correspondence between  locally compact Hausdorff topological spaces and commutative $C^*$-algebras.

\begin{thm}\label{gelfand-naimark}\cite{arveson:c_alg_invt} (Gelfand-Na\u{i}mark). 
Let $A$ be a commutative $C^*$-algebra and let $\mathcal{X}$ be the spectrum of A. There is the natural $*$-isomorphism $\gamma:A \to C_0(\mathcal{X})$.
\end{thm}

\paragraph{}So any (noncommutative) $C^*$-algebra may be regarded as a generalized (noncommutative)  locally compact Hausdorff topological space. Articles \cite{ivankov:qnc,pavlov_troisky:cov} contain noncommutative analogs of coverings. The spectral triple \cite{hajac:toknotes,varilly:noncom} can be regarded as a noncommutative generalization of Riemannian manifold. Having analogs of both coverings and Riemannian manifolds one can proof  a noncommutative generalization of the Proposition \ref{comm_cov_mani}.

  \newpage
Following table contains a list of special symbols.
\newline
\begin{tabular}{|c|c|}
	\hline
	Symbol & Meaning\\
	\hline
	& \\
%	$A^+$  & Unitization of $C^*-$ algebra $A$\\
%	$\widehat{A}$ & Spectrum of  $C^*$ - algebra $A$  with the hull-kernel topology \\
%	& (or Jacobson topology)\\
	$A^G$  & Algebra of $G$-invariants, i.e. $A^G = \{a\in A \ | \ ga=a, \forall g\in G\}$\\
	$\mathrm{Aut}(A)$ & Group * - automorphisms of $C^*$  algebra $A$\\
	$B(\H)$ & Algebra of bounded operators on a Hilbert space $\H$\\
	%$B_{\infty}=B_{\infty}(\{z\in \mathbb{C} \ | \ |z|=1\})$  & Algebra of Borel measured functions on the $\{z\in \mathbb{C} \ | \ |z|=1\}$ set. \\
	$\mathbb{C}$ (resp. $\mathbb{R}$)  & Field of complex (resp. real) numbers \\
	%$\mathbb{C}^*$ & $\{z \in \mathbb{C} \ | \ |z| = 1\}$ \\
	$C(\mathcal{X})$ & $C^*$ - algebra of continuous complex valued \\
	& functions on a space $\mathcal{X}$\\
	$C_0(\mathcal{X})$ & $C^*$ - algebra of continuous complex valued functions on a topological\\
	&    space $\mathcal{X}$ equal to $0$ at infinity\\
	$C_c(\mathcal{X})$ & Algebra of continuous complex valued functions on a \\
	&  topological  space $\mathcal{X}$ with compact support\\
%	$C_b(\mathcal{X})$ & $C^*$ - algebra of bounded  continuous complex valued \\
%	& functions on a topological space $\mathcal{X}$ \\
	%$G_{tors} \subset G$  & The torsion subgroup of an abelian group\\
	$G(\widetilde{\mathcal{X}}~|~\mathcal{X})$ & Group of covering transformations of covering projection  $\widetilde{\mathcal{X}} \to \mathcal{X}$ \cite{spanier:at}  \\
	
	$\H$ &Hilbert space \\
	$\mathcal{K}= \mathcal{K}(\H)$ & $C^*$ - algebra of compact operators \\
%	$\mathcal{K}(X_A)$ & $C^*$ - algebra of compact operators of a Hilbert $A$ module $X_A$ \\
%	$K_i(A)$ ($i = 0, 1$) & $K$ groups of $C^*$-algebra $A$\\
	%$I = [0, 1] \subset \mathbb{R}$ & Closed unit  interval\\
	$K(A)$ & Pedersen ideal of $C^*$-algebra $A$\\
	%$\mathcal{K}(H)$ or $\mathcal{K}$ & Algebra of compact operators on Hilbert space $H$\\
	$\varinjlim$ & Direct limit \\
	$\varprojlim$ & Inverse limit \\
	$M(A)$  & A multiplier algebra of $C^*$-algebra $A$\\
	$\mathbb{M}_n(A)$  & The $n \times n$ matrix algebra over $C^*-$ algebra $A$\\
	$\mathbb{N}$  & A set of positive integer numbers\\
	$\mathbb{N}^0$  & A set of nonnegative integer numbers\\
		$\supp~a$  & The support of the element $a \in C_0\left(\mathcal{X} \right)$ \\
%	$\overline{G/G'}\subset G$  & A set of representatives of a quotient group $G/G'$\\
	%$S^n$ & The $n$-dimensional sphere\\
	%$S^n$ & The $n$-dimensional sphere\\
%	$SU(n)$ & Special unitary group \\
	
	%$\mathscr{P}(\mathcal{X})$  & Fundamental groupoid of a topological space $\mathcal{X}$\\

%	$\mathbb{Q}$  & Field of rational numbers \\
	% $\mathrm{sp}(a)$ & Spectrum of element of $C^*$-algebra $a\in A$  \\
	%  $\mathrm{supp}(f)$ & Support of $f\in C_0(\mathcal{X})$, $\mathrm{supp}(f) = \left\{x \in \mathcal{X} \ | \ f(x)\neq 0 \right\}$   \\
%	$U(H) \subset \mathcal{B}(H) $ & Group of unitary operators on Hilbert space $H$\\
%	$U(A) \subset A $ & Group of unitary operators of algebra $A$\\
%	$U(n) \subset GL(n, \mathbb{C}) $ & Unitary subgroup of general linear group\\
	
	$\mathbb{Z}$ & Ring of integers \\
	
	$\mathbb{Z}_n$ & Ring of integers modulo $n$ \\
	$\overline{k} \in \mathbb{Z}_n$ & An element in $\mathbb{Z}_n$ represented by $k \in \mathbb{Z}$  \\
	%$\Omega$ &  Natural contravariant functor from category  of commutative \\ & $C^*$ - algebras, to category of Hausdorff spaces\\
	$X \backslash A$ & Difference of sets  $X \backslash A= \{x \in X \ | \ x\notin A\}$\\
	$|X|$ & Cardinal number of the finite set\\ $f|_{A'}$& Restriction of a map $f: A\to B$ to $A'\subset A$, i.e. $f|_{A'}: A' \to B$\\ 
	\hline
\end{tabular}

\break

%%%Henceforth  $\left\{x_{\iota}\right\}_{\iota \in I}$ means a set indexed by finite or countable set $I$ of indexes.

\subsection{Topology}

\subsubsection{Locally compact spaces}
\begin{defn}\cite{munkres:topology}
	If $\phi: \mathcal X \to \mathbb{C}$  is continuous then the \textit{support} of $\phi$ is defined to be the closure of the set $\phi^{-1}\left(\mathbb{\C}\backslash \{0\}\right)$ Thus if $x$ lies outside the support, there is some neighborhood of $x$ on which $\phi$ vanishes. Denote by $\supp \phi$ the support of $\phi$.
\end{defn}

There are two equivalent definitions of $C_0\left(\mathcal{X}\right)$ and both of them are used in this article.
\begin{defn}\label{c_c_def_1}
	An algebra $C_0\left(\mathcal{X}\right)$ is the norm closure of the algebra $C_c\left(\mathcal{X}\right)$ of compactly supported continuous functions.
\end{defn}
\begin{defn}\label{c_c_def_2}
	A $C^*$-algebra $C_0\left(\mathcal{X}\right)$ is given by the following equation
	\begin{equation*}
	C_0\left(\mathcal{X}\right) = \left\{\varphi \in C_b\left(\mathcal{X}\right) \ | \ \forall \varepsilon > 0 \ \ \exists K \subset \mathcal{X} \ ( K \text{ is compact}) \ \& \ \forall x \in \mathcal X \backslash K \ \left|\varphi\left(x\right)\right| < \varepsilon  \right\},
	\end{equation*}
	i.e.
	\begin{equation*}
	\left\|\varphi|_{\mathcal X \backslash K}\right\| < \varepsilon.
	\end{equation*}
\end{defn}

\begin{defn}\cite{munkres:topology}
	Let $\left\{\mathcal U_\alpha\in \mathcal X\right\}_{\alpha \in J}$ be an indexed open covering of $\mathcal{X}$. An indexed family of functions 
	\begin{equation*}
	\phi_\alpha : \mathcal X \to \left[0,1\right]
	\end{equation*}
	is said to be a {\it partition of unity }, dominated by $\left\{\mathcal{U}_\alpha \right\}_{\alpha \in J}$, if:
	\begin{enumerate}
		\item $\phi_\alpha\left(\mathcal X \backslash \mathcal U_\alpha\right)= \{0\}$
		\item The family $\left\{\supp\left(\phi_\alpha\right) = \mathfrak{cl}\left(\left\{x \in \mathcal X \ | \ \phi_\alpha > 0 \right\}\right)\right\}$ is locally finite.
		\item $\sum_{\alpha \in J}\phi_\alpha\left(x\right)=1$ for any $x \in \mathcal X$.
	\end{enumerate}
\end{defn}

\begin{thm}\cite{munkres:topology}
	Let $\mathcal X$ be a paracompact Hausdorff space; let $\left\{\mathcal U_\alpha\in \mathcal X\right\}_{\alpha \in J}$ be an indexed open covering of $\mathcal{X}$. Then there exists a partition of unity, dominated by $\left\{\mathcal{U}_\alpha \right\}$.  
\end{thm}

\subsubsection{Coverings}
\begin{defn}\label{comm_cov_pr_defn}\cite{spanier:at}
	Let $\widetilde{\pi}: \widetilde{\mathcal{X}} \to \mathcal{X}$ be a continuous map. An open subset $\mathcal{U} \subset \mathcal{X}$ is said to be {\it evenly covered } by $\widetilde{\pi}$ if $\widetilde{\pi}^{-1}(\mathcal U)$ is the disjoint union of open subsets of $\widetilde{\mathcal{X}}$ each of which is mapped homeomorphically onto $\mathcal{U}$ by $\widetilde{\pi}$. A continuous map $\widetilde{\pi}: \widetilde{\mathcal{X}} \to \mathcal{X}$ is called a {\it covering projection} if each point $x \in \mathcal{X}$ has an open neighborhood evenly covered by $\widetilde{\pi}$. $\widetilde{\mathcal{X}}$ is called the {
		\it covering space} and $\mathcal{X}$ the {\it base space} of the covering.
\end{defn}
%\begin{defn}\label{one_one_cov_defn}
%	Let $\widetilde{\pi}: \widetilde{\mathcal{X}} \to \mathcal{X}$ be a covering. A connected open subset $\widetilde{\mathcal{U}} \subset \widetilde{\mathcal{X}}$ is said to be a {\it one-to-one subset} with respect to $\widetilde{\pi}$ if the restriction $\widetilde{\pi}_{\widetilde{\mathcal{U}}}: \widetilde{\mathcal{U}} \to \widetilde{\pi}\left(\widetilde{\mathcal{U}}\right)$ is a homeomorphism. The family $\left\{\widetilde{\mathcal{U}}_\iota\right\}_{\iota \in I}$ of one-to-one subsets with respect to $\widetilde{\pi}$ such that $\widetilde{\mathcal{X}} = \bigcap_{\iota \in I}\widetilde{\mathcal{U}}_\iota$ is said to be a {\it one-to-one covering}  with respect to $\widetilde{\pi}$.
%\end{defn}
\begin{defn}\cite{spanier:at}
	A fibration $p: \mathcal{\widetilde{X}} \to \mathcal{X}$ with unique path lifting is said to be  {\it regular} if, given any closed path $\omega$ in $\mathcal{X}$, either every lifting of $\omega$ is closed or none is closed.
\end{defn}
\begin{defn}\cite{spanier:at}
	A topological space $\mathcal X$ is said to be \textit{locally path-connected} if the path components of open sets are open.
\end{defn}
\paragraph{} Denote by $\pi_1$ the functor of fundamental group \cite{spanier:at}.
%\begin{lem}\label{lem_3_spanier_lem}\cite{spanier:at}
%Let $p: \widetilde{\mathcal X} \to \mathcal X$ be a fibration with unique path lifting. If $\om$ and $\om'$ are paths such that $\om\left(0\right)=\om'\left(0\right)$ and $p \circ \om$ is homotopic to $p \circ \om'$ then $\om$ is homotopic to $\om'$.  
%\end{lem}
%\paragraph{} It follows from Lemma \ref{lem_3_spanier_lem} that if $p: \widetilde{\mathcal X} \to \mathcal X$ be a fibration with unique path lifting then for any two objects $\widetilde{x}_1$ and $\widetilde{x}_2$ in the fundamental groupoid of $\widetilde{X}$
\begin{thm}\label{locally_path_lem}\cite{spanier:at}
	Let $p: \widetilde{\mathcal X} \to \mathcal X$ be a fibration with unique path lifting and assume that a nonempty $\widetilde{\mathcal X}$ is a locally path-connected space. Then $p$ is regular if and only if for some $\widetilde{x}_0 \in  \widetilde{\mathcal X}$, $\pi_1\left(p\right)\pi_1\left(\widetilde{\mathcal X}, \widetilde{x}_0\right)$ is a normal subgroup of $\pi_1\left(\mathcal X, p\left(\widetilde{x}_0\right)\right)$.
\end{thm}
\begin{defn}\label{cov_proj_cov_grp}\cite{spanier:at}
	Let $p: \mathcal{\widetilde{X}} \to \mathcal{X}$ be a covering.  A self-equivalence is a homeomorphism $f:\mathcal{\widetilde{X}}\to\mathcal{\widetilde{X}}$ such that $p \circ f = p$. This group of such homeomorphisms is said to be the {\it group of covering transformations} of $p$ or the {\it covering group}. Denote by $G\left( \mathcal{\widetilde{X}}~|~\mathcal{X}\right)$ this group.
\end{defn}

%\begin{thm}\label{g_cov_thm}\cite{spanier:at}
%Let $G$ be a properly discontinuous group of homeomorphisms of a space $\mathcal X$. Then the projection $\mathcal X \to \mathcal X/G$ of $\mathcal X$ to the orbit space $\mathcal X/G$ is a covering. If $\mathcal X$ is connected then this covering is regular and $G$ is its group of covering transformations.
%\end{thm}
%\begin{empt}\label{fin_prop_disc_empt}
%	\cite{spanier:at} A finite group action without fixed points on a Hausdorff space is properly discontinuous.
%\end{empt}
\begin{prop}\cite{spanier:at}
	If $p: \mathcal{\widetilde{X}} \to \mathcal{X}$ is a regular covering and $\mathcal{\widetilde{X}}$ is connected and locally path connected, then $\mathcal{X}$ is homeomorphic to space of orbits of $G\left( \mathcal{\widetilde{X}}~|~\mathcal{X}\right)$, i.e. $\mathcal{X} \approx \mathcal{\widetilde{X}}/G\left( \mathcal{\widetilde{X}}~|~\mathcal{X}\right) $. So $p$ is a principal bundle.
\end{prop}
\begin{cor}\label{top_cov_from_pi1_cor}\cite{spanier:at}
	Let $p: \widetilde{\mathcal X} \to \mathcal X$ be a fibration with a unique path lifting. If $ \widetilde{\mathcal X}$ is connected and locally path-connected and $\widetilde{x}_0 \in \widetilde{\mathcal X}$ then $p$ is regular if and only if $G\left(\widetilde{\mathcal X}~|~{\mathcal X} \right)$ transitively acts on each fiber of $p$, in which case 
	$$
	\psi: G\left(\widetilde{\mathcal X}~|~{\mathcal X} \right) \approx \pi_1\left(\mathcal X, p\left( \widetilde{x}_0\right)  \right) / \pi_1\left( p\right)\pi_1\left(\widetilde{\mathcal X}, \widetilde{x}_0 \right).  
	$$
\end{cor}
\begin{rem}
	Above results are copied from \cite{spanier:at}. Below  the \textit{covering projection} word is replaced with \textit{covering}.

\subsubsection{Vector bundles}\label{top_vb_sub_sub}
\paragraph{}
We refer to \cite{karoubi:k} for a notion of {\it (locally trivial) vector bundle} with base $\mathcal X$ and an {\it inverse image} of a vector bundle. 
For any topological space $\mathcal X$ there is a category $\mathrm{Vect}(\mathcal{X})$ of vector bundles with base $\mathcal X$. Denote by $\Ga\left({\mathcal X}, {E}\right)$ the  $C_b\left(\mathcal{X} \right)$-module of continuous sections of $E$. $\Ga\left({\mathcal X}, {E}\right)$ can be regarded as both left and right $C_b\left(\mathcal{X} \right)$-module.
\begin{empt}\label{vb_inv_img_funct}
	If $f: \mathcal X \to \mathcal Y$ is a continuous map then there is an {\it inverse image functor}  $f^*:\mathrm{Vect}(\mathcal{Y})\to\mathrm{Vect}(\mathcal{X})$ (cf. \cite{karoubi:k}). 
\end{empt}
\begin{defn}\label{bundle_supp_defn}
	For any $\xi \in \Ga(\mathcal X, E)$ the closure of the set
	$$
	\mathcal X \backslash \bigcup_{\substack{a \in C_0\left( \mathcal X \right) \\ a \xi = 0}} \supp~a \subset \mathcal X
	$$
	is said to be the \textit{support} of $\xi$. The support is denoted by $\supp~\xi$. For any subset $\mathcal U \subset \mathcal X$ denote by
	\be\nonumber
	\Ga\left(\mathcal U, E|_{\mathcal U}\right)= \left\{\xi \in \Ga\left({\mathcal X}, {E}\right) | \supp~\xi\subset \mathcal U \right\}
	\ee
\end{defn}
\begin{rem}
	For any $a \in C_0\left( \mathcal X\right)$ and $\xi \in   \Ga(\mathcal X, E)$ following condition holds
	\be\label{supp_inc_eqn}
	\supp~a\xi \subset \supp~a.
	\ee
\end{rem}
\begin{empt}\label{lift_constr}
	Let $S$ be a vector bundle over $\mathcal X$.
Let $\pi : \widetilde{\mathcal X} \to \mathcal X$ be a covering, and let  $\widetilde{S}= \pi^*S$ be the inverse image. If $\widetilde{\mathcal U}\subset \widetilde{\mathcal X}$ is an open set which is mapped homeomorphically onto $\pi \left( \widetilde{\mathcal U}\right) = {\mathcal U}$, then there are natural $\C$-isomorphisms
	\be\label{lift_desc_eqn}
	\begin{split}
	\mathfrak{desc}_\pi : C_0\left( \widetilde{\mathcal U}\right) \xrightarrow{\approx} C_0\left( {\mathcal U}\right),\\
	\mathfrak{desc}_\pi : \Ga\left( \widetilde{\mathcal U}, \widetilde{S}_{ \widetilde{\mathcal U}} \right) \xrightarrow{\approx} \Ga\left( {\mathcal U}, {S}_{ {\mathcal U}} \right),\\
	\mathfrak{lift}_{\widetilde{\mathcal U}}  :C_0\left({\mathcal U}\right) \xrightarrow{\approx} C_0\left( \widetilde{\mathcal U}\right),\\
	\mathfrak{lift}_{\widetilde{\mathcal U}} : \Ga\left( {\mathcal U}, {S}_{ \widetilde{\mathcal U}} \right) \xrightarrow{\approx}\Ga\left( \widetilde{\mathcal U}, \widetilde{S}_{ \widetilde{\mathcal U}} \right).
	\end{split}
	\ee
	We use identical notions $\mathfrak{desc}_\pi$, $\mathfrak{lift}_{\widetilde{\mathcal U}}$ for different maps, however the meaning of this maps depends on arguments, as well as the meaning of C++ functions, so the identity will not occur contradictions. Moreover we will often write  $\mathfrak{desc}$ instead $\mathfrak{desc}_\pi$
\end{empt}
\begin{defn}\label{lift_defn}
	Let us consider the situation \ref{lift_constr}.
	If $\xi \in \Ga\left( {\mathcal U}, {S}_{ \widetilde{\mathcal U}} \right)$, $\widetilde{\xi} \in \Ga\left( \widetilde{\mathcal U}, \widetilde{S}_{ \widetilde{\mathcal U}} \right)$, $a \in C_0\left({\mathcal U} \right)$,  $\widetilde{a} \in C_0\left( \widetilde{\mathcal U}\right)$ then we say
	\begin{itemize}
		\item $	\mathfrak{lift}_{\widetilde{\mathcal U}}\left(\xi \right)$ is $\widetilde{\mathcal{U}}$-\textit{lift} of $\xi$, 
			\item $	\mathfrak{lift}_{\widetilde{\mathcal U}}\left(a \right)$ is $\widetilde{\mathcal{U}}$-\textit{lift} of $a$, 
		\item $\mathfrak{desc}_\pi\left( \widetilde{\xi}\right) $ is the $\pi$-\textit{descend} or simply \textit{descend}  of $\widetilde{\xi}$,
	\item $\mathfrak{desc}_\pi\left( \widetilde{a}\right) $ is the $\pi$-\textit{descend} or simply \textit{descend} of $\widetilde{a}$.
	\end{itemize}
\end{defn}
	If $a \in C_0\left(\mathcal{X} \right)$ and $\xi \in \Ga(\mathcal X, E)$ are such that   $\supp a, \supp \xi \in \mathcal{U}$ then
	
	\be\label{lift_product_eqn}
	\mathfrak{lift}_{\widetilde{\mathcal{U}}}\left(a \xi\right)= \mathfrak{lift}_{\widetilde{\mathcal{U}}}\left(a \right)\mathfrak{lift}_{\widetilde{\mathcal{U}}}\left(\xi \right).
	\ee

	If $\widetilde{\xi} \in   \Ga\left( \widetilde{\mathcal X},\widetilde{ E}\right)$ then the following implication holds
	\be\label{supp_lift_desc_eqn}
	\supp~ \widetilde{\xi} \subset \widetilde{\mathcal U} \Rightarrow \widetilde{\xi}=\mathfrak{lift}_{\widetilde{\mathcal{U}}}\left( \mathfrak{desc}_\pi \left(\widetilde{\xi} \right) \right).
	\ee
\begin{definition}\label{local_op_defn}
	A $\mathbb{C}$-(anti)linear map $\varphi$ from $\Ga(\mathcal X, E)$ (resp. dense $\mathbb{C}$-subspace of $X \subset \Ga(\mathcal X, E)$) to $\Ga(\mathcal X, E)$ is said to be \textit{local} if  $\supp ~\varphi \xi \subset \supp \xi$ for any $\xi \in \Ga(\mathcal X, E)$ (resp. $\xi \in X$).
\end{definition}
\begin{empt}\label{top_bundle_defn}
		Let $\pi:\widetilde{\mathcal X} \to \mathcal{X}$ be a covering, and let $E \in \mathrm{Vect}\left(\mathcal{X}\right)$, $\widetilde{E}=\pi^*E$ be an inverse image. If $\widetilde{\mathcal U} \subset \widetilde{\mathcal{X}} $ is an open subset such that a restriction  $\pi |_{\mathcal U}:\widetilde{ \mathcal U} \to \pi (\widetilde{\mathcal U})= \mathcal U$ is a homeomorphism. Then there are *-isomorphism $C\left( \widetilde{\mathcal U} \right) \xrightarrow{\approx} C\left( \mathcal{U}\right) $ and isomorphism $\Ga\left(\widetilde{\mathcal U}, \widetilde{E}|_{\widetilde{\mathcal U}}\right) \approx \Ga\left(\mathcal U, E|_{\mathcal U}\right)$.
	If $\varphi$ is a local $\mathbb{C}$-(anti)linear map  from $\Ga(\mathcal X, E)$ (resp. dense $\mathbb{C}$-subspace of $X \subset \Ga(\mathcal X, E)$) to $\Ga(\mathcal X, E)$ then $\varphi$ naturally induces the local $\mathbb{C}$-(anti)linear map $\widetilde{\varphi}$ from $\Ga\left(\widetilde{\mathcal X}, \widetilde{E}\right)$ (resp. dense $\mathbb{C}$-subspace of $\widetilde{X} \subset \Ga\left(\widetilde{\mathcal X}, \widetilde{E}\right)$) to $\Ga\left(\widetilde{\mathcal X}, \widetilde{E}\right)$, such that there is the following commutative diagram.
	\newline
	\begin{tikzpicture}
	\matrix (m) [matrix of math nodes,row sep=3em,column sep=4em,minimum width=2em]
	{
		\widetilde{X}|_{\widetilde{\mathcal U}}= \widetilde{X} \bigcap \Ga\left(\widetilde{\mathcal U}, E|_{\widetilde{\mathcal U}}\right) & \Ga\left(\widetilde{\mathcal U}, \widetilde{E}|_{\widetilde{\mathcal U}}\right)   \\
		X|_{\mathcal U}= X \bigcap \Ga\left(\mathcal U, E|_{\mathcal U}\right) &  \Ga\left(\mathcal U, \widetilde{E}|_{\mathcal U}\right) \\};
	\path[-stealth]
	(m-1-1) edge node [above] {$\widetilde\varphi|_{\widetilde{\mathcal{U}}}$} (m-1-2)
	(m-2-1) edge node [above] {$\varphi|_{\mathcal{U}}$} (m-2-2)
	(m-1-1) edge node [left]  {$\approx$} (m-2-1)
	(m-1-2) edge node [right] {$\approx$} (m-2-2);
	\end{tikzpicture}
	
\end{empt}

\begin{defn}\label{inv_image_defn}
Let us consider the situation \ref{top_bundle_defn}.	The map $\widetilde\varphi$ is said to be an {\it inverse image} or $\pi$-\textit{lift} of $\varphi$.
\end{defn}

\begin{remark}
	If $G=G\left(\widetilde{\mathcal X}~|~{\mathcal X} \right)$ is the group of covering transformations then for any $g \in G$ and $\xi \in \Ga\left(\widetilde{\mathcal U}, E|_{\widetilde{\mathcal U}}\right)$ following condition holds
	\be\label{lift_g_eqn}
	g\widetilde{\xi} = \mathfrak{lift}_{g\widetilde{\mathcal{U}}}\left( \mathfrak{desc}_\pi \left(\widetilde{\xi} \right) \right).
	\ee
\end{remark}

%Following fact is an implication of definitions.
%\begin{fact}\label{inv_image_fact}
%	A map $\varphi^*$ is an inverse image of $\varphi$ if and only if the above diagram is commutative for any $\mathcal U$ such that $\pi |_{\mathcal U}: \mathcal U \to \pi ({\mathcal U})$ is a homeomorphism.
%\end{fact}

\begin{empt}\label{top_herm_bundle_constr}
	Let $\mathcal X$ be a topological space and $S$ the complex linear bundle on $\mathcal X$. Suppose that for any $x \in \mathcal{X}$ there is the scalar product $\left( \cdot, \cdot \right)_x: S_x \times S_x \to \C$ and there is a measure $\mu_{\mathcal{X}}$ on $\mathcal X$. If $\Ga\left( M, S\right) $ is the space of continuous sections
	of $S$ then we suppose that for any $\xi, \eta \in  \Ga\left( M, S\right)$ the map $\mathcal X \to \C$ given by $x \mapsto \left( \xi_x, \eta_x\right)_x$ is continuous. There is the scalar product $\left( \cdot, \cdot \right) :  \Ga\left( M, S\right) \times \Ga\left( M, S\right) \to \C$ given by
	$$
	\left( \xi, \eta\right)\stackrel{\text{def}}{=} \int_{\mathcal X}\left( \xi_x, \eta_x\right)d~\mu_{\mathcal{X}}
	$$
	Denote by $L^2\left( \mathcal X, S, \mu_{\mathcal X}\right) $ or $L^2\left( \mathcal X, S\right) $ the Hilbert norm completion of $\Ga\left( M, S\right) $. There is the natural representation
	\begin{equation}\label{comm_bundle_repr}
	C_0\left(\mathcal{X} \right) \to B\left( L^2\left( \mathcal X, S\right)\right). 
	\end{equation}
\end{empt}
\begin{defn}
	In the situation of \ref{top_herm_bundle_constr} we say that $S$ is \textit{Hermitian vector bundle}.
\end{defn}
%If $E \to \mathcal X$ is a Hermitian vector bundle then for any $x_0 \in \mathcal X$ there is a  neighborhood $\mathcal U$ and a trivialization $E|_{\mathcal U}\approx \mathcal U \times \C^n$ such that
%$$
%\left(\left(x, e_j \right), \left(x, e_k \right)  \right) = \delta_{jk}
%$$
%where
%$$
%\left(x, e_j \right), \left(x, e_k \right)  \in \mathcal U \times \C^n;
%$$
%$$
%e_j = \left(0,..., \left[\widetilde{D}race{1}_{j^{\text{th}}-\text{place}},...,0 \right);~~e_k = \left(0,..., \left[\widetilde{D}race{1}_{k^{\text{th}}-\text{place}},...,0 \right);
%$$
%\begin{empt}
%	In the above situation we say that $E|_{\mathcal U}\approx \mathcal U \times \C^n$ is \textit{Hermitian trivialization}.
%\end{empt}

\subsection{Inverse limits of coverings}\label{inf_to}

\paragraph{} This subsection is concerned with a topological construction of the inverse limit in the category of coverings. 

\end{rem}
\begin{defn}\label{top_sec_defn}
	The sequence of regular finite-fold coverings  	
	\begin{equation*}
	\mathcal{X} = \mathcal{X}_0 \xleftarrow{}... \xleftarrow{} \mathcal{X}_n \xleftarrow{} ... 
	\end{equation*}
	is said to be a \textit{(topological)  finite covering sequence} if following conditions hold:
	\begin{itemize}
		\item   The space $\mathcal{X}_n$ is a  second-countable \cite{munkres:topology} locally compact connected Hausdorff space for any $n \in \mathbb{N}^0$,
		\item If $k < l < m$ are any nonnegative integer numbers then there is the natural exact sequence
		$$
		\{e\}\to	G\left(\mathcal X_m~|~\mathcal X_l\right) \to 	G\left(\mathcal X_m~|~\mathcal X_k\right)\to 	G\left(\mathcal X_l~|~\mathcal X_k\right)\to \{e\}.
		$$ 
	\end{itemize} 
	For any finite covering sequence we will use a following notation
	\begin{equation*}
	\mathfrak{S} = \left\{\mathcal{X} = \mathcal{X}_0 \xleftarrow{}... \xleftarrow{} \mathcal{X}_n \xleftarrow{} ...\right\}= \left\{ \mathcal{X}_0 \xleftarrow{}... \xleftarrow{} \mathcal{X}_n \xleftarrow{} ...\right\},~~\mathfrak{S} \in \mathfrak{FinTop}.
	\end{equation*}
	
\end{defn}
\begin{defn}\label{top_cov_trans_defn} Let  $\left\{\mathcal{X} = \mathcal{X}_0 \xleftarrow{}... \xleftarrow{} \mathcal{X}_n \xleftarrow{} ...\right\} \in \mathfrak{FinTop}$, and let
	$\widehat{\mathcal{X}} = \varprojlim \mathcal{X}_n$ be the inverse limit  in the category of topological spaces and continuous maps (cf. \cite{spanier:at}). If $\widehat{\pi}_0: \widehat{\mathcal{X}} \to \mathcal{X}_0$ is the natural continuous map then homeomorphism $g$ of the space $\widehat{\mathcal{X}}$ is said to be a \textit{covering  transformation} if the following condition holds
	$$
	\widehat{\pi}_0 = \widehat{\pi}_0 \circ g.
	$$
	The group $\widehat{G}$ of covering homeomorphisms is said to be the \textit{group of  covering  transformations} of $\mathfrak S$. Denote by $G\left(\widehat{\mathcal{X}}~|~\mathcal X \right)\stackrel{\mathrm{def}}{=}\widehat{G}$. 
\end{defn}
\begin{defn}\label{top_coh_defn}
	Let $\mathfrak{S} = \left\{ \mathcal{X}_0 \xleftarrow{}... \xleftarrow{} \mathcal{X}_n \xleftarrow{} ...\right\}$ be 
	a finite covering sequence. The pair $\left(\mathcal{Y},\left\{\pi^{\mathcal Y}_n\right\}_{n \in \mathbb{N}} \right) $ of a (discrete) set $\mathcal{Y}$ with and  
	surjective  maps $\pi^{\mathcal Y}_n:\mathcal{Y} \to \mathcal X_n$ is said to be a \textit{coherent system} if for any $n \in \mathbb{N}^0$ a following diagram  
	\newline
	\begin{tikzpicture}
	\matrix (m) [matrix of math nodes,row sep=3em,column sep=4em,minimum width=2em]
	{
		& \mathcal{Y}  &  \\
		\mathcal{X}_n &  &  \mathcal{X}_{n-1} \\};
	\path[-stealth]
	(m-1-2) edge node [left] {$\pi^{\mathcal Y}_n~$} (m-2-1)
	(m-1-2) edge node [right] {$~\pi^{\mathcal Y}_{n-1}$} (m-2-3)
	(m-2-1) edge node [above] {$\pi_n$}  (m-2-3);
	
	\end{tikzpicture}
	\newline
	is commutative.	
\end{defn}

\begin{defn}\label{comm_top_constr_defn}
	Let $\mathfrak{S} = \left\{ \mathcal{X}_0 \xleftarrow{}... \xleftarrow{} \mathcal{X}_n \xleftarrow{} ...\right\}$ be 
	a topological finite covering sequence. A coherent system $\left(\mathcal{Y},\left\{\pi^{\mathcal Y}_n\right\} \right)$ is said to
	be a \textit{connected covering} of $\mathfrak{S}$ if $\mathcal Y$ is a connected topological space and $\pi^{\mathcal Y}_n$ is a regular covering  for any $n \in \mathbb{N}$.  We will use following notation $\left(\mathcal{Y},\left\{\pi^{\mathcal Y}_n\right\} \right)\downarrow \mathfrak{S}$ or simply $\mathcal{Y} \downarrow \mathfrak{S}$.
\end{defn}
\begin{defn}\label{top_spec_defn}
	Let $\left(\mathcal{Y},\left\{\pi^{\mathcal Y}_n\right\} \right)$ be  a coherent system   of $\mathfrak{S}$ and $y \in \mathcal{Y}$. A subset  $\mathcal V \subset \mathcal{Y}$ is said to be \textit{special} if $\pi^{\mathcal Y}_0\left(\mathcal{V} \right)$ is evenly covered by $\mathcal{X}_1 \to \mathcal{X}_0$ and for any  $n \in \mathbb{N}^0$ following conditions hold:
	\begin{itemize}
		\item $\pi^{\mathcal Y}_n\left(\mathcal{V} \right) \subset \mathcal X_n$ is  an open connected set, 
		\item The restriction $\pi^{\mathcal Y}_n|_{\mathcal V}:\mathcal{V}\to \pi^{\mathcal Y}_n\left( {\mathcal V}\right) $ is a bijection.
	\end{itemize}
	
\end{defn}
\begin{rem}
	If $\left(\mathcal{Y},\left\{\pi^{\mathcal Y}_n\right\} \right)$ is  a covering  of $\mathfrak{S}$ then the topology of $\mathcal{Y}$ is generated by special sets.
\end{rem}

\begin{defn}\label{comm_top_constr_morph_defn}
	Let us consider the situation of the Definition \ref{comm_top_constr_defn}. A \textit{morphism} from $\left(\mathcal{Y}',\left\{\pi^{\mathcal Y'}_n\right\}\right)\downarrow\mathfrak{S}$ to $\left(\mathcal{Y}'',\left\{\pi^{\mathcal Y''}_n\right\}\right)\downarrow\mathfrak{S}$ is a covering  $f: \mathcal{Y}' \to \mathcal{Y}''$ such that
	$$
	\pi_n^{\mathcal Y''} \circ f= \pi_n^{\mathcal Y'} 
	$$
	for any $n \in \N$.
	
\end{defn}
\begin{empt}\label{comm_top_constr}
	There is a category with objects and morphisms described by Definitions \ref{comm_top_constr_defn}, \ref{comm_top_constr_morph_defn}. Denote by $\downarrow \mathfrak S$ this category.
\end{empt}
\begin{lem}\label{top_universal_covering_lem}\cite{ivankov:qnc}
	There is the final object of the category $\downarrow \mathfrak S$  described in \ref{comm_top_constr}.
\end{lem}
\begin{defn}\label{top_topological_inv_lim_defn}
	The final object $\left(\widetilde{\mathcal{X}},\left\{\pi^{\widetilde{\mathcal X}}_n\right\} \right)$ of the category $\downarrow\mathfrak{S}$ is said to be the \textit{(topological) inverse limit} of $\downarrow\mathfrak{S}$.  The notation $\left(\widetilde{\mathcal{X}},\left\{\pi^{\widetilde{\mathcal X}}_n\right\} \right) = \varprojlim \downarrow \mathfrak{S}$ or simply $~\widetilde{\mathcal{X}} =  \varprojlim \downarrow\mathfrak{S}$ will be used. 
\end{defn}

\subsection{Hilbert modules}
\paragraph{} We refer to \cite{blackadar:ko} for definition of Hilbert $C^*$-modules, or simply Hilbert modules.
 For any $\xi, \zeta \in X_A$ let us define an $A$-endomorphism $\theta_{\xi, \zeta}$ given by  $\theta_{\xi, \zeta}(\eta)=\xi \langle \zeta, \eta \rangle_{X_A}$ where $\eta \in X_A$. Operator  $\theta_{\xi, \zeta}$ shall be denoted by $\xi \rangle\langle \zeta$. Norm completion of algebra generated by operators $\theta_{\xi, \zeta}$ is said to be an algebra of compact operators $\mathcal{K}(X_A)$. We suppose that there is a left action of $\mathcal{K}(X_A)$ on $X_A$ which is $A$-linear, i.e. action of  $\mathcal{K}(X_A)$ commutes with action of $A$.
 %\begin{defn}
 %A Hilbert submodule $X \subset \overline{X}_A$ is said to be {\it irreducible} if it is not a direct sum of two Hilbert $A$-modules. If an irreducible submodule is maximal than it is said to be an {\it irreducible component}.
 %\end{defn}
   
 \subsection{$C^*$-algebras and von Neumann algebras}
 \paragraph{}In this section I follow to \cite{pedersen:ca_aut}.
  %\begin{defn}\cite{pedersen:ca_aut}
 %	Let $A$ be a $C^*$-algebra.  The {\it strict topology} on the multiplier algebra $M(A)$ is the topology generated by seminorms $\vertiii{x}_a = \|ax\| + \|xa\|$, ($a\in A$). 
 	%If $x \in M(A)$  and a sequence of partial sums $\sum_{i=1}^{n}a_i$ ($n = 1,2, ...$), ($a_i \in A$) tends to $x$ in the strict topology then we shall write
 % 	\begin{equation*}
 % 	x = \sum_{i=1}^{\infty}a_i.
 %  	\end{equation*}
 % \end{defn}
 \begin{defn}
 	\label{strong_topology}\cite{pedersen:ca_aut} Let $\H$ be a Hilbert space. The {\it strong} topology on $B\left(\H\right)$ is the locally convex vector space topology associated with the family of seminorms of the form $x \mapsto \|x\xi\|$, $x \in B(\H)$, $\xi \in \H$.
 \end{defn}
 \begin{defn}\label{weak_topology}\cite{pedersen:ca_aut} Let $\H$ be a Hilbert space. The {\it weak} topology on $B\left(\H\right)$ is the locally convex vector space topology associated with the family of seminorms of the form $x \mapsto \left|\left(x\xi, \eta\right)\right|$, $x \in B(\H)$, $\xi, \eta \in \H$.
 \end{defn}
 
 \begin{thm}\label{vN_thm}\cite{pedersen:ca_aut}
 	Let $M$ be a $C^*$-subalgebra of $B(\H)$, containing the identity operator. The following conditions are equivalent:
 	\begin{itemize}
 		\item $M=M''$ where $M''$ is the bicommutant of $M$;
 		\item $M$ is weakly closed;
 		\item $M$ is strongly closed.
 	\end{itemize}
 \end{thm}
 
 \begin{defn}
 	Any $C^*$-algebra $M$ is said to be a {\it von Neumann algebra} or a {\it $W^*$- algebra} if $M$ satisfies to the conditions of the Theorem \ref{vN_thm}.
 \end{defn}
 \begin{defn} \cite{pedersen:ca_aut}
 	Let $A$ be a $C^*$-algebra, and let $S$ be the state space of $A$. For any $s \in S$ there is an associated representation $\pi_s: A \to B\left( \H_s\right)$. The representation $\bigoplus_{s \in S} \pi_s: A \to \bigoplus_{s \in S} B\left(\H_s \right)$ is said to be the \textit{universal representation}. The universal representation can be regarded as $A \to B\left( \bigoplus_{s \in S}\H_s\right)$.  
 \end{defn} 
 \begin{defn}\label{env_alg_defn}\cite{pedersen:ca_aut}
 	Let   $A$ be a $C^*$-algebra, and let $A \to B\left(\H \right)$ be the universal representation $A \to B\left(\H \right)$. The strong closure of $\pi\left( A\right)$ is said to be   the  {\it enveloping von Neumann algebra} or  the {\it enveloping $W^*$-algebra}  of $A$. The enveloping  von Neumann algebra will be denoted by $A''$.
 \end{defn}
 %\begin{prop}\label{env_alg_sec_dual}\cite{pedersen:ca_aut}
 %The enveloping von Neumann algebra $A''$ of a $C^*$-algebra $A$ is isomorphic, as a Banach space, to the second dual of $A$, i.e. $A'' \approx A^{**}$.
 %\end{prop}
 
 \begin{thm}\label{env_alg_thm}\cite{pedersen:ca_aut}
 	For each non-degenerate representation $\pi: A \to B\left(\H \right)$ of a $C^*$-algebra $A$ there is a unique  normal morphism of  $A''$ onto $\pi\left( A\right)''$ which extends $\pi$.   
 \end{thm}

 %\begin{defn}\cite{pedersen:ca_aut}
 %Let $B \in B(H)$ be a $C^*$-algebra. Denote by $B''$ the strong closure of $B$ in $B(H)$. $B''$ is an unital weakly closed $C^*$-algebra and if $B$ acts non-degenerately on $H$ then  $B''$ is the {\it bicommutant} of $B$. Any strongly (=weakly) closed algebra is said to be a {\it von Neumann algebra}.
 %\end{defn}
 %\begin{defn}\cite{pedersen:ca_aut}
 %For any $x\in B(H)$ element $|x| \stackrel{\text{def}}{=} (xx^*)^{1/2}$ is said to be the {\it absolute value of} $x$.
 %\end{defn}
 \begin{lem}\label{increasing_convergent_w}\cite{pedersen:ca_aut} Let $\Lambda$ be an increasing net in the partial ordering.  Let $\left\{x_\lambda \right\}_{\la \in \La}$ be an increasing net of self-adjoint operators in $B\left(\H\right)$, i.e. $\la \le \mu$ implies $x_\la \le x_\mu$. If $\left\|x_\la\right\| \le \ga$ for some $\ga \in \mathbb{R}$ and all $\la$ then $\left\{x_\lambda \right\}$ is strongly convergent to a self-adjoint element $x \in B\left(\H\right)$ with $\left\|x_\la\right\| \le \ga$.
 \end{lem}
 %\begin{defn}\label{range_proj_defn}\cite{pedersen:ca_aut}
 %For each $x\in B(\H)$ we define the {\it range projection} of $x$ (denoted by $[x]$) as projection on closure of $x\H$. If $x\ge 0$ then the sequence $\left(\left((1 /n) +x\right)^{-1}x\right)$ is monotone increasing to $[x]$.  If $p$ and $q$ are projections then $p \vee q = [p + q]$ and thus $p \wedge q = 1 - \left[2 - \left(p+q\right)\right]$. Similarly we have $p \setminus q = p - p\wedge q$. Since $[x]\H$ is the orthogonal complement of the null space of $x^*$ we have $[x]=[xx^*]$. 
 %\end{defn}
 %\begin{cor}
 %If $M$ is a $W^*$-algebra and $a \in M$ then the range projection $[a]$ of $a$ lies in $M$.
 %\end{cor}
 %If $\mathcal{M}$ is a von Neumann algebra in $B(\H)$ then $[x]\in \mathcal{M}$ for any $x\in \mathcal{M}$. We next prove a {\it polar decomposition}.
 \paragraph*{}    For each $x\in B(\H)$ we define the {\it range projection} of $x$ (denoted by $[x]$) as projection on the closure of $x\H$. If $M$ is a von Neumann algebra and $x \in M$ then $[x]\in M$.
 
 \begin{prop}\label{polar_decomposition_prop}\cite{pedersen:ca_aut}
 	For each element $x$ in   a von Neumann algebra $M$ there is a unique partial isometry $u\in M$ and positive $\left|x\right| \in M_+$ with $uu^*=[|x|]$ and  $x=|x|u$.
 \end{prop}
 \begin{defn}\label{polar_decomposition_defn}
 	The formula $x=|x|u$ in the Proposition \ref{polar_decomposition_prop} is said to be the \textit{polar decomposition}.
 \end{defn}
 \begin{empt}\label{comm_gns_constr}
 	Any separable $C^*$-algebra $A$ has a state $\tau$ which induces a faithful GNS representation  \cite{murphy}. There is a $\mathbb{C}$-valued product on $A$ given by
 	\begin{equation*}
 	\left(a, b\right)=\tau\left(a^*b\right).
 	\end{equation*}
 	This product induces a product on $A/\mathcal{I}_\tau$ where $\mathcal{I}_\tau =\left\{a \in A \ | \ \tau(a^*a)=0\right\}$. So $A/\mathcal{I}_\tau$ is a pre-Hilbert space. Let denote by $L^2\left(A, \tau\right)$ the Hilbert  completion of $A/\mathcal{I}_\tau$.  The Hilbert space  $L^2\left(A, \tau\right)$ is a space of a  GNS representation  $A\to B\left(L^2\left(A, \tau\right) \right) $. Also there is the natural $\C$-linear map
 	\be\label{from_a_to_l2_eqn}
 	\begin{split}
 		A \to L^2\left(A, \tau\right),\\
 		a \mapsto a + \mathcal{I}_\tau.
 		\end{split}
 	\ee
 	The image of $A$ is a dense subspace of $L^2\left(A, \tau\right)$.
 \end{empt}

 \subsection{Connections}
 \begin{defn}\cite{connes:ncg94}
 	\begin{enumerate}
 		\item [(a)]  A \textit{cycle} of dimension $n$ is a triple $\left(\Om, d, \int \right)$ where $\Om = \bigoplus_{j=0}^n\Om^j$  is a graded algebra over $\C$, $d$ is a graded derivation of degree 1 such that $d^2=0$, and $\int :\Om^n \to \C$ is a closed graded trace on $\Om$,
 		\item[(b)] Let $\A$  be an algebra over $\C$. Then a \textit{cycle over} $\A$ is given by a cycle $\left(\Om, d, \int \right)$	and a homomorphism $\A \to \Om^0$.
 	\end{enumerate}
 \end{defn}
 \begin{defn}\label{conn_defn}\cite{connes:ncg94}
 	Let $\A\xrightarrow{\rho} \Om$ be a cycle over $\A$, and $\E$ a finite projective module over $\A$.
 	Then a \textit{connection} $\nabla$ on $\E$ is a linear map  $\nabla: \E \to \E \otimes_{\A} \Om^1$ such that
 	\be\label{conn_prop_eqn}
 	\nabla\left(\xi x \right) =  \nabla\left(\xi \right) x =  \xi \otimes d\rho\left(x \right) ; ~ \forall \xi \in \E, ~ \forall x \in \A.
 	\ee
 \end{defn}
 Here $\E$ is a right module over $\A$ and $\Om^1$ is considered as a bimodule over $\A$.
 
 \begin{prop}\label{conn_prop}\cite{connes:ncg94}
 	Following conditions hold:	
 	\begin{enumerate} 
 		\item[(a)] 	Let $e \in \End_{\A}\left( \E\right)$ be an idempotent and $\nabla$ is a connection on $\E$; then 
 		\be\label{idem_conn}
 		\xi \mapsto \left(e \otimes 1 \right) \nabla \xi
 		\ee
 		is a connection on $e\E$,
 		\item[(b)] Any finite projective module $\E$ admits a connection,
 		\item[(c)]  The space of connections is an affine space over the vector space $\Hom_{\sA}\left(\E, \E \otimes_{\A} \Om^1 \right)$, 
 		\item[(d)] Any connection $\nabla$ extends uniquely up to a linear map of  $\widetilde{\mathcal E}= \mathcal E \otimes_{\A} \Om$ into itself
 		such that
 		\be
 		\nabla\left(\xi \otimes \om \right) = \nabla\left(\xi \right) \om + \xi \otimes d\om; ~~\forall \xi \in \mathcal E, ~ \om \in \Om. 
 		\ee
 	\end{enumerate}
 \end{prop}

 \subsection{Finite Galois coverings}\label{fin_gal_cov_sec}
 \paragraph*{} Here I follow to \cite{auslander:galois}. Let $A \hookto \widetilde{A}$ be an injective homomorphism of unital algebras, such that
 \begin{itemize}
 	\item $\widetilde{A}$ is a projective finitely generated $A$-module,
 	\item There is an action $G \times \widetilde{A} \to \widetilde{A}$ of a finite group $G$ such that $$A = \widetilde{A}^G=\left\{\widetilde{a}\in \widetilde{A}~|~g\widetilde{a}=\widetilde{a}; ~\forall g \in G\right\}.$$
 \end{itemize}
 Let us consider the category $\mathscr{M}^G_{\widetilde{A}}$ of $G-\widetilde{A}$ modules, i.e.  any object $M \in \mathscr{M}^G_{\widetilde{A}}$ is a $\widetilde{A}$-module with equivariant action of $G$, i.e. for any $m \in M$ a following condition holds
 $$
 g\left(\widetilde{a}m \right)=  \left(g\widetilde{a} \right) \left(gm \right) \text{ for any } \widetilde{a} \in \widetilde{A}, ~ g \in G.
 $$
 Any morphism $\varphi: M \to N$ in the category $\mathscr{M}^G_{\widetilde{A}}$ is $G$- equivariant, i.e.
 $$
 \varphi\left( g m\right)= g \varphi\left( m\right)   \text{ for any } m \in M, ~ g \in G.
 $$
 Let $\widetilde{A}\left[ G\right]$ be an algebra such that $\widetilde{A}\left[ G\right] \approx \widetilde{A}\times G$ as an Abelian group and a multiplication law is given by
 $$
 \left( a, g\right)\left( b, h\right) =\left(a\left(gb \right), gh  \right).
 $$
 The category $\mathscr{M}^G_{\widetilde{A}}$ is equivalent to the category $\mathscr{M}_{\widetilde{A}\left[ G\right]}$ of $\widetilde{A}\left[ G\right]$ modules. Otherwise in \cite{auslander:galois} it is proven that if $\widetilde{A}$ is a finitely generated, projective $A$-module then there is an  equivalence between a category $\mathscr{M}_{A}$ of $A$-modules and the category $\mathscr{M}_{\widetilde{A}\left[ G\right]}$. It turns out that the category $\mathscr{M}^G_{\widetilde{A}}$ is equivalent to the category $\mathscr{M}_{A}$.

\subsection{Spectral triples}
 
 \paragraph{}
 This section contains citations of  \cite{hajac:toknotes}. 
 \subsubsection{Definition of spectral triples}
 \begin{defn}
 	\label{df:spec-triple}\cite{hajac:toknotes}
 	A (unital) {\it {spectral triple}} $(\A, \H, D)$ consists of:
 	\begin{itemize}
 		\item
 		a pre-$C^*$-algebra $\A$ with an involution $a \mapsto a^*$, equipped
 		with a faithful representation on:
 		\item
 		a \emph{Hilbert space} $\H$; and also
 		\item
 		a \emph{selfadjoint operator} $D$ on $\mathcal{H}$, with dense domain
 		$\Dom D \subset \H$, such that $a(\Dom D) \subseteq \Dom D$ for all 
 		$a \in \mathcal{A}$.
 	\end{itemize}
  \end{defn}
There is a set of axioms for  spectral triples described in \cite{hajac:toknotes,varilly:noncom}. In this article the regularity axiom is used only.
\begin{axiom}\label{regularity_axiom}\cite{varilly:noncom}(Regularity) 
For any $a \in \A$, $[D,a]$ is a bounded operator on~$\H$, and both
$a$ and $[D,a]$ belong to the domain of smoothness
$\bigcap_{k=1}^\infty \Dom(\delta^k)$ of the derivation $\delta$
on~$B(\H)$ given by $\delta(T) \stackrel{\mathrm{def}}{=} [\left|D\right|,T]$.
\end{axiom}
 
 \begin{lem}
	\label{lm:proj-approx}\cite{hajac:toknotes}
	Let $\sA$ be an unital Fr\'echet pre-$C^*$-algebra, whose
	$C^*$-completion is~$A$. If $\tilde{q} = \tilde{q}^2 = \tilde{q}^*$ is
	a projection in $A$, then for any $\eps > 0$, we can find a projection
	$q = q^2 = q^* \in \sA$ such that $\|q - \tilde{q}\| < \eps$.
\end{lem}

\subsubsection{Representations of  smooth algebras}\label{s_repr}
 
 \paragraph*{}
Let $(\A, \H, D)$ be a spectral triple. Similarly to \cite{bram:atricle} we  define a representation of $\pi^1:\A \to B(\H^2)$ given by
 \begin{equation}\label{s_diff1_repr_equ}
 \pi^1(a) =  \begin{pmatrix} a & 0\\
 [D,a] & a\end{pmatrix}.
 \end{equation}
 We can inductively construct  representations $\pi^s: \A \to B\left(\H^{2^s}\right)$ for any $s \in \mathbb{N}$. If $\pi^s$ is already constructed then  $\pi^{s+1}: \A \to B\left(\H^{2^{s+1}}\right)$ is given by
 \begin{equation}\label{s_diff_repr_equ}
 \pi^{s+1}(a) =  \begin{pmatrix}  \pi^{s}(a) & 0 \\ \left[D,\pi^s(a)\right] &  \pi^s(a)\end{pmatrix}
 \end{equation}
 where we assume diagonal action of $D$ on $\H^{2^s}$, i.e.
 \begin{equation*}
 D \begin{pmatrix} x_1\\ ... \\ x_{2^s}
 \end{pmatrix}= \begin{pmatrix} D x_1\\ ... \\ D x_{2^s}
 \end{pmatrix}; \ x_1,..., x_{2^s}\in \H.
 \end{equation*}
 For any $s \in \N^0$ there is a seminorm $\left\|\cdot \right\|_s$  on $\A$ given by
 \begin{equation}\label{s_semi_eqn}
\left\|a \right\|_s = \left\| \pi^{s}(a) \right\|.
 \end{equation}
 The definition of spectral triple requires that $\A$ is a Fr\'echet algebra with respect to seminorms $\left\|\cdot \right\|_s$.

\subsubsection{Noncommutative differential forms}\label{ass_cycle_sec}
 \paragraph*{} 
 Any spectral triple $\left( \A, \H, D\right)$  naturally defines a cycle $\rho : \A \to \Om_D$ (cf. Definition \ref{conn_defn}). 
 In particular for any spectral triple there is an $\A$-bimodule $\Om^1_D\subset B\left(\H \right) $ of differential forms which is the $\C$-linear span of operators given by
 \begin{equation}\label{dirac_d_module}
 a\left[D, b \right];~a,b \in \A.
 \end{equation}
 There is the differential map
 \begin{equation}\label{diff_map}
 \begin{split}
 d: \A \to \Om^1_D, \\
 a \mapsto \left[D, a \right].
 \end{split}
 \end{equation}

 \begin{definition}\label{ass_cycle_defn}
 We say that that both the cycle $\rho : \A \to \Om_D$ and the differential \eqref{diff_map} are \textit{associated} with the triple  $\left( \A, \H, D\right)$. We say that  $\A$-bimodule $\Om^1_D$ is the \textit{module of differential forms associated} with the spectral triple  $\left( \A, \H, D\right)$.
 \end{definition}
In case of associated cycles the connection equation \eqref{conn_prop_eqn} has the following form
\be\label{conn_triple_eqn}
\nabla \left(\xi a \right) = \nabla \left(\xi \right) a + \xi \left[ D,a\right] .
\ee

   \section{Noncommutative finite-fold coverings}
   \subsection{Basic construction}

   \begin{definition}
   	If $A$ is a $C^*$- algebra then an action of a group $G$ is said to be {\it involutive } if $ga^* = \left(ga\right)^*$ for any $a \in A$ and $g\in G$. The action is said to be \textit{non-degenerated} if for any nontrivial $g \in G$ there is $a \in A$ such that $ga\neq a$. 
   \end{definition}
   \begin{definition}\label{fin_def_uni}
   	Let $A \hookto \widetilde{A}$ be an injective *-homomorphism of unital $C^*$-algebras. Suppose that there is a non-degenerated involutive action $G \times \widetilde{A} \to \widetilde{A}$ of a finite group $G$, such that $A = \widetilde{A}^G\stackrel{\text{def}}{=}\left\{a\in \widetilde{A}~|~ a = g a;~ \forall g \in G\right\}$. There is an $A$-valued product on $\widetilde{A}$ given by
   	\begin{equation}\label{finite_hilb_mod_prod_eqn}
   	\left\langle a, b \right\rangle_{\widetilde{A}}=\sum_{g \in G} g\left( a^* b\right) 
   	\end{equation}
   	and $\widetilde{A}$ is an $A$-Hilbert module. We say that a triple $\left(A, \widetilde{A}, G \right)$ is an \textit{unital noncommutative finite-fold  covering} if $\widetilde{A}$ is a finitely generated projective $A$-Hilbert module.
   \end{definition}
   \begin{remark}
   	Above definition is motivated by the Theorem \ref{pavlov_troisky_thm}.
   \end{remark}
   \begin{definition}\label{fin_comp_def}
   	Let $A$, $\widetilde{A}$ be $C^*$-algebras and let  $A \hookto \widetilde{A}$ be an inclusion such  that following conditions hold:
   	\begin{enumerate}
   		\item[(a)] 
   		There are unital $C^*$-algebras $B$, $\widetilde{B}$  and inclusions 
   		$A \subset B$,  $\widetilde{A}\subset \widetilde{B}$ such that $A$ (resp. $B$) is an essential ideal of $\widetilde{A}$ (resp. $\widetilde{B}$) and $A = B\bigcap \widetilde{A}$,
   		\item[(b)] There is an unital  noncommutative finite-fold covering $\left(B ,\widetilde{B}, G \right)$,
   		\item[(c)] $G\widetilde{A} = \widetilde{A}$.
   		%\begin{equation}\label{wta_eqn}
   		%\widetilde{A} =  \left\{a\in \widetilde{B}  ~|~ \left\langle \widetilde{B} ,a  \right\rangle_{\widetilde{B} } \in A \right\}.
   		%\end{equation}
   	\end{enumerate}
   	
   	The triple $\left(A, \widetilde{A},G \right)$ is said to be a \textit{noncommutative finite-fold covering with compactification}. 
   \end{definition}
     \begin{remark}
   	Any unital noncommutative finite-fold covering is a noncommutative finite-fold covering with compactification.
   \end{remark}
   \begin{definition}\label{fin_def}
   	Let $A$, $\widetilde{A}$ be $C^*$-algebras, $A\hookto\widetilde{A}$ an injective *-homomorphism and $G\times \widetilde{A}\to \widetilde{A}$ an involutive non-degenerated action of a finite group $G$  such  that following conditions hold:
   	\begin{enumerate}
   		\item[(a)] 
   		$A \cong \widetilde{A}^G \stackrel{\mathrm{def}}{=} \left\{a\in \widetilde{A}  ~|~ Ga = a \right\}$,
   		\item[(b)] 
   		There is a family $\left\{\widetilde{I}_\la \subset \widetilde{A} \right\}_{\la \in \La}$ of closed ideals of $\widetilde{A}$ such that 
   		\be\label{gi-i}
   		G\widetilde{I}_\la = \widetilde{I}_\la.
   		\ee
   		Moreover $\bigcup_{\la \in \La} \widetilde{I}_\la$ (resp. $\bigcup_{\la \in \La} \left( A \bigcap \widetilde{I}_\la\right) $ ) is a dense subset of $\widetilde{A}$ (resp. $A$), and for any $\la \in \La$ there is a natural  noncommutative finite-fold covering with compactification $\left(\widetilde{I}_\la \bigcap A, \widetilde{I}_\la , G \right)$.  
   	\end{enumerate}
   	We say that the triple  $\left(A, \widetilde{A},G \right)$ is a \textit{noncommutative finite-fold covering}.
   \end{definition}

   \begin{remark}
   	The Definition \ref{fin_def} is motivated by the Theorem \ref{comm_fin_thm}.
   \end{remark}
   \begin{remark}
   	Any noncommutative finite-fold covering with compactification is a  noncommutative finite-fold covering.
   \end{remark}
   %\begin{remark}
   %	Any	unital noncommutative finite-fold  covering is a special case of a noncommutative finite-fold  covering.
   %\end{remark}
   \begin{definition}
   	The injective *-homomorphism $A \hookto \widetilde{A}$ %which follows from %\eqref{wtag_eqn}  
   	from the Definition \ref{fin_def}
   	is said to be a \textit{noncommutative finite-fold covering}.
   \end{definition}
   \begin{definition}\label{hilbert_product_defn}
   	Let $\left(A, \widetilde{A}, G\right)$ be a    noncommutative finite-fold covering.  The algebra  $\widetilde{A}$  is a Hilbert $A$-module with an $A$-valued  product given by
   	\begin{equation}\label{fin_form_a}
   	\left\langle a, b \right\rangle_{\widetilde{A}} = 
   	\sum_{g \in G} g(a^*b); ~ a,b \in \widetilde{A}.
   	\end{equation}
   	We say that this structure of Hilbert $A$-module is {\it induced by the covering} $\left(A, \widetilde{A}, G\right)$. Henceforth we shall consider $\widetilde{A}$ as a right $A$-module, so we will write $\widetilde{A}_A$. 
   \end{definition}

   %\begin{empt}\label{dir_sum_constr}
   %	Let $\left(A, \widetilde{A}, G\right)$ be a noncommutative finite covering. If $\widetilde{a} \in \widetilde{A}$ then $\widetilde{a} = a + p$ where $a = \frac{1}{|G|}\sum_{g \in G}g\widetilde{a}$ and $p = \widetilde{a}-a$. It is clear that $\sum_{g \in G}gp = 0$ and for any $G$-invariant $b \in \widetilde{A}$ we have 
   %	\begin{equation*}
   %	\langle b, p \rangle_{\widetilde{A}} = \frac{1}{|G|}\widetilde{b}\sum_{g \in G}gp = 0.
   %	\end{equation*}
   %	Otherwise the set of $G$-invariant elements is just a sublagebra $A \subset \widetilde{A}$.
   %	So  $\widetilde{A}_A$ can be decomposed into the direct orthogonal sum, i.e.
   %	\begin{equation}\label{hilb_mod_direct_sum}
   %	\begin{split}
   %	\widetilde{A}_A= A \oplus P; \  
   %	A \perp P \ , \text{i.e. } \langle \widetilde{b}, p \rangle_{\widetilde{A}}= 0; \text{ for any } a \in A; \ p \in P. 
   %	\end{split}
   %	\end{equation}
   %\end{empt}
   
   The group $G$ is said to be the \textit{covering transformation group} (of $\left(A, \widetilde{A},G \right)$ ) and we use the following notation
   \begin{equation}\label{group_cov_eqn}
   G\left(\widetilde{A}~|~A \right) \stackrel{\mathrm{def}}{=} G.
   \end{equation}
   \subsection{Induced representation}\label{induced_repr_fin_sec}
   
   \begin{empt}\label{induced_repr_constr}
   	Let $\left(A, \widetilde{A}, G\right)$ be a noncommutative finite-fold covering, and let $\rho: A \to B\left(\H\right)$ be a representation. If $X=\widetilde{A}\otimes_A \H$ is the algebraic tensor product then there is a sesquilinear $\C$-valued product $\left(\cdot, \cdot\right)_{X}$ on $X$  given by
   	\begin{equation}\label{induced_prod_equ}
   	\left(a \otimes \xi, b \otimes \eta \right)_{X}= %\frac{1}{\left|G\left(\widetilde{A}~|~A \right) \right| }
   	\left(\xi, 	\left\langle a, b \right\rangle_{\widetilde{A}} \eta\right)_{\H}
   	\end{equation}
   	where $ \left(\cdot, \cdot\right)_{\H}$ means the Hilbert space product on $\H$, and $\left\langle \cdot, \cdot \right\rangle_{\widetilde{A}}$ is given by \eqref{fin_form_a}. So $X$ is a pre-Hilbert space. There is a natural map $\widetilde{A} \times \left( \widetilde{A}\otimes_A \H \right)\to \widetilde{A}\otimes_A \H$ given by
   \be\label{ind_act_form}
   \begin{split}
   	\widetilde{A} \times \left( \widetilde{A}\otimes_A \H \right)\to \widetilde{A}\otimes_A \H,\\
   	(a, b \otimes \xi) \mapsto ab \otimes \xi.
   \end{split}
    \ee
   \end{empt}
   
   \begin{defn}\label{induced_repr_defn}
   	
   	Use notation of the Definition \ref{hilbert_product_defn}, and \ref{induced_repr_constr}.
   	If $\widetilde{\H}$ is the Hilbert completion of  $X=\widetilde{A}\otimes_A \H$ then the map \eqref{ind_act_form} induces the representation $\widetilde{\rho}: \widetilde{A} \to B\left( \widetilde{\H} \right)$. We say that $\widetilde{\rho}$ \textit{is induced by the pair} $\left(\rho,\left(A, \widetilde{A}, G\right)  \right)$.  
   \end{defn}
   \begin{rem}
   	Below any $\widetilde a \otimes \xi\in\widetilde{A}\otimes_A \H$ will be regarded as element in $\widetilde{\H}$.
   \end{rem}
   \begin{lem}\cite{ivankov:qnc}
   	If $A \to B\left(\H \right) $ is faithful then $\widetilde{\rho}: \widetilde{A} \to B\left( \widetilde{\H} \right)$ is faithful. 
   \end{lem}
   
   \begin{empt}
   	Let $\left(A, \widetilde{A}, G\right)$ be a  noncommutative finite-fold covering, let $\rho: A \to B\left(\H \right)$ be a faithful non-degenerated representation, and let  $\widetilde{\rho}: \widetilde{A} \to B\left( \widetilde{\H} \right)$ is induced by the pair $\left(\rho,\left(A, \widetilde{A}, G\right)  \right)$. There is the natural action of $G$ on $\widetilde{\H}$ induced by the map
   	$$
   	g \left( \widetilde{a} \otimes \xi\right)  = \left( g\widetilde{a} \right) \otimes \xi; ~ \widetilde{a} \in \widetilde{A}, ~ g \in G, ~ \xi \in \H. 
   	$$
   	There is the natural orthogonal inclusion 
   	\be\label{hilb_fin_inc_eqn}
   	\H \hookto \widetilde{\H}
   	\ee
   	 induced by inclusions
   	$$
   	A \subset\widetilde{A}; ~~ A \otimes_A \H \subset\widetilde{A} \otimes_A \H.
   	$$
   	If $\widetilde{A}$ is an unital $C^*$-algebra then the inclusion \eqref{hilb_fin_inc_eqn} is given by
   	\be\label{hilb_fin_inc_map_eqn}
   	\begin{split}
   	\varphi:		\H \hookto \widetilde{\H},\\
   		\xi \mapsto 1_{\widetilde{A}} \otimes \xi
   	\end{split}
   	\ee
   	where $1_{\widetilde{A}} \otimes \xi\in \widetilde{A} \otimes_A \H$ is regarded as element of $\widetilde{\H}$. 
   The inclusion \eqref{hilb_fin_inc_map_eqn} is not isometric. From 
   $$
   \left\langle 1_{\widetilde{A}}, 1_{\widetilde{A}} \right\rangle = \sum_{g \in G\left({\widetilde{A}}~|~A \right)} g1^2_{\widetilde{A}} = \left|G\left(\widetilde{A}~|~A \right) \right| 1_{A}
   $$
   it turns out
   \be
   \left(\xi, \eta \right)_{\H}= \frac{1}{\left|G\left(\widetilde{A}~|~A \right) \right|}\left( 	\varphi\left(\xi \right), \varphi\left(\eta \right)\right)_{\widetilde{\H}}; ~~ \forall \xi, \eta \in \H 
   \ee

   	Action of $g\in G\left({\widetilde{A}}~|~A \right)$ on $\widetilde{A}$ can be defined by representation as $g \widetilde{a} = g \widetilde{a} g^{-1}$, i.e.
   	$$
   	(g\widetilde{a}) \xi = g\left(\widetilde{a} \left( g^{-1}\xi \right)  \right);~ \forall \xi \in \widetilde{\H}.
   	$$
   	  	   \end{empt}
   
   	%	Similarly there is the action of $g$ on $ B\left(\widetilde{\H} \right)$ given by $$\left( gb\right) \xi = g \left( b \left( g^{-1} \xi\right) \right) ; ~\forall b \in  B\left(\widetilde{\H} \right), ~\forall \xi \in \widetilde{\H}.$$
   	%	As well as $A = \widetilde{A}^G$ following condition holds
   	%	\begin{equation*}
   	%	B\left(\H \right)= B\left(\widetilde{\H} \right)^G.  
   	%	\end{equation*}
   	
   	%	$B\left(\widetilde{\H} \right)$ is a right - $B\left( \H\right) $ module with a sesquilinear product given by
   	%	\begin{equation*}
   	%	\left\langle a, b \right\rangle_{B\left(\widetilde{\H} \right)} = 
   	%	\sum_{g \in G} g(a^*b) \in B\left(\H \right) .
   	%	\end{equation*}
   	%	There is the natural isomorphism of vector spaces $B\left( \widetilde{\H}\right) \otimes_{B\left(\H \right) } \H \approx \widetilde{\H}$, given by
   	%	$$
   	%	b \otimes \xi \mapsto b \xi; ~ b \in B\left( \widetilde{\H}\right), ~~ \xi \in \H \subset \widetilde{\H}.
   	%	$$

   %\begin{empt}
   %Let $\mathcal P \subset B\left( \widetilde{\H}\right)$ be the set of projectors such that $\left(gp \widetilde{\H}\right)  \perp p\widetilde{\H}$ for any nontrivial $g \in G$. According Zorn's lemma there is the maximal $p_{\max}\in \mathcal P$.
   %\end{empt}
   %\begin{thm}
   %Following conditions hold
   %$$
   %\sum_{g \in G} gp_{\max} = 1_{B\left( \widetilde{\H}\right)}
   %$$
   %\end{thm}
   %\begin{proof}
   %If $\sum_{g \in G} gp_{\max} \neq 1_{B\left( \widetilde{\H}\right)}$ then $q =  1_{B\left( \widetilde{\H}\right)} - \sum_{g \in G} gp_{\max} \neq 1_{B\left( \widetilde{\H}\right)}$ is a nonzero projector. Let  $\mathcal P' = \left\{p' \in \mathcal P~|~ p' < q \right\}$.
   %\end{proof}
   \begin{defn}\label{mult_G_act_defn} If $M\left(\widetilde{A} \right)$ is the multiplier algebra of $\widetilde{A}$ then there is the natural action of $G$ on $M\left(\widetilde{A} \right)$ such that for any $\widetilde{a}\in M\left(\widetilde{A} \right)$, $\widetilde{b}\in\widetilde{A}$ and $g \in G$ a following condition holds
   	$$
   	\left(g \widetilde{a} \right)\widetilde{b} = g\left(\widetilde{a} \left( g^{-1}\widetilde{b} \right)  \right)
   	$$
   	We say that action of $G$ on $M\left(\widetilde{A} \right)$ is \textit{induced} by the action  of $G$ on $\widetilde{A}$.
   \end{defn}

 \subsection{Coverings of spectral triples}\label{triple_fin_cov}

 \paragraph*{}
 
 	Let  $\left( \A, \H, D\right)$ %, \left( \Ga, \right) J\right) $ 
 	be a spectral triple, and let $A$ is the $C^*$-norm completion of $\A$. Let $\left(A, \widetilde{A}, G \right)$ be an unital noncommutative finite-fold covering. Let $\rho: A \to B\left(\H \right)$ be a natural representation given by the spectral triple $\left( \A, \H, D\right)$, and let $\widetilde{\rho}: \widetilde{A} \to B\left( \widetilde{\H}\right)$ be a representation induced by the pair $\left( \rho,\left(A, \widetilde{A}, G \right) \right)$.  
 	The algebra $\widetilde{A}$ is a finitely generated projective $A$-module, it turns out following direct sum
 	$$
 	\widetilde{A} \bigoplus Q \cong A^n.
 	$$
 	of $A$-modules. So there is a projector $p \in \mathbb{M}_n\left(A \right)$ such that
 	$
 		\widetilde{A} \cong pA^n
 	$
 	as $A$-module.  $\A$ is dense in $A$ and $\A$ is closed with respect homomorphic calculus, it turns out that there is a projector $\widetilde{p} \in \mathbb{M}_n\left(\A \right)$ such that $\left\| \widetilde{p} - p\right\|  < 1$, so one has
 	$$
 		\widetilde{A} \cong	\widetilde{p}A^n.
 	$$
 	From $	\widetilde{A} \subset \End_A\left(	\widetilde{A} \right)$ and  $\End_A\left(	\widetilde{A} \right) = \widetilde{p}~\mathbb{M}_n\left(A \right)\widetilde{p} \subset \mathbb{M}_n\left(A \right)$ it follows that there is the following inclusion of $C^*$-algebras
 	$
 	\widetilde{A} \subset \mathbb{M}_n\left(A \right)
 	$. Both $\widetilde{A}$ and $\mathbb{M}_n\left(A \right)$ are finitely generated projective $A$ modules, it turns out that there is an $A$-module $P$ such that
 	$$
 	\widetilde{A} \bigoplus P \cong	\mathbb{M}_n\left(A \right) .
 	$$
 	Taking into account inclusions  $\widetilde{A} \subset \mathbb{M}_n\left(A \right)$ and $\mathbb{M}_n\left(\A \right) \subset \mathbb{M}_n\left(A \right)$ one can define the intersection of algebras
 \begin{equation}\label{a_smooth_eqn}
 	\widetilde{\A}= \widetilde{A} \bigcap \mathbb{M}_n\left(\A \right).
 \end{equation}
 From \cite{varilly_bondia} it turns out that $\mathbb{M}_n\left(\A \right)$ is closed with respect to holomorphic functional calculus.
 	Both $\mathbb{M}_n\left(\A \right)$ and $\widetilde{A}$ are closed with respect to holomorphic functional calculus, so $\widetilde{\A}$ is closed with respect to holomorphic functional calculus, i.e. $\widetilde{\A}$ is a pre-$C^*$-algebra. From 
 	$$
 	\mathbb{M}_n\left(\A \right) \cong \widetilde{\A} \bigoplus \left(P \bigcap 	\mathbb{M}_n\left(\A \right)\right) 
 	$$
 	it turns out that $\widetilde{\A}$ is a finitely generated projective $\A$ module.

  \begin{lemma}\label{dense_lem}
  	The algebra $\widetilde{\A}$ is a dense subalgebra of $\widetilde{A}$ with respect to the $C^*$-norm topology.
  \end{lemma}
\begin{proof}
	There is the isomorphism $A^{n^2}\approx \mathbb{M}_n\left(A \right)$ of right $A$-modules. Since $\widetilde{A}$ is a projective right $A$-module, there is a projector $p_M\in \mathbb{M}_{n^2}\left(A \right)$  it follows that 
	$$
\widetilde{A} \cong p_M A^{n^2}	
	$$
	The algebra $\mathbb{M}_{n^2}\left(\A \right)$ is dense in $\mathbb{M}_{n^2}\left(A \right)$, so there is a projector $\widetilde{p}_M\in \mathbb{M}_{n^2}\left(\A \right)$ such that such that $\left\| \widetilde{p}_M - p_M\right\|  < 1$, it turns out
\be\label{pm_eqn}
	\widetilde{\A} \cong \widetilde{p}_M \A^{n^2}
\ee
For any $\widetilde{a}\in \widetilde{A}$ there is a sequence $\left\lbrace \widetilde{a}_n \in \mathbb{M}_{n}\left(\A \right)\right\rbrace_{n \in \N}$ such that 
$$
\lim_{n \to \infty} \widetilde{a}_n = \widetilde{a}
$$
in sense of $C^*$-norm topology. The sequence can be regarded as a sequence $\left\lbrace \widetilde{a}_n \in \A^{n^2}\right\rbrace_{n \in \N}$. From \eqref{pm_eqn} it turns out that if $\widetilde{b}_n=\widetilde{p}_M \widetilde{a}_n \in \widetilde{A}$ then
$$
\lim_{n \to \infty} \widetilde{b}_n = \lim_{n \to \infty} \widetilde{p}_M\widetilde{a}_n = \widetilde{p}_M\lim_{n \to \infty} \widetilde{a}_n =  \widetilde{p}_M\widetilde{a}=\widetilde{a}.
$$
Otherwise from $\widetilde{a}_n  \in \A^{n^2}$ and $\widetilde{p}_M\in \mathbb{M}_{n^2}\left(\A \right)$ it turns out that  $\widetilde{b}_n =\widetilde{p}_M\widetilde{a}_n  \in \A^{n^2}\cong \mathbb{M}_n\left(\A \right)$, so one has  $\widetilde{b}_n  \in \widetilde{A} \bigcap \mathbb{M}_n\left(\A \right)\cong \widetilde{\A}$. Hence for any $\widetilde{a}\in \widetilde{A}$ there is a sequence $\left\lbrace \widetilde{b}_n \in \widetilde{\A}\right\rbrace_{n \in \N}$ such that
$$
\lim_{n \to \infty} \widetilde{b}_n = \widetilde{a}.
$$
\end{proof}

 \begin{definition}\label{smooth_defn}
 In the above situation we say that the unital noncommutative finite-fold covering $\left(A, \widetilde{A}, G \right)$ is \textit{smoothly invariant} if $G\widetilde{\A} = \widetilde{\A}$.
 \end{definition}
\begin{lem}\label{smooth_matr_lem}
	Let us use the above notation. Suppose that the right $A$-module $\widetilde{A}_A$ is generated by a finite set $\left\{\widetilde{a}_1, \dots, \widetilde{a}_n \right\}$, i.e. 
	$$
	\widetilde{A}_A = \sum_{j =1}^{n} \widetilde{a}_jA,
	$$
	such that following conditions hold:
	\begin{enumerate}
		\item[(a)] $\left\langle \widetilde{a}_j, \widetilde{a}_k\right\rangle_{\widetilde{A}} \in \A$ for any $j,k=1, \dots, n$,
		\item[(b)] The set $\left\{\widetilde{a}_1, \dots, \widetilde{a}_n \right\}$ is $G$-invariant, i.e. $g \widetilde{a}_j \in \left\{\widetilde{a}_1, \dots, \widetilde{a}_n \right\}$ for any $j = 1,\dots,n$ and $g \in G$.
	\end{enumerate}
Then following conditions hold:
\begin{enumerate}
	\item[(i)] 
	\be\label{smooth_cond_eqn}
	\widetilde{A} \bigcap \mathbb{M_n\left(\A \right) }= \left\lbrace 	\widetilde{a}\in 	\widetilde{A}~|~  \left\langle \widetilde{a}_j, \widetilde{a}\widetilde{a}_k\right\rangle_{\widetilde{A}} \in \A; ~\forall j,k =1, \dots, n\right\} 
	\ee
	\item[(ii)] The unital noncommutative finite-fold covering $\left(A, \widetilde{A}, G \right)$ is smoothly invariant.
\end{enumerate}
\end{lem}
\begin{proof}(i)

	If $S \in \End\left( \widetilde{A}\right)_A$ is given by  
\be\label{s_matr}
	S = \sum_{j=1}^{n}\widetilde{a}_j\left\rangle \right\langle\widetilde{a}_j
\ee
	then $S$ is self-adjoint. Moreover $S$ is represented by a matrix $\left\{S_{jk}=\left\langle \widetilde{a}_j, \widetilde{a}_k\right\rangle_{\widetilde{A}} \right\}_{j,k=1,\dots,n} \in \mathbb{M}_n\left( \A\right)$. From the Corollary 1.1.25  of \cite{jensen_thomsen:kk} it turns out that $S$ is strictly positive. Otherwise $\widetilde{A}_A$ is a finitely generated right $A$-module, so from the Exercise  15.O of \cite{wegge_olsen} it follows that $S$ is invertible, i.e. there is $T \in \End\left( \widetilde{A}\right)_A $ such that
	$$
	ST = TS = 1_{\End\left( \widetilde{A}\right)_A}.
	$$
	If we consider $S$ as element of $\mathbb{M}_m\left( \A\right)$ then the spectrum of $S$ is a subset of $\mathcal{U}_0 \bigcup \mathcal{U}_1 \subset \C$ such that
	\begin{itemize}
		\item Both $\mathcal{U}_0$ and $\mathcal{U}_1$ are open sets.
		\item $\mathcal{U}_0 \bigcap \mathcal{U}_1 = \emptyset$,
		\item $0 \in \mathcal{U}_0$ and $0$ is the unique point of the spectrum of $S$ which lies in $\mathcal{U}_0$.
	\end{itemize}
If $\phi, \psi$ are homomorphic functions on $\mathcal{U}_0 \bigcup \mathcal{U}_1$ given by 
\bean\nonumber
\phi|_{\mathcal{U}_0}= \psi|_{\mathcal{U}_0} \equiv 0,\\
\phi|_{\mathcal{U}_1}= 1,\\
\psi|_{\mathcal{U}_1}= z \mapsto \frac{1}{z}~.
\eean
Then following conditions hold:
\begin{itemize}
	\item $p = \phi\left( S\right)$ is a projector, such that $ \widetilde{A}_A \approx p A^n$ as a right $A$-module,
	\item $p \in \mathbb{M}_n\left(\A \right)$,
	\item $\psi\left( S\right) =T\in \mathbb{M}_n\left(\A \right)$ and $TS = ST = 1_{\End\left( \widetilde{A}\right)_A}$. 
\end{itemize}
If $\widetilde{a} \in A$ then from $TS = ST = 1_{\End\left( \widetilde{A}\right)_A}$ it turns out
$$
\widetilde{a} = TS\widetilde{a}ST = T\left( \sum_{j = 1}^n \widetilde{a}_j\left\rangle \right\langle\widetilde{a}_j\right)\widetilde{a}\left( \sum_{k = 1}^n \widetilde{a}_k\left\rangle \right\langle\widetilde{a}_k\right) T= TM^{\widetilde{a}}T
$$
where $M^{\widetilde{a}}\in  \mathbb{M}_n\left(A \right) $ is a matrix given by
$$
M^{\widetilde{a}}=\left\{M^{\widetilde{a}}_{jk}=\left\langle \widetilde{a}_j, \widetilde{a}\widetilde{a}_k\right\rangle_{\widetilde{A}} \right\}_{j,k=1,\dots,n}~.
$$
From $T \in \mathbb{M}_n\left(\A \right)$ it turns out that 
$$
M^{\widetilde{a}} \in \mathbb{M}_n\left(\A \right) \Rightarrow TM^{\widetilde{a}}T\in \mathbb{M}_n\left(\A \right). 
$$
Conversely from $S \in \mathbb{M}_n\left(\A \right)$ it follows that
$$
TM^{\widetilde{a}}T \in \mathbb{M}_n\left(\A \right)\Rightarrow STM^{\widetilde{a}}TS =M^{\widetilde{a}} \in \mathbb{M}_n\left(\A \right), 
$$
so one has
$$
M^{\widetilde{a}} \in \mathbb{M}_n\left(\A \right) \Leftrightarrow TM^{\widetilde{a}}T\in \mathbb{M}_n\left(\A \right). 
$$
Hence
$\widetilde{a} \in \mathbb{M}_n\left(\A \right)$ if and only if $\left\langle\widetilde{a}_j, \widetilde{a}\widetilde{a}_k\right\rangle_{\widetilde{A}}\in \A$ for any $j,k = 1, \dots n$. %From  $TS = ST = 1_{\End\left( \widetilde{A}\right)_A}$ it turns out 
%\be\label{aj_in_sm_eqn}
%\widetilde{a}_j = TS \widetilde{a}_j = T\left( \sum_{k = 1}^n \widetilde{a}_k\left\rangle \right\langle\widetilde{a}_k\right) \widetilde{a}_j = TM^j
%\ee
%where $M^j = \left\{m^j_{kl}\right\}_{k,l=1,\dots,n} \in \mathbb{M}_n\left(\A \right)$ is such that
%$$
%m^j_{kl}=\delta_{kj} \left\langle \widetilde{a}_j, \widetilde{a}_l\right\rangle_{\widetilde{A}} \in \mathbb{M}_n\left(\A \right).
%$$
%From $T \in \mathbb{M}_n\left(\A \right)$ and \eqref{aj_in_sm_eqn} it turns out $\widetilde{a}_j \in \mathbb{M}_n\left(\A \right)$ for any $j =1,\dots n$.
\\
(ii) Note that given by \eqref{finite_hilb_mod_prod_eqn} product is $G$-invariant, i.e. $\left\langle \widetilde{a} , \widetilde{b}  \right\rangle_{\widetilde{A}}= \left\langle g\widetilde{a} , g\widetilde{b}  \right\rangle_{\widetilde{A}}$ for any $g \in G$, it follows that
$$
\left\langle \widetilde{a}_j , \left( g\widetilde{a}\right)  \widetilde{a}_k  \right\rangle_{\widetilde{A}}=\left\langle g^{-1}\widetilde{a}_j , \widetilde{a} \left( g^{-1}\widetilde{a}_k\right)   \right\rangle_{\widetilde{A}}.
$$
Otherwise from  (b) it follows that there are $j',k' \in \left\{1, \dots, n\right\}$ such that
$\widetilde{a}_{j'}=g^{-1}\widetilde{a}_j$ and $\widetilde{a}_{k'}=g^{-1}\widetilde{a}_k$, hence $\left\langle \widetilde{a}_j , \left( g\widetilde{a}\right)  \widetilde{a}_k  \right\rangle_{\widetilde{A}} = \left\langle \widetilde{a}_{j'} ,  \widetilde{a}  \widetilde{a}_{k'}  \right\rangle_{\widetilde{A}}~$. So for any $g \in G$ the condition 
$$
\left\langle \widetilde{a}_j , \left( g\widetilde{a}\right)  \widetilde{a}_k  \right\rangle_{\widetilde{A}}\in \A,~~ \forall j,k =1, \dots n
$$
is equivalent to
$$
\left\langle \widetilde{a}_{j'} , \widetilde{a} \widetilde{a}_{k'}  \right\rangle_{\widetilde{A}}\in \A,~~ \forall j',k' =1, \dots n.
$$
It turns out that $g\left(  \widetilde{A} \bigcap \mathbb{M}_n\left(\A \right)\right) =\widetilde{A} \bigcap \mathbb{M}_n\left(\A \right)
$, or equivalently  $$G\left(  \widetilde{A} \bigcap \mathbb{M}_n\left(\A \right)\right) =\widetilde{A} \bigcap \mathbb{M}_n\left(\A \right)
,$$
hence from \eqref{a_smooth_eqn} one has
$$G\widetilde{\A}  =\widetilde{\A} 
.$$

\end{proof}
In the following text we suppose that  the unital noncommutative finite-fold covering $\left(A, \widetilde{A}, G \right)$ is smoothly invariant.
 	 From the Proposition \ref{conn_prop}  it follows that there is a connection
 	$$
 	\nabla' : \widetilde{\A} \to \widetilde{\A} \otimes_{\A} \Om^1_D.
 	$$
 
Let us define a connection
 	\begin{equation}\label{equiv_eqn}
 	\begin{split}
 	\widetilde{\nabla}:\widetilde{\A} \to \widetilde{\A} \otimes_{\A} \Om^1_D,\\
 	\widetilde{\nabla}\left(\widetilde{a} \right) = \frac{1}{\left|G \right|} \sum_{g \in G} g^{-1}\left(\nabla'\left( g\widetilde{a}\right) . \right)
 	\end{split}
 	\end{equation}
 	The connection $\widetilde{\nabla}$ is $G$-\textit{equivariant}, i.e.
 \be\label{equiv_conn_eqn}
 	\widetilde{\nabla}\left(g\widetilde{a} \right)= g\left( \widetilde{\nabla}\left(\widetilde{a} \right)\right) ; \text{ for any } g \in G, ~\widetilde{a} \in \widetilde{\A}.
 \ee
 	Let  $\H^\infty= \bigcap_{n =1}^\infty \Dom D^n$, and let us define an operator $\widetilde{D} : \widetilde{\A} \otimes_{\A} \H^\infty \to \widetilde{\A} \otimes_{\A} \H^\infty$ such that if $\widetilde{a} \in \widetilde{\A}$ and
 	\begin{equation*}
 	\begin{split}
 	\nabla\left( 	\widetilde{a}\right) = \sum_{j = 1}^m\widetilde{a}_j \otimes \om_j 
 	\end{split}
 	\end{equation*}
 	then
 	\begin{equation}\label{d_defn}
 	\widetilde{D}\left(	\widetilde{a} \otimes \xi \right) = \sum_{j = 1}^m\widetilde{a}_j \otimes \om_j\left( \xi\right) + \widetilde{a} \otimes D\xi \Leftrightarrow \left[\widetilde{D}, \widetilde{a}\right] = \widetilde{\nabla}\left( \widetilde{a}\right).
 	\end{equation}
 	The space $\widetilde{\A} \otimes_{\A} \H^\infty$ is a dense subspace of the Hilbert space $\widetilde{\H} = \widetilde{A} \otimes_{A} \H$.
 	It turns out  $\widetilde{D}$ can be regarded as an unbounded operator on $\widetilde{\H}$. 
 \begin{lem}\label{conn_exist_lem} If the unital noncommutative finite-fold covering $\left(A, \widetilde{A}, G \right)$ is smoothly invariant then there is the unique $G$-equivariant connection $$\widetilde{\nabla}:\widetilde{\A} \to \widetilde{\A} \otimes_{\A} \Om^1_D.$$
\end{lem}
\begin{proof} From the equation \eqref{equiv_eqn} it follows that a $G$-equivariant connection exists. Let us prove the uniqueness of it. 
	It follows from the Proposition \ref{conn_prop} that the space of connections is an affine space over the vector space $\Hom_{\sA}\left(\widetilde{\A}, \widetilde{\A} \otimes_{\A} \Om^1_D \right)$.  The space of $G$-equivariant connections is an affine space over the vector space $\Hom^{G}_{\sA}\left(\widetilde{\A}, \widetilde{\A} \otimes_{\A} \Om^1_D \right)$ of $G$-equivariant morphisms, i.e. morphisms in the category $\mathscr{M}^G_{\widetilde{\A}}$ (cf. \ref{fin_gal_cov_sec}). However from  $\ref{fin_gal_cov_sec}$ it follows that the category  $\mathscr{M}^G_{\widetilde{\A}}$ is equivalent to the category $\mathscr{M}_{\A}$ of $\A$-modules. It turns out that there is a 1-1 correspondence between connections
	$$
	\nabla:	\A \to \A \otimes_{\A} \Om^1_D= \Om^1_D
	$$
	and $G$-equivariant connections
	$$
	\widetilde{\nabla}:	\widetilde{\A} \to \widetilde{\A} \otimes_{\A} \Om^1_D.
	$$
	It follows that thee is the unique $G$-equivariant $\widetilde{\nabla}$ connection which corresponds to 
	\begin{equation*}
	\begin{split}
	\nabla:	\A \to \A \otimes_{\A} \Om^1_D= \Om^1_D,\\
	a \mapsto\left[ D, a\right]. 
	\end{split}
	\end{equation*}
	
\end{proof}
%The given by \eqref{d_defn} operator $\widetilde{D}$ satisfies to conditions (b) and (c) of the Definition \ref{triple_pre_lift_defn}, so one has a following theorem.
%\begin{thm}\label{st_fin_thm}
%If the unital noncommutative finite-fold covering $\left(A, \widetilde{A}, G \right)$ is smoothly invariant, 
% $\widetilde{\A}$ is given by \eqref{a_smooth_eqn}, $\widetilde{\H} = \widetilde{A} \otimes_A \H$ and $\widetilde{D}$ is given by \eqref{d_defn}, then  above construction yields  the unique spectral triple $\left( \widetilde{\A}, \widetilde{\H}, \widetilde{D}\right)$ which is a  $\left(A, \widetilde{A}, G \right)$. 
%	\end{thm}
\begin{defn}\label{triple_conn_lift_defn}
	The operator $\widetilde{D}$ given by \eqref{d_defn} is said to be $\left(A, \widetilde{A}, G \right)$-\textit{lift} of $D$. The spectral triple $\left( \widetilde{\A}, \widetilde{\H}, \widetilde{D}\right)$  is said to be the  $\left(A, \widetilde{A}, G \right)$-\textit{lift} of $\left( \A, \H, D \right)$.
\end{defn}
%\begin{rem}
%	From \eqref{equiv_conn_eqn} it follows that the operator $\widetilde{D}$ is equivariant, i.e.
%	\be\label{equiv_d_eqn}
%	\widetilde{D}\left( g \widetilde{\xi}\right) = g \left( \widetilde{D}\widetilde{\xi}\right);~ \forall \widetilde{\xi} \in \Dom \widetilde{G}, ~ \forall g \in G.
%	\ee
%	There are two equivalent ways of definition of operator $\widetilde{D}$ from the Theorem  \ref{st_fin_thm}:
%	\begin{enumerate}
%		\item [(a)] Looking for an operator $\widetilde{ D}$ which satisfies to the Definition \ref{triple_lift_defn},
%		\item[(b)] Application of the equation \eqref{d_defn}.
%	\end{enumerate}
%\end{rem}

 \section{Noncommutative infinite coverings}
 \subsection{Basic construction}\label{bas_constr}

 This section contains a noncommutative generalization of infinite coverings.
 \begin{definition}\label{comp_defn}
 	Let
 	\begin{equation*}
 	\mathfrak{S} =\left\{ A =A_0 \xrightarrow{\pi_1} A_1 \xrightarrow{\pi_2} ... \xrightarrow{\pi_n} A_n \xrightarrow{\pi^{n+1}} ...\right\}
 	\end{equation*}
 	be a sequence of $C^*$-algebras and noncommutative finite-fold coverings such that:
 	\begin{enumerate}
 		\item[(a)] Any composition $\pi_{n_1}\circ ...\circ\pi_{n_0+1}\circ\pi_{n_0}:A_{n_0}\to A_{n_1}$ corresponds to the noncommutative covering $\left(A_{n_0}, A_{n_1}, G\left(A_{n_1}~|~A_{n_0}\right)\right)$;
 		\item[(b)] If $k < l < m$ then $G\left( A_m~|~A_k\right)A_l = A_l$ (Action of $G\left( A_m~|~A_k\right)$ on $A_l$ means that $G\left( A_m~|~A_k\right)$ acts on $A_m$, so $G\left( A_m~|~A_k\right)$ acts on $A_l$ since $A_l$ a subalgebra of $A_m$);
 		\item[(c)] If $k < l < m$ are nonegative integers then there is the natural exact sequence of covering transformation groups
 		\begin{equation*}
 		\{e\}\to G\left(A_{m}~|~A_{l}\right) \xrightarrow{\iota} G\left(A_{m}~|~A_{k}\right)\xrightarrow{\pi}G\left(A_{l}~|~A_{k}\right)\to\{e\}
 		\end{equation*}
 		where the existence of the homomorphism $G\left(A_{m}~|~A_{k}\right)\xrightarrow{\pi}G\left(A_{l}~|~A_{k}\right)$ follows from (b).
 		
 	\end{enumerate}
 	The sequence
 	$\mathfrak{S}$
 	is said to be an \textit{(algebraical)  finite covering sequence}. 
 	For any finite covering sequence we will use the notation $\mathfrak{S} \in \mathfrak{FinAlg}$.
 \end{definition}
 \begin{definition}\label{equiv_act_defn}
 	Let $\widehat{A} = \varinjlim A_n$  be the $C^*$-inductive limit \cite{murphy}, and suppose that $\widehat{G}= \varprojlim G\left(A_n~|~A \right) $ is the projective limit of groups \cite{spanier:at}. There is the natural action of $\widehat{G}$ on $\widehat{A}$. A non-degenerate faithful representation $\widehat{A} \to B\left( \H\right) $ is said to be \textit{equivariant} if there is an action of $\widehat{G}$ on $\H$ such that for any $\xi \in \H$ and $g \in  \widehat{G}$ the following condition holds
 	\begin{equation}\label{equiv_act_eqn}
 	\left(ga \right) \xi = g\left(a\left(g^{-1}\xi \right)  \right) .
 	\end{equation}
 \end{definition}
\begin{example}
	Let $\rho:\widehat{A} \to B\left(\H \right)$ be a non-degenerate faithful representation, let $\H_g \approx \H$ and let $\rho_g: \widehat{A} \to B\left(\H_g \right)$ is given by $\rho_g\left(a \right) = \rho\left( ga\right)$, for any $g \in \widehat{G}$. There is the natural action of $\widehat{G}$ on $\bigoplus_{g \in  \widehat{G}} \H_g$ which transopses summands of the direct sum and
	$$
	\bigoplus_{g \in  \widehat{G}} \rho_g: \widehat{A} \to B\left(	\bigoplus_{g \in  \widehat{G}} \H_g \right) 
	$$ 
	is an equivariant representation.
\end{example}
%\begin{example}
%	Let $S$ be the state space of $\widehat{A}$, and let $\widehat{A} \to B\left(\bigoplus_{s \in S} \H_s \right)$ be the universal representation. There is the natural action of $\widehat{G}$ on $S$ given by
%	$$
%	\left(gs \right)\left(  a\right)  = s\left( ga\right); ~ s \in S,~ a \in \widehat{A},~ g \in \widehat{G}.
%	$$
%	The action of $\widehat{G}$ on $S$ induces the action of $\widehat{G}$ on $\bigoplus_{s \in S} \H_s$. It follows that the universal representation is equivariant.
%\end{example}
%\begin{example}\label{equiv_exm}
%	Let $s$ be a faithful state which corresponds to the representation $\widehat{A} \to B\left(\H_s \right)$  and $\left\{g_n\in \widehat{G}\right\}_{n \in \N}= \widehat{G}$  is a bijection. The state
%	$$
%	\sum_{n \in \N}\frac{g_ns }{2^{n}}
%	$$
%	corresponds to an equvariant representation $\widehat{A} \to B\left(\bigoplus_{g \in \widehat{G}}\H_{gs} \right)$.
%\end{example}

 \begin{definition}\label{special_el_defn}
 	Let $\pi:\widehat{A} \to B\left( \H\right) $ be an equivariant representation.  A positive element  $\overline{a}  \in B\left(\H \right)_+ $ is said to be \textit{special} (with respect to $\pi$) if following conditions hold:
 	\begin{enumerate}
 		\item[(a)] For any $n \in \mathbb{N}^0$  the following  series 
 		\begin{equation*}
 		\begin{split}
 		a_n = \sum_{g \in \ker\left( \widehat{G} \to  G\left( A_n~|~A \right)\right)} g  \overline{a} 
 		\end{split}
 		\end{equation*}
 		is strongly convergent and the sum lies in   $A_n$, i.e. $a_n \in A_n $;		
 		\item[(b)]
 		If $f_\eps: \R \to \R$ is given by 
 		\begin{equation}\label{f_eps_eqn}
 		f_\eps\left( x\right)  =\left\{
 		\begin{array}{c l}
 		0 &x \le \eps \\
 		x - \eps & x > \eps
 		\end{array}\right.
 		\end{equation}
 		then for any $n \in \mathbb{N}^0$ and for any $z \in A$   following  series 
 		\begin{equation*}
 		\begin{split}
 		b_n = \sum_{g \in \ker\left( \widehat{G} \to  G\left( A_n~|~A \right)\right)} g \left(z  \overline{a} z^*\right) ,\\
 		c_n = \sum_{g \in \ker\left( \widehat{G} \to  G\left( A_n~|~A \right)\right)} g \left(z  \overline{a} z^*\right)^2,\\
 		d_n = \sum_{g \in \ker\left( \widehat{G} \to  G\left( A_n~|~A \right)\right)} g f_\eps\left( z  \overline{a} z^* \right) 
 		\end{split}
 		\end{equation*}
 		are strongly convergent and the sums lie in   $A_n$, i.e. $b_n,~ c_n,~ d_n \in A_n $; 
 		\item[(c)] For any $\eps > 0$ there is $N \in \N$ (which depends on $\overline{a}$ and $z$) such that for any $n \ge N$ the following condition holds
 		\begin{equation}\label{square_condition_equ}
 		\begin{split}
 		\left\| b_n^2 - c_n\right\| < \eps.
 		\end{split}
 		\end{equation}	
 	\end{enumerate}
 	
 	An element  $\overline{   a}' \in B\left( \H\right) $ is said to be \textit{weakly special} if 
 	$$
 	\overline{   a}' = x\overline{a}y; \text{ where }   x,y \in \widehat{A}, \text{ and } \overline{a} \in B\left(\H \right)  \text{ is special}.
 	$$
 	
 \end{definition}
 \begin{lem}\label{stong_conv_inf_lem}\cite{ivankov:qnc}
 	If $\overline{a} \in B\left( \H\right)_+$ is a special element and ${G}_n=\ker\left( \widehat{G} \to  G\left( A_n~|~A \right)\right)$ then from
 	\begin{equation*}
 	\begin{split}
 	a_n = \sum_{g \in {G}_n} g \overline{a},
 	\end{split}
 	\end{equation*}
 	it follows that $\overline{a} = \lim_{n \to \infty} a_n$ in the sense of the strong convergence. Moreover  one has $\overline{a} =\inf_{n \in \N}a_n$.
 \end{lem}

 \begin{cor}\label{special_cor}\cite{ivankov:qnc}
 	Any weakly special element lies in the enveloping von Neumann algebra $\widehat{A}''$ of $\widehat{A}=\varinjlim A_n$. If $\overline{A}_\pi \subset B\left( \H\right)$ is the $C^*$-norm completion of an algebra generated by weakly special elements then $\overline{A}_\pi \subset \widehat{A}''$.
 \end{cor}
 \begin{lem}\cite{ivankov:qnc}
 	If $\overline{a}\in B\left( \H\right)$ is special, (resp.  $\overline{a}'\in B\left( \H\right)$ weakly special) then for any  $g \in \widehat{G}$ the element  $g\overline{a}$ is special, (resp. $g\overline{a}'$ is weakly special).
 \end{lem}
 \begin{cor}\label{disconnect_group_action_cor}\cite{ivankov:qnc}
 	If $\overline{A}_\pi \subset B\left( \H\right)$ is the $C^*$-norm completion of algebra generated by weakly special elements, then there is a natural action of $\widehat{G}$ on $\overline{A}_\pi$.
 \end{cor}

 \begin{definition}\label{main_defn_full}
 	Let $\mathfrak{S} =\left\{ A =A_0 \xrightarrow{\pi^1} A_1 \xrightarrow{\pi^2} ... \xrightarrow{\pi^n} A_n \xrightarrow{\pi^{n+1}} ...\right\}$ be an  algebraical  finite covering sequence. Let  $\pi:\widehat{A} \to B\left( \H\right) $ be an equivariant representation.	 Let $\overline{A}_\pi \subset B\left( \H\right)$ be the $C^*$-norm completion of algebra generated by weakly special elements. We say that $\overline{A}_\pi$ is the {\it disconnected inverse noncommutative limit} of $\downarrow\mathfrak{S}$ (\textit{with respect to $\pi$}). 
 	The triple  $\left(A, \overline{A}_\pi, G\left(\overline{A}_\pi~|~ A\right)\stackrel{\mathrm{def}}{=} \widehat{G}\right)$ is said to be the  {\it disconnected infinite noncommutative covering } of $\mathfrak{S}$ (\textit{with respect to $\pi$}). If $\pi$ is the universal representation then "with respect to $\pi$" is dropped and we will write
 	$\left(A, \overline{A}, G\left(\overline{A}~|~ A\right)\right)$.	
 \end{definition}
 \begin{definition}\label{main_sdefn}
 	Any  maximal irreducible subalgebra $\widetilde{A}_\pi \subset \overline{A}_\pi$  is said to be a {\it connected component} of $\mathfrak{S}$ ({\it with respect to $\pi$}). The maximal subgroup $G_\pi\subset G\left(\overline{A}_\pi~|~ A\right)$ among subgroups $G\subset G\left(\overline{A}_\pi~|~ A\right)$ such that $G\widetilde{A}_\pi=\widetilde{A}_\pi$  is said to be the $\widetilde{A}_\pi$-{\it  invariant group} of $\mathfrak{S}$. If $\pi$ is the universal representation then "with respect to $\pi$" is dropped. 
 \end{definition}
 
 \begin{rem}
 	From the Definition \ref{main_sdefn} it follows that $G_\pi \subset G\left(\overline{A}_\pi~|~ A\right)$ is a normal subgroup.
 \end{rem}
 \begin{definition}\label{good_seq_defn} Let $$\mathfrak{S} = \left\{ A =A_0 \xrightarrow{\pi^1} A_1 \xrightarrow{\pi^2} ... \xrightarrow{\pi^n} A_n \xrightarrow{\pi^{n+1}} ...\right\} \in \mathfrak{FinAlg},$$ and let $\left(A, \overline{A}_\pi, G\left(\overline{A}_\pi~|~ A\right)\right)$ be a disconnected infinite noncommutative covering of $\mathfrak{S}$ with respect to an equivariant representation $\pi: \varinjlim A_n\to B\left(\H \right) $. Let $\widetilde{A}_\pi\subset \overline{A}_\pi$  be a connected component of $\mathfrak{S}$ with respect to $\pi$, and let $G_\pi \subset  G\left(\overline{A}_\pi~|~ A\right)$  be the $\widetilde{A}_\pi$ -  invariant group of $\mathfrak{S}$.
 	Let  $h_n : G\left(\overline{A}_\pi~|~ A\right) \to  G\left( A_n~|~A \right)$ be the natural surjective  homomorphism. The  representation $\pi: \varinjlim A_n\to B\left(\H \right)$ is said to be \textit{good} if it satisfies to following conditions:
 	\begin{enumerate}
 		\item[(a)] The natural *-homomorphism $ \varinjlim A_n \to  M\left(\widetilde{A}_\pi \right)$ is injective,
 		\item[(b)] If $J\subset G\left(\overline{A}_\pi~|~ A\right)$ is a set of representatives of $G\left(\overline{A}_\pi~|~ A\right)/G_\pi$, then the algebraic direct sum
 		\begin{equation*}
 		\bigoplus_{g\in J} g\widetilde{A}_\pi
 		\end{equation*}
 		is a dense subalgebra of $\overline{A}_\pi$,
 		\item [(c)] For any $n \in \N$ the restriction $h_n|_{G_\pi}$ is an epimorphism, i. e. $h_n\left(G_\pi \right) = G\left( A_n~|~A \right)$.
 	\end{enumerate}
 	If $\pi$ is the universal representation we say that $\mathfrak{S}$ is \textit{good}.
 \end{definition}

 \begin{definition}\label{main_defn}
 	Let $\mathfrak{S}=\left\{A=A_0 \to A_1 \to ... \to A_n \to ...\right\} \in \mathfrak{FinAlg}$ be  an algebraical  finite covering sequence. Let $\pi: \widehat{A} \to B\left(\H \right)$ be a good representation.    A connected component $\widetilde{A}_\pi \subset \overline{A}_\pi$  is said to be the {\it inverse noncommutative limit of $\downarrow\mathfrak{S}$ (with respect to $\pi$)}. The $\widetilde{A}_\pi$-invariant group $G_\pi$  is said to be the {\it  covering transformation group of $\mathfrak{S}$} ({\it with respect to $\pi$}).  The triple $\left(A, \widetilde{A}_\pi, G_\pi\right)$ is said to be the  {\it infinite noncommutative covering} of $\mathfrak{S}$  ({\it with respect to $\pi$}). %The algebra $A$  is said to be the {\it base algebra} of $\mathfrak{S}$. 
 	We will use the following notation 
 	\begin{equation*}
 	\begin{split}
 	\varprojlim_\pi \downarrow \mathfrak{S}\stackrel{\mathrm{def}}{=}\widetilde{A}_\pi,\\
 	G\left(\widetilde{A}_\pi~|~ A\right)\stackrel{\mathrm{def}}{=}G_\pi.
 	\end{split}
 	\end{equation*}	
 	If $\pi$ is the universal representation then "with respect to $\pi$" is dropped and we will write $\left(A, \widetilde{A}, G\right)$, $~\varprojlim \downarrow \mathfrak{S}\stackrel{\mathrm{def}}{=}\widetilde{A}$ and  $ G\left(\widetilde{A}~|~ A\right)\stackrel{\mathrm{def}}{=}G$.
 \end{definition}
 %\begin{rem}
 %	Above definition does not depend on choice of $\widetilde{A}_\pi$ up to isomorphism.
 %\end{rem}
 
  \begin{definition}\label{inf_hilb_prod_defn}
 	Let $\mathfrak{S}=\left\{A=A_0 \to A_1 \to ... \to A_n \to ...\right\} \in \mathfrak{FinAlg}$ be  an algebraical  finite covering sequence. Let $\pi: \widehat{A} \to B\left(\H \right)$ be a good representation.   Let $\left(A, \widetilde{A}_\pi, G_\pi\right)$  be the  infinite noncommutative covering of $\mathfrak{S}$  ( with respect to $\pi$). Let $K\left( \widetilde{A}_\pi\right)$ be the Pedersen ideal of  $\widetilde{A}_\pi$. We say that $\mathfrak{S}$ \textit{allows inner product (with respect to $\pi$)} if following conditions hold
 	\begin{enumerate}
 \item[(a)]  Any $\widetilde{a} \in K\left( \widetilde{A}_\pi\right)$ is weakly special,
 \item[(b)]
 	For any $n \in \N$, and $\widetilde{a}, \widetilde{b} \in K\left( \widetilde{A}_\pi\right)$ the series
 		\begin{equation*}
 	\begin{split}
 a_n = \sum_{g \in \ker\left( \widehat{G} \to  G\left( A_n~|~A \right)\right)} g \left(\widetilde{a}^* \widetilde{b}  \right) 
 	\end{split}
 	\end{equation*}
 	is strongly convergent and $a_n \in A_n$.
 	\end{enumerate}
 	
 \end{definition}

\begin{rem}\label{inf_hilb_prod_rem}
	If $\mathfrak{S}$ allows  inner product (with respect to $\pi$) then $K\left( \widetilde{A}_\pi\right)$ is a pre-Hilbert $A$ module such that the inner product is given by
 		\begin{equation*}
\begin{split}
\left\langle \widetilde{a}, \widetilde{b}  \right\rangle  = \sum_{g \in  \widehat{G}} g \left(\widetilde{a}^* \widetilde{b}  \right) \in A 
\end{split}
\end{equation*}
where the above series is strongly convergent. The completion of  $K\left( \widetilde{A}_\pi\right)$ with respect to a norm
\begin{equation*}
\begin{split}
\left\| \widetilde{a}\right\| = \sqrt{\left\| \left\langle \widetilde{a}, \widetilde{a}  \right\rangle\right\|}
\end{split}
\end{equation*}
is an $A$-Hilbert module. Denote by $X_A$ this completion. The ideal $K\left( \widetilde{A}_\pi\right)$ is a left $\widetilde{A}_\pi$-module, so $X_A$ is also $\widetilde{A}_\pi$-module. Sometimes we will write  $_{\widetilde{A}_\pi}X_A$ instead  $X_A$.
	
\end{rem}
   \begin{definition}\label{inf_hilb_mod_defn}
 	Let $\mathfrak{S}=\left\{A=A_0 \to A_1 \to ... \to A_n \to ...\right\} \in \mathfrak{FinAlg}$ and  $\mathfrak{S}$ allows inner product (with respect to $\pi$) then $K\left( \widetilde{A}_\pi\right)$ then we say that given by the Remark  \ref{inf_hilb_prod_rem} $A$-Hilbert module $_{\widetilde{A}_\pi}X_A$ \textit{corresponds to the pair} $\left(\mathfrak{S}, \pi \right) $. If $\pi$ is the universal representation then we say that $_{\widetilde{A}}X_A$ \textit{corresponds to } $\mathfrak{S}$.
  \end{definition}

\subsection{Induced representation}\label{inf_ind_repr_subsection}

\paragraph*{} 	Let $\pi:\widehat{A} \to B\left(\overline{\H}_\pi \right)$ be a good representation. Let  $\left(A, \widetilde{A}_\pi, G_\pi\right)$ be an infinite noncommutative covering  with respect to $\pi$ of $\mathfrak{S}$. Denote by $\overline{W}_\pi \subset B\left(\overline{\H}_\pi \right)$ the $\widehat{A}$-bimodule of weakly special elements, and denote by
\begin{equation}\label{wealky_spec_eqn}
\widetilde{W}_\pi = \overline{W}_\pi \bigcap  \widetilde{A}_\pi.
\end{equation}
If $\pi$ is the universal representation then we write $\widetilde{W}$ instead $\widetilde{W}_\pi$.

\begin{lem}\label{w_conv_lem}\cite{ivankov:qnc}
	If $\widetilde{a}, \widetilde{b} \in \widetilde{W}_\pi $ are weakly special elements  then a series
	$$
	\sum_{g \in G_\pi} g\left(\widetilde{a}^*\widetilde{b} \right)
	$$
	is strongly convergent.
\end{lem}
\begin{defn}\label{ss_defn}
	Element $\widetilde{a} \in \widetilde{A}_\pi$ is said to be \textit{square-summable} if the series 
	\be\label{ss_eqn}
	\sum_{g \in G_\pi} g\left(\widetilde{a}^*\widetilde{a} \right)
	\ee
	is strongly convergent to a bounded operator. Denote by $L^2\left(\widetilde{A}_\pi \right)$ (or 
$L^2\left(\widetilde{A}\right)$) the $\C$-space of square-summable operators.
\end{defn}
\begin{rem}
	If $\widetilde{b} \in \widehat{A}$, and $\widetilde{a}\in L^2\left(\widetilde{A}\right)$ then
	$$
	\left\|	\sum_{g \in G_\pi}g\left(\widetilde{b}\widetilde{a} \right)^* \left(\widetilde{b}\widetilde{a}\right) \right\| \le \left\|\widetilde{b} \right\|^2\left\| 	\sum_{g \in G_\pi} g\left(\widetilde{a}^*\widetilde{a} \right)\right\| ,~~
	\left\|	\sum_{g \in G_\pi}g\left(\widetilde{a}\widetilde{b} \right)^* \left(\widetilde{a}\widetilde{b}\right) \right\| \le \left\|\widetilde{b} \right\|^2\left\| 	\sum_{g \in G_\pi} g\left(\widetilde{a}^*\widetilde{a} \right)\right\| 
	$$
	it turns out
	\be\label{act_on_l2_eqn}
	\widehat{A}L^2\left(\widetilde{A}_\pi\right) \subset L^2\left(\widetilde{A}_\pi\right),~~
		L^2\left(\widetilde{A}_\pi\right)\widehat{A} \subset L^2\left(\widetilde{A}_\pi\right),
	\ee
	i.e. there is the left and right action of $\widehat{A}$ on $L^2\left(\widetilde{A}\right)$.
\end{rem}
\begin{lem}
	If $\widetilde{a}, \widetilde{b} \in L^2\left(\widetilde{A}_\pi \right)$ are square-summable then the series	$\sum_{g \in G_\pi} g\left(\widetilde{a}^*\widetilde{b} \right)$ is weakly convergent. Moreover if
	$a =\sum_{g \in G_\pi} g\left(\widetilde{a}^*\widetilde{b} \right)$  is the sum of the series in the weak topology then $a \in A''$.
\end{lem}
\begin{proof}
	If $\xi, \eta \in \widehat{\H}$ then following series
	\bean
	\sum_{g \in G_\pi}\left(\xi, g\left( \widetilde{a}^*\widetilde{a} \right) \xi\right)_{\overline{\H}_\pi} =  	\sum_{g \in G_\pi}\left(\left( g\widetilde{a}\right) \xi,\left( g\widetilde{a}\right) \xi\right)_{\overline{\H}_\pi}=\sum_{g \in G_\pi}\left\|\left( g\widetilde{a}\right) \xi \right\|^2,\\
		\sum_{g \in G_\pi}\left(\eta, g\left( \widetilde{b}^*\widetilde{b} \right) \eta\right)_{\overline{\H}_\pi} =  	\sum_{g \in G_\pi}\left(\left( g\widetilde{b}\right) \eta,\left( g\widetilde{b}\right) \eta\right)_{\overline{\H}_\pi}=\sum_{g \in G_\pi}\left\|\left( g\widetilde{b}\right) \eta \right\|^2
	\eean
\end{proof}
are convergent. From the Cauchy–Schwarz inequality it turns out that the series
	\bean
\sum_{g \in G_\pi}\left\|\left( g\widetilde{a}\right) \xi \right\|\left\|\left( g\widetilde{b}\right) \eta \right\| \le \sqrt{\sum_{g \in G_\pi}\left\|\left( g\widetilde{a}\right) \xi \right\|^2}\sqrt{\sum_{g \in G_\pi}\left\|\left( g\widetilde{b}\right) \eta \right\|^2}< \infty
\eean
is convergent, and taking into account $$\left|\left(\xi, \left( g\left( \widetilde{a}^*\widetilde{b}\right) \right) \eta\right)_{\overline{\H}_\pi}  \right|=\left|\left(\left( g\widetilde{a}\right) \xi,\left( g\widetilde{b}\right)  \eta\right)_{\overline{\H}_\pi}  \right|\le\left\| \left( g\widetilde{a}\right) \xi \right\|\left\|\left( g\widetilde{b}\right) \eta \right\|$$ we conclude that the series
$$
\sum_{g \in G_\pi}  \left(\xi, \left( g\left( \widetilde{a}\widetilde{b}\right) \right) \eta\right)_{\overline{\H}_\pi}  
$$
is absolutely convergent, or, equivalently the series $\sum_{g \in G_\pi} g\left(\widetilde{a}^*\widetilde{b} \right)$ is weakly convergent. Element  $a$ is $G_\pi$-invariant, so $a \in A''$.

\begin{empt}\label{inf_repr_constr}
	Let $A \to B\left(\H \right)$ be a  representation.  Denote by $\widetilde{\H}$ a Hilbert completion of a pre-Hilbert space
	\begin{equation}\label{inf_ind_prod_eqn}
	\begin{split}
	L^2\left( \widetilde{A}_\pi \right) \otimes_A \H,\\
	\text{with a scalar product} 	\left(\widetilde{a} \otimes \xi, \widetilde{b} \otimes \eta \right)_{\widetilde{\H}} = \left( \xi, \left( \sum_{g \in G_\pi } g \left( \widetilde{a}^*\widetilde{b}\right) \right) \eta \right)_{\H}.
	\end{split}
	\end{equation}  
	There is the left action of $\widehat{A}$ on $L^2\left(\widetilde{A}_\pi\right) \otimes_{A} \H$ given by
\be\label{inf_ind_act_eqn}
	\widetilde{b}\left(\widetilde{a} \otimes \xi \right) = \widetilde{b}\widetilde{a} \otimes \xi
\ee
	where $\widetilde{a} \in 	L^2\left( \widetilde{A}_\pi \right) $, $\widetilde{b} \in \widehat{A}$, $\xi \in \H$.	
	The left action of $\widehat{A}$ on $L^2\left( \widetilde{A}_\pi\right)  \otimes_A \H$  induces following  representations
	\begin{equation*}
	\begin{split}
	\widehat{\rho}:\widehat{A} \to B\left( \widetilde{\H}\right),\\	
	\widetilde{\rho}:\widetilde{A}_\pi \to B\left( \widetilde{\H}\right).	
	\end{split}
	\end{equation*}

\end{empt}
\begin{defn}\label{inf_ind_defn}
	The constructed in \ref{inf_repr_constr} representation  $\widetilde{\rho}:\widetilde{A}_\pi \to  B\left( \widetilde{\H}\right)$ is said to be \textit{induced} by  $\left( \rho, \mathfrak{S}, \pi\right)  $. We also say that  $\widetilde{\rho}$ is  \textit{induced} by $\left( \rho, \left( A, \widetilde{A}_\pi, G\left(\widetilde{A}_\pi~|~A \right) \right), \pi\right) $. If $\pi$ is an universal representation we say that  $\widetilde{\rho}$ is  \textit{induced} by  $\left( \rho, \mathfrak{S}\right)$ and/or $\left( \rho, \left( A, \widetilde{A}, G\left(\widetilde{A}~|~A \right) \right)\right) $.  
\end{defn}
\begin{rem}
	If $\rho$ is faithful, then  $\widetilde{\rho}$ is faithful.
\end{rem}
\begin{rem}\label{a_act_hilb_rem} 
	There is an action of $G_\pi$ on $\widetilde{\H}$ induced by the natural action of $G_\pi$ on the $\widetilde{A}_\pi$-bimodule $L^2\left( \widetilde{A}_\pi\right) $. If the representation $\widetilde A_\pi \to 	B\left( \widetilde{\H} \right)$ is faithful then an action of $ G_\pi$ on $\widetilde A_\pi$ is given by
	$$ 
	\left( g  \widetilde a\right) \xi =   g \left(  \widetilde a  \left(  g^{-1}\widetilde\xi\right) \right); ~ \forall  g  \in {G}, ~ \forall\widetilde a  \in \widetilde{A}_\pi, ~\forall\widetilde \xi \in \widetilde{\H}.
	$$
\end{rem}
\begin{empt}
If $\mathfrak{S}$ allows inner product with respect to $\pi$ then for any representation $A \to B\left( \H\right)$ an algebraic tensor product $_{\widetilde{A}_\pi}X_A \otimes_A \H$ is a pre-Hilbert space with the product given by
\begin{equation*}
\left(a \otimes \xi, b \otimes \eta \right) = \left(\xi, \left\langle a, b \right\rangle\eta \right) 
\end{equation*} 
(cf. Definitions \ref{inf_hilb_mod_defn} and \ref{inf_hilb_prod_defn})
\end{empt}

\begin{lem}\cite{ivankov:qnc}
Suppose $\mathfrak{S}$ allows inner product with respect to $\pi$ and any $\widetilde{a} \in K\left( \widetilde{A}_\pi\right)$ is weakly special. If $\widetilde{\H}$ (resp. $\widetilde{\H}'$) is a Hilbert norm completion of 	$W_\pi \otimes_{A} \H$ (resp. $_{\widetilde{A}_\pi}X_A \otimes_A \H$) then there is the natural isomorphism $\widetilde{\H} \cong \widetilde{\H}'$.
\end{lem}

\begin{empt}\label{h_n_to_h_constr}
	Let $\H_n$ be a Hilbert completion of $A_n \otimes_A \H$ which is constructed in the section \ref{induced_repr_fin_sec}. Clearly
	\begin{equation}\label{tensor_n_equ}
L^2\left(\widetilde{A}_\pi\right)\otimes_{A_n} \H_n = 	L^2\left(\widetilde{A}_\pi\right) \otimes_{A_n} \left( A_n \otimes_A \H\right) = L^2\left(\widetilde{A}_\pi\right) \otimes_A \H.
	\end{equation}
	
	%	From the above equation it turns out the natural action of $B\left( \H_n\right)$ on $X_\pi =   W_\pi \otimes_A \H$. The space $Y \otimes_A \H$ is dense in $\widetilde{\H}$, the action can be uniquely extended up to action of $B\left( \H_n\right)$ on $B\left(\widetilde{ \H}\right)$, i.e. there is a natural homomorphism
	%\begin{equation*}
	%B\left( \H_n\right) \hookto B\left(\widetilde{ \H}\right).
	%\end{equation*}
\end{empt}

 \subsection{Coverings of spectral triples}
 \paragraph*{}
 \begin{defn}\label{cov_sec_triple_defn}
 	Let  $\left(\sA, \sH, D\right)$ be a spectral triple, and let $A$ be the $C^*$-norm completion of $\A$ with the natural representation $A \to B\left( \H\right)$. Let

 	\begin{equation}\label{cov_sec_triple_eqn}
 	\mathfrak{S} =\left\{ A =A_0 \xrightarrow{\pi_1} A_1 \xrightarrow{\pi_2} ... \xrightarrow{\pi_n} A_n \xrightarrow{\pi^{n+1}} ...\right\} \in \mathfrak{FinAlg}
 	\end{equation}
 	be a good algebraical  finite covering sequence. Suppose that for any $n > 0$ there is a spectral triple $\left(\sA_n, \sH_n, D_n\right)$, such that
 	\begin{itemize}
 		\item $A_n$ is the $C^*$-norm completion of $\A_n$,
 		\item There is a good representation $\pi:\varinjlim A_n \to B\left(\H_\pi\right)$, 
 		%	\item The representation $\pi$ is induced by the triple $\left(\rho, \mathfrak{S}, \pi \right)$, 
 		\item For any $k > l \ge 0$  the spectral triple $\left(\sA_k, \sH_k, D_k\right)$ is a $\left(A_l, A_k, G\left(A_k~|A_l \right)  \right)$-lift of $\left(\sA_l, \sH_l, D_l\right)$.
 		
 	\end{itemize}
 	We say that
 	\begin{equation}\label{spectral_triple_sec_eqn}
 	\begin{split}
 	\mathfrak{S}_{\left(\sA, \sH, D\right)}= \{\left(\sA, \sH, D\right)= \left(\sA_0, \sH_0, D_0\right), \left(\sA_1, \sH_1, D_1 \right), ...,\\  \left(\sA_n, \sH_n, D_n \right)\}
 	\end{split}
 	\end{equation}
 	is a \textit{coherent sequence of spectral triples}. We write $\mathfrak{S}_{\left(\sA, \sH, D\right)} \in \mathfrak{CohTriple}$.
 \end{defn}
 Let us consider a coherent sequence of spectral triples.  Let $\left(A, \widetilde{A}_\pi, G_\pi\right)$  be an   infinite noncommutative covering  (with respect to $\pi$) of $\mathfrak{S}$. Denote by $L^2\left(\widetilde{A}_\pi\right) \subset \widetilde{A}_\pi$ the space of square-summable elements, and denote by $J_n = \ker\left( G_\pi \to G\left( A_n~|A\right) \right)$.
 Let us consider a square-summable element $\widetilde{a}  \in L^2\left( \widetilde{A}_\pi\right) $ and denote by
 $$
 a_n = \sum_{g \in J_n} g \widetilde{a} \in A_n. 
 $$
 Let $\widetilde{\rho}:\widetilde{A}_\pi \to \widetilde{\H}$ be induced by $\left( \rho, \left( A, \widetilde{A}_\pi, G\left(\widetilde{A}_\pi~|~A \right) \right), \pi\right)$.
 %Let $A_n \to B\left(\H_n \right)$ be a representation  induced by a pair $\left(\rho, \left(A, A_n, G\left( A_n~|~A\right)  \right) \right)$. Similarly to \ref{s_repr} for any $s, n \in \N$ we consider representations $\pi^s_n: \A_n \to B\left(\H_n^{2^s}\right)$. From \ref{h_n_to_h_constr} it follows that $\pi^s_n$ can be regarded as element in $B\left(\overline \H^{2^s} \right)$
 \begin{defn}\label{smooth_el_defn}
 	Let us consider the above situation 
 	Let $\Om^1_D$ is the {module of differential forms associated} with the spectral triple  $\left( \A, \H, D\right)$ (cf. Definition \ref{ass_cycle_defn}).
 	A weakly special element $\widetilde{a}  \in \widetilde{A}_\pi$ is said to be $\mathfrak{S}_{\left(\sA, \sH, D\right)}$-\textit{smooth with respect to $\pi$} (or \textit{$\mathfrak{S}_{\left(\sA, \sH, D\right)}$-smooth} if $\pi$ is the universal representation) if following conditions hold:
 	\begin{enumerate}
 		\item[(a)] $a_n \in \A_n$ for any $n \in \N$.% following condition holds $$\left[D_n, a_n\right] \in \A_n \otimes_{\A} \
 		\item[(b)] For any $s \in \N$ the sequence $\left\lbrace 1_{M\left(\widetilde{A} \right) } \otimes \pi^s_n\left(a_n \right)\in B\left(\widetilde{\H}^{2^s} \right) \right\rbrace_{n \in \N}$ is strongly convergent. (The representation $\pi^s_n: \A_n \to B\left( \H^{2^s}_n \right)$ is given by \eqref{s_diff_repr_equ}).
 		\item[(c)] The sequence $\left\{1_{M\left(\widetilde{A}\right) } \otimes\left[D_n, a_n\right] \in B\left(\widetilde{\H} \right) \right\}_{n \in \N}$ is strongly convergent and 
 		$$
 		\lim_{n \to \infty} 1_{M\left(\widetilde{A}\right) } \otimes\left[D_n, a_n\right] \in L^2\left(\widetilde{A}_\pi\right) \otimes_{\A} \Om^1_D \subset B\left(\widetilde{\H}  \right), \text{ (cf. Remark \ref{dn_rem})}. 
 		$$
 		\item[(d)] The element $\widetilde{a}$ lies in the Pedersen ideal of $\widetilde{A}_\pi$, i.e. $\widetilde{a}  \in K\left( \widetilde{A}_\pi\right) $.
 	\end{enumerate}
 	Denote by $\widetilde{a}^s= \lim_{n \to \infty} 1_{M\left(\widetilde{A} \right) } \otimes \pi^s_n\left(a_n \right)$ in sense the strong convergence, and denote by $\widetilde{W}^\infty_\pi$ the space of smooth elements. If $\pi$ is the universal representation then we write $\widetilde{W}^\infty$ instead $\widetilde{W}^\infty_\pi$.
 \end{defn}
\begin{rem}\label{dn_rem}
	From \eqref{tensor_n_equ} it follows that  $1_{M\left(\widetilde{A} \right) } \otimes \pi^s_n\left(a_n \right)$ (resp. $1_{M\left(\widetilde{A}\right)} \otimes_{A_n}\left[D_n, a_n\right]$) can be regarded as an operator in  $B\left(\widetilde{\H}^{2^s}\right) $ (resp. $B\left(\widetilde{\H}\right)$).
\end{rem}
 \begin{empt}
 	There is a subalgebra $\widetilde{A}_{\text{smooth}} \subset \widetilde{A}_\pi$ generated by smooth elements. For any $s > 0$ there is a seminorm $\left\| \cdot \right\|_s$  on $\widetilde{A}_{\text{smooth}}$ given by
 	\begin{equation}\label{smooth_seminorms_eqn}
 	\left\| \widetilde{a}\right\|_s = \left\|\widetilde{a}^s \right\|=\left\|\lim_{n \to \infty}1_{M\left(\widetilde{A} \right) } \otimes \pi^s_n\left(a_n \right) \right\|.
 	\end{equation}
 \end{empt}
 \begin{defn}\label{smooth_alg_defn}
 	The completion of $\widetilde{A}_{\text{smooth}}$ in the topology induced by seminorms $\left\| \cdot \right\|_s$ is said to be a \textit{smooth algebra} of the coherent sequence \eqref{spectral_triple_sec_eqn} of spectral triples \textit{(with respect to $\pi$)}. This algebra is denoted by $\widetilde{\A}_\pi$. 
 We say that the sequence of spectral triple is \textit{good (with respect to $\pi$)} if $\widetilde{\A}_\pi$ is dense in $\widetilde{A}_\pi$. If $\pi$ is an universal representation then "with respect to $\pi$" is dropped and we write $\widetilde{\A}$ instead of $\widetilde{\A}_\pi$.
 \end{defn}
 \begin{empt}\label{dirac_inf_constr}
 	
 	For any $\widetilde{a}  \in \widetilde{W}^\infty_\pi $ we denote by
 \begin{equation}\label{a_d_eqn}
 	\widetilde{a}_D = \lim_{n \to \infty}1_{M\left(\widetilde{A}\right) } \otimes_{}\left[D_n, \sum_{g \in J_n}\widetilde{a}\right] = \sum_{j = 1}^k \widetilde{a}_D^j\otimes \om_j \in L^2\left(\widetilde{A}_\pi\right)\otimes_{\A} \Om^1_D.
  \end{equation}
 
 If $\H^\infty = \bigcap_{n = 0}^\infty \Dom D^n$ then for any $\widetilde{a} \otimes \xi \in \widetilde{W}^\infty_\pi \otimes_{\A} \H^\infty$ we denote by
 	\begin{equation}\label{inf_lift_D_eqn}
 	\widetilde{D}\left(\widetilde{a} \otimes \xi \right) \stackrel{\mathrm{def}}{=} \sum_{j = 1}^k \widetilde{a}_D^j\otimes \om_j \left(\xi \right)  +\widetilde{a} \otimes D\xi \in L^2\left(\widetilde{A}_\pi\right) \otimes_A \H,
 	\end{equation}
 	
 	i.e. $\widetilde{D}$ is a $\C$-linear map from $W^\infty_\pi \otimes_{\A} \H^\infty$ to $L^2\left(\widetilde{A}_\pi\right) \otimes_A \H$.
 	The space $\widetilde{W}^\infty_\pi \otimes_{\A} \H^\infty$ is dense in $\widetilde{\H}$, hence the operator $\widetilde{D}$ can be regarded as an unbounded operator on $\widetilde{\H}$.
 \end{empt}
 
 \begin{defn}\label{reg_triple_defn}
 		Let $\mathfrak{S} \in \mathfrak{FinAlg}$ is given by \eqref{cov_sec_triple_eqn}. Let $\pi: \varinjlim A_n \to B\left(\H_\pi \right)$ be a good representation.    Let $\left(A, \widetilde{A}_\pi, G_\pi\right)$  be the  infinite noncommutative covering  (with respect to $\pi$) of $\mathfrak{S}$.
 	Let \eqref{spectral_triple_sec_eqn} be a good coherent sequence of spectral triples (with respect to $\pi$). 
  Let $\widetilde{ D}$ be given by \eqref{inf_lift_D_eqn}.
 	We say that $\left( \widetilde{\A}, \widetilde{\H}, \widetilde{D}\right)$ is a $\left(A, \widetilde{A}_\pi, G_\pi\right)$-\textit{lift} of $\left(\A, \H, D\right)$.  
 \end{defn}

\section{Coverings of commutative spectral triples}
%\subsection{Commutative spectral triples}
\paragraph*{} The Spin-manifold is a Riemannian manifold $M$ with a linear Spin-bundle $S$ described in \cite{hajac:toknotes,varilly:noncom}. The bundle $S$ is Hermitian. Taking into account than any Riemannian manifold has a natural measure $\mu$, there is a Hilbert space  $L^2\left(M,S \right)=L^2\left(M,S, \mu \right)$ described in \ref{top_herm_bundle_constr}.
If $\Ga^\infty\left( M, S\right)$ is a $\Coo\left( M\right)$-module of smooth sections then there is a first order differential operator $\slashed D$ on  $\Ga^\infty\left( M, S\right)$. Locally $\slashed D$ is given by
\be\label{comm_dirac_eqn}
\slashed D \xi = \sum_{j = 1}^n \ga_j\left( x\right)  \frac{\partial}{\partial x_j}
\ee
where $x_j$ ($j =1,...,n$) are local coordinates on $M$, $\ga_j \in \End_{\Coo\left( M\right)}\left(\Ga^\infty\left( M, S\right) \right)$ are described in \cite{hajac:toknotes,varilly:noncom}. Since $\Ga^\infty\left( M, S\right)$ is a dense $\C$-subspace of $L^2\left(M,S \right)$ operator $\slashed D$ can be regarded as an unbounded operator $L^2\left(M,S \right)$. From \eqref{comm_dirac_eqn} it turns out that $\slashed D$ is a local operator (cf. Definition \ref{local_op_defn}). It is shown in \cite{hajac:toknotes,varilly:noncom} that  $\left(C^{\infty}(M), L^2(M, S),\slashed D\right)$ is a spectral triple. 	For any $s \in \N^0$ there is a representation of $\pi^s:\Coo\left({M} \right) \to B\left(  L^2\left(M,S \right)^{2^s}\right) $ given by  \eqref{s_diff_repr_equ}. 
If $\mathcal U \subset M$ is an open subset such that $S$ is trivial over $\mathcal U$, i.e. $S|_{\mathcal U} \cong \mathcal U \times\C^m$  and $a \in \Coo\left( M\right)$ such that $\supp~a \subset \mathcal U$ then that $\left[ {\slashed D}, {a}\right]$ corresponds to a continuous function $f: \mathcal U\to \mathbb{M}_{m} \left(\C \right)$. If $b \in C_b\left( M\right) $  corresponds to a continuous function $g:M\to \C$ then the product $\left[ {\slashed D}, {a}\right]b =fg$ is induced by point-wise product of complex matrix by complex number. The left multiplication of matrix by number coincides with the right one, so one has $fg=gf$, or equivalently 
\begin{equation}\label{comm_matr_x_eqn}
\begin{split}
\left[ {\slashed D}, {a}\right]b= b\left[ {\slashed D}, {a}\right].
\end{split}
\end{equation}
The equation \eqref{comm_matr_x_eqn} can be extended up to any $a \in \Coo\left( M\right)$.

\subsection{Finite-fold coverings}
\paragraph*{} This section contains an algebraic version of the Proposition \ref{comm_cov_mani} in case of finite-fold coverings.
\subsubsection{Coverings of $C^*$-algebras}
\paragraph*{} Following two theorems state equivalence between a topological notion of a covering and an algebraical one.
\begin{thm}\label{pavlov_troisky_thm}\cite{pavlov_troisky:cov}
	Suppose $\mathcal X$ and $\mathcal Y$ are compact Hausdorff connected spaces and $p :\mathcal  Y \to \mathcal X$
	is a continuous surjection. If $C(\mathcal Y )$ is a projective finitely generated Hilbert module over
	$C(\mathcal X)$ with respect to the action
	\begin{equation*}
	(f\xi)(y) = f(y)\xi(p(y)), ~ f \in  C(\mathcal Y ), ~ \xi \in  C(\mathcal X),
	\end{equation*}
	then $p$ is a finite-fold  covering.
\end{thm}
\begin{theorem}\label{comm_fin_thm}\cite{ivankov:qnc}
	If $\mathcal X$, $\widetilde{\mathcal X}$  are  locally compact Hausdorff connected spaces, and  $\pi: \widetilde{\mathcal X}\to \mathcal X$ is a surjective continuous map, then following conditions are equivalent:
	\begin{enumerate}
		\item [(i)] The map $\pi: \widetilde{\mathcal X}\to \mathcal X$ is a finite-fold regular covering,
		\item[(ii)] There is the natural  noncommutative finite-fold covering $\left(C_0\left(\mathcal  X \right), C_0\left(\widetilde{\mathcal X} \right), G    \right)$.
	\end{enumerate}
\end{theorem}
\

   \subsubsection{Topological coverings of spectral triples}
\paragraph*{} Let $\left(C^{\infty}\left( M\right) , L^2\left( M, S\right) ,\slashed D\right)$ be a commutative spectral triple, and let let $\left(C\left( M\right)  , \widetilde{A}, G\right) $ an unital noncommutative finite-fold  covering such that $\widetilde{A}$ is commutative $C^*$-algebra. From the Theorems \ref{gelfand-naimark} and \ref{pavlov_troisky_thm} it turns out that $\widetilde{A} = C\left(\widetilde{M} \right)$ and there is the natural finite-fold covering  $\pi:\widetilde{M} \to M$. From the Proposition \ref{comm_cov_mani} it follows that $\widetilde{M}$ has natural structure of the Riemannian manifold.  Denote by $\widetilde{S} = \pi^*S$ the inverse image of the Spin-bundle $S$ (cf. \ref{top_vb_sub_sub}). Similarly we can define the  inverse image  $\widetilde{\slashed D}= \pi^*\slashed D$ (cf. Definition \ref{inv_image_defn}), and $\left(C^{\infty}\left( \widetilde{M}\right) , L^2\left( \widetilde{M}, \widetilde{S}\right) ,\widetilde{\slashed D}\right)$ is a spectral triple. We would like to proof that $$\left(C^{\infty}\left( \widetilde{M}\right) , L^2\left( \widetilde{M}, \widetilde{S}\right) ,\widetilde{\slashed D}\right)$$ is the $\left(C\left(M \right) , C\left( \widetilde{M}\right) , G\left( \widetilde{M}~|~M\right)  \right)$-lift of $\left(C^{\infty}\left( M\right) , L^2\left( M, S\right) ,\slashed D\right)$.
\subsubsection{Induced representation}\label{induced_comm_finite}
\paragraph*{} Let us consider a family of open connected subsets $\left\lbrace \mathcal{U}_\iota \subset M\right\rbrace_{\iota \in I} $ such that
\begin{itemize}
	\item $\mathcal{U}_\iota$ is evenly covered by $\pi$,
	\item The bundle $S$ is trivial on  $\mathcal{U}_\iota$.
\end{itemize}
The space $M$ is compact, so there is a finite subfamily $\left\lbrace \mathcal{U}_\iota \subset M\right\rbrace_{\iota \in I}$ such that $M = \bigcup_{\iota \in I} \mathcal{U}_\iota$.

\begin{prop}\label{smooth_part_unity_prop}\cite{brickell_clark:diff_m}
	A differential manifold $M$ admits a (smooth) partition of unity if and only if it is paracompact. 
\end{prop}
From the Proposition \ref{smooth_part_unity_prop} it follows that there is a  finite family $\left\lbrace a_\iota \in \Coo\left(  M\right) \right\rbrace_{\iota \in I}$ such that
\begin{equation*}
\begin{split}
\supp ~a_\iota \subset  \mathcal{U}_\iota ,\\
1_{C\left( M\right) }= \sum_{\iota \in I}a_\iota.
\end{split}
\end{equation*}
For any $\iota \in I$ denote by $e_\iota = \sqrt{a_\iota}\in \Coo\left(  M\right)$  If $\xi \in \Ga^\infty\left(M,S \right)$ is a smooth section then
\begin{equation*}
\begin{split}
\xi = \sum_{\iota \in I} \xi_\iota, \text{ where } \xi_\iota = a_\iota \xi = e^2_\iota.
\end{split}
\end{equation*} 
For any $\iota \in I$ we can select an open connected subset $\widetilde{\mathcal{U}}_\iota \subset \widetilde{M}$ such that $\widetilde{ \mathcal{U}}_\iota$ is homeomorphically mapped onto $\mathcal{U}_\iota$. 
Let $\widetilde{a}_\iota = \mathfrak{lift}_{\widetilde{ \mathcal{U}}_\iota} \left(a_\iota \right) \in \Coo\left(\widetilde{M} \right)$, and let $\widetilde{e}_\iota = \mathfrak{lift}_{\widetilde{ \mathcal{U}}_\iota} \left(e_\iota \right)= \sqrt{\widetilde{a}_\iota}\in \Coo\left(\widetilde{M} \right)$ (cf. Definition \ref{lift_defn} ).
If $\widetilde{I} = G \times I$, $~\widetilde{e}_{\left(g, \iota \right) }= g \widetilde{e}_\iota$, for any $\left(g, \iota \right) \in \widetilde{I}$ then from the above construction it turns out
\be\label{comm_sum_1}
\begin{split}
1_{C\left(\widetilde{ M}\right) }= \sum_{g \in G}\sum_{\iota \in I}g\widetilde{e}^2_\iota= \sum_{\widetilde{\iota} \in \widetilde{I}}\widetilde{e}_{\widetilde{\iota}}^2~~,\\
\left( g\widetilde{e}_{\widetilde{\iota}}\right)  \widetilde{e}_{\widetilde{\iota}}= 0; \text{ for any nontrivial } g \in G.\\1_{C\left(\widetilde{ M}\right) } \sum_{g \in G}\sum_{\iota \in I}\widetilde{e}_{\widetilde{\iota}}\left\rangle \right\langle \widetilde{e}_{\widetilde{\iota}},\\
\left\langle \widetilde{e}_{\widetilde{\iota}'}, \widetilde{e}_{\widetilde{\iota}''} \right\rangle 	\in \Coo\left( M\right) 	
\end{split}
\ee
The set $\left\{\widetilde{e}_{\widetilde{\iota}}\right\}_{\widetilde{\iota}\in \widetilde{I}}$ is $G$-invariant, i.e.
\be\label{comm_a_g_inv}
G\left\{\widetilde{e}_{\widetilde{\iota}}\right\}_{\widetilde{\iota}\in \widetilde{I}}=\left\{\widetilde{e}_{\widetilde{\iota}}\right\}_{\widetilde{\iota}\in \widetilde{I}}
\ee

Similarly to the above construction any element in $\widetilde{\xi}\in \Ga^\infty\left(\widetilde{M},\widetilde{S} \right)$ can be represented as
\begin{equation}\label{cov_xi_repr}
\widetilde{\xi} = \sum_{g \in G} \sum_{\iota \in I} \left( g \widetilde{e}^2_\iota\right)  \widetilde{\xi}.
\end{equation}
From $\left(g \widetilde{e}_\iota \right)^2= g\widetilde{e}_\iota e_\iota$ and $$\supp ~\left( g\widetilde{e}_\iota\right)  \widetilde{\xi}\subset \supp~g\widetilde{e}_\iota \subset g \widetilde{\mathcal U}_\iota$$    we can establish a $\C$-linear isomorphism
\begin{equation}\label{comm_smooth_iso_eqn}
\begin{split}
\varphi:\Ga^\infty\left(\widetilde{M},\widetilde{S} \right) \xrightarrow{\approx} \Coo\left(\widetilde{M} \right)  \otimes_{\Coo\left( M\right)}  \Ga^\infty\left(M,S \right), \\
\widetilde{\xi}=\sum_{g \in G} \sum_{\iota \in I} \left( g \widetilde{e}_\iota\right)  \widetilde{e}_\iota \widetilde{\xi}  \mapsto \sum_{g \in G} \sum_{\iota \in I} g \widetilde{e}_\iota \otimes \mathfrak{desc}\left(\left( g\widetilde{e}_\iota\right)  \widetilde{\xi}\right) 
\end{split}
\end{equation}%\label{comm_fin_tens_eqn}
where $\mathfrak{desc}$ means $\pi$-descent (cf. Definition \ref{lift_defn}). Since $\Coo\left(\widetilde{M} \right) $ is dense in $C\left(\widetilde{M} \right) $, and $\Ga^\infty\left(M,S \right)$, (resp. $\Ga^\infty\left(\widetilde{M},\widetilde{S} \right)$) is dense in $L^2\left(M,S \right)$, (resp. $L^2\left(\widetilde{M},\widetilde{S} \right)$) the isomorphism \eqref{comm_smooth_iso_eqn} can be uniquely extended up to $\C$-isomorphism
\be\label{comm_tensor_iso_eqn}
\varphi:L^2\left(\widetilde{M},\widetilde{S} \right) \xrightarrow{\approx} C\left(\widetilde{M} \right)  \otimes_{C\left( M\right)}  L^2\left(M,S \right).
\ee
Above formula coincides with construction \ref{induced_repr_constr} of induced representation. If $\widetilde{a} \otimes \xi, \widetilde{b} \otimes \eta \in C\left(\widetilde{M} \right) \otimes_{C\left(M \right) } \Ga\left( M, S\right) \subset L^2\left(\widetilde{M},\widetilde{S} \right)$, $\mu$ (resp. $\widetilde{\mu}$) is the Riemannian measure (cf.  \cite{do_carmo:rg}) on $M$, (resp. $\widetilde{M}$) then
\begin{equation*}
\begin{split}
\int_{M}\left(\xi,\left\langle \widetilde{a}, \widetilde{b}\right\rangle_{C\left(\widetilde{M} \right) } \eta \right)\left( x\right)d\mu =\left(\xi, \left\langle \widetilde{a}, \widetilde{b}\right\rangle_{C\left(\widetilde{M} \right) } \eta \right)_{L^2\left({M},{S} \right)} = \sum_{\widetilde{\iota}\in \widetilde{I}}\left(\xi, \left\langle \widetilde{a}_{\widetilde{\iota}}\widetilde{a}, \widetilde{b}\right\rangle_{C\left(\widetilde{M} \right) } \eta \right)_{L^2\left({M},{S} \right)}=\\=
\sum_{\widetilde{\iota}\in \widetilde{I}}\left(\xi, \mathfrak{desc}\left( \widetilde{a}_{\widetilde{\iota}} \widetilde{a}, \widetilde{b}\right)  \eta \right)_{L^2\left({M},{S} \right)}= \sum_{\widetilde{\iota}\in \widetilde{I}}\left(1 \otimes \xi, \mathfrak{lift}_{\widetilde{\mathcal{U}}_{\widetilde{\iota}}}\mathfrak{desc}\left(\widetilde{a}_{\widetilde{\iota}}  \widetilde{a} \widetilde{b}\right)\left(  1 \otimes \eta\right)  \right)_{L^2\left(\widetilde{M},\widetilde{S} \right)}= 
\\=\sum_{\widetilde{\iota}\in \widetilde{I}}\left(\widetilde{a}_{\widetilde{\iota}}  \widetilde{a}\left( 1 \otimes \xi\right) ,  \widetilde{b}\left(  1 \otimes \eta\right)  \right)_{L^2\left(\widetilde{M},\widetilde{S} \right)}=
\left( \widetilde{a}\left( 1 \otimes \xi\right) ,  \widetilde{b}\left(  1 \otimes \eta\right) \right)_{L^2\left(\widetilde{M},\widetilde{S} \right)}=\\= \int_{\widetilde{M}}\left( \widetilde{a}\left( 1 \otimes \xi\right) ,  \widetilde{b}\left(  1 \otimes \eta\right) \right)\left( \widetilde{x}\right)  d\widetilde{\mu}.
\end{split} 
\end{equation*}
%\begin{equation*}
%\begin{split}
%\left(\widetilde{a} \otimes \xi, \widetilde{b} \otimes \eta  \right)_{L^2\left(\widetilde{M},\widetilde{S} \right)} = \int_{\widetilde{M}} \widetilde{a}^*\left(\widetilde{x} \right) \widetilde{b}\left( \widetilde{x}\right) \left(\xi, \eta \right)_{\pi\left(\widetilde{x} \right) } d\widetilde{\mu}=\\=
%\int_{M} \sum_{g \in G\left(\widetilde{M}~|~M \right) }\left( g\left( \widetilde{a}^* \widetilde{b}\right) \right)\left( x\right)  \left(\xi, \eta \right)_{x} d\mu = \left(\xi, \left\langle \widetilde{a}, \widetilde{b}\right\rangle_{C\left(\widetilde{M} \right) } \eta \right)_{L^2\left({M},{S} \right)}. 
%\end{split}
%\end{equation*}
Above equation is a version of  \eqref{induced_prod_equ}. If $\varphi$ is extension of given by \eqref{comm_smooth_iso_eqn} map and  $\widetilde{a} \in C\left( \widetilde{M}\right)$ then following condition holds
\be\label{comm_fin_act_eqn}
\varphi\left( \widetilde{a}  \widetilde{\xi}\right) =\sum_{g \in G} \sum_{\iota \in I} \widetilde{a} \left( g  \widetilde{e}_\iota\right)   \otimes \mathfrak{desc}\left(\left( g\widetilde{e}_\iota\right)  \widetilde{\xi}\right). 
\ee
The left part of \eqref{comm_fin_act_eqn} is relevant to the natural action of $C\left(\widetilde{M}\right)$ on $ L^2\left(\widetilde{M},\widetilde{S} \right)$,  the right part  is consistent with \eqref{ind_act_form}. So one has the following lemma.
\begin{lemma}\label{comm_ind_lem}
If the representation  $\widetilde{\rho}: C\left(\widetilde{M}\right) \to  B\left( \widetilde{   \H}\right)  $ is induced by the pair $$\left(C\left(M \right)\to B\left(  L^2\left(M,S \right)\right) ,\left(C\left(M \right) , C\left( \widetilde{M}\right) , G\left(\widetilde{M}~|~M \right) \right)  \right)$$ (cf. Definition \ref{induced_repr_defn}) then following conditions holds
\begin{enumerate}
	\item[(a)] There is the homomorphism of Hilbert spaces $\widetilde{   \H}\cong L^2\left(\widetilde{M},\widetilde{S} \right)$,
	\item[(b)] The representation $\widetilde{\rho}$ is given by the natural action of $C\left(\widetilde{M}\right)$ on $ L^2\left(\widetilde{M},\widetilde{S} \right)$.
\end{enumerate}
\end{lemma}
\begin{proof}
	(a) Follows from \eqref{comm_tensor_iso_eqn},\\
	(b) Follows from \eqref{comm_fin_act_eqn}.
\end{proof}

For any $\widetilde{a} \in \Coo\left(\widetilde{M} \right)$ following condition holds 
$$
\widetilde{e}_{\widetilde{\iota}'}\widetilde{a}\widetilde{e}_{\widetilde{\iota}''}  \in \Coo\left(\widetilde{M} \right),
$$
hence one has
 $$
 \left\langle \widetilde{e}_{\widetilde{\iota}'}~,~~\widetilde{a}\widetilde{e}_{\widetilde{\iota}''}\right\rangle_{C\left(\widetilde{M} \right)}  = \mathfrak{desc}\left(\widetilde{e}_{\widetilde{\iota}'}\widetilde{a}\widetilde{e}_{\widetilde{\iota}''} \right) \in \Coo\left(M \right).
 $$
 Otherwise form $\widetilde{a} \notin \Coo\left(\widetilde{M} \right)$ it turns out that $\exists \widetilde{\iota} \in \widetilde{I}~~ \widetilde{e}^2_{\widetilde{\iota}} \widetilde{a} \notin \Coo\left(\widetilde{M} \right)$, hence
 $$
 \left\langle \widetilde{e}_{\widetilde{\iota}}~,~~\widetilde{a}\widetilde{e}_{\widetilde{\iota}}\right\rangle_{C\left(\widetilde{M} \right)}  = \mathfrak{desc}\left(\widetilde{e}^2_{\widetilde{\iota}}\widetilde{a} \right) \notin \Coo\left({M} \right).
 $$
 Summarize above equations one concludes
 $$
\Coo\left(\widetilde{M} \right)= \left\{\widetilde{a}\in C\left(\widetilde{M} \right)~|~  \left\langle \widetilde{e}_{\widetilde{\iota}'}~,~~\widetilde{a}\widetilde{e}_{\widetilde{\iota}''}\right\rangle_{C\left(\widetilde{M} \right)} \in \Coo\left( M\right); ~ \forall \widetilde{\iota}', \widetilde{\iota}'' \in \widetilde{I} \right\}.
 $$
 From the Lemma \ref{smooth_matr_lem} it turns out that
 \begin{itemize}
 	\item The unital noncommutative finite-fold covering $\left(C\left({M} \right), C\left(\widetilde{M} \right), G\left( \widetilde{M}~|~M \right)\right) $ is {smoothly invariant} (cf. Definition \ref{smooth_defn}),
 \item 
 \be\label{comm_smooth_matr_eqn}
\Coo\left(\widetilde{M} \right) = C\left(\widetilde{M} \right)\bigcap 
 \mathbb{M}_{\left|\widetilde{I} \right| }\left( \Coo\left( M\right) \right).
 \ee 	
 \end{itemize}

\subsubsection{Lift of the Dirac operator}\label{comm_d_sec}
\paragraph*{}
For any $\widetilde{a} \in \Coo\left(\widetilde{M} \right)$ following condition holds
\bean
\widetilde{a} = \sum_{\widetilde{\iota}\in \widetilde{I}} \widetilde{\ga}_{\widetilde{\iota}}= \sum_{\widetilde{\iota}\in \widetilde{I}} \widetilde{\al}_{\widetilde{\iota}}\widetilde{\bt}_{\widetilde{\iota}}; \text{ where } \widetilde{\al}_{\widetilde{\iota}}= \widetilde{e}_{\widetilde{\iota}},~ \widetilde{\bt}_{\widetilde{\iota}}= \widetilde{a}\widetilde{e}_{\widetilde{\iota}},~ \widetilde{\ga}_{\widetilde{\iota}}=\widetilde{a}_{\widetilde{\iota}}\widetilde{e}^2_{\widetilde{\iota}}.
\eean
Denote by $\Om^1_{\slashed D}$  the {module of differential forms associated} with the spectral triple  $$\left(\Coo\left({M} \right), L^2\left(M,S \right), \slashed D\right)$$ (cf. Definition \ref{ass_cycle_defn}). Denote by ${\al}_{\widetilde{\iota}} = \mathfrak{desc} \left(\widetilde{\al}_{\widetilde{\iota}} \right)$, ${\bt}_{\widetilde{\iota}} = \mathfrak{desc} \left(\widetilde{\bt}_{\widetilde{\iota}} \right)$,  ${\ga}_{\widetilde{\iota}} = \mathfrak{desc} \left(\widetilde{\ga}_{\widetilde{\iota}} \right)\in \Coo\left({M} \right)$. Let us define a $\C$-linear map
\bean
\nabla: \Coo\left(\widetilde{M} \right) \to \Coo\left(\widetilde{M} \right) \otimes_{\Coo\left({M} \right)}\Om^1_{\slashed D},\\
\widetilde{a} \mapsto \sum_{\widetilde{\iota}\in \widetilde{I}}\left(  \widetilde{\al}_{\widetilde{\iota}} \otimes \left[\slashed D, {\bt}_{\widetilde{\iota}}  \right] +\widetilde{\bt}_{\widetilde{\iota}} \otimes \left[\slashed D, {\al}_{\widetilde{\iota}}   \right] \right) 
\eean
For any $a \in \Coo\left({M} \right)$ following condition holds
 \bean
 \nabla\left(\widetilde{a}a \right)=  \sum_{\widetilde{\iota}\in \widetilde{I}}\left(  \widetilde{\al}_{\widetilde{\iota}} \otimes \left[\slashed D, {\bt}_{\widetilde{\iota}} a  \right] +\widetilde{\bt}_{\widetilde{\iota}}a \otimes \left[\slashed D, {\al}_{\widetilde{\iota}}   \right] \right) =\\= \sum_{\widetilde{\iota}\in \widetilde{I}} \widetilde{\al}_{\widetilde{\iota}} \otimes \left[\slashed D, {\bt}_{\widetilde{\iota}}  \right]a+\widetilde{\al}_{\widetilde{\iota}} \otimes \widetilde{\bt}_{\widetilde{\iota}}\left[\slashed D,  a  \right] +\widetilde{\bt}_{\widetilde{\iota}} \otimes \left[\slashed D, {\al}_{\widetilde{\iota}}   \right]a=\\= \sum_{\widetilde{\iota}\in \widetilde{I}}\left(  \widetilde{\al}_{\widetilde{\iota}} \otimes \left[\slashed D, {\bt}_{\widetilde{\iota}}   \right] +\widetilde{\bt}_{\widetilde{\iota}} \otimes \left[\slashed D, {\al}_{\widetilde{\iota}}  \right] \right) a+
 \sum_{\widetilde{\iota}\in \widetilde{I}}\widetilde{\al}_{\widetilde{\iota}}\widetilde{\bt}_{\widetilde{\iota}}\left[ \slashed D,a\right]=
 \nabla\left( \widetilde{a}\right) a + \widetilde{a}\otimes  \left[ \slashed D,a\right]. 
 \eean
The comparison of the above equation and \eqref{conn_triple_eqn} states that $\nabla$ is a connection. If $g \in  G\left(\widetilde{M}~|~M \right)$ then from $\mathfrak{desc}\left(g\widetilde{\al}_{\widetilde{\iota}}\right) = {\al}_{\widetilde{\iota}} $ and $\mathfrak{desc}\left(g\widetilde{\bt}_{\widetilde{\iota}}\right) = {\bt}_{\widetilde{\iota}} $ it follows that
\bean
\nabla\left(g\widetilde{a} \right)=  \nabla\left(\sum_{\widetilde{\iota}\in \widetilde{I}} \left( g\widetilde{\al}_{\widetilde{\iota}}\right) \left( g\widetilde{\al}_{\widetilde{\iota}}\right)\right)=\\=\sum_{\widetilde{\iota}\in \widetilde{I}}\left( g \widetilde{\al}_{\widetilde{\iota}} \otimes \left[\slashed D, {\bt}_{\widetilde{\iota}}   \right] +g\widetilde{\bt}_{\widetilde{\iota}} \otimes \left[\slashed D, {\al}_{\widetilde{\iota}}   \right] \right)= g \left( \nabla\left(\widetilde{a} \right)\right),  
\eean
i.e. $\nabla$ is  $G$-{equivariant} (cf. \eqref{equiv_conn_eqn}). 
If $\xi \in \Dom~\slashed D$ then following conditions hold
\bean
\supp~\widetilde{\al}_{\widetilde{\iota}} \otimes \left[\slashed D, {\bt}_{\widetilde{\iota}}   \right]\xi \in \widetilde{\mathcal{U}}_{\widetilde\iota},\\
\supp~\widetilde{\bt}_{\widetilde{\iota}} \otimes \left[\slashed D, {\al}_{\widetilde{\iota}}   \right]\xi \in \widetilde{\mathcal{U}}_{\widetilde\iota},
\eean
hence one has
\bean
\widetilde{\al}_{\widetilde{\iota}} \otimes \left[\slashed D, {\bt}_{\widetilde{\iota}}   \right]\xi = \mathfrak{lift}_{\widetilde{\mathcal{U}}_{\widetilde\iota}} \left({\al}_{\widetilde{\iota}}  \left[\slashed D, {\bt}_{\widetilde{\iota}}   \right]\xi \right),\\
\widetilde{\bt}_{\widetilde{\iota}} \otimes \left[\slashed D, {\al}_{\widetilde{\iota}}   \right]\xi = \mathfrak{lift}_{\widetilde{\mathcal{U}}_{\widetilde\iota}} \left({\bt}_{\widetilde{\iota}}  \left[\slashed D, {\al}_{\widetilde{\iota}}   \right]\xi \right).
\eean
From \eqref{comm_matr_x_eqn} it turns out $\widetilde{\bt}_{\widetilde{\iota}}  \left[\slashed D, {\al}_{\widetilde{\iota}}   \right] =   \left[\slashed D, {\al}_{\widetilde{\iota}}   \right]\widetilde{\bt}_{\widetilde{\iota}}$, hence
\bean
\widetilde{\al}_{\widetilde{\iota}} \otimes \left[\slashed D, {\bt}_{\widetilde{\iota}}   \right]\xi + \widetilde{\bt}_{\widetilde{\iota}} \otimes \left[\slashed D, {\al}_{\widetilde{\iota}}   \right]\xi = \mathfrak{lift}_{\widetilde{\mathcal{U}}_{\widetilde\iota}}\left(\left( \left[\slashed D, {\al}_{\widetilde{\iota}}   \right]{\bt}_{\widetilde{\iota}}+{\al}_{\widetilde{\iota}}  \left[\slashed D, {\bt}_{\widetilde{\iota}}   \right] \right)\xi \right)= \mathfrak{lift}_{\widetilde{\mathcal{U}}_{\widetilde\iota}}\left(\left[\slashed D, {\ga}_{\widetilde{\iota}}   \right] \xi\right) 
\eean
From $\mathfrak{lift}_{\widetilde{\mathcal{U}}_{\widetilde\iota}}\left( {\ga}_{\widetilde{\iota}}\right) = \widetilde{\ga}_{\widetilde{\iota}}$ and $\sum_{\widetilde{\iota}\in \widetilde{I}} \widetilde{\ga}_{\widetilde{\iota}} = \widetilde{a}$ it turns out
\bean
\sum_{\widetilde{\iota}\in \widetilde{I}}\left(  \widetilde{\al}_{\widetilde{\iota}} \otimes \left[\slashed D, {\bt}_{\widetilde{\iota}}  \right] +\widetilde{\bt}_{\widetilde{\iota}} \otimes \left[\slashed D, {\al}_{\widetilde{\iota}}   \right] \right)\xi=\\=\sum_{\widetilde{\iota}\in \widetilde{I}}    
\mathfrak{lift}_{\widetilde{\mathcal U}_{\widetilde{   \iota}}}\left[{\slashed D}, {\ga}_{\widetilde{\iota}}  \right]\left(1 \otimes \xi \right)=\sum_{\widetilde{\iota}\in \widetilde{I}}    \left[\widetilde{\slashed D}, \widetilde{\ga}_{\widetilde{\iota}}  \right]\left(1 \otimes \xi \right)  = \left[\widetilde{\slashed D}, {\widetilde{a}}   \right]\left(1 \otimes \xi \right) 
\eean
where $\widetilde{\slashed D}$ is the $\pi$-lift of ${\slashed D}$. From the above equation it follows that
$$
\nabla\left(\widetilde{a} \right)\xi + \widetilde{a}\otimes  \slashed D \xi = \left[\widetilde{\slashed D}, {\widetilde{a}}   \right]\left(1 \otimes \xi \right) + \widetilde{a} \otimes \slashed D \xi = \widetilde{\slashed D}\left(\widetilde{a} \otimes \xi \right)  
$$
The comparison of the above equation with \eqref{d_defn} states that $\widetilde{\slashed D}$ is  $\left(C\left( M\right) , C\left( \widetilde{M}\right) , G\left(\widetilde{M}~|~M \right)  \right)$-{lift} of $\left( \Coo\left( M\right), L^2\left( M, S\right) , \slashed D \right)$ (cf. Definition \ref{triple_conn_lift_defn}).

\subsubsection{Coverings of spectral triples}
 \paragraph*{}
From \eqref{comm_sum_1} it turns that 
hence $C\left(M \right)$-module $C\left(\widetilde{M} \right)$ is generated by  a finite set $\left\{\widetilde{e}_{\widetilde{\iota}}\right\}_{\widetilde{\iota}\in \widetilde{I}}$, i.e.
$$
C\left(\widetilde{M} \right) = \sum_{\widetilde{\iota}\in \widetilde{I}}  \widetilde{e}_{\widetilde{\iota}}~C(M).
$$ 
Also from \eqref{comm_sum_1} it turns out $\left\langle \widetilde{e}_{\widetilde{\iota'}}, \widetilde{e}_{\widetilde{\iota''}}\right\rangle_{C\left(\widetilde{M} \right) } \in \Coo\left(M \right)$, i.e.  the set $\left\{\widetilde{e}_{\widetilde{\iota}}\right\}_{\widetilde{\iota}\in \widetilde{I}}$, satisfies to the condition (a) of the Lemma \ref{smooth_matr_lem}. From \eqref{comm_a_g_inv} it turns out that 
$\left\{\widetilde{e}_{\widetilde{\iota}}\right\}_{\widetilde{\iota}\in \widetilde{I}}$ is $G$-invariant, i.e.
\be\nonumber
G\left\{\widetilde{e}_{\widetilde{\iota}}\right\}_{\widetilde{\iota}\in \widetilde{I}}=\left\{\widetilde{e}_{\widetilde{\iota}}\right\}_{\widetilde{\iota}\in \widetilde{I}}~,
\ee
 hence the set $\left\{\widetilde{e}_{\widetilde{\iota}}\right\}_{\widetilde{\iota}\in \widetilde{I}}$, satisfies to the condition (b) of the Lemma \ref{smooth_matr_lem}. From the Lemma \ref{smooth_matr_lem} it turns out that  the unital noncommutative finite-fold covering $\left(C\left(M \right) , C\left( \widetilde{M}\right) , G \right)$ is {smoothly invariant}. It is proven in \ref{comm_d_sec} that if $\widetilde{\slashed D}$ is the  $\pi$-lift of $\slashed D$ then $\widetilde{\slashed D}$ is the  $\left(C\left( M\right) , C\left( \widetilde{M}\right) , G\left(\widetilde{M}~|~M \right)  \right)$-{lift} of $\left( \Coo\left( M\right), L^2\left( M, S\right) , \slashed D \right)$.
 So one has the following theorem. 
\begin{thm}\label{fin_sp_tr_thm} A spectral triple $\left(C^{\infty}\left( \widetilde{M}\right) , L^2\left( \widetilde{M}, \widetilde{S}\right) ,\widetilde{\slashed D}\right)$ is the $\left(C\left(M \right) , C\left( \widetilde{M}\right) , G\left( \widetilde{M}~|~M\right)  \right)$-lift of $\left(C^{\infty}\left( M\right) , L^2\left( M, S\right) ,\slashed D\right)$.
\end{thm}

\subsection{Infinite coverings}

\paragraph*{} This section contains an algebraic version of the Proposition \ref{comm_cov_mani} in case of infinite coverings.
\subsubsection{Coverings of $C^*$-algebras}
\paragraph*{} Following theorem states an equivalence between a topological notion of an infinite covering and an algebraical one.
\begin{thm}\label{comm_main_thm}\cite{ivankov:qnc}
	If $\mathfrak{S}_{\mathcal X} = \left\{\mathcal{X} = \mathcal{X}_0 \xleftarrow{}... \xleftarrow{} \mathcal{X}_n \xleftarrow{} ...\right\} \in \mathfrak{FinTop}$ and
	$$\mathfrak{S}_{C_0\left(\mathcal{X}\right)}=
	\left\{C_0(\mathcal{X})=C_0(\mathcal{X}_0)\to ... \to C_0(\mathcal{X}_n) \to ...\right\} \in \mathfrak{FinAlg}$$ is an algebraical  finite covering sequence then following conditions hold:
	\begin{enumerate}
		\item [(i)] $\mathfrak{S}_{C_0\left(\mathcal{X}\right)}$ is good,
		\item[(ii)] There are  isomorphisms:

		\begin{itemize}
			\item $\varprojlim \downarrow \mathfrak{S}_{C_0\left(\mathcal{X}\right)} \approx C_0\left(\varprojlim \downarrow \mathfrak{S}_{\mathcal X}\right)$;
			\item $G\left(\varprojlim \downarrow \mathfrak{S}_{C_0\left(\mathcal{X}\right)}~|~ C_0\left(\mathcal X\right)\right) \approx G\left(\varprojlim \downarrow \mathfrak{S}_{\mathcal{X}}~|~ \mathcal X\right)$.
		\end{itemize}
	\end{enumerate}
(cf. Definitions \ref{top_topological_inv_lim_defn} and \ref{main_defn} for notation).
	
\end{thm}
\subsubsection{The sequence of spectral triples}
\paragraph*{} Let $\left(C^{\infty}\left( M\right) , L^2\left( M, S\right) ,\slashed D\right)$ be a commutative spectral triple, and let $\pi:\widetilde{M} \to M$ be an infinite regular covering. From the Proposition \ref{comm_cov_mani} it follows that $\widetilde{M}$ has natural structure of the Riemannian manifold.  Denote by $\widetilde{S} = \pi^*S$ the inverse image of the Spin-bundle (cf. \ref{top_vb_sub_sub}). Similarly we can define an inverse image $\widetilde{\slashed D}= \pi^*\slashed D$ (cf. Definition \ref{inv_image_defn}). 
Let $G = G\left(\widetilde{M}~|M \right)$ be a  group of covering transformations of $\pi$. Suppose that there is a commutative diagram of group epimorphisms

\begin{tikzpicture}
\matrix (m) [matrix of math nodes,row sep=3em,column sep=4em,minimum width=2em]
{
	G   \\
	G_1    &   \cdots & G_n &   \cdots &  \ \\};
\path[-stealth]
(m-1-1) edge (m-2-1)
(m-1-1) edge  (m-2-3)
(m-2-2) edge (m-2-1)
(m-2-3) edge  (m-2-2)
(m-2-4) edge  (m-2-3);
\end{tikzpicture}

such that
\begin{itemize}
	\item A group $G_n$ is finite for any $n \in \N$,
	\item $\bigcap_{n \in \N} \ker\left(G \to G_n \right)$ is a trivial group. 
\end{itemize}	
If $J_n = \ker\left(G \to G_n \right)$ then there is the following commutative diagram of coverings

\begin{tikzpicture}
\matrix (m) [matrix of math nodes,row sep=3em,column sep=4em,minimum width=2em]
{
	\widetilde{M}   \\
	M & M_1 = \widetilde{M}/J_1 & \cdots   &  M_n=\widetilde{M}/J_n  &  \cdots &  \ \\};
\path[-stealth]
(m-1-1) edge node [left] {$ \pi $} (m-2-1)
(m-1-1) edge node [left] {$ \pi_1~~~~ $} (m-2-2)
(m-1-1) edge node [left]  {$ \pi_n~~~~~~~~$} (m-2-4)
(m-2-2) edge node [right] {$ \ $} (m-2-1)
(m-2-3) edge node [right] {$ \ $} (m-2-2)
(m-2-4) edge node [right] {$ \ $} (m-2-3)
(m-2-5) edge node [right] {$ \ $} (m-2-4);
\end{tikzpicture}

Clearly 
$\mathfrak{S}_{M} = \left\{M = M_0 \xleftarrow{}... \xleftarrow{} M_n \xleftarrow{} ...\right\} \in \mathfrak{FinTop}$ is a topological  finite covering sequence. From the Theorem \ref{comm_main_thm} it turns out that $$\mathfrak{S}_{C\left(M\right)}=
\left\{C(M)=C(M_0)\to ... \to C(M_n) \to ...\right\} $$ is a good algebraical  finite covering sequence, and  the triple $$\left(C\left( M\right) , C_0\left( \widetilde{M}\right) , G=G\left( \widetilde{M}~|~M\right)\right)$$ is an  infinite noncommutative covering   of $\mathfrak{S}_{C\left(M\right)}$. 
Otherwise from the Theorem \ref{fin_sp_tr_thm} it follows that
\begin{equation}\label{comm_triple_seq_eqn}
\begin{split}
\mathfrak{S}_{\left(C^{\infty}\left( M\right) , L^2\left( M, S\right) ,\slashed D\right) } = \{ \left(C^{\infty}\left( M\right) , L^2\left( M, S\right) ,\slashed D\right)=\left(C^{\infty}\left( M_0\right) , L^2\left( M_0, S_0\right) ,\slashed D_0\right), \dots,\\
\left(C^{\infty}\left( M_n\right) , L^2\left( M_n, S_n\right) ,\slashed D_n\right), \dots
\}\in \mathfrak{CohTriple}
\end{split}
\end{equation}
is a coherent sequence of spectral triples. We would like to proof that $\mathfrak{S}_{\left(C^{\infty}\left( M\right) , L^2\left( M, S\right) ,\slashed D\right)}$ is regular and  to find  a $\left(C\left(M \right) , C_0\left( \widetilde{M}\right) , G  \right)$-lift of $\left(C^{\infty}\left( M\right) , L^2\left( M, S\right) ,\slashed D\right)$. Denote by $\widetilde{S} = \pi^*S$ the inverse image of the Spin-bundle $S$. 

\subsubsection{Induced representation}
\paragraph*{} Similarly to \ref{induced_comm_finite}  consider  a finite family $\left\lbrace a_\iota \in \Coo\left(  M\right) \right\rbrace_{\iota \in I}$ of positive elements such that
\begin{equation*}
\begin{split}
a_\iota\left(M \backslash \mathcal{U}_\iota \right) = \{0\},\\
1_{C\left( M\right) }= \sum_{\iota \in I}a_\iota.
\end{split}
\end{equation*} 
Similarly to $\ref{induced_comm_finite}$ for any $\iota \in I$ we choose an open subset $\widetilde{ \mathcal U}_\iota$ which is mapped homeomorphically onto $\mathcal U_\iota$ and define  $\widetilde{a}_\iota= \mathfrak{lift}_{\widetilde{ \mathcal U}_\iota}\left(a_\iota \right)  \in \Coo\left(\widetilde{M} \right)$. Denote by $e_\iota = \sqrt{a_\iota}$, $\widetilde{e}_\iota= \sqrt{\widetilde{a}_\iota}$. If $\Ga^\infty_c\left( \widetilde{M}, \widetilde{S}\right)$ is the space of compactly supported smooth sections then  similarly to \ref{comm_smooth_iso_eqn} there is the $\C$-linear isomorphism
\be\label{comm_smooth_iso_c_eqn}
\begin{split}
\varphi:\Ga_c^\infty\left(\widetilde{M},\widetilde{S} \right) \xrightarrow{\approx} \Coo_c\left(\widetilde{M} \right)  \otimes_{\Coo\left( M\right)}  \Ga^\infty\left(M,S \right), \\
 \sum_{\left(g, \iota\right)  \in I_{\widetilde{\xi}}} \left( g \widetilde{e}_\iota\right)  \left( g \widetilde{e}_\iota\right) \widetilde{\xi}  \mapsto  \sum_{\left(g, \iota\right)  \in I_{\widetilde{\xi}}} g \widetilde{e}_\iota \otimes \mathfrak{desc}\left(\left( g\widetilde{e}_\iota\right)  \widetilde{\xi}\right) 
\end{split}
\ee
where $I_{\widetilde{\xi}}\subset G \times I$ is a finite subset such that
$$
 \sum_{\left(g, \iota\right)  \in I_{\widetilde{\xi}}} \left( g \widetilde{e}_\iota\right)  \left( g \widetilde{e}_\iota\right) \widetilde{\xi} =\widetilde{\xi}. 
$$

Since $\Coo_c\left(\widetilde{M} \right) $ is dense in $L^2\left( C_0\left(\widetilde{M} \right) \right) $, and $\Ga^\infty\left(M,S \right)$, (resp. $\Ga_c^\infty\left(\widetilde{M},\widetilde{S} \right)$) is dense in $L^2\left(M,S \right)$, (resp. $L^2\left(\widetilde{M},\widetilde{S} \right)$) the isomorphism \eqref{comm_smooth_iso_eqn} can be uniquely extended up to $\C$-isomorphism
\be\label{comm_tensor_inf_iso_eqn}
\varphi:L^2\left(\widetilde{M},\widetilde{S} \right) \xrightarrow{\approx} L^2\left( C\left(\widetilde{M} \right)\right)   \otimes_{C\left( M\right)}  L^2\left(M,S \right).
\ee
Denote by $G = G\left(\widetilde{M}~|~M \right)$. 
If $\widetilde{a} \otimes \xi,~ \widetilde{b} \otimes \eta \in L^2\left( C_0\left(\widetilde{M} \right)\right)  \otimes_{C\left(M \right) } \Ga\left( M, S\right) \subset L^2\left(\widetilde{M},\widetilde{S} \right)$, $\mu$ (resp. $\widetilde{\mu}$) is the Riemannian measure (cf.  \cite{do_carmo:rg}) on $M$, (resp. $\widetilde{M}$) then 
\begin{equation*}
\begin{split}
\int_{M}\left(\xi, \left( \sum_{g \in G}  g\left( \widetilde{a}^* \widetilde{b}\right) \right) \eta \right)\left( x\right)d\mu =\left(\xi, \left( \sum_{g \in G\left(\widetilde{M}~|~M \right)}  g\left( \widetilde{a}^* \widetilde{b}\right) \right) \eta  \right)_{L^2\left({M},{S} \right)} = \\= \left(\xi, \left( \sum_{g \in G}  g\left( \sum_{\left(g', \iota\right) \in G \times I} g'\widetilde{a}_\iota \right)\widetilde{a}^*  \widetilde{b} \right) \eta  \right)_{L^2\left({M},{S} \right)}=\\=
\left(\xi,  \left(\sum_{\left(g', \iota\right) \in G \times I} \mathfrak{desc}\left( g'\widetilde{a}_\iota \widetilde{a}^*  \widetilde{b}\right)\right)  \eta  \right)_{L^2\left({M},{S} \right)}
=\\=
\left(1 \otimes \xi,  \left(\sum_{\left(g', \iota\right) \in G \times I} \mathfrak{lift}_{g \widetilde{  \mathcal U}_\iota}\mathfrak{desc}\left( g'\widetilde{a}_\iota \widetilde{a}^*  \widetilde{b}\right)\right)\left(  1 \otimes \eta \right) \right)_{L^2\left(\widetilde{M},\widetilde{S} \right)}
=\\=
\left(1 \otimes \xi,   \widetilde{a}^*  \widetilde{b}\left(  1 \otimes \eta \right) \right)_{L^2\left(\widetilde{M},\widetilde{S} \right)}=\left(\widetilde{a}\left( 1 \otimes \xi\right) ,    \widetilde{b}\left(  1 \otimes \eta \right) \right)_{L^2\left(\widetilde{M},\widetilde{S} \right)}
=\\=
\int_{\widetilde{M}}\left( \widetilde{a}\left( 1 \otimes \xi\right) ,  \widetilde{b}\left(  1 \otimes \eta\right) \right)\left( \widetilde{x}\right)  d\widetilde{\mu},
\end{split} 
\end{equation*}
i.e. scalar product on the Hilbert space of the induced representation coincides with the natural scalar product on $L^2\left(\widetilde{M},\widetilde{S} \right)$.
If $\widetilde{a} \in C_c\left( \widetilde{M}\right)$ then following condition holds
\be\label{comm_inf_act_eqn}
\varphi\left( \widetilde{a}  \xi\right)= \sum_{g \in G} \sum_{\iota \in I} \widetilde{a} g \widetilde{e}_\iota \otimes \mathfrak{desc}\left(\left( g\widetilde{e}_\iota\right)  \widetilde{\xi}\right). 
\ee
The left part of \eqref{comm_inf_act_eqn} is given by the natural action of $C_c\left(\widetilde{M}\right)$ on $ L^2\left(\widetilde{M},\widetilde{S} \right)$,  the right part  is consistent with \eqref{inf_ind_act_eqn}. Since $C_c\left(\widetilde{M}\right)$ is a dense subalgebra of $C_0\left(\widetilde{M}\right)$ the equation \eqref{comm_inf_act_eqn} is true for any $\widetilde{a} \in C_0\left( \widetilde{M}\right)$.
Summarize above equations one has the following lemma.
\begin{lemma}\label{comm_ind1_lem}
	If the representation  $\widetilde{\rho}: C_0\left(\widetilde{M}\right) \to  B\left( \widetilde{   \H}\right)  $ is induced by the pair $$\left(C\left(M \right)\to B\left(  L^2\left(M,S \right)\right) ,\left(C\left(M \right) , C_0\left( \widetilde{M}\right) , G\left(\widetilde{M}~|~M \right) \right)  \right)$$ (cf. Definition \ref{inf_ind_defn}) then following conditions holds
	\begin{enumerate}
		\item[(a)] There is the homomorphism of Hilbert spaces $\widetilde{   \H}\cong L^2\left(\widetilde{M},\widetilde{S} \right)$,
		\item[(b)] The representation $\widetilde{\rho}$ is given by the natural action of $C_0\left(\widetilde{M}\right)$ on $ L^2\left(\widetilde{M},\widetilde{S} \right)$.
	\end{enumerate}
\end{lemma}
\begin{proof}
	(a) Follows from \eqref{comm_tensor_inf_iso_eqn}.\\
	(b) Follows from \eqref{comm_inf_act_eqn}.
\end{proof}

\subsubsection{Smooth elements}\label{comm_smooth_sec}

\paragraph*{}
Consider the coherent sequence
\begin{equation}\label{comm_coh_eqn}
\begin{split}
\mathfrak{S}_{\left(C^{\infty}\left( M\right) , L^2\left( M, S\right) ,\slashed D\right) } = \{ \left(C^{\infty}\left( M\right) , L^2\left( M, S\right) ,\slashed D\right)=\left(C^{\infty}\left( M_0\right) , L^2\left( M_0, S_0\right) ,\slashed D_0\right), \dots,\\
\left(C^{\infty}\left( M_n\right) , L^2\left( M_n, S_n\right) ,\slashed D_n\right), \dots
\}\in \mathfrak{CohTriple}
\end{split}
\end{equation}
of spectral triples given by \eqref{comm_triple_seq_eqn}. 
 For any $n \in \N$ denote by $\widetilde{\pi}_n:\widetilde{M} \to M_n$ the natural covering.
\begin{lemma}\label{comm_c_cp_lem}\cite{ivankov:qnc}
	Following conditions hold:
	\begin{enumerate}
		\item [(i)]	If $\widetilde{  \mathcal U} \subset \widetilde{ M}$ is a compact set then there is $N \in \N$ such that for any $n \ge N$ the restriction $\widetilde{\pi}_n|_{\widetilde{  \mathcal U}}:\widetilde{  \mathcal U} \xrightarrow{\approx} \widetilde{\pi}_n\left( {\widetilde{  \mathcal U}}\right)$ is a homeomorphism,
		
		\item[(ii)]  If $\widetilde{a} \in C_c\left(\widetilde{  M }\right)_+ $ is a positive element then there there is $N \in \N$ such that for any $n \ge \N$ following condition holds
		\begin{equation}\label{comm_a_eqn}
		a_n\left(\widetilde{   \pi }_n \left( \widetilde{x}\right)\right) =\left\{
		\begin{array}{c l}
		\widetilde{   a}\left( \widetilde{x}\right) &\widetilde{x} \in   \supp~ \widetilde{a} ~\&~ \widetilde{   \pi }_n \left( \widetilde{x}\right) \in \supp~ a_n \\
		0 &\widetilde{   \pi }_n \left( \widetilde{x}\right) \notin \supp~ a_n 
		\end{array}\right.
		\end{equation}	
		where 	$$
		a_n  = \sum_{g \in \ker\left( \widehat{G} \to  G_n\right)}g\widetilde{a}.
		$$
	\end{enumerate}
\end{lemma}
\begin{lem}
If $\widetilde{W}^\infty$ is the space of 	$\mathfrak{S}_{\left(C^{\infty}\left( M\right) , L^2\left( M, S\right) ,\slashed D\right) }$-smooth elements (cf. Definition \ref{smooth_el_defn}) then
 $$\widetilde{W}^\infty \subset \Coo_c\left(\widetilde{M} \right)\stackrel{\mathrm{def}}{=} \Coo\left(\widetilde{M} \right) \bigcap C_c\left(\widetilde{M} \right).$$
\end{lem}
\begin{proof}
	It is shown in \cite{pedersen:mea_c} that Pedersen ideal of $C_0\left(\widetilde{M} \right)$ coincides with $C_c\left(\widetilde{M} \right)$, and taking into account the condition (d) of the Definition \ref{smooth_el_defn} one has
	\be\nonumber
	\widetilde{W}^\infty \in C_c\left(\widetilde{M} \right).
	\ee  
	
If $\widetilde{a} \in \widetilde{W}^\infty$ then from the condition (a) of the Definition \ref{smooth_el_defn} it follows that $$a_n= \sum_{g \in \ker\left( \widehat{G} \to  G_n\right)}g\widetilde{a} \in \Coo\left(M_n \right).$$

 From $\widetilde{a} \in\widetilde{W}^\infty \subset C_c\left(\widetilde{M} \right)$  and \eqref{comm_a_eqn} it follows that $\widetilde{a} \in  \Coo\left(\widetilde{M} \right)$, hence  $\widetilde{a} \in  \Coo\left(\widetilde{M} \right)\bigcap C_c\left(\widetilde{M} \right)$.
\end{proof}
\begin{lem}
If $\widetilde{W}^\infty$ is the space of 	$\mathfrak{S}_{\left(C^{\infty}\left( M\right) , L^2\left( M, S\right) ,\slashed D\right) }$-smooth elements (cf. Definition \ref{smooth_el_defn}) then $$\Coo_c\left(\widetilde{M} \right)\subset \widetilde{W}^\infty.$$
\end{lem}
\begin{proof}
	Let $\widetilde{a} \in \Coo_c\left(\widetilde{M} \right)$. One need check that $\widetilde{a}$ satisfies to (a)-(d) of the Definition \ref{smooth_el_defn}.\\
	(a) From $\widetilde{a} \in \Coo_c\left(\widetilde{M} \right)$ it follows that for any $n \in \N^0$ following condition holds
	$$
	\sum_{g \in \ker\left(G\left(\widetilde{ M}~|M \right) \to G\left(\widetilde{ M}~|M_n \right) \right) } g\widetilde{a} \in \Coo\left(M_n \right). 
	$$
	\\
	(b)
	Let  $\pi^s_n:\Coo\left(M_n \right)  \to B\left( \H^{2^s}_n \right)$ be a representation given by \eqref{s_diff_repr_equ}. From \eqref{comm_a_eqn} it turns out
\begin{equation}\label{comm_strong_conv_eqn}
\lim_{n \to \infty} 1_{C_b\left(\widetilde{M} \right)  } \otimes \pi^s_n\left(a_n \right) = \pi_s\left(\widetilde{a} \right)
\end{equation}
in sense of strong convergence.  \\
(c) The equation \eqref{comm_strong_conv_eqn} means that any $\widetilde{a} \in \Coo_c\left(\widetilde{M} \right)$ satisfies to the condition (b) the Definition \eqref{smooth_el_defn}. The manifold $M$ is compact, so there is a finite set $\left\{\mathcal U_\iota\right\}_{\iota \in I}$ of open sets such that any $\mathcal U_\iota$ is evenly covered by $\widetilde{\pi}:\widetilde{M} \to M$ and $M = \bigcup_{\iota \in I} \mathcal U_\iota$.  There is a partition of unity 
\begin{equation}\label{comm_part_eqn}
1_{C\left( M\right) } = \sum_{\iota \in I} a_\iota; \text{ where } a_\iota \in \Coo\left( M\right).
\end{equation}
For any $\iota \in I$ we select $\widetilde{\mathcal U}_\iota$ such that the natural map $\widetilde{\mathcal U}_\iota \xrightarrow{\approx} {\mathcal U}_\iota$ is a homeomorphism. 

Denote by $e_\iota = \sqrt{a_\iota}$ and let $\widetilde{e}_\iota \in \Coo_c\left(\widetilde{M} \right)$ is given by
\begin{equation*}\label{comm_smooth_cov_eqn11}
\widetilde{e}_\iota\left( \widetilde{x}\right)= \left\{
\begin{array}{c l}
e_\iota \left(\widetilde{\pi} \left( \widetilde{x}\right) \right)  & \widetilde{x} \in \widetilde{\mathcal U}_\iota\\
0 & \widetilde{x} \notin \widetilde{\mathcal U}_\iota
\end{array}\right.
\end{equation*}

From \eqref{comm_part_eqn} it follows that
\begin{equation*}
\widetilde{a}= \sum_{\left(g, \iota \right) \in G\left(\widetilde{M}~|~ M\right) \times I} \left(g \widetilde{e}_\iota \right)   \left(g \widetilde{e}_\iota \right) \widetilde{a},
\end{equation*}
and taking into account that $\supp  \widetilde{a}$ is compact there is a finite subset $\widetilde{I} \subset G\left(\widetilde{M}~|~ M\right) \times I$ such that
\be\label{comm_ta_decomp_eqn}
\begin{split}
\widetilde{a}= \sum_{\left(g, \iota \right) \in \widetilde{I}} \left(g \widetilde{e}_\iota \right)   \left(g \widetilde{e}_\iota \right) \widetilde{a} = \sum_{\left(g, \iota \right) \in \widetilde{I}} \widetilde{a}_{\left(g, \iota \right)}=\sum_{\left(g, \iota \right) \in \widetilde{I}} \widetilde{\al}_{\left(g, \iota \right)}\widetilde{\bt}_{\left(g, \iota \right)}, \\ \text{ where } 
\widetilde{\al}_{\left(g, \iota \right)} = g \widetilde{e}_\iota; ~ \widetilde{\bt}_{\left(g, \iota \right)} = g \widetilde{e}_\iota \widetilde{a};~ \widetilde{a}_{\left(g, \iota \right)} = \widetilde{\al}_{\left(g, \iota \right)}\widetilde{\bt}_{\left(g, \iota \right)}= g \widetilde{e}^2_{\left(g, \iota \right)} \widetilde{a},~
\end{split}
\ee

If $a_n,\al_{\left(g, \iota \right)}^n, \bt_{\left(g, \iota \right)}^n \in \Coo\left(M_n \right)$ be given by 
\begin{equation*}\label{c1omm_d_m_eqn}
\begin{split}
a_n = \sum_{g \in \ker\left(G\left(\widetilde{M}~|~M \right)\to G\left(\widetilde{M}~|~M_n \right)  \right)}g \widetilde{a},
\\a_{\left(g, \iota \right)}^n = \sum_{g \in \ker\left(G\left(\widetilde{M}~|~M \right)\to G\left(\widetilde{M}~|~M_n \right)  \right)}g \widetilde{a}_{\left(g, \iota \right)},
\\\\\al_{\left(g, \iota \right)}^n = \sum_{g \in \ker\left(G\left(\widetilde{M}~|~M \right)\to G\left(\widetilde{M}~|~M_n \right)  \right)}g \widetilde{\al}_{\left(g, \iota \right)},
\\
\bt_{\left(g, \iota \right)}^n = \sum_{g \in \ker\left(G\left(\widetilde{M}~|~M \right)\to G\left(\widetilde{M}~|~M_n \right)  \right)}g \widetilde{\bt}_{\left(g, \iota \right)},\\
\end{split}
\end{equation*}
then 
\begin{equation*}
\begin{split}
\left[\slashed D_n, a_n\right] =\sum_{\left(g, \iota \right) \in \widetilde{I}}   \left[ {\slashed D}_n, {a}^n_{\left(g, \iota \right)}\right]=\\=\sum_{\left(g, \iota \right) \in \widetilde{I}} \left(  \left[ {\slashed D}_n, {\al}^n_{\left(g, \iota \right)}\right]{\bt^n}_{\left(g, \iota \right)}   + \al^n_{\left(g, \iota \right)}  \left[ {\slashed D}_n, {\bt}^n_{\left(g, \iota \right)}\right]\right). \end{split}
\end{equation*}
From \eqref{comm_matr_x_eqn} it turns out $$\left[ {\slashed D}_n, {\al}^n_{\left(g, \iota \right)}\right]{\bt^n}_{\left(g, \iota \right)}= {\bt^n}_{\left(g, \iota \right)} \left[ {\slashed D}_n, {\al}^n_{\left(g, \iota \right)}\right],$$
hence one has
\begin{equation*}
\begin{split}
\left[\slashed D_n, a_n\right]=\sum_{\left(g, \iota \right) \in \widetilde{I}} \left( {\bt^n}_{\left(g, \iota \right)} \left[ {\slashed D}_n, {\al}^n_{\left(g, \iota \right)}\right]  + \al^n_{\left(g, \iota \right)}  \left[ {\slashed D}_n, {\bt}^n_{\left(g, \iota \right)}\right]\right).
\end{split}
\end{equation*}
From \eqref{comm_ta_decomp_eqn} it turns out 
	\be\label{comm_al_bt_supn_eqn}
\supp ~{\al}^n_{\left(g, \iota \right)} \subset \widetilde{\pi}_{n}\left( g \widetilde{\mathcal U}_\iota\right) ; ~~~\supp ~{\bt}^n_{\left(g, \iota \right)}\subset \widetilde{\pi}_{n}\left(  g\widetilde{\mathcal U}_\iota\right) .
\ee
From the definition of $\widetilde{\mathcal U}_\iota$ it turns out that for any nontrivial $g' \in G\left( M_n~|~M\right)$ following condition holds
$$
\widetilde{\pi}_{n}\left(  g\widetilde{\mathcal U}_\iota\right) \bigcap g' \widetilde{\pi}_{n}\left(  g\widetilde{\mathcal U}_\iota\right)= \emptyset.
$$ 
and taking into account \eqref{comm_al_bt_supn_eqn} one has
\bean
{\al}^n_{\left(g, \iota \right)}\left(g' {\bt}^n_{\left(g, \iota \right)} \right) = 0,\\
 {\bt}^n_{\left(g, \iota \right)} \left[ {\slashed D}_n, g'{\al}^n_{\left(g, \iota \right)}\right] = 0,\\ \al^n_{\left(g, \iota \right)}  \left[ {\slashed D}_n, g'{\bt}^n_{\left(g, \iota \right)}\right]=0.
\eean
If $e \in G\left(M_n ~|~ M \right)$ is the neutral element of $G\left(M_n ~|~ M \right)$ then from the above equations it turns out
\bean
 {\bt}^n_{\left(g, \iota \right)} \left[ {\slashed D}_n, \sum_{g' \in G\left(M_n ~|~ M \right)} g'{\al}^n_{\left(g, \iota \right)}\right]=\sum_{g' \in G\left(M_n ~|~ M \right)}  {\bt}^n_{\left(g, \iota \right)} \left[ {\slashed D}_n, g'{\al}^n_{\left(g, \iota \right)}\right] = \\ ={\bt}^n_{\left(g, \iota \right)} \left[ {\slashed D}_n, {\al}^n_{\left(g, \iota \right)}\right]+\sum_{g' \in G\left(M_n ~|~ M \right)\backslash\{e\}} {\bt}^n_{\left(g, \iota \right)} \left[ {\slashed D}_n, g'{\al}^n_{\left(g, \iota \right)}\right]=\\=
 {\bt}^n_{\left(g, \iota \right)} \left[ {\slashed D}_n, {\al}^n_{\left(g, \iota \right)}\right] + 0 = {\bt}^n_{\left(g, \iota \right)} \left[ {\slashed D}_n, {\al}^n_{\left(g, \iota \right)}\right].
\eean 
Similarly one has
$$
 {\al}^n_{\left(g, \iota \right)} \left[ {\slashed D}_n, \sum_{g' \in G\left(M_n ~|~ M \right)} g'{\bt}^n_{\left(g, \iota \right)}\right] ={\al}^n_{\left(g, \iota \right)} \left[ {\slashed D}_n, {\bt}^n_{\left(g, \iota \right)}\right],
$$
or if ${a}_{\left(g, \iota \right)}, {\al}_{\left(g, \iota \right)}, {\bt}_{\left(g, \iota \right)}$ are given by
\be\label{comm_al_bt_sum_eqn}
\begin{split}
{a}_{\left(g, \iota \right)} = \sum_{g' \in G\left(M_n ~|~ M \right)} g'{a}^n_{\left(g, \iota \right)}~,\\
{\al}_{\left(g, \iota \right)} = \sum_{g' \in G\left(M_n ~|~ M \right)} g'{\al}^n_{\left(g, \iota \right)}~,\\
{\bt}_{\left(g, \iota \right)} = \sum_{g' \in G\left(M_n ~|~ M \right)} g'{\bt}^n_{\left(g, \iota \right)}\\
\end{split}
\ee
then
\bean
{\bt}^n_{\left(g, \iota \right)} \left[ {\slashed D}_n, {\al}^n_{\left(g, \iota \right)}\right]=\sum_{g' \in G\left(M_n ~|~ M \right)}  {\bt}^n_{\left(g, \iota \right)} \left[ {\slashed D}_n, g'{\al}^n_{\left(g, \iota \right)}\right]={\bt}^n_{\left(g, \iota \right)} \left[ {\slashed D}_n, {\al}_{\left(g, \iota \right)}\right],\\
{\al}^n_{\left(g, \iota \right)} \left[ {\slashed D}_n, {\bt}^n_{\left(g, \iota \right)}\right]=\sum_{g' \in G\left(M_n ~|~ M \right)}  {\al}^n_{\left(g, \iota \right)} \left[ {\slashed D}_n, g'{\bt}^n_{\left(g, \iota \right)}\right]={\al}^n_{\left(g, \iota \right)} \left[ {\slashed D}_n, {\bt}_{\left(g, \iota \right)}\right].\\
\eean
Elements ${\al}_{\left(g, \iota \right)}, ~{\bt}_{\left(g, \iota \right)}$ are invariant with respect to  $G\left(M_n ~|~ M \right)$, one has ${\al}_{\left(g, \iota \right)}, ~{\bt}_{\left(g, \iota \right)} \in \Coo\left( M\right)$. Since ${\slashed D}_n$ is the $\widetilde{   \pi}_n$-lift of $\slashed D$ following conditions hold:

 \be\label{comm_nt_o_eqn}
  {\bt}^n_{\left(g, \iota \right)} \left[ {\slashed D}_n, {\al}^n_{\left(g, \iota \right)}\right] = {\bt}^n_{\left(g, \iota \right)} \left[ {\slashed D}_n, {\al}_{\left(g, \iota \right)}\right] = {\bt}^n_{\left(g, \iota \right)} \otimes \left[ {\slashed D}, {\al}_{\left(g, \iota \right)}\right] \in \Coo\left(M_n \right)\otimes_{\Coo\left(M \right) } \Om^1_{\slashed D}
\ee
where $\Om^1_{\slashed D}$ is the {module of differential forms associated} with the spectral triple $$\left(C^{\infty}\left( M\right) , L^2\left( M, S\right) ,\slashed D\right)$$ (cf. Definition \ref{ass_cycle_sec}).  From the strong limit $\lim_{n \to \infty} {\bt}^n_{\left(g, \iota \right)}=\widetilde{\bt}_{\left(g, \iota \right)}$ it turns out
\bean
\lim_{n \to \infty} 1_{C_b\left(\widetilde{M} \right) }\otimes  {\bt}^n_{\left(g, \iota \right)}  \left[ {\slashed D}, {\al}_{\left(g, \iota \right)}\right]= \widetilde{\bt}_{\left(g, \iota \right)} \otimes \left[ {\slashed D}, {\al}_{\left(g, \iota \right)}\right]
\eean
and taking into account $\widetilde{\bt}_{\left(g, \iota \right)} \in C_c\left( \widetilde{M}\right)\subset L^2\left( C_0\left(\widetilde{M} \right) \right)$ one has
$$
\lim_{n \to \infty} 1_{C_b\left(\widetilde{M} \right) } \otimes {\bt}^n_{\left(g, \iota \right)}  \left[ {\slashed D}, {\al}_{\left(g, \iota \right)}\right]=\widetilde{\bt}_{\left(g, \iota \right)} \otimes \left[ {\slashed D}, {\al}_{\left(g, \iota \right)}\right]\in L^2\left( C_0\left(\widetilde{M} \right) \right) \otimes_{\Coo\left( M\right)}\Om^1_D.
$$
Similarly one can prove
$$
\lim_{n \to \infty} 1_{C_b\left(\widetilde{M} \right) } \otimes {\al}^n_{\left(g, \iota \right)}  \left[ {\slashed D}, {\bt}_{\left(g, \iota \right)}\right]=\widetilde{\al}_{\left(g, \iota \right)} \otimes \left[ {\slashed D}, {\bt}_{\left(g, \iota \right)}\right]\in L^2\left( C_0\left(\widetilde{M} \right) \right) \otimes_{\Coo\left( M\right)}\Om^1_{\slashed D}.
$$
Summarize above equation one has
\be\label{comm_ad_eqn}
\begin{split}
\widetilde{a}^{\slashed D}_{\left(g, \iota \right)}=\lim_{n \to \infty} 1_{C_b\left(\widetilde{M} \right) } \otimes \left[ {\slashed D}_n, a^n_{\left(g, \iota \right)}\right]=\\= \lim_{n \to \infty}  \left( {\bt^n}_{\left(g, \iota \right)} \otimes \left[ {\slashed D}_n, {\al}^n_{\left(g, \iota \right)}\right]  + \al^n_{\left(g, \iota \right)} \otimes  \left[ {\slashed D}_n, {\bt}^n_{\left(g, \iota \right)}\right]\right)=\\=
\sum_{\left(g, \iota \right) \in \widetilde{I}} \left( \widetilde{\bt}_{\left(g, \iota \right)} \otimes \left[ {\slashed D}, {\al}_{\left(g, \iota \right)}\right]  + \widetilde{\al}_{\left(g, \iota \right)} \otimes  \left[ {\slashed D}, {\bt}_{\left(g, \iota \right)}\right]\right).
\end{split}
\ee
From $\widetilde{\al}_{\left(g, \iota \right)}, ~\widetilde{\bt}_{\left(g, \iota \right)}\in L^2\left( C_0\left(\widetilde{M} \right) \right)$, $~\left[ {\slashed D}, {\al}_{\left(g, \iota \right)}\right],\left[ {\slashed D}, {\bt}_{\left(g, \iota \right)}\right]\in \Om^1_{\slashed D}$ it follows that $\widetilde{a}^{\slashed D}_{\left(g, \iota \right)} \in L^2\left( C_0\left(\widetilde{M} \right) \right) \otimes_{\Coo\left( M\right)}\Om^1_{\slashed D}$.
Taking into account that $\widetilde{I}$ is a finite set one concludes that
\be
\widetilde{a}_{\slashed D}=\lim_{n \to \infty} 1_{C_b\left(\widetilde{M} \right) } \otimes \left[ {\slashed D}_n, a_n\right]= \sum_{\left( g, \iota\right)\in \widetilde{I}} \widetilde{a}^{\slashed D}_{\left(g, \iota \right)} \in L^2\left( C_0\left(\widetilde{M} \right) \right) \otimes_{\Coo\left( M\right)}\Om^1_{\slashed D}. 
\ee
(d) It is known that $C_c\left( \widetilde{M}\right)$ the  Pedersen ideal of $C_0\left( \widetilde{M}\right)$ (cf. \cite{pedersen:mea_c}). From $\widetilde{a}\in C_c\left( \widetilde{M}\right)$ it turns out that $\widetilde{a}$ lies in the Pedersen ideal. 
\end{proof}
\subsubsection{Lift of Dirac operator}
\paragraph{}
For any $\xi \in \Ga^\infty\left(M, S \right)$ one has
$$
\supp \widetilde{\bt}_{\left(g, \iota \right)} \otimes \left[ {\slashed D}, {\al}_{\left(g, \iota \right)}\right]\xi \subset g   \widetilde{\mathcal U}_\iota,
$$
so from \eqref{supp_lift_desc_eqn} it turns out
\bean
 \widetilde{\bt}_{\left(g, \iota \right)} \otimes \left[ {\slashed D}, {\al}_{\left(g, \iota \right)}\right]\xi = \mathfrak{lift}_{g\widetilde{\mathcal{U}}}\left( \mathfrak{desc} \left( \widetilde{\bt}_{\left(g, \iota \right)} \otimes \left[ {\slashed D}, {\al}_{\left(g, \iota \right)}\right]\xi \right)\right) =\\= \mathfrak{lift}_{g\widetilde{\mathcal{U}}}\left({\bt}_{\left(g, \iota \right)}  \left[ {\slashed D}, {\al}_{\left(g, \iota \right)}\right]\xi \right)=\mathfrak{lift}_{g\widetilde{\mathcal{U}}}  \left( \left[ {\slashed D}, {\al}_{\left(g, \iota \right)}\right]{\bt}_{\left(g, \iota \right)}\xi \right).
\eean
Similarly one has
\bean
\widetilde{\al}_{\left(g, \iota \right)} \otimes \left[ {\slashed D}, {\bt}_{\left(g, \iota \right)}\right]\xi= \mathfrak{lift}_{g\widetilde{\mathcal{U}}}\left({\bt}_{\left(g, \iota \right)}  \left[ {\slashed D}, {\al}_{\left(g, \iota \right)}\right]\xi\right),\\
 \widetilde{a}_{\left(g, \iota \right)}\otimes  {\slashed D}\xi= \mathfrak{lift}_{g\widetilde{\mathcal{U}}}\left(a_{\left(g, \iota \right)}{\slashed D}\xi\right) 
\eean
From \eqref{comm_ad_eqn} it turns out
\bean
\widetilde{a}^{\slashed D}_{\left(g, \iota \right)}\left(\xi \right)  = \mathfrak{lift}_{g\widetilde{\mathcal{U}}}  \left(\left(  \left[ {\slashed D}, {\al}_{\left(g, \iota \right)}\right]{\bt}_{\left(g, \iota \right)}+{\al}_{\left(g, \iota \right)}\left[ {\slashed D}, {\bt}_{\left(g, \iota \right)}\right]\right) \xi\right)=\\= \mathfrak{lift}_{g\widetilde{\mathcal{U}}}  \left(\left[ {\slashed D}, {\al}_{\left(g, \iota \right)}{\bt}_{\left(g, \iota \right)}\right]\xi\right) =  \mathfrak{lift}_{g\widetilde{\mathcal{U}}}  \left(\left[ {\slashed D}, {a}_{\left(g, \iota \right)}\right]\xi\right) ,\\
\eean
hence one has
\be\label{comm_pre_d_eqn}
\begin{split}
\widetilde{a}^{\slashed D}_{\left(g, \iota \right)}\left(\xi \right) + \widetilde{a}_{\left(g, \iota \right)} \otimes \slashed D \xi = \mathfrak{lift}_{g\widetilde{\mathcal{U}}}  \left(\left( \left[ {\slashed D}, {a}_{\left(g, \iota \right)}\right]+ {a}_{\left(g, \iota \right)}{\slashed D}\right) \xi\right) = \\ = \mathfrak{lift}_{g\widetilde{\mathcal{U}}}  \left({\slashed D} {a}_{\left(g, \iota \right)}\xi\right) = \widetilde{\slashed D}\left( \widetilde{a}_{\left(g, \iota \right)}\otimes\xi\right) 
\end{split}
\ee
Denote by $\slashed D'$ the operator given by \eqref{inf_lift_D_eqn}. From $\widetilde{a}  = \sum_{\left(g, \iota \right) \in \widetilde{I}} \widetilde{a}_{\left(g, \iota \right)}$ it turns out
\bean
\slashed D'\left( \widetilde{a}\otimes \xi\right) = \sum_{\left(g, \iota \right) \in \widetilde{I}} \slashed D'\left( \widetilde{a}_{\left(g, \iota \right)} \otimes \xi\right). 
\eean
From \eqref{inf_lift_D_eqn} it follows that
turns out
\bean
\slashed D'\left( \widetilde{a}\otimes \xi\right) = \sum_{\left(g, \iota \right) \in \widetilde{I}} \slashed D'\left( \widetilde{a}_{\left(g, \iota \right)} \otimes \xi\right)=\sum_{\left(g, \iota \right) \in \widetilde{I}}\left( \widetilde{a}^{\slashed D}_{\left(g, \iota \right)}\left(\xi \right) + \widetilde{a}_{\left(g, \iota \right)} \otimes \slashed D \xi\right) , 
\eean
and taking into account \eqref{comm_pre_d_eqn} one has
\bean
\slashed D'\left( \widetilde{a}\otimes \xi\right) =\sum_{\left(g, \iota \right) \in \widetilde{I}}\left( \widetilde{a}^{\slashed D}_{\left(g, \iota \right)}\left(\xi \right) + \widetilde{a}_{\left(g, \iota \right)} \otimes \slashed D \xi\right)= \sum_{\left(g, \iota \right) \in \widetilde{I}}\widetilde{\slashed D}\left( \widetilde{a}_{\left(g, \iota \right)}\otimes\xi\right) = \widetilde{\slashed D} \left( \widetilde{a}\otimes \xi\right).
\eean
From the above equation it turns out $\slashed D'= \widetilde{\slashed D}$, i.e. in the  specific case of commutative spectral triples the Dirac operator given by \eqref{inf_lift_D_eqn} and the Definition \ref{reg_triple_defn} is the lift of Dirac operator given by the Definition \ref{inv_image_defn}.

\subsubsection{Coverings of spectral triples} 
 Since the space $\Coo\left(\widetilde{M} \right) \bigcap C_c\left(\widetilde{M} \right)$ is dense in $C_0\left(\widetilde{M} \right)$,
 we have the following theorem
 \begin{theorem}
 	Following conditions hold:
 	\begin{itemize}
 		\item		The  sequence of spectral triples
 		\begin{equation*}
 		\begin{split}
 		\mathfrak{S}_{\left(C^{\infty}\left( M\right) , L^2\left( M, S\right) ,\slashed D\right) } = \{ \left(C^{\infty}\left( M\right) , L^2\left( M, S\right) ,\slashed D\right)=\left(C^{\infty}\left( M_0\right) , L^2\left( M_0, S_0\right) ,\slashed D_0\right), \dots,\\
 		\left(C^{\infty}\left( M_n\right) , L^2\left( M_n, S_n\right) ,\slashed D_n\right), \dots
 		\}\in \mathfrak{CohTriple}
 		\end{split}
 		\end{equation*}
 		given by \eqref{comm_triple_seq_eqn} is good (cf. Definition \ref{smooth_alg_defn}),
 		\item The triple $\left( \Coo_0\left(\widetilde{ M}\right) , L^2\left( \widetilde{ M}, \widetilde{\SS}\right),    \widetilde{\slashed D}\right) $ is the $\left(C\left(M \right), C_0\left( \widetilde{ M}\right), G\left(\widetilde{ M}~|~M \right)    \right)$-lift of $$\left( \Coo\left(M\right) , L^2\left( M, {\SS}\right), \slashed D\right)$$ (cf. Definition \ref{reg_triple_defn}). 
 	\end{itemize}
 	
 \end{theorem}

\section{Coverings of  noncommutative tori}\label{nt_sec}  
   \subsection{Fourier transformation}
\paragraph*{}
There is a norm on $\mathbb{Z}^n$ given by
\begin{equation*}
\left\|\left(k_1, ..., k_n\right)\right\|= \sqrt{k_1^2 + ... + k^2_n}.
\end{equation*}
The space of complex-valued Schwartz  functions on $\Z^n$ is given by 
\begin{equation*}
\sS\left(\mathbb{Z}^n\right)= \left\{a = \left\{a_k\right\}_{k \in \mathbb{Z}^n} \in \mathbb{C}^{\mathbb{Z}^n}~|~ \mathrm{sup}_{k \in \mathbb{Z}^n}\left(1 + \|k\|\right)^s \left|a_k\right| < \infty, ~ \forall s \in \mathbb{N} \right\}.
\end{equation*}

Let $\mathbb{T}^n$ be an ordinary $n$-torus. We will often use real coordinates for $\mathbb{T}^n$, that is, view $\mathbb{T}^n$ as $\mathbb{R}^n / \mathbb{Z}^n$. Let $\Coo\left(\mathbb{T}^n\right)$ be an algebra of infinitely differentiable complex-valued functions on $\mathbb{T}^n$. 
There is the bijective Fourier transformations  $\mathcal{F}_\T:\Coo\left(\mathbb{T}^n\right)\xrightarrow{\approx}\sS\left(\mathbb{Z}^n\right)$;  $f \mapsto \widehat{f}$ given by
\begin{equation}\label{nt_fourier}
\widehat{f}\left(p\right)= \mathcal F_\T (f) (p)= \int_{\mathbb{T}^n}e^{- 2\pi i x \cdot p}f\left(x\right)dx
\end{equation}
where $dx$ is induced by the Lebesgue measure on $\mathbb{R}^n$ and   $\cdot$ is the  scalar
product on the Euclidean space $\R^n$.
The Fourier transformation carries multiplication to convolution, i.e.
\begin{equation*}
\widehat{fg}\left(p\right) = \sum_{r +s = p}\widehat{f}\left(r\right)\widehat{g}\left(s\right).
\end{equation*}
The inverse Fourier transformation $\mathcal{F}^{-1}_\T:\sS\left(\mathbb{Z}^n\right)\xrightarrow{\approx} \Coo\left(\mathbb{T}^n\right)$;  $ \widehat{f}\mapsto f$ is given by
$$
f\left(x \right) =\mathcal{F}^{-1}_\T \widehat f\left( x\right)  = \sum_{p \in \Z^n} \widehat f\left( p\right)   e^{ 2\pi i x \cdot p}.
$$
There  is the $\C$-valued scalar product  on $\Coo\left( \T^n\right)$ given by
$$
\left(f, g \right) = \int_{\T^n}fg dx =\sum_{p \in \Z^n}\widehat{f}\left( -p\right) \widehat{g}\left(p \right).  
$$ 
Denote by $\SS\left( \R^{n}\right) $ be the space of
complex Schwartz (smooth, rapidly decreasing) functions on $\R^{n}$. 
\be\label{mp_sr_eqn}
\begin{split}
	\SS\left(\mathbb {R} ^{n}\right)=\left\{f\in C^{\infty }(\mathbb {R} ^{n}):\|f\|_{\alpha  ,\beta )}<\infty \quad \forall \alpha =\left( \al_1,...,\al_n\right) ,\beta =\left( \bt_1,...,\bt_n\right)\in \mathbb {Z} _{+}^{n}\right\},\\
	\|f\|_{{\alpha ,\beta }}=\sup_{{x\in {\mathbb  {R}}^{n}}}\left|x^{\alpha }D^{\beta }f(x)\right|
\end{split}
\ee
where 
\bean
x^\al = x_1^{\al_1}\cdot...\cdot x_n^{\al_n},\\
D^{\beta} = \frac{\partial}{\partial x_1^{\bt_1}}~...~\frac{\partial}{\partial x_n^{\bt_n}}.
\eean
The topology on $\SS\left(\mathbb {R} ^{n}\right)$ is given by seminorms $\|\cdot\|_{{\alpha ,\beta }}$.
\begin{defn}\label{nt_*w_defn}
	Denote by $\SS'\left( \R^{n}\right) $ the vector space dual to $\SS\left( \R^{n}\right) $, i.e. the space of continuous functionals on $\SS\left( \R^{n}\right)$. Denote by $\left\langle\cdot, \cdot \right\rangle:\SS'\left( \R^{n}\right)\times	\SS\left( \R^{n}\right)\to\C$ the natural pairing. We say that $\left\{a_n \in \SS'\left(\mathbb {R} ^{n}\right)\right\}_{n \in \N}$ is \textit{weakly-* convergent} to $a \in \SS'\left(\mathbb {R} ^{n}\right)$ if for any $b \in  \SS\left(\mathbb {R} ^{n}\right)$ following condition holds
	$$
	\lim_{n \to \infty}\left\langle a_n, b \right\rangle = \left\langle a, b \right\rangle.	
	$$
	We say that
	$$
	a = \lim_{n\to \infty}a_n
	$$
	in the \textit{sense of weak-* convergence}.
\end{defn}

Let $\mathcal F$ and $\mathcal F^{-1}$ be the ordinary and inverse Fourier transformations given by
\begin{equation}\label{intro_fourier}
\begin{split}
\left(\mathcal{F}f\right)(u) = \int_{\mathbb{R}^{2N}} f(t)e^{-2\pi it\cdot u}dt,~\left(\mathcal F^{-1}f\right)(u)=\int_{\mathbb{R}^{2N}} f(t)e^{2\pi it\cdot u}dt %,
% \left(\widetilde{F}f\right)(u) = \int f(t)e^{it\cdot Ju}dt.~
\end{split}
\end{equation}
which satisfy  following conditions
$$
\mathcal{F}\circ\mathcal{F}^{-1}|_{\SS\left( \R^{n}\right)} = \mathcal{F}^{-1}\circ\mathcal{F}|_{\SS\left( \R^{n}\right)} = \Id_{\SS\left( \R^{n}\right)}.
$$
There is the $\C$-valued scalar product  on $\SS\left( \R^n\right)$ given by
\begin{equation}\label{fourier_scalar_product_eqn}
\left(f, g \right) = \int_{\R^n}fg dx =\int_{\R^n}\mathcal{F}f\mathcal{F}g dx. 
\end{equation}
which if $\mathcal{F}$-invariant, i.e.
$$
\left(f, g \right) = \left(\mathcal{F}f, \mathcal{F}g \right).
$$

\subsection{Noncommutative torus $\mathbb{T}^n_{\Theta}$}\label{nt_descr_subsec}

\paragraph*{}

Let $\Theta$ be a real skew-symmetric $n \times n$ matrix, we will define a new noncommutative product $\star_{\Theta}$ on $\sS\left(\mathbb{Z}^n\right)$ given by
\begin{equation}\label{nt_product_defn_eqn}
\left(\widehat{f}\star_{\Theta}\widehat{g}\right)\left(p\right)= \sum_{r + s = p} \widehat{f}\left(r\right)\widehat{g}\left(s\right) e^{-\pi ir ~\cdot~ \Theta s}.
\end{equation}
and an involution
\begin{equation*}
\widehat{f}^*\left(p\right)=\overline{\widehat{f}}\left(-p\right)
.
\end{equation*}In result there is an involutive algebra $\Coo\left(\mathbb{T}^n_{\Theta}\right) =\left(\sS\left(\mathbb{Z}^n\right), + , \star_{\Theta}~, ^* \right)$. 
There is a tracial  state on $\Coo\left(\mathbb{T}^n_{\Theta}\right)$ given by
\begin{equation}\label{nt_state_eqn}
\tau\left(f\right)= \widehat{f}\left(0\right).
\end{equation}
From  $\Coo\left(\mathbb{T}^n_{\Theta} \right) \approx \SS\left( \Z^n\right)$ it follows  that there is a $\C$-linear isomorphism 
\begin{equation}\label{nt_varphi_inf_eqn}
\varphi_\infty: \Coo\left(\mathbb{T}^n_{\Theta} \right) \xrightarrow{\approx}  \Coo\left(\mathbb{T}^n \right).
\end{equation} 
such that following condition holds
\begin{equation}\label{nt_state_integ_eqn}
\tau\left(f \right)=  \frac{1}{\left( 2\pi\right)^n }\int_{\mathbb{T}^n} \varphi_\infty\left( f\right) ~dx.
\end{equation}

Similarly to \ref{comm_gns_constr} there is the Hilbert space $L^2\left(\Coo\left(\mathbb{T}^n_{\Theta}\right), \tau\right)$ and the natural representation $\Coo\left(\mathbb{T}^n_{\Theta}\right) \to B\left(L^2\left(\Coo\left(\mathbb{T}^n_{\Theta}\right), \tau\right)\right)$ which induces the $C^*$-norm. The $C^*$-norm completion  $C\left(\mathbb{T}^n_{\Theta}\right)$ of $\Coo\left(\mathbb{T}^n_{\Theta}\right)$ is a $C^*$-algebra and there is a faithful representation
\begin{equation}\label{nt_repr_eqn}
C\left(\mathbb{T}^n_{\Theta}\right) \to B\left( L^2\left(\Coo\left(\mathbb{T}^n_{\Theta}\right), \tau\right)\right) .
\end{equation}
We will write $L^2\left(C\left(\mathbb{T}^n_{\Theta}\right), \tau\right)$ instead of $L^2\left(\Coo\left(\mathbb{T}^n_{\Theta}\right), \tau\right)$. There is the natural $\C$-linear map  $\Coo\left(\mathbb{T}^n_{\Theta}\right) \to L^2\left(C\left(\mathbb{T}^n_{\Theta}\right), \tau\right)$  and since $\Coo\left(\mathbb{T}^n_{\Theta}\right) \approx \sS\left( \mathbb{Z}^n \right)$ there is a linear map $\Psi_\Th:\sS\left( \mathbb{Z}^n \right) \to L^2\left(C\left(\mathbb{T}^n_{\Theta}\right), \tau\right) $. If $k \in \mathbb{Z}^n$ and $U_k \in  \sS\left( \mathbb{Z}^n \right)=\Coo\left(\mathbb{T}^n_{\Theta}\right)$ is such that 
\begin{equation}\label{unitaty_nt_eqn}
U_k\left( p\right)= \delta_{kp}: ~ \forall p \in \mathbb{Z}^n
\end{equation}
then
\begin{equation}\label{nt_unitary_product}
U_kU_p = e^{-\pi ik ~\cdot~ \Theta p} U_{k + p}; ~~~   U_kU_p = e^{-2\pi ik ~\cdot~ \Theta p}U_pU_k.
\end{equation}

If $\xi_k = \Psi_\Th\left(U_k \right)$ then from \eqref{nt_product_defn_eqn}, \eqref{nt_state_eqn} it turns out
\begin{equation}\label{nt_h_product}
\tau\left(U^*_k \star_\Th U_l \right) = \left(\xi_k, \xi_l \right)  = \delta_{kl},   
\end{equation} 
i.e. the subset $\left\{\xi_k\right\}_{k \in \mathbb{Z}^n}\subset L^2\left(C\left(\mathbb{T}^n_{\Theta}\right), \tau\right)$ is an orthogonal basis of  $L^2\left(C\left(\mathbb{T}^n_{\Theta}\right), \tau\right)$.
Hence the Hilbert space  $L^2\left(C\left(\mathbb{T}^n_{\Theta}\right), \tau\right)$ is naturally isomorphic to the Hilbert space $\ell^2\left(\mathbb{Z}^n\right)$ given by
\begin{equation*}
\ell^2\left(\mathbb{Z}^n\right) = \left\{\xi = \left\{\xi_k \in \mathbb{C}\right\}_{k\in \mathbb{Z}^n} \in \mathbb{C}^{\mathbb{Z}^n}~|~ \sum_{k\in \mathbb{Z}^n} \left|\xi_k\right|^2 < \infty\right\}
\end{equation*}
and the $\C$-valued scalar product on $\ell^2\left(\mathbb{Z}^n\right)$ is given by
\begin{equation*}
\left(\xi,\eta\right)_{ \ell^2\left(\mathbb{Z}^n\right)}= \sum_{k\in \mathbb{Z}^n}    \overline{\xi}_k\eta_k.
\end{equation*}
The map $\Psi_\Th:\sS\left( \mathbb{Z}^n \right) \hookto L^2\left(C\left(\mathbb{T}^n_{\Theta}\right), \tau\right) $ is equivalent to the map
\begin{equation}\label{nt_to_hilbert_eqn}
\Psi_\Th:\Coo\left(\mathbb{T}^n_{\Theta}\right) \hookto L^2\left(C\left(\mathbb{T}^n_{\Theta}\right), \tau\right).
\end{equation}
From \eqref{nt_state_integ_eqn} it follows that for any $a, b \in \Coo\left(\mathbb{T}^n_{\Theta}\right)$  the scalar product on $L^2\left(C\left(\mathbb{T}^n_{\Theta}\right), \tau\right)$ is given by
\begin{equation}\label{nt_int_sc_pr_eqn}
\left(a, b \right)= \int_{\T^n} a_{\text{comm}}^*b_{\text{comm}}dx 
\end{equation}
where $a_{\text{comm}}\in \Coo\left(\T^n \right)$ (resp. $b_{\text{comm}})$ is a commutative function which corresponds to $a$ (resp. $b$).
An alternative description of $\C\left(\mathbb{T}^n_{\Theta}\right)$ is such that if
\begin{equation}\label{nt_th_eqn}
\Th = \begin{pmatrix}
0& \th_{12} &\ldots & \th_{1n}\\
\th_{21}& 0 &\ldots & \th_{2n}\\
\vdots& \vdots &\ddots & \vdots\\
\th_{n1}& \th_{n2} &\ldots & 0
\end{pmatrix}
\end{equation}
then $C\left(\mathbb{T}^n_{\Theta}\right)$ is the universal $C^*$-algebra generated by unitary elements   $u_1,..., u_n \in U\left( C\left(\mathbb{T}^n_{\Theta}\right)\right) $ such that following condition holds
\begin{equation}\label{nt_com_eqn}
u_ju_k = e^{-2\pi i \theta_{jk} }u_ku_j.
\end{equation}
Unitary  operators $u_1,..., u_n$ correspond to the standard basis of $\mathbb{Z}^n$.
\begin{defn}\label{nt_uni_defn}
	Unitary elements 
	$u_1,..., u_n \in U\left(C\left(\mathbb{T}^n_{\theta}\right)\right)$ which satisfy the relation \eqref{nt_com_eqn}
	are said to be \textit{generators} of $C\left(\mathbb{T}^n_{\Theta}\right)$. The set $\left\{U_l\right\}_{l \in \Z^n}$ is said to be the \textit{basis} of $C\left(\mathbb{T}^n_{\Theta}\right)$.
\end{defn}
\begin{defn}\label{nt_symplectic_defn}
	If  $\Theta$ is non-degenerated, that is to say,
	$\sigma(s,t) \stackrel{\mathrm{def}}{=} s\.\Theta t$ to be \textit{symplectic}. This implies even
	dimension, $n = 2N$. One then selects
	\begin{equation}\label{nt_simpectic_theta_eqn}
	\Theta = \theta J
	\stackrel{\mathrm{def}}{=} \th \begin{pmatrix} 0 & 1_N \\ -1_N & 0 \end{pmatrix}
	\end{equation}
	where  $\th > 0$ is defined by $\th^{2N} \stackrel{\mathrm{def}}{=} \det\Theta$.
	Denote by $\Coo\left(\mathbb{T}^{2N}_\th\right)\stackrel{\mathrm{def}}{=}\Coo\left(\mathbb{T}^{2N}_\Th\right)$ and $C\left(\mathbb{T}^{2N}_\th\right)\stackrel{\mathrm{def}}{=}C\left(\mathbb{T}^{2N}_\Th\right)$.
\end{defn}
\subsection{Geometry of noncommutative tori}\label{nt_geom_sec}

\paragraph*{}
Denote by $\delta_{\mu}$ ($\mu = 1,\dots, n$) the analogues of the partial derivatives $ \frac{1}{i}\frac{\partial}{\partial x^{\mu}}$ on $C^{\infty}(\mathbb{T}^n)$ which are derivations on the algebra $C^{\infty}(\mathbb{T}^n_{\Theta})$ given by
$$\delta_{\mu}(U_k)=k_{\mu} U_k.$$
These derivations have the following property
$$\delta_{\mu}(a^*)=-(\delta_{\mu}a)^*,$$
and also satisfy  the integration by parts formula
$$\tau(a\delta_{\mu}b)=-\tau((\delta_{\mu}a)b),\quad a,b\in C^{\infty}(\mathbb{T}^n_{\Theta}).$$

The spectral triple describing the noncommutative geometry of noncommutative $n$-torus consists of the algebra $C^{\infty}(\mathbb{T}^n_{\Theta})$, the Hilbert space $\mathcal{H}=L^2\left(C\left(\mathbb{T}^n_{\Theta}\right), \tau\right)\otimes\mathbb{C}^{m}$, where $m=2^{[n/2]}$ with  the representation  $\pi\otimes 1_{B\left(\C^m \right) }:C^{\infty}(\mathbb{T}^n_{\Theta})\to B\left( \H\right) $  where $\pi: C^{\infty}(\mathbb{T}^n_{\Theta})\to B\left( L^2\left(C\left(\mathbb{T}^n_{\Theta}\right), \tau\right)\right)$ is given by \eqref{nt_repr_eqn}.
The Dirac operator is given by
\begin{equation}\label{nt_dirac_eqn}
D=\slashed{\partial}\stackrel{\mathrm{def}}{=}\sum_{\mu =1}^{n}\partial_{\mu}\otimes\gamma^{\mu}\cong\sum_{\mu =1}^{n}\delta_{\mu} \otimes\gamma^{\mu},
\end{equation}
where $\partial_{\mu}=\delta_{\mu}$, seen as an unbounded self-adjoint operator on $L^2\left(C\left(\mathbb{T}^n_{\Theta}\right), \tau\right) $ and $\gamma^{\mu}$s are Clifford (Gamma) matrices in $\mathbb{M}_n(\mathbb{C})$ satisfying the relation
\be\label{nt_gam_eqn}\gamma^i\gamma^j+\gamma^j\gamma^i=2\delta^{ij} I_N.\ee
There is a spectral triple 
\begin{equation}\label{nt_sp_tr_eqn}
\left( C^{\infty}(\mathbb{T}^n_{\Theta}),L^2\left(C\left(\mathbb{T}^n_{\Theta}\right), \tau\right)\otimes\mathbb{C}^{m},D\right).
\end{equation}

%There is useful an alternative description of $D$. Note that there are
%\begin{itemize}
%	\item A $\C$- linear isomorphism $\Coo\left(\T \right) \xrightarrow{\approx} \Coo\left(\T_\Th \right)$,
%	\item An infinite covering $\R^n \to \R^n/2\pi\Z^n = \T^n$.
%\end{itemize}
%Above objects enables represent elements of $\Coo\left(\T_\Th \right)$ by smooth $2\pi$ periodic functions on $\R^n$. If $x_1, ..., x_n$ are Cartesian coordinates then one has
%\begin{equation}
%\delta_{\mu} = \frac{\partial}{\partial x_\mu}.
%\end{equation}
There is an alternative description of $D$. The space  $\Coo\left(\T^n \right)$ (resp. $\Coo\left(\T^n_\Th \right)$ ) is dense in $L^2\left( \T^n\right)$ (resp. $L^2\left( \Coo\left(\T^n_\Th \right), \tau\right)$ ), hence from the $\C$-linear isomorphism $\varphi_\infty:\Coo\left(\mathbb{T}^n_{\Theta} \right) \xrightarrow{\approx}  \Coo\left(\mathbb{T}^n \right)$ given by \eqref{nt_varphi_inf_eqn} it follows isomorphism of Hilbert spaces 
\begin{equation*}
\varphi: L^2\left(C\left(\mathbb{T}^n_{\Theta}\right), \tau\right) \xrightarrow{\approx} L^2\left(\T^n \right). 
\end{equation*}
Otherwise $\T^n$ admits a Spin-bundle $S$ such that $L^2\left(\T^2,S \right) \approx L^2\left(\T^n \right) \otimes \mathbb{C}^{m}$. It turns out an isomorphism of Hilbert spaces
\begin{equation*}
\varPhi: L^2\left(C\left(\mathbb{T}^n_{\Theta}\right), \tau\right)\otimes \mathbb{C}^{m} \xrightarrow{\approx} L^2\left(\T^n,S \right). 
\end{equation*}
There is a commutative spectral triple 
\begin{equation}\label{nt_ct_tr_eqn}
\left( \Coo\left(\mathbb{T}^n \right),  L^2\left(\T^n,S \right), \slashed D\right)
\end{equation} 
such that $D$ is given by
\begin{equation}\label{nt_comm_dirac_eqn}
D = \Phi^{-1} \circ \slashed D \circ \Phi.
\end{equation}
Noncommutative geometry replaces differentials with commutators such that the differential $df$ corresponds to $\frac{1}{i}\left[ \slashed D, f\right]$ and the well known equation
\begin{equation*}
d f = \sum_{\mu = 1}^n \frac{\partial f}{\partial x_\mu} dx_\mu
\end{equation*}
is replaced with 
\begin{equation}\label{nt_comm_diff_eqn}
\left[\slashed D, f \right]  = \sum_{\mu = 1}^n \frac{\partial f}{\partial x_\mu} \left[\slashed D, x_\mu \right]
\end{equation}
 In case of commutative torus we on has
\begin{equation*}
dx_\mu = iu^*_\mu du_\mu
\end{equation*}
where $u_\mu = e^{-ix_\mu}$, so what equation \eqref{nt_comm_diff_eqn} can be written by the following way
\begin{equation}\label{nt_comm_diff_mod_eqn}
\left[\slashed D, f \right]  = \sum_{\mu = 1}^n \frac{\partial f}{\partial x_\mu} u^*_\mu \left[\slashed D, u_\mu \right]
\end{equation}
We  would like to prove a noncommutative analog of \eqref{nt_comm_diff_mod_eqn}, i.e. for any $a \in \Coo\left(\mathbb{T}^n_\Th \right)$ following condition holds
\begin{equation}\label{nt_diff_mod_eqn}
\left[D, a \right]  = \sum_{\mu = 1}^n \frac{\partial a}{\partial x_\mu} u^*_\mu \left[ D, u_\mu \right]
\end{equation}
From \eqref{nt_dirac_eqn} it follows that \eqref{nt_diff_mod_eqn} is true if and only if 
\begin{equation}\label{nt_mu_mod_eqn}
\left[\delta_{\mu}, a \right]  = \sum_{\mu = 1}^n \frac{\partial a}{\partial x_\mu} u^*_\mu \left[ \delta_{\mu}, u_\mu \right]; ~ \mu =1,\dots n.
\end{equation}
In the above equation $\frac{\partial a}{\partial x_\mu}$ means that one considers $a$ as element of $\Coo\left(\mathbb{T}^n\right)$, takes $\frac{\partial}{\partial x_\mu}$ of it and then the result of derivation considers as element of $\Coo\left(\mathbb{T}^n_\Th \right)$. Since the linear span of elements $U_k$ is dense in both $\Coo\left(\mathbb{T}^n_\Th \right)$ and $L^2\left(C\left(\mathbb{T}^n_{\Theta}\right), \tau\right)$ the equation \eqref{nt_mu_mod_eqn} is true if  for any $k, l \in \Z^n$ following condition holds
\begin{equation*}\label{nt_mu_mod_eqn1}
\left[\delta_{\mu}, U_k \right]U_l  =  \frac{\partial U_k}{\partial x_\mu} u^*_\mu \left[ \delta_{\mu}, u_\mu \right]U_l.
\end{equation*}
The above equation is a consequence of the following calculations: 
\begin{equation*}
\begin{split}
\left[\delta_{\mu}, U_k \right]U_l = \delta_{\mu} U_k U_l - U_k \delta_k U_l = (k + l) U_k U_l - l U_k U_l = k U_kU_l, \\
\frac{\partial U_k}{\partial x_\mu} u^*_\mu \left[ \delta_{\mu}, u_\mu \right]U_l = kU_k u^*\left(\delta_{\mu} u_\mu U_l - u_\mu \delta U_l \right)= kU_k u^*\left(\left(l + 1 \right)  u_\mu U_l - l u_\mu \delta U_l \right) =\\=
k U_ku^*_\mu u_\mu U_l = k U_kU_l.
\end{split}
\end{equation*}
For any $k \in \Z^n$ following condition holds
$$
u^*_\mu \left[ \delta_{\mu}, u_\mu \right]U_k = u^*_\mu \delta_{\mu} u_\mu U_k - u^*_\mu \ u_\mu \delta_{\mu} U_k = u^*_\mu \left(\left(k + 1 \right)  u_\mu U_k - k  u_\mu U_k\right) = U_k
$$
it turns out
\begin{equation}\label{nt_diff_1_eqn}
u^*_\mu \left[ \delta_{\mu}, u_\mu \right]= 1_{C\left(\mathbb{T}^n_{\Theta}\right)}; ~~ u^*_\mu \left[ D, u_\mu \right] = \ga^\mu.
\end{equation}

From \eqref{nt_dirac_eqn}, \eqref{nt_diff_mod_eqn} and \eqref{nt_diff_1_eqn} it turns out

\begin{equation}\label{nt_comm_diff_mod_gamma_eqn}
\left[D, f \right]  = \sum_{\mu = 1}^n \frac{\partial f}{\partial x_\mu} \ga^\mu.
\end{equation}
\begin{equation}\label{nt_comm_diff_gamma_in_eqn}
\ga^\mu \in \Om^1_D
\end{equation}
where $\Om^1_D$ is the {module of differential forms associated} with the spectral triple  \eqref{nt_sp_tr_eqn} (cf. Definition \ref{ass_cycle_defn}).

\subsection{Finite-fold coverings}\label{nt_fin_cov}
\subsubsection{Basic construction}
\paragraph{}  In this section we write $ab$ instead $a\star_\Th b$.
%\subsubsection{Construction of coverings} \label{nt_cov_constr_sub_sub_sec} 
Let $\Th$ be given by \eqref{nt_th_eqn}, and let $C\left(\mathbb{T}^n_\Theta\right)$ be a noncommutative torus. If  $\overline{k} = \left(k_1, ..., k_n\right) \in \mathbb{N}^n$ and
$$
\widetilde{\Theta} = \begin{pmatrix}
0& \widetilde{\theta}_{12} &\ldots & \widetilde{\theta}_{1n}\\
\widetilde{\theta}_{21}& 0 &\ldots & \widetilde{\theta}_{2n}\\
\vdots& \vdots &\ddots & \vdots\\
\widetilde{\theta}_{n1}& \widetilde{\theta}_{n2} &\ldots & 0
\end{pmatrix}
$$
is a skew-symmetric matrix such that
\begin{equation*}
e^{-2\pi i \theta_{rs}}= e^{-2\pi i \widetilde{\theta}_{rs}k_rk_s}
\end{equation*}
then there is a *-homomorphism $C\left(\mathbb{T}^n_\Th\right)\to C\left(\mathbb{T}^n_{\widetilde{\Th}}\right)$ given by
\begin{equation}\label{nt_cov_eqn}
u_j \mapsto v~^{k_j}_j; ~ j = 1,...,n
\end{equation}
where $u_1,..., u_n \in C\left(\mathbb{T}^n_{\Th}\right)$ (resp. $v_1,..., v_n \in C\left(\mathbb{T}^n_{\widetilde{\Th}}\right)$) are unitary generators of $C\left(\mathbb{T}^n_{\Th}\right)$ (resp. $C\left(\mathbb{T}^n_{\widetilde{\Th}}\right)$).	
There is an involutive action of $G=\mathbb{Z}_{k_1}\times...\times\mathbb{Z}_{k_n}$ on $C\left(\mathbb{T}^n_{\widetilde{\Th}}\right)$ given by
\begin{equation*}
\left(\overline{p}_1,..., \overline{p}_n\right)v_j = e^{\frac{2\pi i p_j}{k_j}}v_j,
\end{equation*}
and a following condition holds $C\left(\mathbb{T}^n_{\Th}\right)=C\left(\mathbb{T}^n_{\widetilde{\Th}}\right)^G$.
Otherwise there is a following $C\left(\mathbb{T}^n_{\Th}\right)$ - module isomorphism
\be\label{nt_gen_eqn}
C\left(\mathbb{T}^n_{\widetilde{\Th}}\right) = \bigoplus_{\left(\overline{p}_1, ... \overline{p}_n \right)\in\mathbb{Z}_{k_1}\times...\times\mathbb{Z}_{k_n} } v_1^{p_1} \cdot ... \cdot v_n^{p_n} C\left(\mathbb{T}^n_{\Th}\right) \approx C\left(\mathbb{T}^n_{\Th}\right)^{k_1\cdot ... \cdot k_n}
\ee
i.e. $C\left(\mathbb{T}^n_{\widetilde{\Th}}\right)$ is a finitely generated projective Hilbert $C\left(\mathbb{T}^n_{\Th}\right)$-module.
It turns out the following theorem. 
\begin{thm}\label{nt_fin_cov_thm}\cite{ivankov:qnc}												
	The triple $\left(C\left(\mathbb{T}^n_{\Th}\right), C\left(\mathbb{T}^n_{\widetilde{\Th}}\right),~\mathbb{Z}_{k_1}\times...\times\mathbb{Z}_{k_n}\right)$  is an unital noncommutative finite-fold  covering.
\end{thm}

\subsubsection{Induced representation}\label{nt_i_sec}
\paragraph*{}
Denote by  $\left\{U_l\in C\left(\mathbb{T}^n_{{\Theta}}\right)\right\}_{l \in \Z^n}$ (resp. $\left\{\widetilde{U}_l\in C\left(\mathbb{T}^n_{\widetilde{\Theta}}\right)\right\}_{l \in \Z^n}$) the basis of  
$C\left(\mathbb{T}^n_{{\Theta}}\right)$ (resp. $ C\left(\mathbb{T}^n_{\widetilde{\Theta}}\right)$) (cf. Definition \ref{nt_uni_defn}).
One has
\be\label{nt_hm_prod_eqn}
\begin{split}
\left\langle \widetilde{U}_{l'= \left( l'_1, \dots, l'_n\right) }, \widetilde{U}_{l''= \left( l''_1, \dots, l''_n\right)} \right\rangle_{C\left(\mathbb{T}^n_{\widetilde{\Th}}\right)} =\\= \sum_{\left( \overline{p}_1, \dots \overline{p}_n\right)\in \mathbb{Z}_{k_1}\times...\times\mathbb{Z}_{k_n} }  e^{2\pi i \frac{p_1\left(l'_1-l''_1 \right) }{k_1}}, \dots, e^{2\pi i \frac{p_1\left(l'_1-l''_1 \right) }{k_1}}\widetilde{U}_{l''}\widetilde{U}_{l'}=\\
= \left\{\begin{array}{c l}
	\left|\mathbb{Z}_{k_1}\times...\times\mathbb{Z}_{k_n} \right| 1_{C\left(\mathbb{T}^n_{{\Th}}\right)}  & l'=l'' \\
	0 & l'\neq l''
\end{array}\right..
\end{split}
\ee 

There is the natural dense inclusion $\Psi_\Th:C\left(\mathbb{T}^n_{{\Theta}}\right) \to L^2\left(C\left(\mathbb{T}^n_{{\Theta}}\right), \tau \right)$ given by \eqref{from_a_to_l2_eqn}. 	
Similarly to  \eqref{induced_prod_equ} we consider following pre-Hilbert space
$$
C\left(\mathbb{T}^n_{\widetilde{\Theta}}\right)\otimes_{C\left(\mathbb{T}^n_{\Theta}\right)} L^2\left(C\left(\mathbb{T}^n_{\Theta}\right), \tau\right)
$$
and denote by $\widetilde{\H}$ its Hilbert completion.
There are dense subspaces 
 $\Coo\left(\mathbb{T}^n_{\widetilde{\Theta}}\right) \subset C\left(\mathbb{T}^n_{\widetilde{\Theta}}\right)$, $\Coo\left(\mathbb{T}^n_{{\Theta}}\right) \subset L^2\left(C\left(\mathbb{T}^n_{\Theta}\right), \tau\right)$, hence the composition
$$
\Coo\left(\mathbb{T}^n_{\widetilde{\Theta}}\right) \otimes_{\Coo\left(\mathbb{T}^n_{{\Theta}}\right)} \Coo\left(\mathbb{T}^n_{{\Theta}}\right) \subset C\left(\mathbb{T}^n_{\widetilde{\Theta}}\right)\otimes_{C\left(\mathbb{T}^n_{\Theta}\right)} L^2\left(C\left(\mathbb{T}^n_{\Theta}\right), \tau\right) \subset \widetilde{   \H}
$$
is the dense inclusion. From $ \Coo\left(\mathbb{T}^n_{\widetilde{\Theta}}\right) \otimes_{\Coo\left(\mathbb{T}^n_{{\Theta}}\right)} \Coo\left(\mathbb{T}^n_{{\Theta}}\right) \cong \Coo\left(\mathbb{T}^n_{\widetilde{\Theta}}\right)$ it follows that there is the dense inclusion 
\bean
\Psi_{\widetilde \Th}:\Coo\left(\mathbb{T}^n_{\widetilde{\Theta}}\right) \hookto \widetilde{\H},\\
\widetilde{U}_l \mapsto \widetilde{U}_l \otimes  \Psi_\Th\left( 1_{C\left(\mathbb{T}^n_{{\Theta}}\right)}\right) \in  C\left(\mathbb{T}^n_{\widetilde{\Theta}}\right)\otimes_{C\left(\mathbb{T}^n_{\Theta}\right)} L^2\left(C\left(\mathbb{T}^n_{\Theta}\right), \tau\right).
\eean
If $\widetilde{\xi}_l = \Psi_{\widetilde{ \Th}}\left( \widetilde{U}_l\right) $ then from \eqref{induced_prod_equ} and \eqref{nt_hm_prod_eqn} it turns out
\be
\begin{split}
\left(\widetilde{\xi}_{l'}, \widetilde{\xi}_{l''} \right)_{\widetilde{\H}}= \left(\Psi_\Th\left( 1_{C\left(\mathbb{T}^n_{{\Theta}}\right)}\right), \left\langle \widetilde{U}_{l' }, \widetilde{U}_{l''} \right\rangle_{C\left(\mathbb{T}^n_{\widetilde{\Th}}\right)}\Psi_\Th\left( 1_{C\left(\mathbb{T}^n_{{\Theta}}\right)}\right) \right)_{\H}=\\=\left\{\begin{array}{c l}
	\left|\mathbb{Z}_{k_1}\times...\times\mathbb{Z}_{k_n} \right|  & l'=l'' \\
	0 & l'\neq l''
\end{array}\right.,
\end{split}
\ee
i.e. the set $\left\{\widetilde{\xi}_{l}\right\}_{l \in \Z^n}$ is an orthogonal basis $\widetilde{\H}$ such that $\left\|\widetilde{\xi}_{l} \right\|= \sqrt{\left|\mathbb{Z}_{k_1}\times...\times\mathbb{Z}_{k_n} \right|}$. From the product formula \eqref{nt_unitary_product} it turns out that action of $C\left(\mathbb{T}^n_{\widetilde{\Theta}}\right)$ on $\widetilde{\H}$ is given bu
\be\label{nt_t_act}
\widetilde{U}_k\widetilde{\xi}_p = e^{-\pi ik ~\cdot~ \Theta p}\widetilde{\xi}_{k + p}.
\ee
It is not difficult to prove that there is the natural isomorphism of $C\left(\mathbb{T}^n_{\widetilde{\Theta}}\right)$ modules $\widetilde{\H} \to L^2\left( C\left(\mathbb{T}^n_{\widetilde{\Theta}}\right), \widetilde{ \tau}\right)$ however this isomorphism is not isometric, i.e. it is not an isomorphism of Hilbert spaces.
From \eqref{nt_gen_eqn} it turns out that  $C\left(\mathbb{T}^n_{\widetilde{\Theta}}\right)_{C\left(\mathbb{T}^n_{{\Theta}}\right)}$ is a right $C\left(\mathbb{T}^n_{{\Theta}}\right)$-module generated by a finite set
$$
\Xi=\left\{ \widetilde{U}_{j=\left(j_1, \dots, j_n \right) }\right\}_{0 \le j_1 < k_1,\dots, 0 \le j_n < k_n}.
$$
From \eqref{nt_hm_prod_eqn} it turns out
$$
\left\langle  \widetilde{U}_{j'=\left(j'_1, \dots, j'_n \right)}~,~  \widetilde{U}_{j''=\left(j''_1, \dots, j'''_n \right)} \right\rangle_{C\left(\mathbb{T}^n_{\widetilde{\Theta}}\right)} = \delta_{j'_1j''_1} \cdot ... \cdot \delta_{j'_nj''_n} \cdot 1_{C\left(\mathbb{T}^n_{{\Theta}}\right)} \in \Coo\left(\mathbb{T}^n_{{\Theta}}\right).
$$
If $g = \left(\overline{p}_1, \dots \overline{p}_n\right) \in \mathbb{Z}_{k_1}\times...\times\mathbb{Z}_{k_n}$ then
\be\label{nt_act_eqn}
g\left(  \widetilde{U}_{j=\left(j_1, \dots, j_n \right)}\right)  = e^{\frac{2\pi i j_1 p_1}{k_1}}\cdot ... \cdot e^{\frac{2\pi i j_np_n}{k_n}} \cdot \widetilde{U}_{j=\left(j_1, \dots, j_n \right)} ,
\ee
hence one has
\be\label{nt_fin_eqn}
\begin{split}
	\widetilde{\Xi} = \left( \mathbb{Z}_{k_1}\times...\times\mathbb{Z}_{k_n}\right)  \Xi== \\=\left\{e^{\frac{2\pi i j_1 p_1}{k_1}}\cdot ... \cdot e^{\frac{2\pi i j_np_n}{k_n}} \widetilde{U}_{j=\left(j_1, \dots, j_n \right) }\right\}_{\left(\overline{p}_1, \dots \overline{p}_n\right)\in \mathbb{Z}_{k_1}\times...\times\mathbb{Z}_{k_n}}
\end{split}
\ee

Similarly from $g' = \left(\overline{p}'_1, \dots \overline{p}'_n\right), g'' = \left(\overline{p}''_1, \dots \overline{p}''_n\right) \in \mathbb{Z}_{k_1}\times...\times\mathbb{Z}_{k_n}$ it turns out

\bean
\left\langle g'\left(  \widetilde{U}_{j'=\left(j'_1, \dots, j'_n \right)}\right) ~,~ g''\left( \widetilde{U}_{j''=\left(j''_1, \dots, j''_n \right)}\right)  \right\rangle_{C\left(\mathbb{T}^n_{\widetilde{\Theta}}\right)} =\\=\left| \mathbb{Z}_{k_1}\times...\times\mathbb{Z}_{k_n} \right|  \delta_{j'_1j''_1} \cdot ... \cdot \delta_{j'_nj''_n} \cdot e^{\frac{2\pi i\left(  j''_1p''_1- j'_1 p'_1\right) }{k_1}}\cdot ... \cdot e^{\frac{2\pi i \left(  j''_np''_n- j'_n p'_n\right)}{k_n}}\cdot 1_{C\left(\mathbb{T}^n_{{\Theta}}\right)} \in \Coo\left(\mathbb{T}^n_{{\Theta}}\right).
\eean
From the above equation it follows that the set $\widetilde{\Xi}$
satisfies to condition (a) of the Definition \ref{smooth_matr_lem}. From \eqref{nt_act_eqn}, \eqref{nt_fin_eqn} it turns out that
$$
\left(\mathbb{Z}_{k_1}\times...\times\mathbb{Z}_{k_n} \right) \widetilde{\Xi} = \widetilde{\Xi},
$$
i.e. $\widetilde{\Xi}$ satisfies to the condition (b) of the Lemma \ref{smooth_matr_lem}. From (ii) of the Lemma \ref{smooth_matr_lem} it turns out that  the unital noncommutative finite-fold covering $$\left(C\left(\mathbb{T}^n_{\Th}\right), C\left(\mathbb{T}^n_{\widetilde{\Th}}\right),\mathbb{Z}_{k_1}\times...\times\mathbb{Z}_{k_n}\right)$$ is smoothly invariant. If $\widetilde{a} \in \Coo\left(\mathbb{T}^n_{\widetilde{\Th}}\right)$ then
\be
\widetilde{a} = \sum_{\mu \in \Z^n} a_\mu \widetilde{U}_\mu
\ee
where $\left\{a_\mu \in \C\right\}_{\mu \in \Z^n}$ has rapid decay.
 If $j'=\left( j'_1, \dots, j'_n\right)$ is such that $0 \le j'_1 < k_1,~\dots,~ 0 \le j'_n < k_n$ and  $j''=\left( j''_1, \dots, j''_n\right)$ is such that $0 \le j''_1 < k_1,~\dots,~ 0 \le j''_n < k_n$ then
 \be\label{nt_rd}
 \begin{split}
 \left\langle \widetilde{U}_{j'},\widetilde{a}  \widetilde{U}_{j''} \right\rangle_{C\left(\mathbb{T}^n_{\widetilde{\Th}}\right)} =\\= \left| \mathbb{Z}_{k_1}\times...\times\mathbb{Z}_{k_n} \right|\sum_{\mu=\left(\mu_1,..., \mu_n \right)  \in \Z^n}  e^{\pi i \varphi_{ j'j''\mu }} a_{\left( \mu_1k_1 - j'_1 + j''_1,~ ...~, \mu_nk_n-j'_1 + j''_1\right) }  U_\mu
\end{split}
  \ee
 where 
 $$
  \varphi_{j'j''\mu} =\left(  \begin{pmatrix} j''_1  \\ 
  \dots\\
    j''_n\end{pmatrix}\widetilde{\Th} \begin{pmatrix} \mu_1k_1 - j'_1 + j''_1  \\ 
    \dots\\
    \mu_nk_n - j'_n + j''_n\end{pmatrix}\right) \cdot\left( 
    \begin{pmatrix} \mu_1k_1 - j'_1   \\ 
    \dots\\
    \mu_nk_n - j'_n \end{pmatrix}  \widetilde{\Th} \begin{pmatrix} j'_1  \\ 
    \dots\\
    j'_n\end{pmatrix}\right) 
 $$
 From $\left| e^{\pi i \varphi_{ j'j''\mu }} \right|=1$ it follows an implication
\bean
 \left\{a_\mu \in \C\right\}_{\mu \in \Z^n} \text{ has rapid decay } \Rightarrow \\ \Rightarrow \left\{e^{\pi i \varphi_{\left( j',j'',\mu\right) }} a_{\left( \mu_1k_1 - j'_1 + j''_1,~ ...~, \mu_nk_n-j'_1 + j''_1\right) } \in \C\right\}_{\mu \in \Z^n} \text{ has rapid decay },
\eean
so from \eqref{nt_rd} it turns out  $\left\langle \widetilde{U}_{j'},\widetilde{a}  \widetilde{U}_{j''} \right\rangle_{\left(\mathbb{T}^n_{\widetilde{\Th}}\right)} \in \Coo\left( \T_{\Th}\right)$.
Otherwise if right part of \eqref{nt_rd} has rapid decay for any $j'=\left( j'_1, \dots, j'_n\right)$  such that $0 \le j'_1 < k_1,~\dots,~ 0 \le j'_n < k_n$ and  $j''=\left( j''_1, \dots, j''_n\right)$  such that $0 \le j''_1 < k_1,~\dots,~ 0 \le j''_n < k_n$ then the sequence  $\left\{a_\mu \in \C\right\}_{\mu \in \Z^n}$ has rapid decay, it follows that $\widetilde{a} = \sum_{\mu \in \Z^n} a_\mu \widetilde{U}_\mu \in \Coo\left( \T^n_{\widetilde{\Th}}\right)$. From (i) of the Lemma \ref{smooth_matr_lem} it follows that
$$
\Coo\left(\mathbb{T}^n_{\widetilde{\Th}}\right) = C\left(\mathbb{T}^n_{ \widetilde{\Th}}\right)\bigcap \mathbb{M}_{k_1\cdot ... \cdot k_n}\left(\Coo\left(\mathbb{T}^n_{{\Th}}\right) \right). 
$$

\subsubsection{Lift of Dirac operator} 
\paragraph*{}
The Hilbert space of relevant to spectral triple is $\widetilde{\H}^m= \widetilde{\H}\otimes \C^m$ where $m=2^{[n/2]}$ and $\widetilde{\H}$ is described in \ref{nt_i_sec} and the action of $C\left(\mathbb{T}^n_{\widetilde{\Th}} \right)$ on $\widetilde{\H}$ is given by \eqref{nt_t_act}. The action of $C\left(\mathbb{T}^n_{\widetilde{\Th}} \right)$ on  $\widetilde{\H}^m$ is diagonal. Let $\Om^1_D$ is the {module of differential forms associated} with the spectral triple  $\left( \A, \H, D\right)$ (cf. Definition \ref{ass_cycle_defn}) Let us consider a map
\bean
\nabla : \Coo\left(\mathbb{T}^n_{\widetilde{\Th}} \right) \to \Coo\left(\mathbb{T}^n_{\widetilde{\Th}} \right) \otimes_{\Coo\left(\mathbb{T}^n_{{\Th}} \right)} \Om^1_D,\\
\widetilde{a} \mapsto \sum_{\mu = 1}^n \frac{\partial a}{\partial x_\mu} \otimes u^*_\mu \left[ D, u_\mu \right]= \sum_{\mu = 1}^n \frac{\partial a}{\partial x_\mu} \otimes \ga^\mu
\eean
where $\ga^\mu$ Clifford (Gamma) matrices in $\mathbb{M}_n(\mathbb{C})$ satisfying the relation
\eqref{nt_gam_eqn}

The map satisfies  to the following equation
\be\label{nt_nabla_u_eqn}
\nabla \widetilde{ U}_{\widetilde{l} = \left(\widetilde{l}_1, \dots, \widetilde{l}_n \right) } = \sum_{\mu = 1}^n\frac{\widetilde{l}_\mu}{k_\mu} \widetilde{ \mathcal U}_{\widetilde{l}} \otimes \ga^\mu.
\ee
If  $U_{l = \left(l_1,\dots,l_n \right) }\in \Coo\left(\mathbb{T}^n_{{\Th}} \right)$, $\widetilde{l}'=\left( l_1k_1,\dots, l_nk_n\right)$, $\widetilde{l}'' = \widetilde{l} + \widetilde{l}'$ then  from \eqref{nt_unitary_product} and \eqref{nt_cov_eqn} it turns out
\bean
\widetilde{  U}_{\widetilde{l}} U_l e^{-\pi i\widetilde{l} ~\cdot~ \widetilde{\Theta} \widetilde{l}'} U_{\widetilde{l}''=\widetilde{l} + \widetilde{l}'}
\eean
and taking into account \eqref{nt_nabla_u_eqn} one has
\bean
\nabla \left( \widetilde{  U}_{\widetilde{l}} U_l\right) = e^{-\pi i\widetilde{l} ~\cdot~ \widetilde{\Theta} \widetilde{l}'}\sum_{\mu = 1}^n\frac{l_\mu k_\mu+\widetilde{l}_\mu}{k_\mu} \widetilde{ \mathcal U}_{\widetilde{l}''} \otimes \ga^\mu= \\=e^{-\pi i\widetilde{l} ~\cdot~ \widetilde{\Theta} \widetilde{l}'}\sum_{\mu = 1}^n\frac{\widetilde{l}_\mu}{k_\mu} \widetilde{ \mathcal U}_{\widetilde{l}''} \otimes \ga^\mu + e^{-\pi i\widetilde{l} ~\cdot~ \widetilde{\Theta} \widetilde{l}'}\sum_{\mu = 1}^nl_\mu \widetilde{ \mathcal U}_{\widetilde{l}''} \otimes \ga^\mu=\\= \sum_{\mu = 1}^n\frac{\widetilde{l}_\mu}{k_\mu} \widetilde{  U}_{\widetilde{l}} U_l \otimes \ga^\mu + \sum_{\mu = 1}^nl_\mu \widetilde{  U}_{\widetilde{l}} U_l \otimes \ga^\mu =\\= \sum_{\mu = 1}^n\frac{\widetilde{l}_\mu}{k_\mu}\left(  \widetilde{  U}_{\widetilde{l}} \otimes \ga^\mu\right)  U_l +  \widetilde{  U}_{\widetilde{l}} \otimes \sum_{\mu = 1}^nl_\mu U_l \otimes \ga^\mu =\\=
\left(\nabla\widetilde{  U}_{\widetilde{l}} \right) U_l + \widetilde{  U}_{\widetilde{l}} \otimes \left[D, U_l \right] = \left(\nabla\widetilde{  U}_{\widetilde{l}} \right) U_l + \widetilde{  U}_{\widetilde{l}} \otimes d U_l.
\eean
where $d U_l$ is given by  \eqref{diff_map}. Right part of the above equation is a specific case of \eqref{conn_defn}, i.e. elements $\widetilde{  U}_{\widetilde{l}}$ satisfy condition $\eqref{conn_defn}$. Since $\C$-linear span of $\left\{\widetilde{  U}_{\widetilde{l}}\right\} $ (resp. $\left\{{  U}_{{l}}\right\}$ ) is dense in $\Coo\left(\mathbb{T}^n_{\widetilde{\Th}} \right)$ (resp. $\Coo\left(\mathbb{T}^n_{{\Th}} \right)$) the map $\nabla$ is a connection. If $\left( \overline{p}_1, \dots,\overline{p}_n\right) \in \Z_{k_1}\times \dots \times \Z_{k_1}$ then from
$$
\left( \overline{p}_1, \dots,\overline{p}_n\right) \widetilde{ U}_{\widetilde{l}} = e^{2\pi i \frac{p_1l_1}{k_1} }\cdot \dots\cdot  e^{2\pi i \frac{p_nl_n}{k_n} } \widetilde{ U}_{\widetilde{l}}
$$
it turns out 
\bean
\nabla\left( \left( \overline{p}_1, \dots,\overline{p}_n\right) \widetilde{ U}_{\widetilde{l}}\right)  = e^{2\pi i \frac{p_1l_1}{k_1} }\cdot ...\cdot  e^{2\pi i \frac{p_nl_n}{k_n} } \sum_{\mu = 1}^n\frac{\widetilde{l}_\mu}{k_\mu} \widetilde{ \mathcal U}_{\widetilde{l}} \otimes \ga^\mu=\left( \overline{p}_1, \dots,\overline{p}_n\right)\left( \nabla \widetilde{ U}_{\widetilde{l}}\right), 
\eean
i.e. $\nabla$ is $\Z_{k_1}\times \dots \times \Z_{k_1}$-equivariant, i.e it satisfies to \eqref{equiv_conn_eqn}. 
If $\widetilde{\xi}_l = \Psi_{\widetilde{ \Th}}\left( \widetilde{U}_{\widetilde{l}}\right)=  \widetilde{U}_l \otimes 1_{C\left(\mathbb{T}^n_{\Th}\right)}$ then from $\left[D,  1_{C\left(\mathbb{T}^n_{\Th}\right)} \right] = 0$ it turns out that given by \eqref{d_defn} Dirac operator satisfy to following condition
\bean
\widetilde{D} \left( \widetilde{\xi}_{\widetilde{l}} \otimes x\right) = \nabla\left(\widetilde{U}_{\widetilde{l}} \right) \left( 1_{C\left(\mathbb{T}^n_{\Th}\right)} \otimes x\right) + \widetilde{U}_{\widetilde{l}} D\left(   1_{C\left(\mathbb{T}^n_{\Th}\right)} \otimes x\right)= \\= \sum_{\mu = 1}^n\frac{\widetilde{l}_\mu}{k_\mu} \widetilde{ \xi}_{\widetilde{l}} \otimes \ga^\mu x_\mu ; ~~\forall x =\begin{pmatrix} x_1 \\\dots\\ x_m \end{pmatrix} \in \C^m. 
\eean

\subsubsection{Coverings of spectral triples}
\paragraph*{} Let us consider following objects
\begin{itemize}
	\item The spectral triple $\left( C^{\infty}(\mathbb{T}^n_{\Theta}),L^2\left(C\left(\mathbb{T}^n_{\Theta}\right), \tau\right)\otimes\mathbb{C}^{m},D\right) $ given by \eqref{nt_sp_tr_eqn},
	\item An unital noncommutative finite-fold  covering $\left(C\left(\mathbb{T}^n_{\Th}\right), C\left(\mathbb{T}^n_{\widetilde{\Th}}\right),\mathbb{Z}_{k_1}\times...\times\mathbb{Z}_{k_n}\right)$ given by the Theorem \ref{nt_fin_cov_thm}.
\end{itemize}

Summarize above constructions one has the following theorem.

\begin{thm}\label{nt_st_fin_cov_thm}
The triple $\left(\Coo\left(\mathbb{T}^n_{\widetilde{\Theta}}\right),  L^2\left( C\left(\mathbb{T}^n_{\widetilde{\Theta}}\right)\right) , \widetilde{ D} \right)$ is the $\left(C\left(\mathbb{T}^n_{\Th}\right), C\left(\mathbb{T}^n_{\widetilde{\Th}}\right),\mathbb{Z}_{k_1}\times...\times\mathbb{Z}_{k_n}\right)$-lift of $\left( C^{\infty}(\mathbb{T}^n_{\Theta}),L^2\left(C\left(\mathbb{T}^n_{\Theta}\right), \tau\right)\otimes\mathbb{C}^{m},D\right) $.
\end{thm} 
%Similarly to \eqref{nt_diff_mod_eqn} for any $\widetilde{a} \in \Coo\left(\mathbb{T}^n_{\widetilde{\Th}}\right)$ following condition holds
%\begin{equation}\label{nt_diff_cov_mod_eqn}
%\left[\widetilde{D}, \widetilde{a} \right]  = \sum_{\mu = 1}^n \frac{\partial \widetilde{a}}{\partial x_\mu} u^*_\mu \left[ D, u_\mu \right]
%\end{equation}
%where $u_1, \dots, u_n$ are unitary generators of $\Coo\left(\mathbb{T}^n_{{\Th}}\right)$.

\subsection{Moyal plane and a representation of the noncommutative torus}\label{nt_ind_repr_subsubsec}
\begin{defn}
	Denote the \textit{Moyal plane} product $\star_\th$ on $\SS\left(\R^{2N} \right)$ given by
\begin{equation}\label{mp_prod_eqn}
	\left(f \star_\th h \right)\left(u \right)= \int_{y \in \R^{2N} } f\left(u - \frac{1}{2}\Th y\right) g\left(u + v \right)e^{2\pi i y \cdot v }  dydv
\end{equation}

	where $\Th$ is given by \eqref{nt_simpectic_theta_eqn}.
\end{defn}
% Denition 2.8 
\begin{defn}\label{mp_mult_defn}\cite{varilly_bondia:phobos}
	\label{df:Moyal-alg}	Denote by $\SS'\left( \R^{n}\right) $ the vector space dual to $\SS\left( \R^{n}\right) $, i.e. the space of continuous functionals on $\SS\left( \R^{n}\right)$.
	The Moyal product can be defined, by duality, on larger sets than
	$\SS\left(\R^{2N}\right)$. For $T \in \SS'\left(\R^{2N}\right)$, write the evaluation on $g \in \SS\left(\R^{2N}\right)$ as
	$\<T, g> \in \C$; then, for $f \in \SS$ we may define $T \star_{\theta} f$ and
	$f \star_{\theta} T$ as elements of~$\SS'\left(\R^{2N}\right)$ by
	\begin{equation}\label{mp_star_ext_eqn}
	\begin{split}
\<T \star_{\theta} f, g> \stackrel{\mathrm{def}}{=} \<T, f \star_{\theta} g>\\
\<f \star_{\theta} T, g> \stackrel{\mathrm{def}}{=} \<T, g \star_{\theta} f>	\end{split}
	\end{equation}  using the continuity of the
	star product on~$\SS\left(\R^{2N}\right)$. Also, the involution is extended to  by
	$\<T^*,g> \stackrel{\mathrm{def}}{=} \overline{\<T,g^*>}$.	
%	We shall soon argue~\cite{varilly_bondia:phobos} that if $T \in \SS'\left(\R^{2N}\right)$ and	$f \in \SS\left(\R^{2N}\right)$, then $T \star_{\theta} f,\,f \star_{\theta} T \in C^\infty(\R^{2N})$.
	Consider the left and right multiplier algebras:
 \begin{equation}\label{mp_multi_eqn}
\begin{split}
\M_L^\th
&\stackrel{\mathrm{def}}{=} \set{T \in \SS'(\R^{2N}) : T \star_{\theta} h \in \SS(\R^{2N})
	\text{ for all } h \in \SS(\R^{2N})},
\\
\M_R^\th
&\stackrel{\mathrm{def}}{=} \set{T \in \SS'(\R^{2N}) : h \star_{\theta} T \in \SS(\R^{2N})
	\text{ for all } h \in \SS(\R^{2N})},\\
\M^\th &\stackrel{\mathrm{def}}{=} \M_L^\th \cap \M_R^\th.\\
\end{split}
\end{equation} 
\end{defn}
In \cite{varilly_bondia:phobos} it is proven that
\begin{equation}\label{mp_mult_distr}
\M_R^\th \star_{\theta} \SS'\left(\R^{2N}\right) = \SS'\left(\R^{2N}\right) \text{ and }
\SS'\left(\R^{2N}\right) \star_{\theta} \M_L^\th = \SS'\left(\R^{2N}\right).
\end{equation}

	It is known \cite{moyal_spectral} that the domain of  the Moyal plane product can be extended up to $L^2\left(\R^{2N} \right)$. 

\begin{lem}\label{nt_l_2_est_lem}\cite{moyal_spectral}
	If $f,g \in L^2 \left(\R^{2N} \right)$, then $f\star_\th g \in L^2 \left(\R^{2N} \right)$ and $\left\|f\right\|_{\mathrm{op}} < \left(2\pi\th \right)^{-\frac{N}{2}} \left\|f\right\|_2$.
	where	$\left\|\cdot\right\|_{2}$ be the $L^2$-norm given by
	\begin{equation}\label{nt_l2_norm_eqn}
	\left\|f\right\|_{2} \stackrel{\mathrm{def}}{=} \left|\int_{\R^{2N}} \left|f\right|^2 dx \right|^{\frac{1}{2}}.
	\end{equation}
	and the operator norm
	\be\label{mp_op_eqn}
	\|T\|_{\mathrm{op}} \stackrel{\mathrm{def}}{=}\sup\set{\|T \star g\|_2/\|g\|_2 : 0 \neq g \in L^2\left( \R^{2N})\right) }
	\ee 
\end{lem}
%\begin{rem}\label{norm_remark}
%	Since $\SS\left(\R{2N} \right)$ is dense in $L^2\left(\R{2N} \right)$ the operator norm is given by $\|T\|_{\mathrm{op}} =\sup\set{\|T \star g\|_2/\|g\|_2 : 0 \neq g \in \SS\left( \R^{2N}\right) }$ 
%\end{rem}

\begin{defn}\label{mp_star_alg_defn}
	Denote by $\SS\left(\R^{2N}_\th \right)$  (resp. $L^2\left(\R^{2N}_\th \right)$ ) the operator algebra  which is $\C$-linearly isomorphic to $\SS\left(\R^{2N} \right)$  (resp. $L^2\left(\R^{2N} \right)$ ) and product coincides with  $\star_\th$. Both  $\SS\left(\R^{2N}_\th \right)$ and $L^2\left(\R^{2N}_\th \right)$ act on the Hilbert space $L^2\left(\R^{2N} \right)$. Denote by
	\begin{equation}\label{mp_psi_th_eqn}
	\Psi_\th:  \SS\left(\R^{2N} \right)\xrightarrow{\approx}\SS\left(\R^{2N}_\th \right)
	\end{equation}
	the natural $\C$-linear isomorphism.  
\end{defn}
\begin{empt}
	There is the tracial property \cite{moyal_spectral} of the Moyal product
	\begin{equation}\label{nt_tracial_prop}
	\int_{\R^{2N}} \left( f\star_\th g\right) \left(x \right)dx =  \int_{\R^{2N}}  f\left(x \right) g\left(x \right)dx.
	\end{equation}
	The Fourier transformation of the star product satisfies to the following condition.
	\begin{equation}\label{mp_fourier_eqn}
	\mathcal{F}\left(f \star_\th g\right) \left(x \right) =    	\int_{\R^{2N}}\mathcal{F}{f}\left(x-y \right) \mathcal{F}{g}\left(y\right)e^{ \pi i  y \cdot \Th x }~dy.
	\end{equation}
	
\end{empt} 
\begin{defn}\label{r_2_N_repr}\cite{moyal_spectral} Let $\SS'\left(\R^{2N} \right)$ be a vector space dual to $\SS\left(\R^{2N} \right)$. Denote by $C_b\left(\R^{2N}_\th\right)\stackrel{\mathrm{def}}{=} \set{T \in \SS'\left(\R^{2N}\right) : T \star_\th g \in L^2\left(\R^{2N}\right) \text{ for all } g \in L^2(\R^{2N})}$, provided with the operator norm \begin{equation}\label{mp_op_norm_eqn}
\|T\|_{\mathrm{op}} \stackrel{\mathrm{def}}{=}\sup\set{\|T \star_\th g\|_2/\|g\|_2 : 0 \neq g \in L^2(\R^{2N})}.
\end{equation}
Denote by $C_0\left(\R^{2N}_\th \right)$ the operator norm completion of $\SS\left(\R^{2N}_\th \right).$  
\end{defn}
\begin{rem}
	Obviously $\SS\left(\R^{2N}_\th\right)  \hookto C_b\left(\R^{2N}_\th\right)$. But $\SS\left(\R^{2N}_\th\right)$ is not dense in $C_b\left(\R^{2N}_\th\right)$, i.e. $C_0\left(\R^{2N}_\th\right) \subsetneq C_b\left(\R^{2N}_\th\right)$  (cf. \cite{moyal_spectral}).
\end{rem}

\begin{rem}
	$L^2\left(\R^{2N}_\th\right)$ is the $\|\cdot\|_2$ norm completion of $\SS\left(\R^{2N}_\th\right)$ hence 
	from the Lemma \ref{nt_l_2_est_lem} it follows that 
	\begin{equation}\label{mp_2_op_eqn}
	L^2\left(\R^{2N}_\th\right) \subset C_0\left(\R^{2N}_\th\right).
	\end{equation} 
\end{rem}
\begin{rem}
	Notation of the Definition \ref{r_2_N_repr} differs from \cite{moyal_spectral}. Here symbols $A_\th, \A_\th, A^0_\th$ are replaced with $C_b\left(\R^{2N}_\th\right), \SS\left(\R^{2N}_\th\right), C_0\left(\R^{2N}_\th\right)$ respectively.
\end{rem}
\begin{rem}
	The $\C$-linear space $C_0\left(\R^{2N}_\th \right)$ is not isomorphic to $C_0\left(\R^{2N}\right)$. 
\end{rem}

There are elements $\left\{f_{nm}\in\SS\left(\R^{2} \right)\right\}_{m,n \in \N^0}$, described in \cite{varilly_bondia:phobos}, which satisfy to the following proposition.

%Theorem 6: 
\begin{prop}\label{mp_fmn}\cite{moyal_spectral,varilly_bondia:phobos}
	Let $N = 1$. Then $\SS\left(\R^{2N}_\th\right)=\SS\left(\R^{2}_\th\right) $ has a Fr\'echet algebra isomorphism with
	the matrix algebra of rapidly decreasing double sequences
	$c = (c_{mn})$ of complex numbers such that, for each $k \in \N$,
			\begin{equation}\label{mp_matr_norm}
	r_k(c) \stackrel{\mathrm{def}}{=} \biggl( \sum_{m,n=0}^\infty
\th^{2k}  \left( m+\half\right)^k \left( n+\half\right)^k |c_{mn}|^2 \biggr)^{1/2}
	\end{equation}

	is finite, topologized by all the seminorms $(r_k)$; via the
	decomposition $f = \sum_{m,n=0}^\infty c_{mn} f_{mn}$ of~$\SS(\R^2)$ in
	the $\{f_{mn}\}$ basis.
	 The twisted product $f \star_\th g$ is
	the matrix product $ab$, where
	\begin{equation}\label{mp_mult_eqn}
	\left( ab\right)_{mn} \stackrel{\mathrm{def}}{=} \sum_{k= 0}^{\infty} a_{mk}b_{kn}.
	\end{equation}
	
	For $N > 1$, $\Coo\left(\R^{2N}_\th\right)$ is isomorphic to the (projective) tensor product
	of $N$ matrix algebras of this kind, i.e.
		\begin{equation}\label{mp_tensor_prod}
	\SS\left(\R^{2N}_\th\right) \cong \underbrace{\SS\left(\R^{2}_\th\right)\otimes\dots\otimes\SS\left(\R^{2}_\th\right)}_{N-\mathrm{times}}
	\end{equation}
	with the projective topology induced by seminorms $r_k$ given by \eqref{mp_matr_norm}.	
\end{prop}
\begin{rem}
	If $A$ is  $C^*$-norm completion of the matrix algebra with the norm \eqref{mp_matr_norm} then $A \approx \mathcal K$, i.e.
	\begin{equation}\label{mp_2_eqn}
C_0\left(\R^{2}_\th\right) \approx \mathcal K.
	\end{equation}
	Form \eqref{mp_tensor_prod} and \eqref{mp_2_eqn} it follows that
	\begin{equation}\label{mp_2N_eqn}
C_0\left(\R^{2N}_\th\right) \cong \underbrace{C_0\left(\R^{2}_\th\right)\otimes\dots\otimes C_0\left(\R^{2}_\th\right)}_{N-\mathrm{times}} \approx \underbrace{\mathcal K\otimes\dots\otimes\mathcal K}_{N-\mathrm{times}} \approx \mathcal K
\end{equation}
where $\otimes$ means minimal or maximal tensor product ($\mathcal{K}$ is nuclear hence both products coincide).

\end{rem}

\begin{empt}\cite{moyal_spectral}
	By plane waves we understand all functions of the form
	$$
	x \mapsto \exp(ik\cdot x) 
	$$
	for $k\in \R^{2N}$.  One obtains for the Moyal
	product of plane waves:
	\begin{equation}\label{mp_wave_prod_eqn}
	\begin{split}
	\exp\left(ik\cdot\right) \star_{\Theta}\exp\left(ik\cdot\right)=\exp\left(ik\cdot\right) \star_{\theta}\exp\left(ik\cdot\right)= \exp\left(i\left( k+l\right) \cdot\right) e^{-\pi i k \cdot \Th l}.
	\end{split}
	\end{equation}
It is proven in \cite{moyal_spectral} that plane waves lie in $C_b\left(\R^{2N}_\th \right)$. 	
\end{empt}

\begin{empt}
	The equation \eqref{mp_wave_prod_eqn} is similar to  the equation \eqref{nt_unitary_product} which defines $C\left(\mathbb{T}^n_{\Theta}\right)$. This fact enables us to construct a representation $\pi: C\left(\mathbb{T}^n_{\Theta}\right) \to B\left(L^2\left( \R^{2N}\right) \right)$
	\begin{equation}\label{nt_l2r_eqn}
	\begin{split}
	\pi: C\left(\mathbb{T}^n_{\Theta}\right) \to B\left(L^2\left( \R^{2N}\right) \right), \\
	U_k \mapsto \exp\left(2\pi ik\cdot\right)
	\end{split}
	\end{equation}
	
	where $U_k\in C\left(\mathbb{T}^n_{\Theta}\right)$ is given by the Definition \ref{nt_uni_defn}.
\end{empt} 
\begin{empt}\label{mp_scaling_constr}

Let us consider the unitary dilation operators $E_a$ given
by
$$
E_af(x) \stackrel{\mathrm{def}}{=} a^{N/2} f(a^{1/2}x),
$$
It is proven in \cite{moyal_spectral} that
\begin{equation}\label{eq:starscale}
f {\star_{\theta}} g =
(\th/2)^{-N/2} E_{2/\th}(E_{\th/2}f \star_2 E_{\th/2}g).
\end{equation}
We can simplify our construction by setting $\th = 2$. Thanks to
the scaling relation~\eqref{eq:starscale} any qualitative result can is true if it is true in case of 
$\th = 2$. We use the following notation
\begin{equation}\label{mp_times_eqn}
f {\times} g\stackrel{\mathrm{def}}{=}f {\star_{2}} g
\end{equation}
\end{empt}

\begin{defn}
	\label{df:Gst}\cite{moyal_spectral}
	We may as well introduce more Hilbert spaces $\G_{st}$ (for
	$s,t \in \R$) of those 
	$$
	f \in \SS'(\R^2) = \sum_{m,n = 0}^\infty c_{mn} f_{mn}
	$$ for which the following sum
	is finite:
	$$
	\|f\|_{st}^2 \stackrel{\mathrm{def}}{=} \sum_{m,n=0}^\infty
	 (m+\half)^s (n+\half)^t |c_{mn}|^2.
	$$
%	We define $\G_{st}$, for $s,t$ now in $\R^N$, as the tensor product	of Hilbert spaces $\G_{s_1t_1} \oxyox \G_{s_Nt_N}$. In other	words, the elements $(2\pi)^{-N/2} 	(m+\half)^{-s/2} (n+\half)^{-t/2} f_{mn}$ (with an obvious multiindex 	notation), for $m,n \in \N^N$, are declared to be an orthonormal basis
	for~$\G_{st}$.
\end{defn}
%\begin{rem}\cite{varilly_bondia:phobos}
%	Observe that 
%	\begin{equation*}
%\G_{0,0} = L^2\left( \R^2\right)
%	\end{equation*}
	
%	$\G_{0,0} = L^2\left( \R^2\right)$. 
%\end{rem}
\begin{rem}\label{mp_l2_rem}
	It is proven in \cite{varilly_bondia:phobos} $f, g \in L^2\left( \R^2\right)$, then $f\times g \in L^2\left( \R^2\right)$ and $\left\|f\times g  \right\|\le\left\|f \right\|\left\| g  \right\|$.
	Moreover, $f \times g$ lies in $C_0\left( \R^2\right)$ : the continuity follows by adapting the analogous argument for (ordinary) convolution. 
\end{rem}
%\begin{rem}\label{mp_smooth_rem}\cite{varilly_bondia:deimos}
%The Leibniz formula assures us that $f \times g \in C_0^m\left(\R^2 \right)$ 
% whenever
% $$
%f, g \in \bigcup \left\{\G_{st}~|~ s > 2m, t > 2m, s + t > 4m + 2\right\}
%$$
%analogously to what happens in the usual Sobolev spaces, the distributions in $\G_{st}$ grow more
%regular as $s, t$ become larger in a suitable way.
%\end{rem}
\begin{rem}\label{mp_ss_rem}
It is shown in \cite{varilly_bondia:phobos} that 
\begin{equation}\label{mp_ssin_eqn}
\SS\left( \R^2\right) = \bigcap_{s, t \in \R} \G_{st}.
\end{equation}
\end{rem}

\begin{empt}\label{mp_coord_constr}
	This part contains a useful equations proven in \cite{varilly_bondia:phobos}.
 There are coordinate functions $p,q$ on $\R^2$ such that for any $f \in \SS\left( \R^2\right)$ following conditions hold
 \begin{equation}\label{mp_pq_mult_eqn}
 \begin{split}
q \times f = \left(q + i \frac{\partial}{\partial p} \right)  f; ~~p \times f = \left(p - i \frac{\partial}{\partial q} \right)  f;\\
f \times q = \left(q - i \frac{\partial}{\partial p} \right)  f; ~~f \times p = \left(p + i \frac{\partial}{\partial q} \right)  f.
 \end{split}
 \end{equation}
 From $q \times f, f \times q,p  \times f, f \times p \in \SS\left(\R^{2N} \right)$ it follows that $p, q \in \M^2$ (cf. \eqref{mp_multi_eqn}). From 
  \eqref{mp_mult_distr} it follows that
  \begin{equation}\label{mp_pq_mult}
 \begin{split}
 q \times \SS'\left(\R^{2N} \right) \subset \SS'\left(\R^{2N} \right); ~~p \times \SS'\left(\R^{2N} \right) \subset \SS'\left(\R^{2N} \right);\\
\SS'\left(\R^{2N} \right) \times q \subset \SS'\left(\R^{2N} \right); ~~\SS'\left(\R^{2N} \right) \times p \subset \SS'\left(\R^{2N} \right).
 \end{split}
  \end{equation}
  If $f \in \SS'\left( \R^2\right)$ then from \eqref{mp_pq_mult_eqn} it follows that
  \begin{equation}\label{mp_partial_eqn}
  \frac{\partial}{\partial p} f = -iq \times f + i f \times q, ~~  \frac{\partial}{\partial q} f = ip \times f - i f \times p
  \end{equation}

If 
 \begin{equation}\label{mp_ham_eqn}
 \begin{split}
a \stackrel{\mathrm{def}}{=} \frac{q + ip}{\sqrt{2}}, ~~\overline{a} \stackrel{\mathrm{def}}{=} \frac{q - ip}{\sqrt{2}},\\
\frac{\partial}{\partial a}\stackrel{\mathrm{def}}{=} \frac{\partial_q + i\partial_p}{\sqrt{2}}, ~~ \frac{\partial}{\partial \overline{a}}\stackrel{\mathrm{def}}{=} \frac{\partial_q - i\partial_p}{\sqrt{2}},\\
H \stackrel{\mathrm{def}}{=}a \overline{a}= \frac{1}{2}\left(p^2 + q^2 \right) , \\
\overline{a}\times a = H - 1, ~~ a \times \overline{a} = H + 1
\end{split}
\end{equation}
then
 \begin{equation}\label{mp_ap_eqn}
\begin{split}
a \times f = af + \frac{\partial f}{\partial \overline{a}}, ~~ f \times a = af - \frac{\partial f}{\partial \overline{a}}, \\
\overline{a} \times f = \overline{a}f - \frac{\partial f}{\partial a}, ~~ f \times a = \overline{a}f + \frac{\partial f}{\partial a}, \\
\end{split}
\end{equation} \begin{equation}\label{mp_ham_act_eqn}
\begin{split}
H\times f_{mn} = (2m + 1) f_{mn};~~f_{mn} \times H = 2(n + 1)f_{mn}
\end{split}
\end{equation} 
 \begin{equation}\label{mp_hamd_act_eqn}
\begin{split}
a \times f_{mn} = \sqrt{2m}f_{m-1,n};~~ f_{mn}\times a = \sqrt{2n + 2} f_{m, n+1};\\
\overline{a} \times f_{mn} = \sqrt{2m+2}f_{m+1,n};~~ f_{m+1,n}\times \overline{a} = \sqrt{2n} f_{m, n-1}.
\end{split}
\end{equation}

 It is proven in \cite{varilly_bondia:phobos} that
 \begin{equation}\label{mp_part_prod_eqn}
\partial_j\left(f \times g \right)  = \partial_j f \times g + f \times \partial_jg;
 \end{equation}
 where $\partial_j = \frac{\partial}{ \partial x_{j}}$ is the partial derivation in $\SS\left(\R^{2N} \right)$. 
\end{empt}

\subsection{Infinite coverings}
\paragraph{} Let us consider a sequence
\begin{equation}\label{nt_long_seq_eqn}
\mathfrak{S}_{C\left( \T^n_\Th\right) } =\left\{ C\left( \T^n_\Th\right)  =C\left( \T^n_{\Th_0}\right) \xrightarrow{\pi^1}  ... \xrightarrow{\pi^j} C\left( \T^n_{\Th_j}\right) \xrightarrow{\pi^{j+1}} ...\right\}.
\end{equation}
of finite coverings of noncommutative tori.
The sequence \eqref{nt_long_seq_eqn} satisfies to the Definition  \ref{comp_defn}, i.e. $\mathfrak{S}_{C\left( \T^n_\Th\right) } \in \mathfrak{FinAlg}$.

\begin{empt}\label{nt_mp_prel_lim}
	Let $\Th = J \th$ where $\th \in \R \backslash \Q$ and
	$$
	J = \begin{pmatrix} 0 & 1_N \\ -1_N & 0 \end{pmatrix}.
	$$
	Denote by $C\left( \T^{2N}_\th\right) \stackrel{\text{def}}{=} C\left( \T^{2N}_\Th\right)$.
	Let $\{p_k \in \mathbb{N}\}_{k \in \mathbb{N}}$ be an infinite sequence of natural numbers such that $p_k > 1$ for any $k$, and let $m_j = \Pi_{k = 1}^{j} p_k$. From the \ref{nt_fin_cov} it follows that there is a  sequence of *-homomorphisms
	\begin{equation}\label{nt_long_seq_spec_eqn}
	\mathfrak{S}_\th= \left\{	C\left(\mathbb{T}^{2N}_\th\right) \to	C\left(\mathbb{T}^{2N}_{\th/m_1^{2}}\right) \to C\left(\mathbb{T}^{2N}_{\th/m_2^{2}}\right) \to... \to C\left(\mathbb{T}^{2N}_{\th/m_j^{2}}\right)\to ...\right\}
	\end{equation}
	such that
	\begin{enumerate}
		\item[(a)] 	For any $j \in \mathbb{N}$ there are generators $u_{j-1,1},..., u_{j-1,2N}\in U\left(C\left(\mathbb{T}^{2N}_{\th/m_{j-1} ^{2}}\right)\right)$ and generators $u_{j,1},..., u_{j,2N}\in U\left(C\left(\mathbb{T}^{2N}_{\th/m_j^{2}}\right)\right)$ such that the *-homomorphism $ C\left(\mathbb{T}^{2N}_{\th/m_{j-1}^{2}}\right)\to C\left(\mathbb{T}^{2N}_{\th/m_j^{2}}\right)$ is given by
		$$
		u_{j-1,k} \mapsto u^{p_j}_{j,k}; ~~ \forall k =1,..., 2N.
		$$
		There are generators $u_1,...,u_{2N} \in U\left( C\left(\mathbb{T}^{2N}_\th\right)\right) $ such that *-homomorphism $C\left(\mathbb{T}^{2N}_\th\right) \to C\left(\mathbb{T}^{2N}_{\th/m_1^{2}}\right)$ is given by
		$$
		u_{j} \mapsto u^{p_1}_{1,j}; ~~ \forall j =1,..., 2N,
		$$
		
		\item[(b)] For any $j \in \N$ the triple $\left(C\left(\mathbb{T}^{2N}_{\th/m_{j - 1}^{2}}, C\left(\mathbb{T}^{2N}_{\th/m_j^{2}}\right), \Z_{p_j}\right) \right)$ is a noncommutative finite-fold covering,
		\item[(c)] There is the sequence of groups and epimorphisms
		\begin{equation*}
		\mathbb{Z}^{2N}_{m_1} \leftarrow\mathbb{Z}^{2N}_{m_2} \leftarrow ...
		\end{equation*}
		which is equivalent to the sequence
		\begin{equation*}
		\begin{split}
		G\left(C\left(\mathbb{T}^{2N}_{\th/m_1^{2}}\right)~|~ C\left(\mathbb{T}^{2N}_{\th}\right)\right)\leftarrow G\left(C\left(\mathbb{T}^{2N}_{\th/m_2^{2}}\right)~|~ C\left(\mathbb{T}^{2N}_{\th}\right)\right)\leftarrow... \\ \leftarrow G\left(C\left(\mathbb{T}^{2N}_{\th/m_j^{2}}\right)~|~ C\left(\mathbb{T}^{2N}_{\th}\right)\right)\leftarrow...\ .
		\end{split}
		\end{equation*}
	\end{enumerate}
	The sequence \eqref{nt_long_seq_spec_eqn}, is a specialization of \eqref{nt_long_seq_eqn}, hence $\mathfrak{S}_\th \in \mathfrak{FinAlg}$. Denote by $\widehat{C\left(\mathbb{T}^{2N}_{\th}\right) } \stackrel{\text{def}}{=} \varinjlim C\left(\mathbb{T}^{2N}_{\th/m_j^{2}}\right)$, $\widehat{G} \stackrel{\text{def}}{=} \varprojlim G\left(C\left(\mathbb{T}^{2N}_{\th/m_j^{2}}\right)~|~ C\left(\mathbb{T}^{2N}_{\th}\right)\right)$ . The group $\widehat{G}$ is Abelian because it is the inverse limit of Ablelian groups. Denote by $0_{\widehat{G}}$ (resp. "+") the neutral element of $\widehat{G}$ (resp. the product operation of $\widehat{G}$).
\end{empt}

\subsubsection{Inverse noncommutative limit}\label{nt_inv_lim_sec}
\paragraph{}

 There are the equivariant representation
\begin{equation}\label{mp_equ_eqn}
\widehat{\pi}^\oplus: \widehat{C\left(\mathbb{T}^{2N}_{\th}\right) } \to \bigoplus_{g \in J} g L^2\left(\R^{2N}_\th \right) 
\end{equation}
and an inclusion $\Z^{2N} \to \widehat{G}$ described in  \cite{ivankov:qnc}.

\begin{thm}\label{nt_inf_cov_thm}\cite{ivankov:qnc} Following conditions hold:
	\begin{enumerate}
		\item[(i)] The representation $\widehat{\pi}^\oplus$ is good,
		\item[(ii)] 
		\begin{equation*}	\begin{split}
		\varprojlim_{\widehat{\pi}^\oplus} \downarrow \mathfrak{S}_\th = C_0\left(\R^{2N}_\th\right), \\
		G\left(\varprojlim_{\widehat{\pi}^\oplus} \downarrow \mathfrak{S}_\th~|~ C\left(\mathbb{T}^{2N}_\th \right)\right)  = \Z^{2N}.
		\end{split}
		\end{equation*}
		\item[(iii)] The triple $\left(C\left(\mathbb{T}^{2N}_\th \right), C_0\left(\R^{2N}_\th\right), \Z^{2N} \right)$ is an  infinite noncommutative covering of $\mathfrak{S}_\th$ with respect to $\widehat{\pi}^\oplus$.
	\end{enumerate}

\end{thm}
\subsubsection{Induced representation}\label{nt_inf_ind_repr_constr_sec}
\paragraph*{}
Denote by $L^2\left(C_0\left(\R^{2N}_\th \right)  \right) \subset C_0\left(\R^{2N}_\th \right) $ the space of square-summable elements (cf. Definition \ref{ss_defn}). Clearly $\SS\left(\R^{2N}_\th \right) \subset L^2\left(C_0\left(\R^{2N}_\th \right)  \right)$ and since $\SS\left(\R^{2N}_\th \right)$ is dense in $L^2\left(\R^{2N}_\th\right)$ in the topology of the Hilbert space, $ L^2\left(C_0\left(\R^{2N}_\th \right)  \right)$ is also dense in  $L^2\left(\R^{2N}_\th\right)$.	
Similarly to \eqref{inf_ind_prod_eqn} we consider following pre-Hilbert space
$$
L^2\left(C_0\left(\R^{2N}_\th \right)  \right) \otimes_{C\left(\mathbb{T}^{2N}_{\theta}\right)} L^2\left(C\left(\mathbb{T}^{2N}_{\theta}\right), \tau\right)
$$
and denote by $\widetilde{\H}$ its Hilbert completion. From the dense inclusions $\SS\left(\R^{2N}_\th \right)\subset L^2\left(C_0\left(\R^{2N}_\th \right)  \right)$, $\Coo\left(\mathbb{T}^{2N}_{\theta}\right) \subset L^2\left(C\left(\mathbb{T}^{2N}_{\theta}\right), \tau\right)$ it follows that  the composition
$$
\SS\left(\R^{2N}_\th \right)\otimes_{\Coo\left(\mathbb{T}^{2N}_{\theta}\right)}\Coo\left(\mathbb{T}^{2N}_{\theta}\right) \subset  L^2\left(C_0\left(\R^{2N}_\th \right)  \right) \otimes_{C\left(\mathbb{T}^{2N}_{\theta}\right)} L^2\left(C\left(\mathbb{T}^{2N}_{\theta}\right)\right)  \subset \widetilde{   \H}
$$
is the dense inclusion. Otherwise $\SS\left(\R^{2N}_\th \right)\otimes_{\Coo\left(\mathbb{T}^{2N}_{\theta}\right)}\Coo\left(\mathbb{T}^{2N}_{\theta}\right) \cong \SS\left(\R^{2N}_\th \right)$ it follows that there is the dense (with respect to the topology of the Hilbert space) inclusion
$$
\SS\left(\R^{2N}_\th \right) \subset \widetilde{   \H}.
$$
 $\widetilde{ a}, \widetilde{b} \in \SS\left(\R^{2N}_\th \right)$ then from  \eqref{inf_ind_prod_eqn} it turns out  it turns out
\begin{equation*}
\begin{split}
\left( \widetilde{ a} \otimes 1_{C\left(\mathbb{T}^{2N}_{\theta}\right)} , \widetilde{ b} \otimes \Psi_\th\left(1_{C\left(\mathbb{T}^{2N}_{\theta}\right)} \right)\right) _{\widetilde{ \H}} =\\  =\left( 1_{C\left(\mathbb{T}^{2N}_{\theta}\right)} , \sum_{g \in \Z^{2N}}g\left(\widetilde{ a}^*\widetilde{ b} \right)  \Psi_\th\left(1_{C\left(\mathbb{T}^{2N}_{\theta}\right)} \right)\right)_{L^2\left(C\left(\mathbb{T}^{2N}_{\theta}\right), \tau\right)} =\\=
\int_{\T^{2N}}\left(\sum_{g \in \mathbb{Z}^{2N}}  g\left(\widetilde{a}_{\text{comm}}^*~\widetilde{b}_{\text{comm}} \right)\right) \left(x \right) ~dx  =
\int_{\R^{2N}}\widetilde{a}_{\text{comm}}^*~\widetilde{b}_{\text{comm}}\left( \widetilde{x}\right) ~ d\widetilde{x}
\end{split}
\end{equation*}
where $\widetilde{a}_{\text{comm}}\in \SS\left(\R^{2N} \right)$ (resp. $\widetilde{b}_{\text{comm}}\in \SS\left(\R^{2N} \right)$) is a commutative function which corresponds to $\widetilde{a}$ (resp. $\widetilde{b}$). Above equation coincides with \eqref{nt_tracial_prop}. Taking into account that $\SS\left(\R^{2N}_\th \right)$ is dense in $L^2\left(\R^{2N}_\th \right)$ one has an isomorphism
$$
\widetilde{ \H}\approx L^2\left( \mathbb{R}^{2N}_{\widetilde{\theta}}\right)
$$
of Hilbert spaces. Thus if $\rho: C\left(\mathbb{T}^{2N}_{\theta}\right) \to L^2\left(C\left(\mathbb{T}^{2N}_{\theta}\right), \tau\right)$ then both
\begin{equation*}
\begin{split}
\widehat{\rho}: \widehat{ C\left(\mathbb{T}^{2N}_{\theta}\right)} \to B\left( L^2\left( \mathbb{R}^{2N}_{\widetilde{\theta}}\right)\right),\\
\widetilde{\rho}: C_0\left( \R^{2N}_\th\right)  \to B\left( L^2\left( \mathbb{R}^{2N}_{\widetilde{\theta}}\right)\right) 
\end{split}
\end{equation*}
are induced by 
$\left( \rho, \mathfrak{S}_\th, \widehat{\pi}^\oplus \right).  $

\subsubsection{The sequence of spectral triples}

\paragraph*{}

Let us consider following objects
\begin{itemize}
\item A spectral triple of a noncommutative torus $\left(C^{\infty}(\mathbb{T}^{2N}_{\theta}),\mathcal{H}=L^2\left(C\left(\mathbb{T}^{2N}_{\theta}\right), \tau\right)\otimes\mathbb{C}^{2^N},D\right)$,
\item  A good algebraical  finite covering sequence given by
	\begin{equation*}
\mathfrak{S}_\th= \left\{	C\left(\mathbb{T}^{2N}_\th\right) \to	C\left(\mathbb{T}^{2N}_{\th/m_1^{2}}\right) \to C\left(\mathbb{T}^{2N}_{\th/m_2^{2}}\right) \to... \to C\left(\mathbb{T}^{2N}_{\th/m_j^{2}}\right)\to ...\right\}\in \mathfrak{FinAlg}.
\end{equation*}
given by \eqref{nt_long_seq_spec_eqn}.
\end{itemize}
Otherwise from the Theorem \ref{nt_st_fin_cov_thm} it follows that
\begin{equation}\label{nt_triple_seq_eqn}
\begin{split}
\mathfrak{S}_{\left(C^{\infty}\left( \T_\th\right) , L^2\left(C\left(\T^{2N}_\th \right), \tau\right) \otimes \C^{2^N} , D\right) } = \{ \left(C^{\infty}\left( \T_\th\right) , L^2\left(C\left(\T^{2N}_\th \right), \tau\right) \otimes \C^{2^N} , D\right), \dots,\\
 \left(C^{\infty}\left( \T^{2N}_{\th/m_j^2}\right) , L^2\left(C\left( \T^{2N}_{\th/m_j^2} \right), \tau_j\right) \otimes \C^{2^N} , D_j \right), \dots
\}\in \mathfrak{CohTriple}
\end{split}
\end{equation}
is a coherent sequence of spectral triples. We would like to proof that $$\mathfrak{S}_{\left(C^{\infty}\left( \T_\th\right) , L^2\left(C\left(\T^{2N}_\th \right), \tau\right) \otimes \C^{2^N} , D\right) }$$ is regular and  to find  a $\left(C\left(\T^{2N}_\th \right) , C\left(\R^{2N}_\th \right) , \Z^{2N} \right)$-lift of $\left(C^{\infty}\left( \T_\th\right) , L^2\left(C\left(\T^{2N}_\th \right), \tau\right) \otimes \C^{2^N} \right)$. If $\rho: C\left( \T_\th\right) \to  B\left( L^2\left(C\left(\T^{2N}_\th \right), \tau\right)\right)  \otimes \C^{2^N}$ is the natural representation then from the \ref{nt_inf_ind_repr_constr_sec} it turns out that
$$
\widetilde{\rho}:C\left(\R^{2N}_\th \right) \to B\left(L^2\left(\R^{2N}_\th \right)\otimes \C^{2^N} \right) 
$$  
is induced by $\left(\rho, \mathfrak{S}_\th, \widehat{ \pi}^\oplus  \right)$. Let us consider a topological covering $\varphi: \R^{2N} \to \T^{2N}$ and a commutative spectral triple $\left( \Coo\left(\mathbb{T}^{2N} \right),  L^2\left(\T^{2N},S \right), \slashed D\right)$ given by \eqref{nt_ct_tr_eqn}.  Denote by $\widetilde{S} = \varphi^*S$, $\widetilde{\slashed D} = \varphi^*\slashed D$ inverse images of the Spin-bundle $S$ and the Dirac operator $\slashed D$ (cf. \ref{top_vb_sub_sub}, \ref{inv_image_defn}). 
Otherwise there is a natural isomorphism of Hilbert spaces
$$
\widetilde{\Phi}: L^2\left(\R^{2N}_\th\right)\otimes \C^{2^N} \xrightarrow{\approx}L^2\left(\R^{2N}\right)\otimes \C^{2^N}.
$$
Denote by 
\begin{equation*}
\widetilde{D} = \widetilde{\Phi}^{-1} \circ \widetilde{\slashed D} \circ \widetilde{\Phi}.
\end{equation*}
\subsubsection{Smooth elements}

\paragraph*{}
Following lemmas will be used for the construction of the smooth algebra.
\begin{lem}\label{mp_weak_lem}\cite{ivankov:qnc}
	Following conditions hold:
	\begin{enumerate}
		\item[(i)]
		Let $\left\{a_n \in  C_b\left(\R^{2N}_\th\right)\right\}_{n \in \N}$ be a sequence such that
		\begin{itemize}
			\item $\left\{a_n \right\}$ is weakly-* convergent (cf. Definition \ref{nt_*w_defn}),
			\item If $a = \lim_{n \to \infty} a_n$ in the sense of weak-* convergence then  $a \in C_b\left(\R^{2N}_\th\right)$.
		\end{itemize}
		Then the sequence $\left\{a_n \right\}$ is convergent in sense of weak topology $\left\{a_n \right\}$ (cf. Definition \ref{weak_topology}) and $a$ is limit of $ \left\{a_n \right\}$ with respect to the weak topology. Moreover if $\left\{a_n \right\}$ is increasing or decreasing sequence of self-adjoint elements then $\left\{a_n \right\}$ is convergent in sense of strong topology  
		(cf. Definition \ref{strong_topology})  and $a$ is limit of $ \left\{a_n \right\}$ with respect to the strong topology.
	\item[(ii)] If $\left\{a_n \right\}$ is strongly and/or weakly convergent (cf. Definitions \ref{strong_topology}, \ref{weak_topology}) and  $a = \lim_{n \to \infty} a_n$ is strong and/ or weak limit then $\left\{a_n \right\}$ is  weakly-* convergent and $a$ is the limit of $\left\{a_n \right\}$ in the sense of weakly-* convergence.
		\end{enumerate}

\end{lem}

\begin{lem}\label{mp_strong_lem}\cite{ivankov:qnc}
	Let $\overline{G}_j = \ker\left( \Z^{2N} \to \Z^{2N}_{m_j}\right)$.
	Let $\widetilde{a} \in \SS\left( \R^{2N}_\th\right)$ and let
	\be\label{mp_aj_eqn}
	a_j = \sum_{g \in\overline{G}_j } g\widetilde{a} 
	\ee
	where the sum the series means weakly-* convergence. Following conditions hold:
	\begin{enumerate}
		\item [(i)] $a_j \in \Coo\left(\R^{2N} \right)$,
		\item[(ii)]  The series \eqref{mp_aj_eqn} is convergent with respect to the strong topology (cf. Definition \ref{strong_topology}),
		\item[(iii)] There is a following strong limit
		\be\label{mp_ta_eqn}
		\widetilde{a} = \lim_{j \to \infty} a_j.
		\ee
	\end{enumerate}
	
\end{lem}\label{mp_strong_rem}
\begin{remark}
In (i) of the Lemma \ref{mp_strong_lem} $a_j$ is regarded as element of $\SS'\left(\R^{2N} \right)$, it follows that $a_j \in \Coo\left( \T^{2N}_\th\right)$. 
\end{remark}

\begin{lem}\label{mp_semin_lem}
	The system of seminorms $\left\| \cdot\right\|_s$ on $\SS\left( \R^{2N}_\th\right)$   given by \eqref{smooth_seminorms_eqn} is equivalent to the system of seminorms $\left\| \cdot\right\|_{\left(t_1, \dots, t_{2N} \right) }$ given by
	\begin{equation*}
\left\|\widetilde{a}\right\|_{\left(t_1, \dots, t_{2N} \right) } \stackrel{\mathrm{def}}{=} \left\| \frac{\partial^{t_1+ \dots + t_{2N}} }{\partial x^{t_1}_{1}\dots\partial x^{t_{2N}}_{2N}}\widetilde{a}\right\|_{\mathrm{op}} 
	\end{equation*}
	where %$t_1,\dots t_{2N}\in \N^0$, and   $\frac{\partial^{t_1+ \dots + t_{2N}} }{\partial x^{t_1}_{1}\dots\partial x^{t_{2N}}_{2N}}$  means partial derivation of $\widetilde{a}$  regarded as element in $\SS'\left( \R^{2N}\right)$ and 
	$\left\|\cdot\right\|_{\mathrm{op}}$ is the operator norm given by  \eqref{mp_op_norm_eqn}.  
\end{lem}
\begin{proof}
  Operators $1_{C_b\left(\R^{2N}_\th \right) } \otimes \pi^s_j\left(a_j \right)\in B\left(\widetilde{\H}^{2^s} \right)=B\left(\left( L^2\left(\R^{2N}\right)\right)^{2^N2^s}\right)$ from the condition (b) of the Definition \ref{smooth_el_defn} and \eqref{s_diff_repr_equ} can be regarded as matrices in $\mathbb{M}_{2^s2^N}\left(B\left(L^2\left(\R^{2N}\right)  \right) \right)$, so we will write $$1_{C_b\left(\R^{2N}_\th \right) } \otimes \pi^s_j\left(a_j \right) = \left(m^j_{\al\bt}\right)_{\al,\bt = 1, \dots 2^s2^N}\in \mathbb{M}_{2^s2^N}\left(B\left(L^2\left(\R^{2N}\right)  \right) \right).$$
	From \eqref{nt_comm_diff_mod_gamma_eqn} it follows that
	\begin{equation}\label{nt_diff_cov_mod_eqn1}
	\left[D_j, a_j \right]  = \sum_{\mu = 1}^{2N} \frac{\partial a_j}{\partial x_\mu} u^*_\mu \left[ D, u_\mu \right] = \sum_{\mu = 1}^{2N} \gamma^\mu\frac{\partial a_j}{\partial x_\mu} u^*_\mu \left[ \delta_\mu, u_\mu \right]= \sum_{\mu = 1}^{2N} \frac{\partial a_j}{\partial x_\mu}\gamma^\mu 
	\end{equation}
	where $u_1, \dots, u_n$ are unitary generators of $\Coo\left(\mathbb{T}^{2N}_{{\th}}\right)$.
	If $s = 1$ then from \eqref{nt_dirac_eqn} \eqref{nt_diff_mod_eqn} it follows that for any $\al, \bt$ element
	$m_{\al\bt}$ is given by
	$$
	m^j_{\al\bt} = a_j
	$$
	or there is $\mu \in \left\{1,\dots, 2N\right\}$ such that
	\begin{equation}\nonumber%\label{mp_pr_eqn}
	m^j_{\al\bt} = \frac{\partial a_j}{\partial x_\mu} u^*_\mu \left[\delta_\mu, u_\mu \right]
	\end{equation}
	and taking into account  \eqref{nt_diff_1_eqn} one has
	\begin{equation}\label{mp_pr_d_eqn}
	m^j_{\al\bt} = \frac{\partial a_j}{\partial x_\mu}.
	\end{equation}
	
From $\widetilde{a} \in \SS\left( \R^{2N}_\th\right)$ it turns out $ \frac{\partial \widetilde{a}}{\partial x_\mu} \in \SS\left( \R^{2N}_\th\right)$. Taking into account  (ii) of the Lemma \ref{mp_strong_lem} for any $\mu =1,\dots 2N$ the sequence
	$$\left\{\frac{\partial a_j}{\partial x_\mu}\in \Coo\left( \T^{2N}_{\th/m^2_j}\right)\right\}_{j \in \N}$$ is strongly convergent and from (iii) of the Lemma \ref{mp_strong_lem} there is the following strong limit
	$$
	\lim_{j \to \infty} \frac{\partial a_j}{\partial x_\mu} =  \frac{\partial \widetilde{a}}{\partial x_\mu}.
	$$
	Hence if $	m^j_{\al\bt}$ is given by \eqref{mp_pr_d_eqn} then there is a following  strong limit
	\begin{equation}\label{mp_lim_eqn}
	\lim_{j \to \infty} m^j_{\al\bt} = \widetilde{m}_{\al\bt}= \frac{\partial \widetilde{a}}{\partial x_\mu}u_\mu^*\left[\delta_\mu, u_\mu \right] = \frac{\partial \widetilde{a}}{\partial x_\mu}
	\end{equation}
	where $\widetilde{m}_{\al\bt}$ is element of matrix which represent $\pi^1\left( \widetilde{a}\right)$. From  \eqref{nt_diff_cov_mod_eqn1} and (iii) of the Lemma \ref{mp_strong_lem} one has a  strong limit
	\begin{equation}\label{mp_b_eqn}
	\lim_{j \to \infty} 1_{C_b\left(\R^{2N}\right) } \otimes\left[D_j, a_j\right] =    \sum_{\mu = 1}^{2N} \frac{\partial \widetilde{a}}{\partial x_\mu} \gamma^\mu.
	\end{equation}
	It follows  from  \eqref{s_diff_repr_equ} and \eqref{mp_pr_d_eqn} that if $s = 2$ then  the matrix which corresponds to $\pi^2_j\left( a_j\right)$   for any $\mu = 1,\dots, 2N$ contains a submatrix
	$$
	\left[D_j, \frac{\partial a_j}{\partial x_\mu} \right].
	$$
	For any $\nu = 1, \dots, 2N$ the above submatrix contains an element given by
	\begin{equation}\label{mp_pr2_eqn}
	m^j_{\al\bt} = \frac{\partial^2 a_j}{\partial x_\nu\partial x_\mu} u^*_\nu \left[\delta_\nu, u_\nu \right] = \frac{\partial^2 a_j}{\partial x_\nu\partial x_\mu}.
	\end{equation}
	
	From $ \frac{\partial^2 \widetilde{a}}{\partial x_\nu\partial x_\mu} \in \SS\left(\R^{2N} \right)  $ and (iii) of the  Lemma \ref{mp_strong_lem}  one has a strong limit
	$$
	\lim_{j \to \infty}\frac{\partial^2 a_j}{\partial x_\nu\partial x_\mu} = \frac{\partial^2 \widetilde{a}}{\partial x_\nu\partial x_\mu},
	$$
	so the matrix $\left\{m_{\al\bt}\right\}$ contains an element $\frac{\partial^2 \widetilde{a}}{\partial x_\nu\partial x_\mu}$. 
	 Similarly for any multiindex $\left( t_1,..., t_{2N}\right)\in \left(\N^0 \right)^{2M}$ there is $s \in \N$ such that
	$1_{C_b\left(\R^{2N}_\th \right) } \otimes \pi^s_j\left(a_j \right)\in B\left(\widetilde{\H}^{2^s} \right)$ is represented by a matrix $$ \left(m^j_{\al\bt}\right)_{\al,\bt = 1, \dots 2^s2^N}\in \mathbb{M}_{2^s2^N}\left(B\left(L^2\left(\R^{2N}\right)  \right) \right)$$ such that there are $\al$, $\bt$ such that
	\begin{equation}
	m^j_{\al\bt}= \frac{\partial^{t_1+ \dots + t_{2N}} a_j}{\partial x^{t_1}_{1}\dots\partial x^{t_{2N}}_{2N}}.
	\end{equation} 
From $\frac{\partial^{t_1+ \dots + t_{2N}} \widetilde{a}}{\partial x^{t_1}_{1}\dots\partial x^{t_{2N}}_{2N}} \in \SS\left(\R^{2N} \right)$ the and (iii) of the Lemma \ref{mp_strong_lem} it follows that for any $s \in \N$ there are strong limits
	\begin{equation}\label{mp_multi_lim_eqn}
	\widetilde{m}_{\al\bt} =\lim_{j \to \infty}	m^j_{\al\bt} = \frac{\partial^{t_1+ \dots + t_{2N}} \widetilde{a}}{\partial x^{t_1}_{1}\dots\partial x^{t_{2N}}_{2N}},
	\end{equation}
so one has the strong limit $\pi^s\left( \widetilde{a}\right)= \lim_{j \to \infty}\pi^s\left( a_j\right)$ for any $s \in \N$.  The operator	$1_{C_b\left(\R^{2N}_\th \right) } \otimes \pi^s_j\left(a_j \right)\in B\left(\widetilde{\H}^{2^s} \right)$ is represented by a matrix $ \left(m^j_{\al\bt}\right)_{\al,\bt = 1, \dots 2^s2^N}\in \mathbb{M}_{2^s2^N}\left(B\left(L^2\left(\R^{2N}\right)  \right) \right)$ it follows that the norm $\left\| \pi^s\left( \widetilde{a}\right)\right\| $ of $\pi^s\left( \widetilde{a}\right)$ is equivalent to the system of operator norms of its matrix elements given by
	\begin{equation}\label{mp_multi_norm_lim_eqn}
\left\| \widetilde{m}_{\al\bt} \right\|  = \left\| \frac{\partial^{t_1+ \dots + t_{2N}} \widetilde{a}}{\partial x^{t_1}_{1}\dots\partial x^{t_{2N}}_{2N}}\right\|_{\mathrm{op}}.
\end{equation}	
	
\end{proof}

\begin{lem}\label{mp_ss_lem}
	Any $\widetilde{a} \in \SS\left(\R^{2N}_\th \right)$ satisfies to the conditions (b), (c) of the Definition \ref{smooth_el_defn}.
\end{lem}
\begin{proof}(b) Follows from the Lemma \ref{mp_semin_lem}.\\
(c) From \eqref{mp_b_eqn} it turns out that
\begin{equation}\label{mp_lim_diff_eqn}
\lim_{j \to \infty} 1_{C_b\left(\R^{2N}_\th\right) } \otimes\left[D_j, a_j\right] =  \sum_{\mu = 1}^{2N} \frac{\partial \widetilde{a}}{\partial x_\mu} u^*_\mu \left[ D, u_\mu \right].
\end{equation}
If $L^2\left( C_0\left(\R^{2N}_\th\right)\right)$ is a space of square-summable elements (cf. Definition \ref{ss_defn}) then 
$\SS\left( \R^{2N}_\th\right) \subset L^2\left( C_0\left(\R^{2N}_\th\right)\right) $.  Taking into account $~~\frac{\partial \widetilde{a}}{\partial x_\mu}\in \SS\left( \R^{2N}_\th\right)$, $~~u^*_\mu \left[ D, u_\mu \right] \in \Omega^1_D$ and \eqref{mp_lim_diff_eqn} one has
 $$
 \lim_{j \to \infty} 1_{C_b\left(\R^{2n}\right) } \otimes\left[D_j, a_j\right]  \in L^2\left( C_0\left(\R^{2N}_\th\right)\right)  \otimes_{C\left(\T_{\th}^{2N} \right) } \Om^1_D.
 $$	
\end{proof}

\begin{cor}\label{mp_in_cor}
	Let $f_{mn} \in \SS\left(\R^2_\th \right)$ be given by the Proposition \ref{mp_fmn}. If $\widetilde{a} \in \SS\left( \R^{2N}_\th\right)$ is such that
	\begin{equation}
	\widetilde{a} = f_{m_1n_1} \otimes \dots \otimes f_{m_Nn_N}; ~~ (\mathrm{cf. } ~\eqref{mp_2N_eqn})	
	\end{equation} 
then $\widetilde{a}$ is a $\mathfrak{S}_{\left(C^{\infty}\left( \T_\th\right) , L^2\left(C\left(\T^{2N}_\th \right), \tau\right) \otimes \C^{2^N} , D\right) }$-smooth element with respect to $\widehat{\pi}^\oplus$.
\end{cor}
\begin{proof}
	We need check conditions (a)-(d) of the Definition  \ref{smooth_el_defn}.
 The condition (a) follows from (i) of the Lemma \ref{mp_strong_lem}. Conditions (b), (c) follow from the Lemma  \eqref{mp_ss_lem}
	From the Proposition \ref{mp_fmn} it turns out that	$f_{m_1n_1} , \dots , f_{m_Nn_N}$ are rank-one operators, hence $\widetilde{a} = f_{m_1n_1} \otimes \dots \otimes f_{m_Nn_N}$ is also a rank-one operator. However any finite-rank operator lies in the Pedersen ideal (it is proven in \cite{pedersen:mea_c}).	So $\widetilde{a}$ lies in the Pederesen ideal of $C_0\left(\R^{2N}_\th \right)$, i.e. $\widetilde{a}$ satisfies to the condition (d).
\end{proof}
\begin{empt}\label{mp_d_constr}
If $\widetilde{a}$ is a $\mathfrak{S}_{\left(C^{\infty}\left( \T_\th\right) , L^2\left(C\left(\T^{2N}_\th \right), \tau\right) \otimes \C^{2^N} , D\right) }$-smooth element with respect to $\widehat{\pi}^\oplus$ then from \eqref{mp_lim_diff_eqn} it follows that if $\widetilde{a}_D$ is given by  \eqref{a_d_eqn} then
\begin{equation}\label{mp_lim_diff_eqn1}
\widetilde{a}_D =\lim_{j \to \infty} 1_{C_b\left(\R^{2n}\right) } \otimes\left[D_j, a_j\right] =  \sum_{\mu = 1}^{2N} \frac{\partial \widetilde{a}}{\partial x_\mu} u^*_\mu \left[ D, u_\mu \right]
\end{equation}
For any $\xi = \left( \xi_1, \dots, \xi_{2^N}\right) \in \Coo\left(\T^{2N}_\Th \right) \otimes \C^{2^N} \subset L^2\left(\Coo\left(\T^{2N}_\Th \right), \tau \right) \otimes \C^{2^N}$  the operator $\widetilde{ D}$ given by \eqref{inf_lift_D_eqn} satisfies to the following condition
$$
\widetilde{D}\left(\widetilde{a} \otimes \xi \right) = \sum_{\mu = 1}^{2N} \frac{\partial \widetilde{a}}{\partial x_\mu} u^*_\mu \left[ D, u_\mu \right]\xi + \widetilde{a} D\xi = \sum_{\mu =1}^{2N}\frac{\partial \widetilde{a}}{\partial x_\mu} \otimes \gamma^\mu \xi + \widetilde{a} \otimes   \sum_{\mu =1}^{2N}\gamma^\mu\frac{\partial \xi}{\partial x_\mu},
$$
and taking into account \eqref{mp_part_prod_eqn} one has.
\begin{equation}\label{mp_dirac_eqn}
\widetilde{D}=   \sum_{\mu =1}^{2N}\gamma^\mu\frac{\partial }{\partial x_\mu}.
\end{equation}
The given by \eqref{smooth_seminorms_eqn} seminorms $\left\| \cdot\right\|_s$ satisfy to following equation
\begin{equation}\label{mp_s_n_eqn}
\left\| \widetilde{a}\right\|_s = \left\| \pi^s\left( \widetilde{a}\right) \right\|_{\mathrm{op}}
\end{equation}
where
$$
\pi^s\left( \widetilde{a}\right)  =  \begin{pmatrix}  \pi^{s-1}\left( \widetilde{a}\right) & 0 \\ \left[\sum_{\mu =1}^{2N}\gamma^\mu\frac{\partial }{\partial x_\mu}~,\pi^{s-1}\left( \widetilde{a}\right)\right] &  \pi^{s-1}\left( \widetilde{a}\right)\end{pmatrix}.
$$

\end{empt}
 \begin{lem}\label{mp_more_lem}
 	If $\widetilde{a}$ is a $\mathfrak{S}_{\left(C^{\infty}\left( \T_\th\right) , L^2\left(C\left(\T^{2N}_\th \right), \tau\right) \otimes \C^{2^N} , D\right) }$-smooth element with respect to the representation $\widehat{\pi}^\oplus$ given by \eqref{mp_equ_eqn} (cf. Definition \ref{smooth_el_defn}) then  
 	$
 	\widetilde{a} \in \SS\left(\R^{2N}_\th \right).
 	$
 \end{lem}
\begin{proof}
Let $\widetilde{a}$ is a $\mathfrak{S}_{\left(C^{\infty}\left( \T_\th\right) , L^2\left(C\left(\T^{2N}_\th \right), \tau\right) \otimes \C^{2^N} , D\right) }$-smooth element with respect to  $\widehat{\pi}^\oplus$ given by \eqref{mp_equ_eqn}. From  the condition (d) of the Definition  \ref{smooth_el_defn} and the Theorem \ref{nt_inf_cov_thm} it follows that $\widetilde{a} \in K\left(C_0\left(\R^{2N}_\th \right)  \right) $. Otherwise from $C_0\left(\R^{2N}_\th\right) \approx \mathcal K$ (cf. \eqref{mp_2N_eqn} )  and taking into account that any $b \in K\left(\mathcal K \right)$ is a finite-rank operator (in \cite{pedersen:mea_c} is proven that any operator in $K\left(\mathcal K \right)$ is a finite-rank operator), one concludes that $\widetilde{a}$ is a finite-rank operator . From this fact and \eqref{mp_2N_eqn} it turns out that
\begin{equation}\label{mp_fin_eqn}
\begin{split}
\widetilde{a} = \sum_{j = 1}^M \widetilde{a}^j, \text{ where } \widetilde{a}^j = \widetilde{a}^j_1 \otimes\dots \otimes \widetilde{a}^j_N \in \underbrace{C\left(\R^{2}_\th\right)\otimes\dots\otimes C\left(\R^{2}_\th\right)}_{N-\mathrm{times}} , 
\end{split}
\end{equation}
where $ \widetilde{a}^j_1, \dots  \widetilde{a}^j_N \in C\left(\R^{2}_\th\right)$ are finite-rank operators. Let us select the  representation \eqref{mp_fin_eqn} such that $M$ is minimal.
If $a \in C_0\left(\R^{2}_\th\right)$ is a finite-rank operator then it can be represented by the following matrix
\begin{equation}\label{mp_fr_eqn}
\begin{split}
a = u \begin{pmatrix}
\la_1& 0 &\ldots & 0 & 0&\ldots \\
0& \la_2 &\ldots & 0 &  0 &\ldots \\
\vdots& \vdots &\ddots  & \vdots& \vdots  &\ldots\\
0& 0 &\ldots & \la_r& 0  &\ldots  \\
0& 0 &\ldots & 0& 0  &\ldots  \\
\vdots& \vdots &\ldots & \vdots & \vdots &\ddots  \\
\end{pmatrix} v 
\end{split}
\end{equation}
where $u, v$ are finite-rank partial isometries. Above operator can be represented by following way
\begin{equation}\label{mp_fact_eqn}
\begin{split}
a = \sum_{k = 1}^r \al_k \bt_k, \text{ where } \al_k, \bt_k \text{ are given by}\\ \\
\al_k = u \begin{pmatrix}
0& \ldots & 0 & 0 & \dots \\
\vdots& \ddots & \vdots &\vdots & \ldots\\
0& \ldots &\left|\la_k\right| & 0&  \ldots \\
0& \ldots &0 & 0 & \ldots  \\
\vdots& \vdots &\ldots & \vdots & \ddots \\
\end{pmatrix}, \\
\bt_k =  \begin{pmatrix}
0& \ldots & 0 & 0 & \dots \\
\vdots& \ddots & \vdots &\vdots & \ldots\\
0& \ldots &\frac{\la_k}{\left|\la_k\right|} & 0&  \ldots \\
0& \ldots &0 & 0 & \ldots  \\
\vdots& \vdots &\ldots & \vdots & \ddots \\
\end{pmatrix}v. \\
\end{split}
\end{equation}
Above equation is equivalent to
\begin{equation}\label{mp_row_col_eqn}
\begin{split}
\al_k = \sum_{m = 0}^\infty \al_{mk} f_{mk},  \\
\bt_k = \sum_{m = 0}^\infty \bt_{km} f_{km}  \\
\end{split}
\end{equation}
where $\left\lbrace f_{mk} \in \SS\left( \R^{2}_\th\right) \right\rbrace_{m,k \in \N^0} $ are given by the Proposition \ref{mp_fmn}.
From the above equation it follows that $\al_k \bt_k$ is a bounded operator if and only if
\begin{equation*}
\begin{split}
\sum_{m = 0}^\infty \left| \al_{mk}\right|^2 < \infty,  \\
\sum_{m = 0}^\infty \left| \bt_{km}\right|^2 < \infty.  \\
\end{split}
\end{equation*}
From the Remark \ref{mp_l2_rem} it turns out  $\al_k \bt_k \in C_0\left( \R^2\right)$. From this fact and taking to account equation \eqref{mp_fact_eqn} one concludes that any term of the finite sum \eqref{mp_fin_eqn} lies in $C_0\left(\R^2 \right)$. It follows that $\widetilde{a} \in C_0\left(\R^{2N} \right)$.
Denote by 
$$
a_j = \sum_{g \in \ker\left(\Z^{2N}\to\Z^{2N}_{m_j} \right) } \widetilde{a}.
$$
From the condition (a) of the Definition \ref{smooth_el_defn} it turns out that $a_j \in \Coo\left(\T^{2N}_{\th/m_j} \right)$. It follows that $a_j$  corresponds to a smooth function in $\Coo\left(\T^{2N}_{m_j}\right) $. Otherwise $a_j$ can be regarded as a smooth periodic function on $\R^{2N}$ and $a_j$ is a distribution, i.e. $a_j \in \SS'\left(\R^{2N} \right)$. %One has $\lim_{j \to \infty} a_j = \widetilde{a}$ in the sense of weak-* convergence. Otherwise $\widetilde{a}$ is bounded operator with respect to the given by \eqref{mp_op_norm_eqn} $C^*$-norm $\left\|\cdot \right\|_{\text{op}}$, it follows from the Lemma \ref{mp_strong_lem}  the strong limit $\lim_{j \to \infty} a_j = \widetilde{a}$. 
From the proof of the  Lemma \ref{mp_semin_lem} one can write
$$1_{C_b\left(\R^{2N}_\th \right) } \otimes \pi^s_j\left(a_j \right) = \left(m^j_{\al\bt}\right)_{\al,\bt = 1, \dots 2^s2^N}\in \mathbb{M}_{2^s2^N}\left(B\left(L^2\left(\R^{2N}\right)  \right) \right).$$ 
From \eqref{mp_multi_lim_eqn} it follows that for any multiindex $\left( t_1,\dots t_{2N}\right)\in \left(\N^0 \right)^{2N}$ there is $s \in \N$ such that
$1_{C_b\left(\R^{2N}_\th \right) } \otimes \pi^s_j\left(a_j \right)\in B\left(\widetilde{\H}^{2^s} \right)$ is represented by a matrix $$ \left(m^j_{\al\bt}\right)_{\al,\bt = 1, \dots 2^s2^N}\in \mathbb{M}_{2^s2^N}\left(B\left(L^2\left(\R^{2N}\right)  \right) \right)$$ such that there are $\al$, $\bt$ which satisfy to the following equation
\begin{equation}
m^j_{\al\bt}= \frac{\partial^{t_1+ \dots + t_{2N}} a_j}{\partial x^{t_1}_{1}\dots\partial x^{t_{2N}}_{2N}}.
\end{equation} 
%For any $\al, \bt$ there is a following weak-* limit 
%\begin{equation}\label{mp_ss_eqn}
%\widetilde{m}_{\al\bt} =\lim_{j \to \infty}	m^j_{\al\bt} %=\frac{\partial^{t_1+ \dots + t_{2N}} \widetilde{a}}{\partial x^{t_1}_{1}\dots\partial x^{t_{2N}}_{2N}}.
%\end{equation}

  From the condition (b) of the Definition \ref{smooth_el_defn} following conditions hold:
\begin{itemize}
\item For any  multiindex $\left( t_1, \dots, t_{2N}\right)\in \left(\N^0 \right)^{2N}$ there is the following limit
\begin{equation*}
\widetilde{m}_{\al\bt} =\lim_{j \to \infty}	m^j_{\al\bt} 
\end{equation*}
 in the strong topology (cf. Definition \ref{strong_topology}).
\item The limit corresponds to a bounded operator with respect to the operator norm \eqref{mp_op_norm_eqn}. 
\end{itemize}
From (ii) of the Lemma \ref{mp_weak_lem} it follows that the strong topology limit is $\widetilde{m}_{\al\bt} =\lim_{j \to \infty}	m^j_{\al\bt}$ is the limit in sense of the weak-* convergence, so one has
\begin{equation}\label{mp_ss_eqn}
\widetilde{m}_{\al\bt} =\lim_{j \to \infty}	m^j_{\al\bt} =\frac{\partial^{t_1+ \dots + t_{2N}} \widetilde{a}}{\partial x^{t_1}_{1}\dots\partial x^{t_{2N}}_{2N}}.
\end{equation}
and right part of \eqref{mp_ss_eqn} corresponds to a bounded operator, i.e. $\frac{\partial^{t_1+ \dots + t_{2N}} \widetilde{a}}{\partial x^{t_1}_{1}\dots\partial x^{t_{2N}}_{2N}} \in B\left(L^2\left(\R^{2N} \right)  \right)$.
If one considers a factorization \eqref{mp_fin_eqn}
\begin{equation*}
\begin{split}
\widetilde{a} = \sum_{j = 1}^M \widetilde{a}^j, \text{ where } \widetilde{a}^j = \widetilde{a}^j_1 \otimes\dots \otimes \widetilde{a}^j_N \in \underbrace{C\left(\R^{2}_\th\right)\otimes\dots\otimes C\left(\R^{2}_\th\right)}_{N-\mathrm{times}}  
\end{split}
\end{equation*}
such that $M$ is minimal then all partial tensor products
\begin{equation*}
\begin{split}
P_j =
 \widetilde{a}^j_2 \otimes\dots \otimes \widetilde{a}^j_N \in \underbrace{C\left(\R^{2}_\th\right)\otimes\dots\otimes C\left(\R^{2}_\th\right)}_{\left( N-1\right) -\mathrm{times}};~~ j = 1,\dots, M 
\end{split}
\end{equation*}
are linearly independent. Similarly to \ref{mp_coord_constr}  for $j^{\text{th}}$ term of the tensor product $$\underbrace{C\left(\R^{2}_\th\right)\otimes\dots\otimes C\left(\R^{2}_\th\right)}_{N-\mathrm{times}}$$ we denote by $p_j, q_j$ coordinates which satisfy to \eqref{mp_pq_mult_eqn} - \eqref{mp_pq_mult}. 
One has
\begin{equation*}
\frac{\partial}{\partial p_1} \widetilde{a} = \sum_{j = 1}^M \frac{\partial}{\partial p_1} \widetilde{a}^j_1 \otimes \widetilde{a}^j_2   \otimes\dots \otimes \widetilde{a}^j_N .
\end{equation*}
Elements  $P_j$ are linearly independent, it follows that if any term in the above sum is unbounded with respect to the norm  \eqref{mp_op_norm_eqn} then all sum is unbounded.  So $\frac{\partial}{\partial p_1} \widetilde{a}^j_1$ is bounded for any $j = 1,\dots M$. Otherwise $\widetilde{a}^1_1$ is a finite-rank operator it follows that  $\widetilde{a}^1_1$ can be represented by \eqref{mp_fact_eqn}, i.e.
\begin{equation}\label{mp_a11_eqn}
\begin{split}
\widetilde{a}^1_1 = \sum_{k = 1}^r \al_k \times \bt_k\\ 
\end{split}
\end{equation}
Otherwise taking into account \eqref{mp_partial_eqn} one has
\begin{equation}\label{mp_indep_sum}
\begin{split}
\frac{\partial}{\partial p_1} \widetilde{a}^j_1 = \sum_{k = 1}^r -iq_1 \times \al_k \times \bt_k + \sum_{k = 1}^r  \al_k \times \bt_k \times iq_1.
\end{split}
\end{equation}
From  \eqref{mp_fact_eqn} it turns out that all terms in \eqref{mp_indep_sum} are linearly independent, so if one or more terms are unbounded (with respect to the norm \eqref{mp_op_norm_eqn}) then the whole sum is unbounded.
Otherwise $q_1 \times \al_k \times \bt_k$ is unbounded if and only if $q_1 \times \al_k$ is unbounded. Similarly $ \al_k \times \bt_k \times q_1$ is unbounded if and only if  $\bt_k \times q_1$ is unbounded. From this fact it turns out that all operators 
\begin{equation*}
\begin{split}
q_1 \times \al_k, ~~\bt_k \times q_1
\end{split}
\end{equation*}
are bounded. Similarly one can prove that following operators
\begin{equation*}
\begin{split}
p_1 \times \al_k, ~~\bt_k \times p_1
\end{split}
\end{equation*}
are bounded. Clearly if \begin{equation*}
\begin{split}
a_1 = \frac{q_1 + ip_1}{\sqrt{2}}; ~~\overline{a}_1 = \frac{q_1 - ip_1}{\sqrt{2}}\\
\end{split}
\end{equation*}
then operators 
\begin{equation*}
\begin{split}
a_1  \times \al_k, ~~\bt_k \times a_1, ~ \overline{a}_1  \times \al_k, ~~\bt_k \times \overline{a}_1, 
\end{split}
\end{equation*}
are bounded. Similarly to \eqref{mp_ham_eqn} we define $H_1 = \overline{a}_1\times a_1 - 1$. For any $m, n \in \N$ a distribution  $\frac{\partial^m}{\partial^m p_1}\frac{\partial^n}{\partial^n q_1} \widetilde{a}^1_1$ is a bounded  (with respect to \eqref{mp_op_norm_eqn}) it follows that for any $l \in \N$ following distributions
\begin{equation*}
\begin{split}
\underbrace{H_1 \times \dots \times H_1}_{l -\mathrm{times}} \times \al_k, ~~\bt_k \times \underbrace{H_1 \times \dots \times H_1}_{l -\mathrm{times}}
\end{split}
\end{equation*}
are bounded operators. From  \eqref{mp_ham_act_eqn} and \eqref{mp_row_col_eqn} it follows that
\begin{equation*}
\begin{split}
\al_k = \sum_{m = 0}^\infty \al_{mk} f_{mk},  \\
\bt_k = \sum_{m = 0}^\infty \bt_{km} f_{km},  \\
\underbrace{H_1 \times \dots\times  H_1}_{l -\mathrm{times}} \times \al_k = \sum_{m = 0}^\infty \left(2 m + 1 \right)^l \al_{mk} f_{mk}, \\
\bt_k \times \underbrace{H_1 \times \dots \times H_1}_{l -\mathrm{times}}=  \sum_{m = 0}^\infty \left(2 m + 1 \right)^l \bt_{km} f_{km},
\end{split}
\end{equation*}
hence operators $\underbrace{H_1 \times \dots \times H_1}_{l -\mathrm{times}} \times \al_k$ and $\bt_k \times \underbrace{H_1 \times \dots \times H_1}_{l -\mathrm{times}}$ are bounded if following conditions hold:
\begin{equation*}
\begin{split}
 \sum_{m = 0}^\infty \left(2 m + 1 \right)^{2l} \left| \al_{mk}\right|^2 < \infty, \\
\sum_{m = 0}^\infty \left(2 m + 1 \right)^{2l} \left| \bt_{km}\right|^2 < \infty
\end{split}
\end{equation*}
Form the Definition \ref{df:Gst} it follows that for any $s \in \N$ following conditions hold:
\begin{equation*}
\begin{split}
\al_k \in \G_{2l,s}, ~~
\bt_k \in \G_{s,2l}
\end{split}
\end{equation*}
Since we can select arbitrary $l$ and taking into account \eqref{mp_ssin_eqn} one has
$$
\al_k \times \bt_k \in \SS\left(\R^2 \right).
$$
From \eqref{mp_a11_eqn} it turns out that
$$
\widetilde{a}^1_1 \in \SS\left(\R^2 \right)
$$
If we consider representation \eqref{mp_fin_eqn}
\begin{equation*}
\begin{split}
\widetilde{a} = \sum_{j = 1}^M \widetilde{a}^j, \text{ where } \widetilde{a}^j = \widetilde{a}^j_1 \otimes\dots \otimes \widetilde{a}^j_N \in \underbrace{C\left(\R^{2}_\th\right)\otimes\dots\otimes C\left(\R^{2}_\th\right)}_{N-\mathrm{times}}  
\end{split}
\end{equation*}
then similarly to the above construction one can prove that
$$
\widetilde{a}^j_k \in \SS\left(\R^2 \right); ~ j = 1,\dots, M, ~~ k=1, \dots N.
$$
From \eqref{mp_fin_eqn} it follows that 
$$
\widetilde{a} \in \underbrace{\SS\left(\R^{2}_\th\right)\otimes\dots\otimes \SS\left(\R^{2}_\th\right)}_{N-\mathrm{times}} \subset\SS\left(\R^{2N} \right).
$$
where one means the algebraic tensor product.
\end{proof}

\begin{lem}\label{mp_ss_in_smooth}
	Denote by $\Coo_0\left(\R^{2N}_\th \right) $  the smooth algebra of $\mathfrak{S}_{\left(C^{\infty}\left( \T_\th\right) , L^2\left(C\left(\T^{2N}_\th \right), \tau\right) \otimes \C^{2^N} , D\right)}$ with respect to $\widehat{\pi}^\oplus$. Following condition holds
	\begin{equation*}
	\SS\left( \R^{2N}\right) \subset \Coo_0\left(\R^{2N}_\th \right).
	\end{equation*}
\end{lem}
\begin{proof}
	Let $I_0 = \left(\N^0\right)^2$ and let $I = I_0^N$. For any $\nu = \left(\left(m^\nu_1, n^\nu_1\right), \dots , \left(m^\nu_N, n^\nu_N\right)  \right) \in I$ we denote 
	\begin{equation*}
	f_\nu \stackrel{\text{def}}{=} f_{m^\nu_1n^\nu_1} \otimes \dots \otimes f_{m^\nu_Nn^\nu_N} \in \underbrace{\SS\left(\R^{2}\right)\otimes\dots\otimes \SS\left(\R^{2}\right)}_{N-\mathrm{times}} \subset \SS\left( \R^{2N}\right) 
	\end{equation*}
	where we mean the algebraic tensor product. Indeed $\SS\left( \R^{2N}\right)$ is the projective completion of the above algebraic tensor product with respect to seminorms $r_{k }$ given by \eqref{mp_ss_norm_eqn}. From the seminorms $(r_k)$ given by \eqref{mp_matr_norm} it turns out that $\SS\left( \R^{2N}\right)$ is a space of $\C$-linear combinations
	\begin{equation*}
	\sum_{\nu \in I} c_\nu f_\nu; \text{ where } c_\nu = c_{\left(\left(m^\nu_1, n^\nu_1\right), \dots , \left(m^\nu_N, n^\nu_N\right)  \right)} \in \C
	\end{equation*}
	such that for any $k = \left(k_1, \dots, k_N \right) \in \left( \N^0\right)^N$ following condition holds 
	\begin{equation}\label{mp_ss_norm_eqn}
r_k	\left(\sum_{\nu \in I} 
	c_\nu f_\nu\right) = \biggl(  \th^{2\left( k_1 + \dots + k_N\right)} \sum_{\nu \in I}\left|c_\nu \right|^2 \prod_{p = 1}^N  \left( m^\nu_p+\half\right) ^{k_p} \left( n^\nu_p+\half\right) ^{k_p} \biggr)^{1/2} < \infty.
	\end{equation}
	 If $M \in \N$ and 
	and $I_M \subset I$ is a finite subset such that
	\begin{equation}\label{mp_I_M}
	I_M = \left\{\left(\left(m^\nu_1, n^\nu_1\right), \dots , \left(m^\nu_N, n^\nu_N\right)  \right) \in I~|~m^\nu_1, n^\nu_1,\dots,m^\nu_N, n^\nu_N \le M \right\}
	\end{equation}
	then \eqref{mp_ss_norm_eqn} is equivalent to
	\begin{equation*}
	\sum_{\nu \in I\backslash I_M}\left|c_\nu \right| \prod_{p = 1}^N  \left( m^\nu_p+\half\right) ^{k_p} \left( n^\nu_p+\half\right)^{k_p} < \infty.
	\end{equation*}
	From the above equation and
	$
	\left(m+n+1 \right)^k < 2^k \left(m+\half \right)^k\left(n + \half \right)^k, \forall m, n, k > 1.
	$ it follows that for any $M > 1$ and $l > 1$ following condition holds
	\begin{equation}\label{mp_l_eqn}
\sum_{\nu \in I\backslash I_M}\left|c_\nu \right| \prod_{p = 1}^N  \left( m^\nu_p+ n^\nu_p+1\right)^{l} < \infty.
	\end{equation}

	From  the Lemma \ref{mp_semin_lem} it turns out that 
the system of seminorms $\left\| \cdot\right\|_s$  given by \eqref{smooth_seminorms_eqn} is equivalent to the system of seminorms $\left\| \cdot\right\|_{\left(t_1, \dots, t_{2N} \right) }$ given by
	\begin{equation*}
	\left\|\widetilde{a}\right\|_{\left(t_1, \dots, t_{2N} \right) } \stackrel{\mathrm{def}}{=} \left\| \frac{\partial^{t_1+ \dots + t_{2N}} }{\partial x^{t_1}_{1}\dots\partial x^{t_{2N}}_{2N}}\widetilde{a}\right\|_{\mathrm{op}} 
	\end{equation*}
	This lemma is true if for any ${\left(t_1, \dots, t_{2N} \right) }\in \left( \N^0\right)^{2N} $ from
	$$
	\left\| \frac{\partial^{t_1+ \dots + t_{2N}} }{\partial x^{t_1}_{1}\dots\partial x^{t_{2N}}_{2N}}\widetilde{a}\right\|_{\mathrm{op}} < \infty
	$$
 it follows that for any $\eps > 0$ there is a finite subset $I_f \subset I$ such that
	\begin{equation}\label{mp_eps_eqn}
	\sum_{\nu \in I\backslash I_f} \left|c_\nu \right|\left\| \frac{\partial^{t_1+ \dots + t_{2N}} }{\partial x^{t_1}_{1}\dots\partial x^{t_{2N}}_{2N}}f_\nu\right\|_{\mathrm{op}} < \eps.
	\end{equation}
		Let us replace coordinates $x_1, \dots, x_{2N}$ with coordinates $p_1, q_1, \dots, p_N, q_N$ such that $p_j, q_j$ are coordinates on $j^{\text{th}}$ term of the product
	$\underbrace{\SS\left(\R^{2}\right)\otimes\dots\otimes \SS\left(\R^{2}\right)}_{N-\mathrm{times}} \subset \SS\left( \R^{2N}\right) $. From the equation
	\begin{equation*}
	\frac{\partial^{t_1+ \dots + t_{2N}} }{\partial x^{t_1}_{1}\dots\partial x^{t_{2N}}_{2N}}f_\nu = \frac{\partial^{t_1, t_{2}} }{\partial p^{t_1}_{1}\partial q^{t_2}_{1}} f_{m^\nu_1, n^\nu_1} \otimes \dots \otimes \frac{\partial^{t_{2N-1}, t_{2N}} }{\partial p^{t_1}_{N}\partial q^{t_2}_{N}} f_{m^\nu_N, n^\nu_N}
	\end{equation*}
	it follows that
	\begin{equation*}
	\left\| \frac{\partial^{t_1+ \dots + t_{2N}} }{\partial x^{t_1}_{1}\dots\partial x^{t_{2N}}_{2N}}f_\nu\right\|_{\text{op}}  =\left\|  \frac{\partial^{t_1, t_{2}} }{\partial p^{t_1}_{1}\partial q^{t_2}_{1}} f_{m^\nu_1, n^\nu_1} \right\|_{\text{op}} \cdot ... \cdot \left\| \frac{\partial^{t_{2N-1}, t_{2N}} }{\partial p^{t_1}_{N}\partial q^{t_2}_{N}} f_{m^\nu_N, n^\nu_N}\right\|_{\text{op}}.
	\end{equation*}
	
Our proof can be simplified if we use scaling construction \ref{mp_scaling_constr}, i.e. we set $\th = 2$. Indeed many quantitative results does not depend on $\th$. Similarly to \eqref{mp_times_eqn} we write $f {\times} g$ instead of $f {\star_{2}} g$. Moreover one can use given by \eqref{mp_ham_eqn} coordinates $a, \overline{a}$ instead of $p, q$. From \eqref{mp_ap_eqn} it follows that
\begin{equation*}
\frac{\partial f}{\partial a} = -\overline{a} \times f+ f \times \overline{a},~~\frac{\partial f}{\partial \overline{a}} = a \times f - f \times a,
\end{equation*}
and taking into account \eqref{mp_hamd_act_eqn} one has
\begin{equation*}\label{mp_deriv_a_eqn}
\begin{split}
\frac{\partial f_{mn}}{\partial a}=-\sqrt{2m+2}f_{m+1,n}+\sqrt{2n} f_{m, n-1},\\ \frac{\partial f_{mn}}{\partial \overline{a}}= \sqrt{2m}f_{m-1,n}-\sqrt{2n + 2} f_{m, n+1}.
\end{split}
\end{equation*}
If $t_1, t_2 \in \N^0$ and $\left|t \right|=t_1 + t_2$ then from the above equations and $\left\|f_{mn}\right\|_{\text{op}}=1$ it follows that
\begin{equation}\label{mp_4mn_eqn}
\begin{split}
\left\|\frac{\partial^{t_1, t_{2}} }{\partial a^{t_1}\partial \overline{a}^{t_2}} f_{mn}\right\|_{\text{op}} \le \\
\le\left(\sqrt{2m+2 + \left|t \right|} + \sqrt{2n + \left|t \right|} \right)^{t_1}\left(\sqrt{2m + \left|t \right|} + \sqrt{2n + 2 + \left|t \right|} \right)^{t_2} <\\<\left( 2m+2n+2+2\left|t \right|\right)^{t_1} \left( 2m+2n+2+2\left|t \right|\right)^{t_2} =
\left(2m + 2n + 2 + 2 \left|t \right|\right)^{\left|t \right|}. 
\end{split}
\end{equation}
If $m, n \ge M$ then
\be\label{mp_nM_eqn}
\left(2m + 2n + 2 + 2 \left|t \right|\right)^{\left|t \right|} \le \left( \frac{4M+2+2 \left|t \right|}{4M+2} \right)^{\left|t \right|} \left(2m + 2n + 2 \right)^{\left|t \right|}.
\ee
Denote by $\left|t \right| = t_1 + \dots + t_{2N}$.
If $M \in \N$   and $I_M$ is given by \eqref{mp_I_M} then from \eqref{mp_nM_eqn} it turns out
\be\nonumber
\begin{split}
\sum_{\nu \in I \backslash I_M} \left|c_\nu \right| \prod_{p = 1}^N  \left( 2m^\nu_p+ 2n^\nu_p+2 + 2 \left|t \right|\right)^{\left|t \right|} <\\ <\left( \frac{4M+2+2 \left|t \right|}{4M+2} \right)^{\left|t \right|}  \sum_{\nu \in I\backslash I_M}\left|c_\nu \right| \prod_{p = 1}^N  \left( 2m^\nu_p+ 2n^\nu_p+2\right)^{\left|t \right|}.
\end{split}
\ee
From \eqref{mp_l_eqn}  it follows that right part of the above equation is convergent, hence one has
\begin{equation}\label{mp_sss_eqn}
\sum_{\nu \in I} \left|c_\nu \right| \prod_{p = 1}^N  \left( 2m^\nu_p+ 2n^\nu_p+2 + 2 \left|t \right|\right)^{\left|t \right|} < \infty.
\end{equation}
If $P$ is a differential operator given by
\begin{equation}\nonumber
P =	\frac{\partial^{t_1+ \dots + t_{2N}} }{\partial a^{t_1}_{1}\partial \overline{a}^{t_2}_{1}\dots\partial a^{t_{2N-1}}_{N}\partial \overline{a}^{t_{2N}}_{N}} = P_1 \otimes \dots \otimes P_N= \frac{\partial^{t_1, t_{2}} }{\partial a^{t_1}_{1}\partial \overline{a}^{t_2}_{1}}  \otimes \dots \otimes \frac{\partial^{t_{2N-1}, t_{2N}} }{\partial a^{t_{2N-1}}_{N}\partial \overline{a}^{t_{2N}}_{N}},
\end{equation}
 then from \eqref{mp_4mn_eqn} and \eqref{mp_sss_eqn} it follows that the series 
\begin{equation*}
\sum_{\nu \in I} c_\nu P f_\nu
\end{equation*}
is $\left\|\cdot\right\|_{\text{op}}$-norm convergent (cf. \eqref{mp_op_eqn}). Otherwise from the Lemma \ref{mp_semin_lem} it follows that the family of seminorms $\sum_{\nu \in I} c_\nu  f_\nu \mapsto \left\| \sum_{\nu \in I} c_\nu P f_\nu \right\|_{\mathrm{op}}$ is equivalent to the family of  seminorms $\left\| \cdot\right\|_s$ given by 
\eqref{smooth_seminorms_eqn}. It turns out $\sum_{\nu \in I} c_\nu  f_\nu $ is convergent with respect to seminorms $\left\| \cdot\right\|_s$, i.e. $$\sum_{\nu \in I} c_\nu  f_\nu\in \Coo_0\left(\R^{2N}_\th \right).$$

\end{proof}

\subsubsection{Covering of spectral triple}
\paragraph*{}

Following theorem completely describes infinite coverings of noncommutative tori.

\begin{theorem}
	Let $\mathfrak{S}_{\left(C^{\infty}\left( \T_\th\right) , L^2\left(C\left(\T^{2N}_\th \right), \tau\right) \otimes \C^{2^N} , D\right) }\in \mathfrak{CohTriple}$ be a coherent sequence of spectral triples given by \eqref{nt_triple_seq_eqn}. Let $\widehat{\pi}^\oplus: \widehat{C\left(\mathbb{T}^{2N}_{\th}\right) } \to \bigoplus_{g \in J} g L^2\left(\R^{2N}_\th \right)$ be an equivariant representation given by \eqref{mp_equ_eqn}.
	Following condition holds:
	\begin{enumerate}
		\item [(i)] If $\Coo_0\left(\R^{2N}_\th\right)$ is the smooth algebra of $\mathfrak{S}_{\left(C^{\infty}\left( \T_\th\right) , L^2\left(C\left(\T^{2N}_\th \right), \tau\right) \otimes \C^{2^N} , D\right) }$ with respect to $\widehat{\pi}^\oplus$ then $\Coo_0\left(\R^{2N}_\th\right)$ is the completion of $\SS\left(\R^{2N}_\th\right)$ with respect to seminorms given by \eqref{mp_s_n_eqn},
		\item [(ii)] The sequence $\mathfrak{S}_{\left(C^{\infty}\left( \T_\th\right) , L^2\left(C\left(\T^{2N}_\th \right), \tau\right) \otimes \C^{2^N} , D\right)}$ is regular with respect to $\widehat{\pi}^\oplus$,
		\item[(iii)] If $\widetilde{D}$ is given by \eqref{mp_dirac_eqn} then the triple $$\left(\Coo_0\left(\R^{2N}_\th\right), L^2\left(\R^{2N}\right)\otimes \mathbb{C}^{2^N}  , \widetilde{D}\right)$$ is the $\left(C\left(\mathbb{T}^{2N}_\th \right), C_0\left(\R^{2N}_\th\right), \Z^{2N} \right)$-lift of  $\left(C\left(\mathbb{T}^{2N}_\th \right), L^2\left(\T^{2N}\right)\otimes \mathbb{C}^{2^N}, D \right)$.
	\end{enumerate}
\end{theorem}
\begin{proof}(i)
		From the   Lemma $\ref{mp_more_lem}$ it turns out $\Coo_0\left(\R^{2N}_\th\right)$ is contained in  the completion of 
	$\SS\left(\R^{2N}_\th\right)$ with respect to seminorms $\left\| \cdot\right\|_s$.
	From the Lemma  \ref{mp_ss_in_smooth} it follows that  $\Coo_0\left(\R^{2N}_\th\right)$ contains the completion of $\SS\left(\R^{2N}_\th\right)$ with respect to seminorms given by \eqref{mp_s_n_eqn}.
 \\
	(ii) The algebra $\SS\left(\R^{2N}_\th\right)$ is dense in $C_0\left(\R^{2N}_\th\right)$, so $\Coo_0\left(\R^{2N}_\th\right)$ is dense in $C_0\left(\R^{2N}_\th\right)$.\\
(iii) Follows from the construction \ref{mp_d_constr}.
\end{proof}

\section{Isospectral deformations and their coverings}

\paragraph*{}A very general construction of isospectral
deformations
of noncommutative geometries is described in \cite{connes_landi:isospectral}. The construction
implies in particular that any
compact Spin-manifold $M$ whose isometry group has rank
$\geq 2$ admits a
natural one-parameter isospectral deformation to noncommutative geometries
$M_\theta$.
We let $(\Coo\left(M \right)  , \H = L^2\left(M,S \right)  , \slashed D)$ be the canonical spectral triple associated with a
compact spin-manifold $M$. We recall that $\mathcal{A} = C^\infty(M)$ is
the algebra of smooth
functions on $M$, $S$ is the spinor bundle and $\slashed D$
is the Dirac operator.
Let us assume that the group $\mathrm{Isom}(M)$ of isometries of $M$ has rank
$r\geq2$.
Then, we have an inclusion
\begin{equation*}
\mathbb{T}^2 \subset \mathrm{Isom}(M) \, ,
\end{equation*}
with $\mathbb{T}^2 = \mathbb{R}^2 / 2 \pi \mathbb{Z}^2$ the usual torus, and we let $U(s) , s \in
\mathbb{T}^2$, be
the corresponding unitary operators in $\H = L^2(M,S)$ so that by construction
\begin{equation*}
U(s) \, \slashed D = \slashed D \, U(s).
\end{equation*}
Also,
\begin{equation}\label{isospectral_sym_eqn}
U(s) \, a \, U(s)^{-1} = \alpha_s(a) \, , \, \, \, \forall \, a \in \mathcal{A} \, ,
\end{equation}
where $\alpha_s \in \mathrm{Aut}(\mathcal{A})$ is the action by isometries on the
algebra of functions on
$M$.

\noindent
We let $p = (p_1, p_2)$ be the generator of the two-parameters group $U(s)$
so that
\begin{equation*}
U(s) = \exp(i(s_1 p_1 + s_2 p_2)) \, .
\end{equation*}
The operators $p_1$ and $p_2$ commute with $D$.
Both $p_1$ and $p_2$
have integral spectrum,
\begin{equation*}
\mathrm{Spec}(p_j) \subset \mathbb{Z} \, , \, \, j = 1, 2 \, .
\end{equation*}

\noindent
One defines a bigrading of the algebra of bounded operators in $\H$ with the
operator $T$ declared to be of bidegree
$(n_1,n_2)$ when,
\begin{equation*}
\alpha_s(T) = \exp(i(s_1 n_1 + s_2 n_2)) \, T \, , \, \, \, \forall \, s \in
\mathbb{T}^2 \, ,
\end{equation*}
where $\alpha_s(T) = U(s) \, T \, U(s)^{-1}$ as in \eqref{isospectral_sym_eqn}.
\paragraph{}
Any operator $T$ of class $C^\infty$ relative to $\alpha_s$ (i. e. such that
the map $s \rightarrow \alpha_s(T) $ is of class $C^\infty$ for the
norm topology) can be uniquely
written as a doubly infinite
norm convergent sum of homogeneous elements,
\begin{equation*}
T = \sum_{n_1,n_2} \, \widehat{T}_{n_1,n_2} \, ,
\end{equation*}
with $\widehat{T}_{n_1,n_2}$ of bidegree $(n_1,n_2)$ and where the sequence
of norms $||
\widehat{T}_{n_1,n_2} ||$ is of
rapid decay in $(n_1,n_2)$.
Let $\lambda = \exp(2 \pi i \theta)$. For any operator $T$ in $\H$ of
class $C^\infty$ we define
its left twist $l(T)$ by
\begin{equation}\label{l_defn}
l(T) = \sum_{n_1,n_2} \, \widehat{T}_{n_1,n_2} \, \lambda^{n_2 p_1} \, ,
\end{equation}
and its right twist $r(T)$ by
\begin{equation*}
r(T) = \sum_{n_1,n_2} \, \widehat{T}_{n_1,n_2} \, \lambda^{n_1 p_2} \, ,
\end{equation*}
Since $|\lambda | = 1$ and $p_1$, $p_2$ are self-adjoint, both series
converge in norm. Denote by $\Coo\left(M \right)_{n_1, n_2} \subset \Coo\left(M \right) $ the $\C$-linear subspace of elements of bidegree $\left( n_1, n_2\right) $. \\
One has,
\begin{lem}\label{conn_landi_iso_lem}\cite{connes_landi:isospectral}
	\begin{itemize}
		\item[{\rm a)}] Let $x$ be a homogeneous operator of bidegree $(n_1,n_2)$
		and $y$ be
		a homogeneous operator of  bidegree $(n'_1,n'_2)$. Then,
		\begin{equation}
		l(x) \, r(y) \, - \,  r(y) \, l(x) = (x \, y \, - y \, x) \,
		\lambda^{n'_1 n_2} \lambda^{n_2 p_1 + n'_1 p_2}
		\end{equation}
		In particular, $[l(x), r(y)] = 0$ if $[x, y] = 0$.
		\item[{\rm b)}] Let $x$ and $y$ be homogeneous operators as before and
		define
		\begin{equation*}
		x * y = \lambda^{n'_1 n_2} \, x y \, ; \label{star}
		\end{equation*}
		then $l(x) l(y) = l(x * y)$.
	\end{itemize}
\end{lem}

\noindent
The product $*$ defined in (\ref{star}) extends by linearity
to an associative product on the linear space of smooth operators and could
be called a $*$-product.
One could also define a deformed `right product'. If $x$ is homogeneous of
bidegree
$(n_1,n_2)$ and $y$ is homogeneous of bidegree $(n'_1,n'_2)$ the product is
defined by
\begin{equation*}
x *_{r} y = \lambda^{n_1 n'_2} \, x y \, .
\end{equation*}
Then, along the lines of the previous lemma one shows that $r(x) r(y) = r(x
*_{r} y)$.

We can now define a new spectral triple where both $\H$ and the operator
$\slashed D$ are unchanged while the
algebra $\Coo\left(M \right)$  is modified to $l(\Coo\left(M \right))$ . By
Lemma~{\ref{conn_landi_iso_lem}}~b) one checks that  $l\left( \Coo\left(M \right)\right) $ is still an algebra. Since $\slashed D$ is of bidegree $(0,0)$ one has,
\begin{equation}\label{isospectral_ld_eqn}
[\slashed D, \, l(a) ] = l([\slashed D, \, a]) 
\end{equation}
which is enough to check that $[\slashed D, x]$ is bounded for any $x \in l(\mathcal{A})$. There is a spectral triple $\left(l\left( \Coo\left(M \right)\right) , \H, \slashed D\right)$.
\paragraph{} Denote by $\Coo\left( M_\th\right)$ (resp.  $C\left(M_\th \right)$) the algebra $l\Coo\left( M\right)$ (resp. the operator norm completion of $l\left(\Coo\left( M\right)  \right)  $). Denote by $\rho: C\left(M\right) \to L^2\left( M, S\right) $ (resp. $\pi_\th: C\left(M_\th\right) \to B\left( L^2\left( M, S\right)\right) $ ) natural representations.

\subsection{Finite-fold coverings}\label{isosectral_fin_cov}
\subsubsection{Basic construction}
\paragraph{}
Let $M$ be a spin - manifold with the smooth action of $\T^2$. 
 Let $\widetilde{x}_0 \in \widetilde{M}$ and $x_0=\pi\left(\widetilde{x}_0 \right)$. Denote by $\varphi: \R^2 \to \R^2 / \Z^2 = \T^2$ the natural covering. There are two closed paths $\om_1, \om_2: \left[0,1 \right]\to M$ given by
\begin{equation*}
\begin{split}
\om_1\left(t \right) = \varphi\left(t, 0 \right) x_0,~
\om_2\left(t \right) = \varphi\left(0, t \right) x_0.
\end{split}
\end{equation*} 
There are  lifts of these paths, i.e. maps $\widetilde{\om}_1 , \widetilde{\om}_2: \left[0,1 \right] \to\widetilde{M}$ such that
\begin{equation*}
\begin{split}
\widetilde{\om}_1\left(0 \right)= \widetilde{\om}_2\left(0 \right)=\widetilde{x}_0,~ \\
\pi\left( \widetilde{\om}_1\left(t \right)\right)  = \om_1\left(t\right),\\
\pi\left( \widetilde{\om}_2\left(t \right)\right)  = \om_2\left(t\right).
\end{split}
\end{equation*}
Since $\pi$ is a finite-fold covering there are $N_1, N_2 \in \N$ such that if 
$$
\gamma_1\left(t \right) = \varphi\left(N_1t, 0 \right) x_0,~
\gamma_2\left(t \right) = \varphi\left(0, N_2t \right) x_0.
$$
and $\widetilde{\gamma}_1$ (resp. $\widetilde{\gamma}_2$) is the lift of $\gamma_1$ (resp. $\gamma_2$) then both $\widetilde{\gamma}_1$, $\widetilde{\gamma}_2$ are closed. Let us select minimal values of $N_1, N_2$. If $\text{pr}_n: S^1 \to S^1$ is an $n$ listed covering and $\text{pr}_{N_1, N2}$ the covering given by
$$
\widetilde{\T}^2 = S^1 \times S^1 \xrightarrow{\text{pr}_{N_1}\times\text{pr}_{N_2}} \to S^1 \times S^1 = \T^2
$$
then there is the action $\widetilde{\T}^2 \times \widetilde{M} \to  \widetilde{M}$ such that

\begin{tikzpicture}
\matrix (m) [matrix of math nodes,row sep=3em,column sep=4em,minimum width=2em]
{
	\widetilde{\T}^2 \times \widetilde{M}  	&  	 & \widetilde{M} \\
	\T^2. \times M 	 & 	 & M   \\};
\path[-stealth]
(m-1-1) edge node [above] {$ \ $} (m-1-3)
(m-1-1) edge node [right] {$\mathrm{pr}_{N_1N_2} \times \pi $} (m-2-1)
(m-1-3) edge node [right] {$\pi$} (m-2-3)
(m-2-1) edge node [above] {$ \ $} (m-2-3);

\end{tikzpicture}

where  $\widetilde{\T}^2 \approx \T^2$.
Let $\widetilde{p} = \left( \widetilde{p}_1, \widetilde{p}_2\right) $ be the generator of the associated with $\widetilde{\T}^2$ two-parameters group $\widetilde{U}\left(s \right) $
so that
\begin{equation*}
\widetilde{U}\left(s \right) = \exp\left( i\left( s_1 \widetilde{p}_1 + s_2 \widetilde{p}_2\right)\right).
\end{equation*}	
The covering $\widetilde{M} \to M$ induces an involutive injective homomorphism
\begin{equation*}
\varphi:\Coo\left(M \right)\to\Coo\left( \widetilde{M} \right).
\end{equation*}

Suppose $M \to M/\T^2$ is submersion, and suppose there is  a weak fibration $\T^2 \to M \to M/\T^2$ (cf. \cite{spanier:at})
There is the exact \textit{homotopy sequence of the weak fibration}
\bean
\dots \to \pi_n\left( \T^2, e_0\right) \xrightarrow{i_{\#}} \pi_n\left( M, e_0\right)  \xrightarrow{p_{\#}} \pi_n\left( M/\T^2, b_0\right) \xrightarrow{\overline\partial} \pi_{n-1}\left( \T^2, e_0\right) \to \dots\\
\dots \to \pi_2\left( M/\T^2, b_0\right) \xrightarrow{\overline\partial} \pi_1\left( \T^2, e_0\right) \xrightarrow{i_{\#}} \pi_1\left( M, e_0\right)  \xrightarrow{p_{\#}} \pi_1\left( M/\T^2, b_0\right) \xrightarrow{\overline\partial} \pi_0\left( \T^2, e_0\right) \to \dots 
\eean
(cf. \cite{spanier:at}) where $\pi_n$ is the $n^{\mathrm{th}}$ homotopical group for any $n\in \N^0$.
If $\pi:\widetilde{M} \to M$ is a finite-fold regular covering then there is the natural surjective homomorphism $
\pi_1\left( M, e_0\right) \to G\left(\widetilde{M}~|~M \right)$. If $\pi:\widetilde{M} \to M$ induces a covering $\pi:\widetilde{M}/\T^2 \to M / \T^2$  then 
the  homomorphism $
\varphi:\pi_1\left( M, e_0\right) \to G\left(\widetilde{M}~|~M \right)$ can be included into the following commutative diagram.

\begin{tikzpicture}
\matrix (m) [matrix of math nodes,row sep=3em,column sep=4em,minimum width=2em]
{
	\pi_1\left(\T^2, e_0 \right)\cong \Z^2 &  \pi_1\left( M, e_0\right) &  \pi_1\left( M/\T^2, b_0\right) &  \{e\} \\
	G\left(\widetilde{\T}^2~|\T^2 \right)  & G\left(\widetilde{M}~|~M \right) & G\left(\widetilde{M}/\widetilde{\T}^2~|~M/\T^2\right)  & \{e\} \ \\};
\path[-stealth]
(m-1-1)  edge node [right] {$\varphi'$} (m-2-1)
(m-1-2)  edge node [right] {$\varphi$} (m-2-2)
(m-1-3)  edge node [right] {$\varphi''$} (m-2-3)
(m-1-1)  edge node [above] {$i_{\#} $} (m-1-2)
(m-1-2)  edge node [above] {$p_{\#} $} (m-1-3)
(m-2-1)  edge node [above] {$i_*$} (m-2-2)
(m-2-2)  edge node [above] {$p_*$} (m-2-3)
(m-1-3)  edge node [right] {$ $} (m-1-4)
(m-1-4)  edge node [right] {$ $} (m-2-4)
(m-2-3)  edge node [right] {$ $} (m-2-4);
\end{tikzpicture}

Denote by $G \stackrel{\mathrm{def}}{=}G\left(\widetilde{M}~|~M \right)$, $~G' \stackrel{\mathrm{def}}{=}G\left(\widetilde{\T}^2~|\T^2 \right)$, $~G'' \stackrel{\mathrm{def}}{=}G\left(\widetilde{M}/\widetilde{\T}^2~|~M/\T^2\right)$. From the above construction it turns out that $G' =G\left(\widetilde{\T}^2~|\T^2 \right) = \Z_{N_1} \times \Z_{N_2}$. Otherwise there is an inclusion of Abelian groups $G\left(\widetilde{\T}^2~|\T^2 \right) \subset \widetilde{\T}^2$. The action $\widetilde{\T}^2 \times \widetilde{M} \to  \widetilde{M}$ is free, so the action $G' \times \widetilde{M} \to $ is free, so the natural homomorphism $G' \to G$ is injective, hence there is an exact sequence of groups
\be\label{isospectral_gr_eqn}
\{e\} \to G' \to G \to G'' \to \{e\}.
\ee

Let $\th, \widetilde{\th} \in \R$ be such that
$$
\widetilde{\th}= \frac{\th +  n}{N_1N_2}, \text{ where }n \in \Z.
$$
If $\lambda= e^{2\pi i \th}$, $\widetilde{\lambda}= e^{2\pi i \widetilde{\th}}$ then
$
\lambda = \widetilde{\lambda}^{N_1N_2}.
$
There are isospectral deformations $\Coo\left(M_\th \right), \Coo\left( \widetilde{M}_{\widetilde{\th}} \right)$ and $\C$-linear isomorphisms
$l:\Coo\left(M \right) \to \Coo\left(M_\th \right)$, $\widetilde{l}:\Coo\left( \widetilde{M} \right) \to \Coo\left( \widetilde{M}_{\widetilde{\th}} \right)$.
These isomorphisms and the inclusion $\varphi$ induces the inclusion
\begin{equation*}
\begin{split}
\varphi_\th:\Coo\left(M_\th \right)\to\Coo\left( \widetilde{M}_{\widetilde{\th}} \right),
\\
\varphi_{\widetilde{\th}}\left(\Coo\left(M_\th \right) \right)_{n_1,n_2} \subset \Coo\left( \widetilde{M}_{\widetilde{\th}} \right)_{n_1N_1,~ n_2N_2}.
\end{split}
\end{equation*}
Denote by $G = G\left(\widetilde{M}~|~M \right)$ the group of covering transformations.   Since $\widetilde{l}$ is a $\C$-linear isomorphism the action of $G$ on $\Coo\left( \widetilde{M} \right)$ induces a $\C$-linear action  $G \times \Coo\left( \widetilde{M}_{ \widetilde{\th}} \right)  \to \Coo\left( \widetilde{M}_{ \widetilde{\th}} \right)$. According to the definition of the action of $\widetilde{\T}^2$ on $\widetilde{M}$ it follows that the action of $G$ commutes with the action of $\widetilde{\T}^2$.
It turns out
$$
g \Coo\left( \widetilde{M} \right)_{n_1,n_2} = \Coo\left( \widetilde{M} \right)_{n_1,n_2}
$$
for any $n_1, n_2 \in \Z$ and $g \in G$.
If $\widetilde{a} \in \Coo\left( \widetilde{M} \right)_{n_1,n_2}$, $\widetilde{b} \in \Coo\left( \widetilde{M} \right)_{n'_1,n'_2}$ then  $g\left( \widetilde{a}\widetilde{b}\right)= \left(g\widetilde{a} \right) \left(g\widetilde{b} \right)\in \Coo\left( \widetilde{M} \right)_{n_1+n'_1,n_2+n'_2} $. One has
\begin{equation*}
	\begin{split}
		\widetilde{l}\left(\widetilde{a}\right)\widetilde{l}\left(\widetilde{b}\right)= \widetilde{\la}^{n'_1n_2}\widetilde{l}\left(\widetilde{a}\widetilde{b}\right), \\
		\widetilde{\la}^{n_2\widetilde{p}_1}l\left( \widetilde{b}\right) = \widetilde{\la}^{n'_1n_2}l\left( \widetilde{b}\right) \widetilde{\la}^{n_2\widetilde{p}_1},\\
		\widetilde{l}\left(g \widetilde{a}\right)\widetilde{l}\left(g \widetilde{b}\right)= g \widetilde{a}\widetilde{\la}^{n_2\widetilde{p}_1}g \widetilde{b}\widetilde{\la}^{n'_2\widetilde{p}_1}= \widetilde{\la}^{n'_1n_2} g\left(\widetilde{a}\widetilde{b} \right) \widetilde{\la}^{\left( n_2+n_2'\right) \widetilde{p}_1}.
	\end{split}
\end{equation*}
On the other hand
\begin{equation*}
	\begin{split}
		g\left( \widetilde{l}\left(\widetilde{a}\right)\widetilde{l}\left(\widetilde{b}\right)\right) = g\left( \widetilde{\la}^{n'_1n_2}\widetilde{l}\left(\widetilde{a}\widetilde{b}\right)\right)= \widetilde{\la}^{n'_1n_2} g\left(\widetilde{a}\widetilde{b} \right) \widetilde{\la}^{\left( n_2+n_2'\right) \widetilde{p}_1}. 
	\end{split}
\end{equation*}
From above equations it turns out
$$
\widetilde{l}\left(g \widetilde{a}\right)\widetilde{l}\left(g \widetilde{b}\right) = g\left( \widetilde{l}\left(\widetilde{a}\right)\widetilde{l}\left(\widetilde{b}\right)\right),
$$
i.e. $g$ corresponds to automorphism of $\Coo\left( \widetilde{M}_{ \widetilde{\th}}\right)$. It turns out that $G$ is the group of automorphisms of $\Coo\left( \widetilde{M}_{ \widetilde{\th}}\right)$. From $\widetilde{a} \in \Coo\left( \widetilde{M}_{ \widetilde{\th}}\right)_{n_1,n_2}$ it follows that $\widetilde{a}^* \in \Coo\left( \widetilde{M}_{ \widetilde{\th}}\right)_{-n_1,-n_2}$. One has
$$
g\left(\left( \widetilde{l}\left(\widetilde{a}\right)\right)^* \right) =  g\left( \widetilde{\la}^{-n_2\widetilde{p_1}}\widetilde{a}^*\right) =
g \left(\widetilde{\la}^{n_1 n_2} \widetilde{a}^*\widetilde{\la}^{-n_2\widetilde{p_1}}\right) = \widetilde{\la}^{n_1 n_2} g\left(\widetilde{l}\left(\widetilde{a}^* \right)  \right). 
$$
On the other hand
$$
\left(g \widetilde{l}\left(\widetilde{a}\right) \right)^*= \left(\left( g \widetilde{a}\right)\widetilde{\la}^{n_2\widetilde{p_1}}  \right)^*=\widetilde{\la}^{-n_2\widetilde{p_1}}\left(ga^* \right) = \widetilde{\la}^{n_1 n_2}\left(ga^* \widetilde{\la}^{-n_2\widetilde{p_1}}\right)= \widetilde{\la}^{n_1 n_2} g\left(\widetilde{l}\left(\widetilde{a}^* \right)  \right),
$$
i.e. $g\left(\left( \widetilde{l}\left(\widetilde{a}\right)\right)^* \right) = \left(g \widetilde{l}\left(\widetilde{a}\right) \right)^*$.
It follows that $g$ corresponds to the involutive automorphism of $\Coo\left( \widetilde{M}_{ \widetilde{\th}}\right)$. Since  $\Coo\left( \widetilde{M}_{ \widetilde{\th}}\right)$ is dense in $C\left( \widetilde{M}_{ \widetilde{\th}}\right)$ there is the unique involutive action $G \times C\left( \widetilde{M}_{ \widetilde{\th}}\right) \to C\left( \widetilde{M}_{ \widetilde{\th}}\right)$.
For any $y_0 \in M/\T^2$ there is a point $x_0 \in M$ mapped onto $y_0$ and a connected submanifold $\mathcal U \subset M$ such that:
\begin{itemize}
	\item $\dim \mathcal U= \dim M-2$,
	\item $\mathcal U$ is transversal to orbits of $\T^2$-action,
	\item The  fibration $\T^2 \to \mathcal U \times \T^2 \to \mathcal U \times \T^2 / \T^2$ is the restriction of the fibration $\T^2 \to M \to M/\T^2$,
	\item The image  $\mathcal V_{y_0} \in M/\T^2$ of $\mathcal U \times \T^2$ in $M/\T^2$ is an open neighborhood of $y_0$,
	\item $\mathcal V_{y_0}$ is evenly covered by  $\widetilde{M}/ \widetilde{\T}^2 \to M / \T^2$.
\end{itemize}
It is clear that
$$
M/\T^2 = \bigcup_{y_0 \in M/\T^2} \mathcal V_{y_0}.
$$
Since $M/\T^2$ is compact there is a finite subset $I \in M/\T^2$ such that
$$
M/\T^2 = \bigcup_{y_0 \in I} \mathcal V_{y_0}.
$$
Above equation will be rewritten as
\be\label{isospectral_p_eqn}
M/\T^2 = \bigcup_{\iota \in I} \mathcal V_{\iota}
\ee
where $\iota$ is just an element of the  finite set $I$ and we denote corresponding transversal submanifold by $\mathcal{U}_\iota$. There is a smooth partition of unity subordinated to \eqref{isospectral_p_eqn}, i.e. there is a set $\left\{a_\iota \in \Coo\left(M /\T^2\right)  \right\}_{\iota \in I}$ of positive elements such that
\be\label{isospectral_part_eqn}
1_{C\left(M /\T^2\right) }= \sum_{\iota \in I} a_\iota~,\\
\ee
\be\nonumber
a_\iota\left(\left( M/\T^2\right)  \backslash \mathcal V_{\iota}\right) = \{0\}.
\ee
Denote by 
\be\label{isospectral_e_eqn}e_\iota \stackrel{\mathrm{def}}{=} \sqrt{a_\iota} \in \Coo\left( M /\T^2\right).
\ee  For any $\iota \in I$ we select an open subset $\widetilde{\mathcal V}_{\iota} \subset \widetilde{M}/\T^2$ which is homeomorphically mapped onto $\mathcal V_{\iota}$.
If $\widetilde{I}= G'' \times I$ then for any $\left(g'', \iota \right) \in \widetilde{I}$ we define
\be\label{isospectral_vgi_eqn}
\begin{split}
	\widetilde{\mathcal V}_{\left(g'', \iota \right)} = g''	\widetilde{\mathcal V}_{\iota}.
\end{split}
\ee
 Similarly we select a transversal submanifold
\be\nonumber%\label{isospectral_u_eqn}
\widetilde{\mathcal U}_{\iota} \subset \widetilde{M}
\ee 
 which is homeomorphially mapped onto $\mathcal U_{\iota}$. For any $\left(g'', \iota \right) \in \widetilde{I}$ we define
 \be\label{isospectral_wgi_eqn}
 \begin{split}
 	\widetilde{\mathcal U}_{\left(g'', \iota \right)} = g	\widetilde{\mathcal U}_{\iota}.
 \end{split}
 \ee
 where $g \in G$ is an arbitrary element mapped to $g''$.
   The set $\mathcal V_{\iota}$ is evenly covered by $\pi'':\widetilde{M}/\widetilde{\T}^2\to M/\T^2$, so one has
\be\label{isostectral_empt_eqn}
g\widetilde{\mathcal V}_{\iota}\bigcap \widetilde{\mathcal V}_{\iota} = \emptyset; \text{ for any nontrivial } g \in G''.
\ee

If $\widetilde{e}_\iota \in \Coo\left(\widetilde M / \widetilde \T^2 \right) $ is given by 
$$
\widetilde{e}_\iota\left(\widetilde{x} \right) = 
\left\{\begin{array}{c l}
e_\iota\left( \pi''\left( \widetilde{x} \right) \right)  & \widetilde{x} \in \widetilde{\mathcal U}_\iota \\
0 & \widetilde{x} \notin \widetilde{\mathcal U}_\iota
\end{array}\right.
$$
then from \eqref{isospectral_part_eqn} and \eqref{isostectral_empt_eqn} it turns out
\begin{equation}
\label{isospectral_partg_eqn}
\begin{split}
1_{C\left(\widetilde M/ \widetilde \T^2 \right) } = \sum_{g \in G''} \sum_{\iota \in I} \widetilde{e}^2_{\iota},\\
\left( g\widetilde{e}_{\iota}\right)\widetilde{e}_{\iota} = 0; \text{ for any nontrivial } g \in G''. 
\end{split}
\end{equation}
If $\widetilde{I}= G'' \times I$ and $\widetilde{e}_{\left(g, \iota\right)} = g \widetilde{e}_{ \iota}$ for any $\left(g, \iota\right)  \in G'' \times I$ then from  \eqref{isospectral_partg_eqn} it turns out
\begin{equation}
\label{isospectral_parti_eqn}
\begin{split}
1_{C\left( \widetilde M / \widetilde \T^2 \right) } =\sum_{\widetilde{\iota} \in \widetilde{I}} \widetilde{e}^2_{\widetilde{\iota}},\\
\left( g\widetilde{e}_{\widetilde{\iota}}\right)\widetilde{e}_{\widetilde{\iota}} = 0; \text{ for any nontrivial } g \in G'',\\
1_{C\left( \widetilde M / \widetilde \T^2 \right) } =  \sum_{\widetilde{\iota} \in \widetilde{I}} \widetilde{e}_{\widetilde{\iota}}\left\rangle \right\langle\widetilde{e}_{\widetilde{\iota}}~.
\end{split}
\end{equation}
It is known that $C\left(\T^2 \right)$ is an universal commutative $C^*$-algebra generated by two unitary elements $u, v$, i.e. there are following relations
\be\label{isospectral_gen_eqn}
\begin{split}
uu^*=u^*u=vv^*=v^*v= 1_{C\left(\T^2 \right)},\\
uv = vu, ~u^*v = vu^*,~uv^*= v^*u, ~ u^*v^*=v^*u^*.
\end{split}
\ee
If $J = I \times \Z \times \Z$ then for any $\left({\iota}, j, k \right) \in J$ there is an element ${f}'_{\left({\iota}, j, k \right)} \in \Coo\left( {\mathcal U}_{{\iota}} \times {\T}^2\right) $ given by
\be\label{isospectral_f's_eqn}
{f}'_{\left({\iota}, j, k \right)} = {e}_{{\iota}} {u}^j{
	v}^k
\ee
where ${e}_{{\iota}}\in \Coo\left( {M}/ {\T}^2\right)$ is regarded as element of $\Coo\left( {M}\right)$.  
Let $p: {M} \to {M} / {\T}^2$.
Denote by ${f}_{\left({\iota}, j, k \right)}\in \Coo\left({M} \right)$ an element given by
\be\label{isospectral_fs_eqn}
\begin{split}
{f}_{\left({\iota}, j, k \right)}\left( {x}\right) =
\left\{\begin{array}{c l}
	{f}'_{\left({\iota}, j, k \right)}\left( {x}\right) & p\left({x} \right)  \in {\mathcal V}_{{\iota}} \\
	0 & p\left( {x}\right)  \notin {\mathcal V}_{{\iota}} 
\end{array}\right.,\\
\text{where the right part of the above equation  assumes the inclusion } \mathcal U_{{\iota}} \times {\T}^2 \hookto M.
\end{split}
\ee 
If we denote by $\widetilde{u}, \widetilde{v} \in U\left( C\left( \widetilde{\T^2} \right) \right)$ unitary generators of $C\left( \widetilde{\T^2} \right)$ then the covering $\pi':\widetilde{\T^2} \to \T^2$ corresponds to a *-homomorphism $C\left(\T^2 \right) \to  C\left( \widetilde{\T^2}\right)$ given by
\bean
u \mapsto \widetilde{u}^{N_1},\\
v \mapsto \widetilde{v}^{N_2}.
\eean
There is the natural action of $G\left( \widetilde{\T}^2~|~{\T^2}\right)\cong \Z_{N_1} \times \Z_{N_2}$ on $C\left(\T^2 \right)$ given by
\be\label{isospectral_uv_eqn}
\begin{split}
\left(\overline{k}_1, \overline{k}_2 \right) \widetilde{u} = e^{\frac{2\pi i k_1}{N_1}}\widetilde{u},\\
\left(\overline{k}_1, \overline{k}_2 \right) \widetilde{v} = e^{\frac{2\pi i k_2}{N_1}}\widetilde{v}
\end{split}
\ee
where $\left(\overline{k}_1, \overline{k}_2 \right) \in \Z_{N_1} \times \Z_{N_2}$. 
If we consider $C\left( \widetilde{\T}^2\right)_{C\left( {\T^2}\right)}$ as a right Hilbert module which corresponds to a finite-fold noncommutative covering then one has
\be\label{isospectral_tor_eqn}
\begin{split}
\left\langle \widetilde{u}^{j'} \widetilde{v}^{k'}, \widetilde{u}^{j''} \widetilde{v}^{k''}  \right\rangle_{C\left( \widetilde{\T^2}\right)} = N_1N_2\delta_{j'j''} \delta_{k'k''}1_{C\left( {\T^2}\right)},\\
1_{C\left( \widetilde{\T^2}\right)}= \frac{1}{N_1N_2}\sum_{\substack{j = 0\\ k = 0}}^{\substack{j = N_1\\ k = N_2}} \widetilde{u}^{j} \widetilde{v}^{k}\left\rangle \right\langle \widetilde{u}^{j} \widetilde{v}^{k}.
\end{split}
\ee
If $\widetilde{J} = \widetilde{I} \times \left\{0, \dots, N_1-1\right\} \times \left\{0, \dots, N_2-1\right\}$ then for any $\left(\widetilde{\iota}, j, k \right) \in \widetilde{J}$ there is an element $\widetilde{f}'_{\left(\widetilde{\iota}, j, k \right)} \in \Coo\left( \widetilde{\mathcal U}_{\widetilde{\iota}} \times \widetilde{\T}^2\right) $ given by
\be\label{isospectral_f'_eqn}
\widetilde{f}'_{\left(\widetilde{\iota}, j, k \right)} = \widetilde{e}_{\widetilde{\iota}} \widetilde{u}^j\widetilde{
v}^k
\ee
where $\widetilde{e}_{\widetilde{\iota}}\in \Coo\left( \widetilde{M}/ \widetilde{\T}^2\right)$ is regarded as element of $\Coo\left( \widetilde{M}\right)$.  
Let $p: \widetilde{M} \to \widetilde{M} / \widetilde{\T}^2$.
Denote by $\widetilde{f}_{\left(\widetilde{\iota}, j, k \right)}\in \Coo\left(\widetilde{M} \right)$ an element given by
\be\label{isospectral_f_eqn}
\begin{split}
\widetilde{f}_{\left(\widetilde{\iota}, j, k \right)}\left( \widetilde{x}\right) =
\left\{\begin{array}{c l}
	\widetilde{f}'_{\left(\widetilde{\iota}, j, k \right)}\left( \widetilde{x}\right) & p\left(\widetilde{x} \right)  \in \widetilde{\mathcal V}_{\widetilde{\iota}} \\
	0 & p\left( \widetilde{x}\right)  \notin \widetilde{\mathcal V}_{\widetilde{\iota}} .
\end{array}\right.\\
\text{where right the part of the above equation assumes the inclusion }  \widetilde{\mathcal U}_{ \widetilde{\iota}} \times  \widetilde{\T}^2 \hookto  \widetilde{M}.
\end{split}
\ee 
Any element $\widetilde{e}_{\widetilde{\iota}} \in C\left(\widetilde{M}/\widetilde{   \T}^2 \right)$ is regarded as element of $C\left(\widetilde{M}\right)$. From $\widetilde{e}_{\widetilde{\iota}} \in \Coo\left(\widetilde{M}\right)_{0,0}$ it turns out $\widetilde{l}\left(\widetilde{e}_{\widetilde{\iota}} \right)  = \widetilde{e}_{\widetilde{\iota}}$,  $~~\left\langle \widetilde{l}\left( \widetilde{e}_{\widetilde{\iota}'}\right) , \widetilde{l}\left( \widetilde{e}_{\widetilde{\iota}''}\right)  \right\rangle_{C\left( \widetilde{M}_{ \widetilde{\th}}\right) } = \left\langle  \widetilde{e}_{\widetilde{\iota}'} ,  \widetilde{e}_{\widetilde{\iota}''}  \right\rangle_{C\left( \widetilde{M}\right) } $.

From \eqref{isospectral_parti_eqn}-\eqref{isospectral_f_eqn} it follows that
\be\label{isospectral_undec_eqn}
1_{C\left(\widetilde{M} \right) }= \frac{1}{N_1N_2}\sum_{\widetilde{\iota} \in \widetilde{I}} \sum_{\substack{j = 0\\ k = 0}}^{\substack{j = N_1-1\\ k = N_2-1}} \widetilde{f}_{\left(\widetilde{\iota}, j, k \right)}\left\rangle \right\langle \widetilde{f}_{\left(\widetilde{\iota}, j, k \right)},
\ee
\be\nonumber%\label{isospectral_fprod_eqn}
\left\langle \widetilde{f}_{\left(\widetilde{\iota}', j', k' \right)}, \widetilde{f}_{\left(\widetilde{\iota}'', j'', k'' \right)} \right\rangle_{C\left( \widetilde{M}\right) } = \frac{1}{N_1N_2} \delta_{j'j''}\delta_{k'k''}\left\langle \widetilde{e}_{\widetilde{\iota}'}, \widetilde{e}_{\widetilde{\iota}''} \right\rangle_{C\left( \widetilde{M}\right) }\in \Coo\left(M \right),
\ee
\be\label{isospectral_fprodl_eqn}
\left\langle \widetilde{l}\left( \widetilde{f}_{\left(\widetilde{\iota}', j', k' \right)}\right) , \widetilde{l}\left( \widetilde{f}_{\left(\widetilde{\iota}'', j'', k'' \right)}\right)  \right\rangle_{C\left( \widetilde{M}_{\widetilde{   \th}}\right) } =\frac{1}{N_1N_2} \delta_{j'j''}\delta_{k'k''}\left\langle \widetilde{e}_{\widetilde{\iota}'}, \widetilde{e}_{\widetilde{\iota}''} \right\rangle_{C\left( \widetilde{M}_{\widetilde{   \th}}\right) }\in  \Coo\left(M_\th \right).
\ee
From the \eqref{isospectral_undec_eqn} it turns out that $C\left(\widetilde{M} \right)$  is a right $C\left(M \right)$ module generated by  finite set of elements $\widetilde{f}_{\left(\widetilde{\iota}, j, k \right)}$ where $\left(\widetilde{\iota}, j, k \right) \in \widetilde{J}$, i.e. any $\widetilde{a} \in C\left(\widetilde{M} \right)$ can be represented as
\be\label{isospectral_wa_eqn}
\widetilde{a} = \sum_{\widetilde{\iota} \in \widetilde{I}} \sum_{\substack{j = 0\\ k = 0}}^{\substack{j = N_1-1\\ k = N_2-1}} \widetilde{f}_{\left(\widetilde{\iota}, j, k \right)} a_{\left(\widetilde{\iota}, j, k \right)}; \text{ where } a_{\left(\widetilde{\iota}, j, k \right)} \in C\left( M\right). 
\ee
Moreover if $\widetilde{a} \in \Coo\left(\widetilde{M} \right)$ then one can select $a_{\left(\widetilde{\iota}, j, k \right)} \in \Coo\left(M \right)$. However any $a_{\left(\widetilde{\iota}, j, k \right)} \in \Coo\left(M \right)$ can be uniquely
written as a doubly infinite
norm convergent sum of homogeneous elements,
\begin{equation*}
a_{\left(\widetilde{\iota}, j, k \right)} = \sum_{n_1,n_2} \, \widehat{T}_{n_1,n_2} \, ,
\end{equation*}

with $\widehat{T}_{n_1,n_2}$ of bidegree $(n_1,n_2)$ and where the sequence
of norms $||
\widehat{T}_{n_1,n_2} ||$ is of
rapid decay in $(n_1,n_2)$. One has
\bea\label{isospectral_sec_eqn}
\widetilde{l}\left(\widetilde{f}_{\left(\widetilde{\iota}, j, k \right)} a_{\left(\widetilde{\iota}, j, k \right)}\right) = \sum_{n_1, n_2}  \widetilde{f}_{\left(\widetilde{\iota}, j, k \right)} \widehat{T}_{n_1,n_2} \widetilde{\lambda}^{\left( N_2n_2+j \right) \widetilde{p}_1}= \sum_{n_1, n_2} \widetilde{f}_{\left(\widetilde{\iota}, j, k \right)} \widetilde{\lambda}^{j \widetilde{p}_1}  \widetilde{\lambda}^{kN_1n_1}\widehat{T}_{n_1,n_2} \widetilde{\lambda}^{N_2n_2 \widetilde{p}_1} 
\eea
%\bea\label{isospectral_sect_eqn}
%\widetilde{l}\left(\widetilde{f}_{\left(\widetilde{\iota}, j, k \right)} a_{\left(\widetilde{\iota}, j, k \right)}\right) = \sum_{n_1, n_2}  \widetilde{f}_{\left(\widetilde{\iota}, j, k \right)} \widehat{T}_{n_1,n_2} \widetilde{\lambda}^{\left( N_2n_2+j \right) \widetilde{p}_1}= \sum_{n_1, n_2} \widetilde{f}_{\left(\widetilde{\iota}, j, k \right)} \widetilde{\lambda}^{j \widetilde{p}_1}  \widetilde{\lambda}^{kN_1n_1}\widehat{T}_{n_1,n_2} \widetilde{\lambda}^{N_2n_2 \widetilde{p}_1} 
%\eea

the sequence
of norms $||
\widetilde{\lambda}^{kN_2n_2}\widehat{T}_{n_1,n_2} ||=||
\widehat{T}_{n_1,n_2} ||$ is of
rapid decay in $(n_1,n_2)$ it follows that
\begin{equation}\label{isospectral_a_eqn}
a'_{\left(\widetilde{\iota}, j, k \right)} = \sum_{n_1,n_2} \, \widetilde{\lambda}^{kN_1n_1} \widehat{T}_{n_1,n_2} \in \Coo\left(M \right) 
\end{equation}
From \eqref{isospectral_wa_eqn} - \eqref{isospectral_a_eqn} it turns out
\be\label{isospectral_sum_eqn}
\widetilde{l}\left( \widetilde{a}\right)  = \sum_{\widetilde{\iota} \in \widetilde{I}} \sum_{\substack{j = 0\\ k = 0}}^{\substack{j = N_1-1\\ k = N_2-1}} \widetilde{l}\left( \widetilde{f}_{\left(\widetilde{\iota}, j, k \right)} \right) l\left( a'_{\left(\widetilde{\iota}, j, k \right)}\right) ; \text{ where } l\left( a'_{\left(\widetilde{\iota}, j, k \right)}\right) \in \Coo\left( M_\th \right).
\ee
However $C\left(\widetilde{M}_{\widetilde{\th}} \right)$ is the norm  completion of $\Coo\left(\widetilde{M}_\th \right)$, so from \eqref{isospectral_sum_eqn} it turns out that  $C\left(\widetilde{M}_{\widetilde{\th}} \right)$ is a right Hilbert $C\left({M}_{{\th}} \right)$-module generated by a finite set
\be\label{isospectral_xi_eqn}
\Xi=\left\{\widetilde{l}\left(\widetilde{f}_{\left(\widetilde{\iota}, j, k \right)} \right) \in   \Coo\left(\widetilde{M}_{\widetilde{\th}} \right)   \right\}_{\left(\widetilde{\iota}, j, k \right) \in \widetilde{J}}
\ee  From the Kasparov stabilization theorem it follows that the module is projective.
So one has the following theorem.
\begin{thm}\label{isospectral_fin_thm}\cite{ivankov:qnc}
	The triple $\left( C\left( M_\th\right), C\left( \widetilde{M}_{ \widetilde{\th}}\right), G\left(\widetilde{M}~|~ M \right)\right)   $ is an unital noncommutative finite-fold  covering.
\end{thm}
\subsubsection{Induced representation}
\paragraph*{}

From \eqref{isospectral_fprodl_eqn} it turns out
\be\label{nt_hilb_eqn}
\begin{split}
\left\langle \widetilde{l}\left(  \widetilde{f}_{\left(\widetilde{\iota}', j', k' \right)}\right) , \widetilde{l}\left(\widetilde{f}_{\left(\widetilde{\iota}'', j'', k'' \right)}\right)  \right\rangle_{C\left(\widetilde{ M}_{\widetilde{\th}}\right) } = \delta_{j'j''}\delta_{k'k''}\left\langle \widetilde{l}\left( \widetilde{e}_{\widetilde{\iota}'}\right) , \widetilde{l}\left( \widetilde{e}_{\widetilde{\iota}''}\right)  \right\rangle_{C\left( \widetilde{M}_{ \widetilde{\th}}\right) }= \\
=  \delta_{j'j''}\delta_{k'k''}\left\langle  \widetilde{e}_{\widetilde{\iota}'}, \widetilde{e}_{\widetilde{\iota}''}\right\rangle_{C\left( \widetilde{M}_{ \widetilde{\th}}\right)}={e}_{\widetilde{\iota}'} {e}_{\widetilde{\iota}''} \in  \Coo\left(M_\th \right).
\end{split}
\ee

Let $g \in G$ be any element, and let $\widetilde{l}\left(\widetilde{f}_{\left(\widetilde{\iota}, j, k \right)}\right)  \in \Xi$ where $\Xi$ is given by \eqref{isospectral_xi_eqn}. From \eqref{isospectral_f'_eqn}  and \eqref{isospectral_f_eqn} it turns out 
\be
\widetilde{l}\left( \widetilde{f}_{\left(\widetilde{\iota}, j, k \right)}\right)  = \widetilde{e}_{\widetilde{\iota}} \widetilde{l}\left( \widetilde{u}^j\widetilde{
	v}^k\right) = \widetilde{e}_{\widetilde{\iota}}  \widetilde{u}^j\widetilde{
	v}^k\widetilde{\la}^{k\widetilde{p}_1}.
\ee
If $g \in G$ be any element, then from the exact sequence \eqref{isospectral_gr_eqn} $\{e\} \to G' \to G \to G'' \to \{e\}$ it follows that there is $g''\in G''$ which is the image of $g$.  For any $\widetilde{\iota}\in \widetilde{I} = G''\times I$ there is $\widetilde{\iota}' \in \widetilde{I}$ such that $g''$ transforms $\widetilde{\iota}$ to  $\widetilde{\iota}~'$. If $\widetilde{\mathcal V}_{\widetilde{\iota}} \in \widetilde{M}/\widetilde{\T}^2$ is given by \eqref{isospectral_vgi_eqn}, and $\widetilde{\mathcal U}_{\widetilde{\iota}}\in \widetilde{M}$ is given by \eqref{isospectral_wgi_eqn} then there is $g' \in G'\cong \Z_{N_1} \times \Z_{N_2}$ such that
\be\nonumber
\begin{split}
g'' \widetilde{\mathcal V}_{\widetilde{\iota}}=\widetilde{\mathcal V}_{\widetilde{\iota}'},\\
g \widetilde{\mathcal U}_{\widetilde{\iota}}= g' \widetilde{\mathcal U}_{\widetilde{\iota}'}. 
\end{split}
\ee
If $g'$ corresponds  to $\left(\overline{k}_1, \overline{k}_2 \right) \in \Z_{N_1} \times \Z_{N_2}$ then from $g''\widetilde{e}_{\widetilde{\iota}} =\widetilde{e}_{\widetilde{\iota}'}$ and \eqref{isospectral_uv_eqn} it turns out
\be\label{isospectral_trans_eqn}
\begin{split}
	g \widetilde{f}_{\left(\widetilde{\iota}, j, k \right)} = g\left( \widetilde{e}_{\widetilde{\iota}}  \widetilde{u}^j\widetilde{
	v}^k\right)= \left( g''\widetilde{e}_{\widetilde{\iota}}\right) \left( g'\left( \widetilde{u}^j\widetilde{
	v}^k\right)\right) =\\= \widetilde{e}_{\widetilde{\iota}'} \left( e^{\frac{2\pi i jk_1}{N_1}}e^{\frac{2\pi i kk_2}{N_1}}\widetilde{u}^j\widetilde{
	v}^k\right)=e^{\frac{2\pi i jk_1}{N_1}}e^{\frac{2\pi i kk_2}{N_1}}\widetilde{f}_{\left(\widetilde{\iota}', j, k \right)}. 
\end{split}
	\ee
Form above equation it follows that for any $\widetilde{\iota}_1, \widetilde{\iota}_2 \in \widetilde{I}$, $~g_1, g_2 \in G$, $~j',j'' = 0,\dots N_1-1$, $~k',k'' = 0,\dots N_2-1$ where are $\widetilde{\iota}'_1, \widetilde{\iota}'_2 \in \widetilde{I}$, $l_1, l_2 \in \Z$ such that
\be\label{isospectral_smooth_eqn}
\begin{split}
\left\langle g_1 \widetilde{f}_{\left(\widetilde{\iota}_1, j', k' \right)} , g_2 \widetilde{f}_{\left(\widetilde{\iota}_2, j'', k'' \right)}  \right\rangle_{C\left( \widetilde{M}\right) } = e^{\frac{2\pi i l_1}{N_1}}e^{\frac{2\pi i l_2}{N_2}} \delta_{j'j''}\delta_{k'k''}\left\langle \widetilde{e}_{\widetilde{\iota}_1'}~,~\widetilde{e}_{\widetilde{\iota}_2'} \right\rangle_{C\left( \widetilde{M}\right) }\in \Coo\left(M \right), \\
\left\langle g_1\left( \widetilde{l}\left(  \widetilde{f}_{\left(\widetilde{\iota}_1, j', k' \right)}\right) \right)  , g_2 \left( \widetilde{l}\left( \widetilde{f}_{\left(\widetilde{\iota}_2, j'', k'' \right)} \right) \right)  \right\rangle_{C\left( \widetilde{M}_{\widetilde{   \th}}\right) } = \\=e^{\frac{2\pi i l_1}{N_1}}e^{\frac{2\pi i l_2}{N_2}} \delta_{j'j''}\delta_{k'k''}\left\langle \widetilde{e}_{\widetilde{\iota}_1'}~,~\widetilde{e}_{\widetilde{\iota}_2'} \right\rangle_{C\left( \widetilde{M}_{\widetilde{   \th}}\right) }\in \Coo\left(M_\th \right). 
\end{split}
\ee
%From the above equation it turns out that for any $g_1, g_2 \in G$ there are $l_1, l_2 $\left(\widetilde{\iota}', j', k' \right)$, $\left(\widetilde{\iota}'', j'', k'' \right)$ there is $\eta \in \R$ such that

From \eqref{isospectral_smooth_eqn} it turns out that the finite set $\widetilde{\Xi}=G\Xi$ satisfies to the condition (a) of the Lemma \ref{smooth_matr_lem}, i.e.
\be\label{isospectral_smooth_fin_eqn}
\left\langle \widetilde{l}\left( \widetilde{f}_{\left(\widetilde{\iota}', j', k' \right)}\right),\widetilde{l}\left( \widetilde{f}_{\left(\widetilde{\iota}'', j'', k'' \right)}\right)  \right\rangle_{C\left( \widetilde{M}_{\widetilde{   \th}}\right) } =  \Coo\left(M_\th \right); ~~ \forall \widetilde{l}\left( \widetilde{f}_{\left(\widetilde{\iota}', j', k' \right)}\right),\widetilde{l}\left( \widetilde{f}_{\left(\widetilde{\iota}'', j'', k'' \right)}\right) \in \widetilde{\Xi}.
\ee

\begin{lemma}\label{isospectral_smooth_lem}
	For any $\widetilde{a} \in C\left(  \widetilde{M}_{ \widetilde{\th}}\right) $ following conditions are equivalent
	\begin{enumerate}
		\item[(a)] $\widetilde{a} \in  \Coo\left(  \widetilde{M}_{ \widetilde{\th}}\right)$,
		\item[(b)] $\left\langle \widetilde{l} \left( \widetilde{f}_{\left(\widetilde{\iota}', j', k' \right)}\right) , \widetilde{a}\widetilde{l}\left( \widetilde{f}_{\left(\widetilde{\iota}'', j'', k'' \right)}\right) \right\rangle_{C\left(  \widetilde{M}_{ \widetilde{\th}}\right)} \in \Coo\left(M_\th \right); ~~\forall \widetilde{l} \left( \widetilde{f}_{\left(\widetilde{\iota}', j', k' \right)}\right) , \widetilde{l} \left( \widetilde{f}_{\left(\widetilde{\iota}'', j'', k'' \right)}\right)  \in \Xi$. 
	\end{enumerate}
\end{lemma} 
\begin{proof} (a)$\Rightarrow$(b) For any $g \in G$ following condition holds  $g\widetilde{l}\left( \widetilde{f}_{\left(\widetilde{\iota}', j', k' \right)}\right) , g\widetilde{l}\left( \widetilde{f}_{\left(\widetilde{\iota}'', j'', k'' \right)}\right) \in \Coo\left(  \widetilde{M}_{ \widetilde{\th}}\right)$, hence
	$$ 
	\left\langle \widetilde{l}\left( \widetilde{f}_{\left(\widetilde{\iota}', j', k' \right)}\right) , \widetilde{a}\widetilde{l}\left( \widetilde{f}_{\left(\widetilde{\iota}'', j'', k'' \right)}\right) \right\rangle_{C\left(  \widetilde{M}_{ \widetilde{\th}}\right)}= \sum_{g\in G}g\left(\widetilde{l}\left( \widetilde{f}^*_{\left(\widetilde{\iota}', j', k' \right)}\right) , \widetilde{a}\widetilde{l}\left( \widetilde{f}_{\left(\widetilde{\iota}'', j'', k'' \right)} \right)\right) \in \Coo\left(  \widetilde{M}_{ \widetilde{\th}}\right). 
	$$
	Since $\left\langle \widetilde{l}\left( \widetilde{f}_{\left(\widetilde{\iota}', j', k' \right)}\right),\widetilde{a}\widetilde{l}\left( \widetilde{f}_{\left(\widetilde{\iota}'', j'', k'' \right)}\right)\right\rangle_{C\left(  \widetilde{M}_{ \widetilde{\th}}\right)}$ is $G$-invariant one has $$\left\langle \widetilde{l}\left( \widetilde{f}_{\left(\widetilde{\iota}', j', k' \right)}\right) ,\widetilde{a}\widetilde{l}\left(  \widetilde{f}_{\left(\widetilde{\iota}'', j'', k'' \right)}\right) \right\rangle_{C\left(  \widetilde{M}_{ \widetilde{\th}}\right)}\in \Coo\left(M \right).$$\\
	(b)$\Rightarrow$(a)
	There is the following equivalence
	$$
	\widetilde{e}_{\widetilde{\iota}} \widetilde{a} 	\widetilde{e}_{\widetilde{\iota}} \in \Coo\left(  \widetilde{M}_{ \widetilde{\th}}\right) \Leftrightarrow  \left\langle 	\widetilde{e}_{\widetilde{\iota}}, \widetilde{a} 	\widetilde{e}_{\widetilde{\iota}} \right\rangle_{C\left(  \widetilde{M}_{ \widetilde{\th}}\right)} \in \Coo\left({M}_{ {\th}}\right).
	$$
	Taking into account $\widetilde{e}_{\widetilde{\iota}} = \widetilde{l}\left(\widetilde{f}_{\left(\widetilde{\iota}, 0, 0 \right)} \right) $ one has a following logical equation
\bean
\forall\widetilde{\iota} \in \widetilde{I}~~	\left\langle 	\widetilde{l}\left(\widetilde{f}_{\left(\widetilde{\iota}, 0, 0 \right)}\right) , \widetilde{a} 	\widetilde{l}\left(\widetilde{f}_{\left(\widetilde{\iota}, 0, 0 \right)}\right)  \right\rangle_{C\left(  \widetilde{M}_{ \widetilde{\th}}\right)} \in \Coo\left({M}_{ {\th}}\right) \Leftrightarrow \\ \Leftrightarrow \widetilde{l}\left(\widetilde{f}_{\left(\widetilde{\iota}, 0, 0 \right)}\right)  \widetilde{a} 	\widetilde{l}\left(\widetilde{f}_{\left(\widetilde{\iota}, 0, 0 \right)}\right) \in \Coo\left(  \widetilde{M}_{ \widetilde{\th}}\right)\Rightarrow  \\ \Rightarrow
\widetilde{a} = \sum_{\widetilde{\iota} \in \widetilde{I}} 	\widetilde{e}_{\widetilde{\iota}}	\widetilde{a}	\widetilde{e}_{\widetilde{\iota}} = \sum_{\widetilde{\iota} \in \widetilde{I}} 	\widetilde{l}\left(\widetilde{f}_{\left(\widetilde{\iota}, 0, 0 \right)}\right) 	\widetilde{a}\widetilde{l}\left(\widetilde{f}_{\left(\widetilde{\iota}, 0, 0 \right)}\right) \in \Coo\left(  \widetilde{M}_{ \widetilde{\th}}\right).
\eean

 \end{proof}
\begin{corollary}\label{isospectral_cor}
	Following conditions hold:
	\begin{itemize}
		\item $C\left(\widetilde{M}_{\widetilde{\th}} \right) \bigcap \mathbb{M}_n\left(\Coo\left(M_\th\right) \right) = \Coo\left(\widetilde{M}_{\widetilde{\th}} \right)$,
		\item The unital noncommutative finite-fold covering  $$\left( C\left( M_\th\right), C\left( \widetilde{M}_{ \widetilde{\th}}\right), G\left(\widetilde{M}~|~ M \right)\right)$$ is {smoothly invariant}.
	\end{itemize}
\end{corollary}
\begin{proof}
	Follows from \eqref{isospectral_smooth_fin_eqn}, and Lemmas \ref{smooth_matr_lem}, \ref{isospectral_smooth_lem}.
\end{proof}

From \eqref{isospectral_undec_eqn} for any $\widetilde{a} \in C\left( \widetilde{M}_{ \widetilde{\th}}\right)$ following condition holds
\be\label{isospectral_d_eqn}
\widetilde{a} = \frac{1}{N_1N_2}\sum_{\widetilde{\iota} \in \widetilde{I}} \sum_{\substack{j = 0\\ k = 0}}^{\substack{j = N_1-1\\ k = N_2-1}} \widetilde{l}\left( \widetilde{f}_{\left(\widetilde{\iota}, j, k \right)}\right) \left\langle \widetilde{a},\widetilde{l} \left( \widetilde{f}_{\left(\widetilde{\iota}, j, k \right)}\right) \right\rangle_{C\left( \widetilde{M}_{ \widetilde{\th}}\right)}= \sum_{\widetilde{\iota} \in \widetilde{I}} \sum_{\substack{j = 0\\ k = 0}}^{\substack{j = N_1-1\\ k = N_2-1}} \widetilde{l}\left( \widetilde{f}_{\left(\widetilde{\iota}, j, k \right)}\right)a_{\left(\widetilde{\iota}, j, k \right)}
\ee
where $a_{\left(\widetilde{\iota}, j, k \right)}=\frac{1}{N_1N_2}\left\langle \widetilde{a},\widetilde{l} \left( \widetilde{f}_{\left(\widetilde{\iota}, j, k \right)}\right) \right\rangle_{C\left( \widetilde{M}_{ \widetilde{\th}}\right)}\in C\left( {M}_{ {\th}}\right)$. If $\xi \in L^2\left(M, S \right)$ then
\be\label{isospectral_decomp_eqn}
\begin{split}
\widetilde{a} \otimes \xi = \sum_{\widetilde{\iota} \in \widetilde{I}}~~ \sum_{\substack{j = 0\\ k = 0}}^{\substack{j = N_1-1\\ k = N_2-1}} \widetilde{l}\left( \widetilde{f}_{\left(\widetilde{\iota}, j, k \right)}\right)a_{\left(\widetilde{\iota}, j, k \right)}\otimes \xi = \sum_{\widetilde{\iota} \in \widetilde{I}}~~ \sum_{\substack{j = 0\\ k = 0}}^{\substack{j = N_1-1\\ k = N_2-1}} \widetilde{l}\left( \widetilde{f}_{\left(\widetilde{\iota}, j, k \right)}\right)\otimes a_{\left(\widetilde{\iota}, j, k \right)} \xi=\\
= \sum_{\widetilde{\iota} \in \widetilde{I}} ~~\sum_{\substack{j = 0\\ k = 0}}^{\substack{j = N_1-1\\ k = N_2-1}} \widetilde{l}\left( \widetilde{f}_{\left(\widetilde{\iota}, j, k \right)}\right)\otimes \xi_{\left(\widetilde{\iota}, j, k \right)} \in C\left( \widetilde{M}_{ \widetilde{\th}}\right) \otimes_{C\left( {M}_{ {\th}}\right)} L^2\left(M, S \right),\\ \text{ where } \xi_{\left(\widetilde{\iota}, j, k \right)} = a_{\left(\widetilde{\iota}, j, k \right)}\xi \in   L^2\left(M, S \right).
\end{split}
\ee
Denote by $ \widetilde{\H}=C\left( \widetilde{M}_{ \widetilde{\th}}\right) \otimes_{C\left( {M}_{ {\th}}\right)} L^2\left(M, S \right)$ and let $\left(\cdot, \cdot \right)_{\widetilde{\H}}$ is the given by \eqref{induced_prod_equ} Hilbert product. If $\xi, \eta \in L^2\left(M, S \right)$ then from \eqref{nt_hilb_eqn} it turns out
\be\label{isospectral_hilb_p_eqn}
\left(\widetilde{l} \left( \widetilde{f}_{\left(\widetilde{\iota}', j', k' \right)}\right)  \otimes \xi, ~\widetilde{l} \left( \widetilde{f}_{\left(\widetilde{\iota}'', j'', k'' \right)}\right)  \otimes \eta \right)_{\widetilde{\H}}= N_1N_2\delta_{j'j''}\delta_{k'k''}\left(\xi,{e}_{\widetilde{\iota}'} {e}_{\widetilde{\iota}''} \eta  \right)_{ L^2\left(M, S \right)}.
\ee
From \eqref{isospectral_hilb_p_eqn} it turns out the orthogonal decomposition
\bean
\widetilde{\H} = \bigoplus_{\substack{j = 0\\ k = 0}}^{\substack{j = N_1-1\\ k = N_2-1}}\widetilde{\H}_{jk}~, ~~ \\ \text{ where } \widetilde{\H}_{jk}=\left\{\widetilde{\xi} \in \widetilde{\H}~|~\widetilde{\xi}= \sum_{\widetilde{\iota} \in \widetilde{I}}\widetilde{l}\left( \widetilde{f}_{\left(\widetilde{\iota}, j, k \right)}\right)\otimes \xi_{\left(\widetilde{\iota}, j, k \right)}\in C\left( \widetilde{M}_{ \widetilde{\th}}\right) \otimes_{C\left( {M}_{ {\th}}\right)} L^2\left(M, S \right) \right\}.
\eean
From \eqref{isospectral_hilb_p_eqn} it turns out than for any $0\le j',j''<N_1$ and $0\le k',k''<N_2$ there is an isomorphism of Hilbert spaces given by
\be\label{isospectral_hilb_iso_eqn}
\begin{split}
\Phi^{j'k'}_{j''k''}:\widetilde{\H}_{j'k'}\xrightarrow{\approx}\widetilde{\H}_{j''k''},\\
\widetilde{l} \left( \widetilde{f}_{\left(\widetilde{\iota}, j', k' \right)}\right)  \otimes \xi \mapsto \widetilde{l} \left( \widetilde{f}_{\left(\widetilde{\iota}, j'', k'' \right)}\right)  \otimes \xi.
\end{split}
\ee
Similarly if  $\widetilde{\H}^{\mathrm{comm}}=C\left( \widetilde{M}\right) \otimes_{C\left( {M}\right)} L^2\left(M, S \right)$ then there is the decomposition
\bean
\widetilde{\H}^{\mathrm{comm}} = \bigoplus_{\substack{j = 0\\ k = 0}}^{\substack{j = N_1-1\\ k = N_2-1}}\widetilde{\H}^{\mathrm{comm}}_{jk}~, ~~ \\ \text{ where } \widetilde{\H}^{\mathrm{comm}}_{jk}=\left\{\widetilde{\xi} \in \widetilde{\H}^{\mathrm{comm}}~|~\widetilde{\xi}= \sum_{\widetilde{\iota} \in \widetilde{I}} \widetilde{f}_{\left(\widetilde{\iota}, j, k \right)}\otimes \xi_{\left(\widetilde{\iota}, j, k \right)}\in C\left( \widetilde{M}\right) \otimes_{C\left( {M}\right)} L^2\left(M, S \right) \right\}.
\eean
From the Lemma \eqref{comm_ind_lem} it turns out $\widetilde{\H}^{\mathrm{comm}} = L^2\left( \widetilde{M},\widetilde{S}\right)$ the induced representation is given by the natural action of $C\left( {M}\right)$ on $L^2\left( \widetilde{M},\widetilde{S}\right)$. 
Similarly to \eqref{isospectral_hilb_iso_eqn} for any $0\le j',j''<N_1$ and $0\le k',k''<N_2$ there is an isomorphism of Hilbert spaces given by
\be\label{isospectral_hilb_iso_comm_eqn}
\begin{split}
	\Psi^{j'k'}_{j''k''}:\widetilde{\H}^{\mathrm{comm}}_{j'k'}\xrightarrow{\approx}\widetilde{\H}^{\mathrm{comm}}_{j''k''},\\
 \widetilde{f}_{\left(\widetilde{\iota}, j', k' \right)}  \otimes \xi \mapsto   \widetilde{f}_{\left(\widetilde{\iota}, j'', k'' \right)}  \otimes \xi.
\end{split}
\ee
From $\widetilde{l}\left( \widetilde{f}_{\left(\widetilde{\iota}, 0, 0 \right)}\right)= \widetilde{f}_{\left(\widetilde{\iota}, 0, 0 \right)}=\widetilde{e}_{\widetilde{\iota}}$ it turns out the natural ismomorphism  
\be\label{isospectral_nat_eqn}
\widetilde{\H}^{\mathrm{comm}}_{0,0} \cong \widetilde{\H}_{0,0}
\ee
We would like to define the isomorphism $\varphi:\widetilde{\H} \approx C\left( \widetilde{M}_{ \widetilde{\th}}\right) \otimes_{C\left( {M}_{ {\th}}\right)} L^2\left(M, S \right) \xrightarrow{\approx}\widetilde{\H}^{\mathrm{comm}}$ such that for any $\widetilde{a} \in C\left(\widetilde{M}_{\widetilde{\th}} \right)$, and $\xi \in L^2\left(M,S \right)$ following condition holds 
\be\label{isospectral_varphi_eqn}
\varphi\left( \widetilde{a}\otimes \xi \right)= \widetilde{a}\left(1_{C\left( M\right) } \otimes \xi \right) 
\ee
where the right part of the above equation assumes the action of $C\left(\widetilde{M}_{\widetilde{\th}} \right)$ on $L^2\left( \widetilde{M},\widetilde{S}\right)$ by operators  \eqref{l_defn}. From the decomposition \eqref{isospectral_d_eqn} it turns the equation  \eqref{isospectral_varphi_eqn} is true if and only if it is true for any $\widetilde{a} = \widetilde{l}\left( \widetilde{f}_{\left(\widetilde{\iota}, j, k \right)}\right) a$  where $a \in C\left( {M}_{ \th}\right)$, i.e.
$$
\varphi\left( \widetilde{l}\left( \widetilde{f}_{\left(\widetilde{\iota}, j, k \right)}\right) a\otimes \xi \right)= \widetilde{l}\left( \widetilde{f}_{\left(\widetilde{\iota}, j, k \right)}\right) a\left(1_{C\left( M\right) } \otimes \xi \right),
$$
or equivalently
\be\label{isospectral_varphi_p_eqn}
\varphi\left( \widetilde{l}\left( \widetilde{f}_{\left(\widetilde{\iota}, j, k \right)}\right) \otimes a\xi \right)= \widetilde{l}\left( \widetilde{f}_{\left(\widetilde{\iota}, j, k \right)}\right) \left(1_{C\left( M\right) } \otimes a\xi \right)
\ee
From \eqref{isospectral_hilb_iso_comm_eqn} the right part of \eqref{isospectral_varphi_p_eqn} is given by
$$
\widetilde{l}\left( \widetilde{f}_{\left(\widetilde{\iota}, j, k \right)}\right) \left(1_{C\left( M\right) } \otimes a\xi \right)= \Psi^{0,0}_{jk}\left(\widetilde{\la}^{kp_1} \left(\widetilde{f}_{\left(\widetilde{\iota}, 0, 0\right)} \otimes a\xi \right) \right) 
$$
From the above equation it turns out that if $\eta \in \widetilde{\H}_{j,k}$ then
\be\label{isospectral_jk_p_eqn}
\varphi\left( \eta\right)=\varphi|_{\widetilde{\H}_{j,k}}\left( \eta\right)=  \Psi^{0,0}_{jk}\left( \widetilde{\la}^{kp_1}\Phi^{jk}_{0,0} \left(\eta \right) \right) 
\ee
and 
\bean
\varphi = \bigoplus_{\substack{j = 0\\ k = 0}}^{\substack{j = N_1-1\\ k = N_2-1}} \varphi_{\widetilde{\H}_{j,k}}:\widetilde{\H}\xrightarrow{\approx} \widetilde{\H}^{\mathrm{comm}}
\eean
then from above construction it turns out that $\varphi$ satisfies to \eqref{isospectral_varphi_eqn}. In result we have the following lemma.
\begin{lemma}
	If $\widetilde{\rho}:C\left( \widetilde{M}_{ \widetilde{\th}}\right)\to B\left(\widetilde{\H} \right) $ is induced by $\left(\rho, \left( C\left( M_\th\right), C\left( \widetilde{M}_{ \widetilde{\th}}\right), G\left(\widetilde{M}, M \right)\right)\right)$ then $\widetilde{\rho}$ can be represented by action of $C\left( \widetilde{M}_{ \widetilde{\th}}\right)$ on $ L^2\left( \widetilde{M},\widetilde{S}\right)$ by operators \eqref{l_defn}.
\end{lemma}
\subsubsection{Lift of the Dirac operator}
\paragraph*{}
If ${f}_{\left({\iota}, j, k \right)} \in \Coo \left(M \right)$ is given by \eqref{isospectral_fs_eqn} then from then from \eqref{isospectral_f's_eqn} it turns out 
\be\nonumber
{f}'_{\left({\iota}, j, k \right)} = {e}_{{\iota}} {u}^j{
	v}^k \in C_0\left({\mathcal{U}}_{\iota} \right)= C_0\left({\mathcal{V}}_{\iota} \times \T^2 \right)
\ee
In the above formula the product ${u}^j{
	v}^k$ can be regarded as element of both $C\left( \T^2\right)$ and  $C_b\left({\mathcal{V}}_{\iota} \times \T^2 \right) =  C_b\left({\mathcal{W}}_{\iota} \right)$, where $\mathcal{W}_{\iota}\subset M$ is the homeomorphic image of ${\mathcal{V}}_{\iota} \times \T^2$. Since the Dirac operator $\slashed D$ is invariant with respect to transformations 
$
u \mapsto e^{i\varphi_u}u, ~~ v \mapsto e^{i\varphi_v}v
$
one has
\be\label{isospectral_jdb_eqn}
\left[\slashed D, u \right]= d^{\iota}_uu,~~ \left[\slashed D, v \right]= d^{\iota}_vv
\ee
where $d^\iota_u, d^\iota_v: \mathcal V_\iota \to \mathbb{M}_{\dim S} \left(\C \right)$ are continuous matrix-valued functions. I would like to avoid functions in $C_b\left({\mathcal{U}}_{\iota} \right)$, so instead \eqref{isospectral_jdb_eqn} the following evident consequence of it will be used
\be\label{isospectral_jde_eqn}
\left[\slashed D,a u^jv^k \right]= \left[\slashed D,a \right] u^jv^k  + a\left( jd^\iota_u+kd^\iota_v\right) u^jv^k; ~~ a \in C_0\left( M/ \T^2\right), ~~ \supp a \subset \mathcal{V}_{\iota}.
\ee
In contrary to \eqref{isospectral_jdb_eqn} the equation \eqref{isospectral_jde_eqn} does not operate with $C_b\left({\mathcal{U}}_{\iota} \right)$, it operates with $C_0\left({\mathcal{U}}_{\iota} \right)$. Let $\pi: \widetilde M \to M$ and let $\widetilde{\slashed D}$ be the $\pi$-lift of $\slashed D$ (cf. Definition  \ref{inv_image_defn}). Suppose $\widetilde{  \mathcal V}_{\widetilde{\iota}} \subset \widetilde{M}/\widetilde{   \T}^2$ is mapped onto ${  \mathcal V}_{{\iota}} \subset {M}/{   \T}^2$. Then we set $d^{\widetilde{\iota}}_u \stackrel{\mathrm{def}}{=} d^{{\iota}}_u$, $d^{\widetilde{\iota}}_v \stackrel{\mathrm{def}}{=} d^{{\iota}}_v$. The covering $\pi$ maps $\widetilde{\mathcal{V}}_{ \widetilde\iota} \times \widetilde{\T}^2$ onto ${\mathcal{V}}_{\iota} \times \T^2$. If $\widetilde{u}, \widetilde{v} \in C\left(\widetilde{\T}^2 \right)$ are natural generators, then the covering $\widetilde{\T}^2 \to \T^2$ is given by
\bean
C\left(\T^2 \right) \to C\left(\widetilde{\T}^2 \right),\\
u \mapsto \widetilde{u}^{N_1},~~v \mapsto \widetilde{v}^{N_2}.
\eean 
From the above equation and taking into account \eqref{isospectral_jdb_eqn} one has
\be\label{isospectral_du_eqn}
\begin{split}
\left[\widetilde{\slashed D}, \widetilde{u} \right]= \frac{d^{\widetilde{\iota}}_u}{N_1}\widetilde{u},~~ \left[\widetilde{\slashed D}, \widetilde{v} \right]= \frac{d^{\widetilde{\iota}}_v}{N_2}\widetilde{v},\\
\left[\widetilde{\slashed D}, \widetilde{a} \widetilde{u}^j\widetilde{v}^k \right]=\left[\widetilde{\slashed D},\widetilde{a} \right] \widetilde{u}^j\widetilde{v}^k  + \widetilde{a} \left(\frac{j^{\widetilde{\iota}}_u}{N_1} +\frac{k^{\widetilde{\iota}}_v}{N_2}\right) \widetilde{u}^j\widetilde{v}^k;\\ ~~ \widetilde{a} \in C_0\left( \widetilde{M}/ \widetilde{\T}^2\right), ~~ \supp \widetilde{a} \subset \widetilde{\mathcal{V}}_{\widetilde{\iota}}.
\end{split}
\ee
For any $\widetilde{a}\in \Coo\left( \widetilde{M}\right) $ such that $\supp \widetilde{a} \subset  \widetilde{\mathcal{W}}_{\widetilde{\iota}}$ following condition holds
\be\label{isospectral_da_eqn}
\left[\widetilde{\slashed D}, \widetilde{a} \right]= \mathfrak{lift}_{\widetilde{\mathcal{W}}_{\widetilde{\iota}}}\left(\left[{\slashed D}, \mathfrak{desc}\left( \widetilde{a}\right)  \right] \right) 
\ee
If $b \in \Coo\left(M/\T^2 \right)\subset \Coo\left(M \right)$ then from \eqref{isospectral_du_eqn} and \eqref{isospectral_da_eqn} it turns out
\be\nonumber
\begin{split}
\left[\widetilde{\slashed D},  \widetilde{f}_{\left(\widetilde{\iota}, j, k \right)}b u^{j'}v^{k'}\right]   = \sqrt{\widetilde{e}_{\widetilde{\iota}}}\widetilde{u}^j\widetilde{v}^k  u^{j'}v^{k'}\otimes \left[ \slashed D, \sqrt{{e}_{\widetilde{\iota}}} b  \right]+ \\ +  b\sqrt{\widetilde{e}_{\widetilde{\iota}}}\widetilde{u}^j\widetilde{v}^k  u^{j'}v^{k'}\otimes \left[ \slashed D, \sqrt{{e}_{\widetilde{\iota}}}\right]+\\+b
	\sqrt{\widetilde{e}_{\widetilde{\iota}}}\widetilde{u}^j\widetilde{v}^k  u^{j'}v^{k'}\otimes  \sqrt{{e}_{\widetilde{\iota}}}\left(j'd_u+\frac{ jd_u}{N_1} +k'd_v + \frac{kd_v}{N_2}\right)\\
\end{split}
\ee 
and taking into account $\left[ \widetilde{\slashed D}, \widetilde{l}\right]=0$ one has
\be\label{isospectral_predu_eqn}
\begin{split}
	\left[\widetilde{\slashed D}, \widetilde{l}\left(  \widetilde{f}_{\left(\widetilde{\iota}, j, k \right)}b u^{j'}v^{k'}\right) \right]   = \sqrt{\widetilde{e}_{\widetilde{\iota}}}\widetilde{l}\left(\widetilde{u}^j\widetilde{v}^k  u^{j'}v^{k'}\right)\otimes \left[ \slashed D, \sqrt{{e}_{\widetilde{\iota}}} b  \right]+ \\ +  b\sqrt{\widetilde{e}_{\widetilde{\iota}}}\widetilde{l}\left(\widetilde{u}^j\widetilde{v}^k  u^{j'}v^{k'}\right)\otimes \left[ \slashed D, \sqrt{{e}_{\widetilde{\iota}}}\right]+\\+b
	\sqrt{\widetilde{e}_{\widetilde{\iota}}}\widetilde{l}\left(\widetilde{u}^j\widetilde{v}^k  u^{j'}v^{k'}\right)\otimes  \sqrt{{e}_{\widetilde{\iota}}}\left(j'd_u+\frac{ jd_u}{N_1} +k'd_v + \frac{kd_v}{N_2}\right)\\
\end{split}
\ee
Taking into account that  
$
\widetilde{l}\left(\widetilde{u}^j\widetilde{v}^k \right)\widetilde{l}\left( u^{j'}v^{k'}\right)= \widetilde{\la}^{j'N_2k} \widetilde{l}\left(\widetilde{u}^j\widetilde{v}^k  u^{j'}v^{k'}\right)
$ the equation \eqref{isospectral_predu_eqn} is equivalent to

\be\label{isospectral_pred_eqn}
\begin{split}
	\left[\widetilde{\slashed D},  \widetilde{l}\left( \widetilde{f}_{\left(\widetilde{\iota}, j, k \right)}\right) b \widetilde{l}\left( u^{j'}v^{k'}\right) \right]   = \sqrt{\widetilde{e}_{\widetilde{\iota}}}\widetilde{l}\left(\widetilde{u}^j\widetilde{v}^k \right)\widetilde{l}\left( u^{j'}v^{k'}\right)\otimes \left[ \slashed D, \sqrt{{e}_{\widetilde{\iota}}} b  \right]+ \\ +  b\sqrt{\widetilde{e}_{\widetilde{\iota}}}\widetilde{l}\left(\widetilde{u}^j\widetilde{v}^k \right) \widetilde{l}\left( u^{j'}v^{k'}\right)\otimes \left[ \slashed D, \sqrt{{e}_{\widetilde{\iota}}}\right]+\\+b
	\sqrt{\widetilde{e}_{\widetilde{\iota}}}\widetilde{l}\left(\widetilde{u}^j\widetilde{v}^k \right) \widetilde{l}\left( u^{j'}v^{k'}\right)\otimes  \sqrt{{e}_{\widetilde{\iota}}}\left(j'd_u+\frac{ jd_u}{N_1} +k'd_v + \frac{kd_v}{N_2}\right)\\
\end{split}
\ee

For any $\widetilde{a} \in C\left( \widetilde{M}_{ \widetilde{\th}}\right)$ there is the decomposition given by \eqref{isospectral_d_eqn}, i.e.
\be\nonumber 
\widetilde{a} = \sum_{\widetilde{\iota} \in \widetilde{I}} \sum_{\substack{j = 0\\ k = 0}}^{\substack{j = N_1-1\\ k = N_2-1}} \widetilde{l}\left( \widetilde{f}_{\left(\widetilde{\iota}, j, k \right)}\right)a_{\left(\widetilde{\iota}, j, k \right)}
\ee
Let $\Om^1_{\slashed D}$ be the {module of differential forms associated} with the spectral triple  $\left(\Coo\left(M_\th \right) , L^2\left(M,S \right) , \slashed D\right)$ (cf. Definition \ref{ass_cycle_defn}). Let us define a $\C$-linear map
\be\label{isospectral_conn_eqn}
\begin{split}
\nabla: \Coo\left(\widetilde{M}_{\widetilde{\th}} \right) \to \Coo\left(\widetilde{M}_{\widetilde{\th}} \right) \otimes_{\Coo\left(M_\th \right)}\Om^1_{\slashed D}~,\\
\sum_{\widetilde{\iota} \in \widetilde{I}}~~  \sum_{\substack{j = 0\\ k = 0}}^{\substack{j = N_1-1\\ k = N_2-1}} \widetilde{l}\left( \widetilde{f}_{\left(\widetilde{\iota}, j, k \right)}\right)a_{\left(\widetilde{\iota}, j, k \right)}\mapsto \sum_{\widetilde{\iota} \in \widetilde{I}}~~  \sum_{\substack{j = 0\\ k = 0}}^{\substack{j = N_1-1\\ k = N_2-1}} \sqrt{\widetilde{e}_{\widetilde{\iota}}}\widetilde{l}\left(\widetilde{u}^j\widetilde{v}^k \right)\otimes \left[ \slashed D, \sqrt{{e}_{\widetilde{\iota}}} a_{\left(\widetilde{\iota}, j, k \right)}\right]+ \\ + \sum_{\widetilde{\iota} \in \widetilde{I}}~~  \sum_{\substack{j = 0\\ k = 0}}^{\substack{j = N_1-1\\ k = N_2-1}} \sqrt{\widetilde{e}_{\widetilde{\iota}}}\widetilde{l}\left(\widetilde{u}^j\widetilde{v}^k \right) a_{\left(\widetilde{\iota}, j, k \right)}\otimes \left[ \slashed D, \sqrt{{e}_{\widetilde{\iota}}}\right]+\\+
 \sum_{\widetilde{\iota} \in \widetilde{I}}~~  \sum_{\substack{j = 0\\ k = 0}}^{\substack{j = N_1-1\\ k = N_2-1}}\sqrt{\widetilde{e}_{\widetilde{\iota}}}\widetilde{l}\left(\widetilde{u}^j\widetilde{v}^k \right) a_{\left(\widetilde{\iota}, j, k \right)}\otimes  \sqrt{{e}_{\widetilde{\iota}}}\left(\frac{jd_u}{N_1} +\frac{kd_v}{N_2}\right). 
\end{split}
\ee
For any $a \in \Coo\left( M_\th\right)$ following condition holds 
\be\label{isospectral_conn_p_eqn}
\begin{split}
\nabla\left(  \widetilde{l}\left( \widetilde{f}_{\left(\widetilde{\iota}, j, k \right)}\right)a_{\left(\widetilde{\iota}, j, k \right)}a\right) =  \sqrt{\widetilde{e}_{\widetilde{\iota}}}\widetilde{l}\left(\widetilde{u}^j\widetilde{v}^k \right)\otimes \left[ \slashed D, \sqrt{{e}_{\widetilde{\iota}}} a_{\left(\widetilde{\iota}, j, k \right)}a\right]+ \\+ \sqrt{\widetilde{e}_{\widetilde{\iota}}}\widetilde{l}\left(\widetilde{u}^j\widetilde{v}^k \right) a_{\left(\widetilde{\iota}, j, k \right)}a\otimes \left[ \slashed D, \sqrt{{e}_{\widetilde{\iota}}}\right]+
\sqrt{\widetilde{e}_{\widetilde{\iota}}}\widetilde{l}\left(\widetilde{u}^j\widetilde{v}^k \right) a_{\left(\widetilde{\iota}, j, k \right)}a\otimes  \sqrt{{e}_{\widetilde{\iota}}}\left(\frac{jd_u}{N_1} +\frac{kd_v}{N_2}\right)=\\=
 \sqrt{\widetilde{e}_{\widetilde{\iota}}}\widetilde{l}\left(\widetilde{u}^j\widetilde{v}^k \right)\otimes \left[ \slashed D, \sqrt{{e}_{\widetilde{\iota}}} a_{\left(\widetilde{\iota}, j, k \right)}\right]a+ {\widetilde{e}_{\widetilde{\iota}}}\widetilde{l}\left(\widetilde{u}^j\widetilde{v}^k \right)  a_{\left(\widetilde{\iota}, j, k \right)}\otimes \left[ \slashed D,a\right]+ \\+ \sqrt{\widetilde{e}_{\widetilde{\iota}}}\widetilde{l}\left(\widetilde{u}^j\widetilde{v}^k \right) a_{\left(\widetilde{\iota}, j, k \right)}a\otimes \left[ \slashed D, \sqrt{{e}_{\widetilde{\iota}}}\right]+
\sqrt{\widetilde{e}_{\widetilde{\iota}}}\widetilde{l}\left(\widetilde{u}^j\widetilde{v}^k \right) a_{\left(\widetilde{\iota}, j, k \right)}a\otimes  \sqrt{{e}_{\widetilde{\iota}}}\left(\frac{jd_u}{N_1} +\frac{kd_v}{N_2}\right)=\\
=\nabla\left(  \widetilde{l}\left( \widetilde{f}_{\left(\widetilde{\iota}, j, k \right)}\right)a_{\left(\widetilde{\iota}, j, k \right)}\right)a +\widetilde{l}\left( \widetilde{f}_{\left(\widetilde{\iota}, j, k \right)}\right)a_{\left(\widetilde{\iota}, j, k \right)}\left[ \slashed D,a\right]. 
\end{split}
\ee
From \eqref{isospectral_conn_eqn}, \eqref{isospectral_conn_p_eqn} and taking into account \eqref{conn_triple_eqn} one concludes that $\nabla$ is a connection (cf. Definition \ref{conn_triple_eqn}). If $\left(\overline{l}_1, \overline{l}_2 \right) \in \Z_{N_1}\times \Z_{N_2}$ and $\al=e^{\frac{2\pi i l_1j}{N_1}}e^{\frac{2\pi i l_2k}{N_2}}\in \C$ then following condition holds
\bean
\nabla\left(\left(\overline{l}_1, \overline{l}_2 \right) \left(   \widetilde{l}\left( \widetilde{f}_{\left(\widetilde{\iota}, j, k \right)}\right)a_{\left(\widetilde{\iota}, j, k \right)}\right)\right)=  \sqrt{\widetilde{e}_{\widetilde{\iota}}}\al\widetilde{l}\left(\widetilde{u}^j\widetilde{v}^k \right)\otimes \left[ \slashed D, \sqrt{{e}_{\widetilde{\iota}}} a_{\left(\widetilde{\iota}, j, k \right)}\right]+ \\+ \sqrt{\widetilde{e}_{\widetilde{\iota}}}\al\widetilde{l}\left(\widetilde{u}^j\widetilde{v}^k \right) a_{\left(\widetilde{\iota}, j, k \right)}\otimes \left[ \slashed D, \sqrt{{e}_{\widetilde{\iota}}}\right]+
\sqrt{\widetilde{e}_{\widetilde{\iota}}}\al\widetilde{l}\left(\widetilde{u}^j\widetilde{v}^k \right) a_{\left(\widetilde{\iota}, j, k \right)}\otimes  \sqrt{{e}_{\widetilde{\iota}}}\left(\frac{jd_u}{N_1} +\frac{kd_v}{N_2}\right)  =\\=
\al\nabla\left(   \widetilde{l}\left( \widetilde{f}_{\left(\widetilde{\iota}, j, k \right)}\right)a_{\left(\widetilde{\iota}, j, k \right)}\right)=\left(\overline{l}_1, \overline{l}_2 \right) \nabla\left(   \widetilde{l}\left( \widetilde{f}_{\left(\widetilde{\iota}, j, k \right)}\right)a_{\left(\widetilde{\iota}, j, k \right)}\right),
\eean 
i.e. the connection $\nabla$ is $\Z_{N_1}\times \Z_{N_2}$ equivariant (cf. \eqref{equiv_conn_eqn}).
If $a_{\left( \widetilde{\iota}, j',k'\right) } \in \Coo\left(M_\th \right)_{j',k'}$ is an element of bidegree $\left(j',k' \right)$ then there is $b\in \Coo\left(M_\th \right)_{0,0}$ such that
\be\label{isospectral_bas_eqn}
\widetilde{e}_{\widetilde{\iota}}a_{\left( \widetilde{\iota}, j',k'\right) }= \widetilde{e}_{\widetilde{\iota}}b u^{j'}v^{k'}.
\ee
From \eqref{isospectral_conn_p_eqn} it it follows that
\bean
\nabla\left(\widetilde{l}\left(  \widetilde{f}_{\left(\widetilde{\iota}, j, k \right)}a_{\left( \widetilde{\iota}, j',k'\right) }\right)\right)   = \sqrt{\widetilde{e}_{\widetilde{\iota}}}\widetilde{l}\left(\widetilde{u}^j\widetilde{v}^k \right)\otimes \left[ \slashed D, \sqrt{{e}_{\widetilde{\iota}}} b \widetilde{l}\left( u^{j'}v^{k'}\right) \right]+ \\ +  b\sqrt{\widetilde{e}_{\widetilde{\iota}}}\widetilde{l}\left(\widetilde{u}^j\widetilde{v}^k \right) \widetilde{l}\left( u^{j'}v^{k'}\right)\otimes \left[ \slashed D, \sqrt{{e}_{\widetilde{\iota}}}\right]+\\+b
\sqrt{\widetilde{e}_{\widetilde{\iota}}}\widetilde{l}\left(\widetilde{u}^j\widetilde{v}^k \right) \widetilde{l}\left( u^{j'}v^{k'}\right)\otimes  \sqrt{{e}_{\widetilde{\iota}}}\left(\frac{jd_u}{N_1} +\frac{kd_v}{N_2}\right),\\
%\nabla\left( \left(   \widetilde{l}\left( \widetilde{f}_{\left(\widetilde{\iota}, j, k \right)}\right)a_{\left(\widetilde{\iota}, j, k \right)}\right)\right)=  \sqrt{\widetilde{e}_{\widetilde{\iota}}}\widetilde{l}\left(\widetilde{u}^j\widetilde{v}^k \right)\otimes \left[ \slashed D, \sqrt{{e}_{\widetilde{\iota}}} a_{\left(\widetilde{\iota}, j, k \right)}\right]+ \\+ \sqrt{\widetilde{e}_{\widetilde{\iota}}}\widetilde{l}\left(\widetilde{u}^j\widetilde{v}^k \right) a_{\left(\widetilde{\iota}, j, k \right)}\otimes \left[ \slashed D, \sqrt{{e}_{\widetilde{\iota}}}\right]+
%\sqrt{\widetilde{e}_{\widetilde{\iota}}}\widetilde{l}\left(\widetilde{u}^j\widetilde{v}^k \right) a_{\left(\widetilde{\iota}, j, k \right)}\otimes  \sqrt{{e}_{\widetilde{\iota}}}\left(\frac{jd_u}{N_1} +\frac{kd_v}{N_2}\right)
\eean
and taking into account
\bean
\begin{split}
	\sqrt{\widetilde{e}_{\widetilde{\iota}}}\widetilde{l}\left(\widetilde{u}^j\widetilde{v}^k \right)\otimes \left[ \slashed D, \sqrt{{e}_{\widetilde{\iota}}} b \widetilde{l}\left( u^{j'}v^{k'}\right) \right]= \sqrt{\widetilde{e}_{\widetilde{\iota}}}\widetilde{l}\left(\widetilde{u}^j\widetilde{v}^k \right) \widetilde{l}\left( u^{j'}v^{k'}\right)\otimes \left[ \slashed D, \sqrt{{e}_{\widetilde{\iota}}} b  \right]+\\+
	b\sqrt{\widetilde{e}_{\widetilde{\iota}}}\widetilde{l}\left(\widetilde{u}^j\widetilde{v}^k \right) \widetilde{l}\left( u^{j'}v^{k'}\right)\otimes\sqrt{{e}_{\widetilde{\iota}}}\left(j'd_u+k'd_v\right)\\
\end{split}
\eean

% \left[ \slashed D, \sqrt{{e}_{\widetilde{\iota}}} b  \right]+ b\left[ \slashed D, \sqrt{{e}_{\widetilde{\iota}}}   \right]= b\left[ \slashed D, \sqrt{{e}_{\widetilde{\iota}}}    b\right],

one has
\be\label{isospectral_basq_eqn}
\begin{split}
\nabla\left(\widetilde{l}\left(  \widetilde{f}_{\left(\widetilde{\iota}, j, k \right)}a_{\left( \widetilde{\iota}, j',k'\right) }\right)\right)   = \sqrt{\widetilde{e}_{\widetilde{\iota}}}\widetilde{l}\left(\widetilde{u}^j\widetilde{v}^k \right)\widetilde{l}\left( u^{j'}v^{k'}\right)\otimes \left[ \slashed D, \sqrt{{e}_{\widetilde{\iota}}} b  \right]+ \\ +  b\sqrt{\widetilde{e}_{\widetilde{\iota}}}\widetilde{l}\left(\widetilde{u}^j\widetilde{v}^k \right) \widetilde{l}\left( u^{j'}v^{k'}\right)\otimes \left[ \slashed D, \sqrt{{e}_{\widetilde{\iota}}}\right]+\\+b
\sqrt{\widetilde{e}_{\widetilde{\iota}}}\widetilde{l}\left(\widetilde{u}^j\widetilde{v}^k \right) \widetilde{l}\left( u^{j'}v^{k'}\right)\otimes  \sqrt{{e}_{\widetilde{\iota}}}\left(j'd_u+\frac{ jd_u}{N_1} +k'd_v + \frac{kd_v}{N_2}\right)\\
\end{split}
\ee
From \eqref{isospectral_basq_eqn}  and \eqref{isospectral_pred_eqn} it turns out that $\nabla\left(\widetilde{l}\left(  \widetilde{f}_{\left(\widetilde{\iota}, j, k \right)}a_{\left( \widetilde{\iota}, j',k'\right) }\right)\right) = \left[\widetilde{\slashed D}, \widetilde{l}\left(  \widetilde{f}_{\left(\widetilde{\iota}, j, k \right)}\right)a_{\left( \widetilde{\iota}, j',k'\right) } \right]$. Any $\widetilde{a}\in \Coo\left(\widetilde{M}_{\widetilde{\th}} \right) $ is an infinite sum of elements $\widetilde{f}_{\left(\widetilde{\iota}, j, k \right)}a_{\left( \widetilde{\iota}, j',k'\right) }$ it turns out
\be\label{isospectral_dfin_eqn}
\nabla\left(\widetilde{a} \right)= \left[\widetilde{\slashed D}, \widetilde{a}  \right]. 
\ee
Taking into account the Corollary \ref{isospectral_cor} one has the following theorem.
\begin{theorem}
The noncommutative spectral triple   $$\left( \Coo\left(\widetilde{M}_{\widetilde{\th}}\right) , L^2\left(\widetilde{M},\widetilde{S} \right), \widetilde{ \slashed D}  \right)$$ is a $\left( C\left( M_\th\right), C\left( \widetilde{M}_{ \widetilde{\th}}\right), G\left(\widetilde{M}~|~ M \right)\right)   $-lift of $\left( \Coo\left(M_\th\right) , L^2\left(M,S \right), \slashed D  \right)$. 

\end{theorem}

\subsection{Infinite coverings}
\paragraph{} Let $\mathfrak{S}_M =\left\{M = M^0 \leftarrow M^1 \leftarrow ... \leftarrow M^n \leftarrow ... \right\} \in \mathfrak{FinTop}$ be an infinite sequence of spin  - manifolds and regular finite-fold covering. Suppose that there is the action $\T^2 \times M \to M$ given by \eqref{isospectral_sym_eqn}. From the Theorem \ref{isospectral_fin_thm} it follows that there is the  algebraical  finite covering sequence
\begin{equation*}\label{isospectral_sequence_eqn}
\mathfrak{S}_{C\left(M_\th \right) }  = \left\{C\left(M_\th \right)\to ... \to C\left(M^n_{\th_n} \right)\to ...\right\}.
\end{equation*}

So one can calculate a finite noncommutative limit of the above sequence. This article does not contain detailed properties of this noncommutative limit, because it is not known yet by the author of this article.

\section*{Acknowledgment}

\paragraph*{}
I am very grateful to Prof. Joseph C Varilly  and Arup Kumar Pal
 for advising me on the properties of Moyal planes resp. equivariant spectral triples. Author would like to acknowledge members of the Moscow State University Seminar
 “Noncommutative geometry and topology” leaded by professor A. S. Mishchenko and others for a discussion
 of this work.

\end{document}